\pdfoutput=1
\documentclass[a4paper,11pt]{amsart}
\usepackage[hmarginratio={1:1},vmarginratio={1:1},lmargin=80.0pt,tmargin=90.0pt]{geometry}

\usepackage{calligra}
\usepackage{fontenc}
\usepackage[ugly]{nicefrac}

\usepackage{mathtools}
\mathtoolsset{mathic}
\usepackage[final]{microtype}
\usepackage{multicol}
\usepackage{tabularx,tabu}
\usepackage{latexsym,exscale,enumitem,amsfonts,amssymb,mathtools}
\usepackage{amsmath,amsthm,amsfonts,amssymb,amscd,textcomp,bbm}
\usepackage{stmaryrd}
\SetSymbolFont{stmry}{bold}{U}{stmry}{m}{n}
\usepackage[normalem]{ulem}
\usepackage{thmtools}
\usepackage{etex}
\usepackage{fancybox}
\usepackage[table]{xcolor}
\usepackage{mathrsfs}
\DeclareMathAlphabet{\mathscrbf}{OMS}{mdugm}{b}{n}

\usepackage{caption} 
\usepackage[all]{xy}
\SelectTips{cm}{}

\definecolor{myorange}{RGB}{240,127,0}
\definecolor{mygreen}{RGB}{0,240,0}
\definecolor{mypurple}{RGB}{140,0,140}
\definecolor{myred}{RGB}{255,0,0}
\definecolor{myblue}{RGB}{0,0,210}
\definecolor{myyellow}{RGB}{210,210,0}
\definecolor{mycream}{RGB}{200,200,150}

\definecolor{dummy}{RGB}{10,10,10}
\definecolor{mygray}{gray}{0.7}

\definecolor{orchid}{RGB}{143,40,194}
\definecolor{lava}{RGB}{207,16,32}

\definecolor{mydarkblue}{RGB}{10,10,170}

\usepackage{tikz}
\tikzset{anchorbase/.style={baseline={([yshift=-0.5ex]current bounding box.center)}},
  tinynodes/.style={font=\tiny,text height=0.75ex,text depth=0.15ex},
  bstrand/.style={line width=1.5, color=myblue},
  rstrand/.style={line width=1.5, color=myred},
  ystrand/.style={line width=1.5, color=myyellow},
  dstrand/.style={line width=1.5, color=black},
}
\usetikzlibrary{cd}
\usetikzlibrary{decorations}
\usetikzlibrary{decorations.markings}
\usetikzlibrary{decorations.pathreplacing}
\usetikzlibrary{decorations.pathmorphing}
\usetikzlibrary{shapes,positioning,matrix,calc}
\usetikzlibrary{shapes.callouts}
\tikzstyle directed=[postaction={decorate,decoration={markings,
    mark=at position #1 with {\arrow[draw=black, line width=0.3mm]{>}}}}]
\tikzstyle rdirected=[postaction={decorate,decoration={markings,
    mark=at position #1 with {\arrow[draw=black, line width=0.3mm]{<}}}}]
\tikzstyle ddirected=[postaction={decorate,decoration={markings,
    mark=at position #1 with {\arrow[draw=black, line width=0.3mm]{<>}}}}]
\tikzstyle{snakeline} = [decorate, decoration={pre length=0.2cm,
                         post length=0.2cm, snake, amplitude=.4mm,
                         segment length=2mm},thick]
\tikzstyle marked=[postaction={decorate,decoration={markings,
    mark=at position #1 with {\fill[mygray] (.5pt,2pt)--(-.5pt,2pt)--(-.5pt,-2pt)--(.5pt,-2pt)--cycle;}}}]                         
\tikzstyle Xmarked=[postaction={decorate,decoration={markings,
    mark=at position #1 with {\arrow[line width=0.25mm, black]{>}}}}]
\tikzstyle Ymarked=[postaction={decorate,decoration={markings,
    mark=at position #1 with {\arrow[line width=0.25mm, black]{<}}}}]
\tikzstyle crossmarked=[postaction={decorate,decoration={markings,
    mark=at position #1 with {\draw[ultra thin, white, fill=white] (0,0) circle (.175cm);}}}]
\tikzstyle crossmarkedd=[postaction={decorate,decoration={markings,
    mark=at position #1 with {\draw[ultra thin, white, fill=white] (0,0) circle (.1cm);}}}]

\makeatletter
\newcommand{\setword}[2]{%
  \phantomsection
  #1\def\@currentlabel{\unexpanded{#1}}\label{#2}%
}
\makeatother

\newcommand{\neatfrac}[2]{\tfrac{#1}{#2}}
\newcommand{\sneatfrac}[2]{\scalebox{.75}{$\tfrac{#1}{#2}$}}

\newcommand{\placeholder}{\underline{\phantom{a}}}

\newcommand{\bluebox}{\,\tikz[baseline=-.05,scale=0.25]{\draw[myblue,fill=myblue] (0,0) rectangle (1,1);}\,}
\newcommand{\redbox}{\,\tikz[baseline=-.05,scale=0.25]{\draw[myred,fill=myred] (0,0) rectangle (1,1);}\,}
\newcommand{\yellowbox}{\,\tikz[baseline=-.05,scale=0.25]{\draw[myyellow,fill=myyellow] (0,0) rectangle (1,1);}\,}

\newcommand{\greenbox}{\,\tikz[baseline=-.05,scale=0.25]{\draw[mygreen,fill=mygreen] (0,0) rectangle (1,1);}\,}
\newcommand{\orangebox}{\,\tikz[baseline=-.05,scale=0.25]{\draw[myorange,fill=myorange] (0,0) rectangle (1,1);}\,}
\newcommand{\purplebox}{\,\tikz[baseline=-.05,scale=0.25]{\draw[mypurple,fill=mypurple] (0,0) rectangle (1,1);}\,}

\newcommand{\C}{\mathbb{C}}
\newcommand{\Z}{\mathbb{Z}}
\newcommand{\R}{\mathbb{R}}

\newcommand{\N}{\mathbb{N}}

\newcommand{\zeetwo}{\Z/2\Z}
\newcommand{\zeethree}{\Z/3\Z}

\newcommand{\varstuff}[1]{\mathtt{#1}}
\newcommand{\algstuff}[1]{\mathrm{#1}}
\newcommand{\catstuff}[1]{\mathcal{#1}}
\newcommand{\twocatstuff}[1]{\mathscrbf{#1}}

\newcommand{\obstuff}[1]{\mathtt{#1}}
\newcommand{\morstuff}[1]{\mathrm{#1}}
\newcommand{\twomorstuff}[1]{\mathsf{#1}}

\newcommand{\modcat}[1]{\catstuff{M}\mathrm{od}_{#1}}
\newcommand{\modtwocat}[1]{\twocatstuff{M}\mathrm{od}_{#1}}

\newcommand{\K}{\algstuff{R}}

\newcommand{\iunit}{\mathsf{i}}
\newcommand{\zetaroot}{\mathtt{\zeta}}

\newcommand{\vpar}{\varstuff{v}}
\newcommand{\intvpar}{\scalebox{.6}{$[\vpar]$}}
\newcommand{\Cv}{\C_{\vpar}}
\newcommand{\vnumber}[1]{[#1]_{\vpar}}
\newcommand{\vfrac}[1]{[#1]_{\vpar}!\,}
\newcommand{\aformv}{\Z_{\intvpar}}
\newcommand{\aformvN}{\N_{\intvpar}}
\newcommand{\vbin}[2]{{\textstyle\genfrac{[}{]}{0pt}{}{#1}{#2}}_{\vpar}}
\newcommand{\vbinn}[2]{\genfrac{[}{]}{0pt}{}{#1}{#2}_{\vpar}}

\newcommand{\qpar}{\varstuff{q}}
\newcommand{\intqpar}{\scalebox{.6}{$[\qpar]$}}
\newcommand{\qqpar}{\varstuff{\eta}}
\newcommand{\Cq}[1][\qpar]{\C_{#1}}

\newcommand{\qnumber}[1]{[#1]_{\qpar}}

\newcommand{\aformq}{\C_{\intqpar}}

\newcommand{\qqnumber}[1]{[#1]_{\qqpar}}

\newcommand{\qbinq}[2]{{\textstyle\genfrac{[}{]}{0pt}{}{#1}{#2}}_{\qqpar}}

\newcommand{\typeA}{\mathsf{A}}
\newcommand{\typeD}{\mathsf{D}}
\newcommand{\typeE}{\mathsf{E}}

\newcommand{\typea}[1]{\mathsf{A}_{#1}}

\newcommand{\typeat}[1]{\tilde{\mathsf{A}}_{#1}}
\newcommand{\typedt}[1]{\tilde{\mathsf{D}}_{#1}}
\newcommand{\typeet}[1]{\tilde{\mathsf{E}}_{#1}}

\newcommand{\typei}[1][e+2]{\mathsf{I}_{2}(#1)}

\DeclareRobustCommand{\ADE}{\mathsf{ADE}}

\DeclareRobustCommand{\graphA}[1]{\boldsymbol{\mathsf{A}}_{#1}}
\DeclareRobustCommand{\graphD}[1]{\boldsymbol{\mathsf{D}}_{#1}}
\DeclareRobustCommand{\graphC}[1]{\boldsymbol{\mathsf{cA}}_{#1}}
\DeclareRobustCommand{\graphE}[1]{\boldsymbol{\mathsf{E}}_{#1}}

\newcommand{\bc}{{\color{myblue}b}}
\newcommand{\rc}{{\color{myred}r}}
\newcommand{\yc}{{\color{myyellow}y}}

\newcommand{\gc}{{\color{mygreen}g}}
\newcommand{\oc}{{\color{myorange}o}}
\newcommand{\pc}{{\color{mypurple}p}}

\newcommand{\bcblack}{{\color{black}b}}

\newcommand{\ycblack}{{\color{black}y}}

\newcommand{\gcblack}{{\color{black}g}}

\newcommand{\wc}{\emptyset}

\newcommand{\duc}{{\color{dummy}c}}
\newcommand{\dudc}{{\color{dummy}d}}

\newcommand{\tduc}{{\color{dummy}\textbf{u}}}
\newcommand{\tdudc}{{\color{dummy}\textbf{v}}}

\newcommand{\Gg}{\boldsymbol{\Gamma}}

\newcommand{\oGg}{\Gg}

\newcommand{\bulletg}{\scalebox{1.05}{\text{{\color{mygreen}$\bullet$}}}}
\newcommand{\bulleto}{\scalebox{.7}{\text{{\color{myorange}$\blacksquare$}}}}
\newcommand{\bulletp}{\scalebox{.7}{\text{{\color{mypurple}$\blacklozenge$}}}}

\newcommand{\bulletb}{\scalebox{.7}{\text{{\color{myblue}$\blacktriangledown$}}}}
\newcommand{\bulletr}{\scalebox{.7}{\text{{\color{myred}$\bigstar$}}}}
\newcommand{\bullety}{\scalebox{.7}{\text{{\color{myyellow}$\blacktriangle$}}}}

\newcommand{\bulletgstart}{\scalebox{1.5}{\text{{\color{mygreen}$\star$}}}}

\newcommand{\Gset}{{\color{mygreen}G}}
\newcommand{\Oset}{{\color{myorange}O}}
\newcommand{\Pset}{{\color{mypurple}P}}

\newcommand{\Bset}{{\color{myblue}B}}
\newcommand{\Rset}{{\color{myred}R}}
\newcommand{\Yset}{{\color{myyellow}Y}}

\newcommand{\Prset}{\Bset\Rset\Yset}
\newcommand{\Seset}{\Gset\Oset\Pset}

\newcommand{\zig}[1][e]{\nabla_{#1}}

\newcommand{\zigmod}[1][e]{\zig[#1]\text{-}\mathrm{p}\catstuff{M}\mathrm{od}^{\mathrm{gr}}}

\newcommand{\circled}[1]{\tikz{
	\draw[thin,fill=mycream] (0,0) circle (.35cm);
    \node at (0,-.1) {#1};}}

\newcommand{\dvert}[1]{\circled{${\color{black}\mathtt{#1}}$}}
\newcommand{\somevert}[1]{\mathtt{#1}}

\newcommand{\pathx}[2]{\mathtt{#2}\vert\mathtt{#1}}
\newcommand{\pathy}[2]{\mathtt{#2}\vert\mathtt{#1}}

\newcommand{\pathxx}[3]{\mathtt{#3}\vert\mathtt{#2}\vert\mathtt{#1}}

\newcommand{\pathxxx}[4]{\mathtt{#4}\vert\mathtt{#3}\vert\mathtt{#2}\vert\mathtt{#1}}

\newcommand{\pathxxxx}[5]{\mathtt{#5}\vert\mathtt{#4}\vert\mathtt{#3}\vert\mathtt{#2}\vert\mathtt{#1}}

\newcommand{\loopy}[1]{\alpha_{\mathtt{#1}}}

\newcommand{\lpro}[1][\somevert{i}_{m,n}]{\algstuff{P}_{#1}}
\newcommand{\rpro}[1][\somevert{i}_{m,n}]{{}_{#1}\algstuff{P}}

\newcommand{\thetaf}[1]{\boldsymbol{\Theta}_{#1}}

\newcommand{\zigideal}[1]{\algstuff{K}_{#1}}

\newcommand{\hecke}[1][\vpar]{\algstuff{H}_{#1}}
\newcommand{\subquo}[1][\infty]{\algstuff{T}_{#1}}

\newcommand{\dihquo}[1][\infty]{\algstuff{D}_{#1}}

\newcommand{\Wgroup}{\algstuff{W}}

\newcommand{\rklg}[1]{\algstuff{h}\hspace*{-.015cm}{}^{#1}_{\gc}}
\newcommand{\rklo}[1]{\algstuff{h}\hspace*{-.015cm}{}^{#1}_{\oc}}
\newcommand{\rklp}[1]{\algstuff{h}\hspace*{-.015cm}{}^{#1}_{\pc}}
\newcommand{\rklx}[1]{\algstuff{h}\hspace*{-.015cm}{}^{#1}_{\tduc}}

\newcommand{\RKLg}[1]{\algstuff{c}\hspace*{-.015cm}{}^{#1}_{\gc}}
\newcommand{\RKLo}[1]{\algstuff{c}\hspace*{-.015cm}{}^{#1}_{\oc}}
\newcommand{\RKLp}[1]{\algstuff{c}\hspace*{-.015cm}{}^{#1}_{\pc}}
\newcommand{\RKLx}[1]{\algstuff{c}\hspace*{-.015cm}{}^{#1}_{\tduc}}
\newcommand{\RKLy}[1]{\algstuff{c}\hspace*{-.015cm}{}^{#1}_{\tdudc}}

\newcommand{\klx}[1]{{}^{#1}\hspace*{-.2cm}{}_{\tduc}\algstuff{h}}
\newcommand{\kly}[1]{{}^{#1}\hspace*{-.2cm}{}_{\tdudc}\algstuff{h}}

\newcommand{\KLx}[1]{{}^{#1}\hspace*{-.2cm}{}_{\tduc}\algstuff{c}}
\newcommand{\KLy}[1]{{}^{#1}\hspace*{-.2cm}{}_{\tdudc}\algstuff{c}}

\newcommand{\basisH}[1][\infty]{\algstuff{H}\hspace*{-.015cm}{}^{#1}}
\newcommand{\Hbasis}[1][\infty]{{}^{#1}\algstuff{H}}
\newcommand{\basisC}[1][\infty]{\algstuff{C}\hspace*{-.015cm}{}^{#1}}
\newcommand{\Cbasis}[1][\infty]{{}^{#1}\algstuff{C}}

\newcommand{\killideal}[1]{\algstuff{I}_{#1}}
\newcommand{\vanideal}[1]{\algstuff{J}_{#1}}
\newcommand{\vanset}[1]{\algstuff{V}_{#1}}

\newcommand{\posalg}{\algstuff{P}}
\newcommand{\posmod}{\algstuff{M}}
\newcommand{\posbasis}[1][\posalg]{\algstuff{B}^{#1}}
\newcommand{\posmodbasis}[1][\posmod]{\algstuff{B}^{#1}}

\newcommand{\lcell}{\mathsf{L}}
\newcommand{\rcell}{\mathsf{R}}
\newcommand{\tcell}{\mathsf{J}}

\newcommand{\Lcell}{\mathrm{L}}
\newcommand{\Rcell}{\mathrm{R}}
\newcommand{\Tcell}{\mathrm{J}}

\newcommand{\lgeq}{\geq_{\Lcell}}
\newcommand{\rgeq}{\geq_{\Rcell}}
\newcommand{\tgeq}{\geq_{\Tcell}}

\newcommand{\lgeqs}{>_{\Lcell}}

\newcommand{\lsim}{\sim_{\Lcell}}
\newcommand{\rsim}{\sim_{\Rcell}}
\newcommand{\tsim}{\sim_{\Tcell}}

\newcommand{\qDema}{\partial}

\newcommand{\fdual}{\star}

\newcommand{\rootb}{\alpha_{\bc}}
\newcommand{\rootr}{\alpha_{\rc}}
\newcommand{\rooty}{\alpha_{\yc}}
\newcommand{\rootdu}{\alpha_{\duc}}

\newcommand{\frobel}[2]{\mu_{#1}^{#2}}
\newcommand{\Frobel}[2]{\Delta_{#1}^{#2}}

\newcommand{\somebasis}{\algstuff{B}}

\newcommand{\dubasis}{\algstuff{B}_{\duc}}
\newcommand{\dutbasis}{\algstuff{B}_{\tduc}}
\newcommand{\dubcbasis}{\algstuff{B}^{\duc}_{\tduc}}
\newcommand{\dubasisd}{\algstuff{B}_{\duc}^{\fdual}}
\newcommand{\dutbasisd}{\algstuff{B}_{\tduc}^{\fdual}}
\newcommand{\dubcbasisd}{(\algstuff{B}^{\duc}_{\tduc})^\fdual}

\newcommand{\slt}{\mathfrak{sl}_{3}}
\newcommand{\Uslt}{\algstuff{U}_{\qpar}(\mathfrak{sl}_{3})}
\newcommand{\Uqslt}{\algstuff{U}_{\qqpar}(\mathfrak{sl}_{3})}

\newcommand{\Ll}{\algstuff{L}}

\newcommand{\fu}{\varstuff{X}}
\newcommand{\fud}{\varstuff{Y}}

\newcommand{\sltcatpre}[1][\qpar]{\algstuff{U}_{#1}(\mathfrak{sl}_{3})\text{-}\catstuff{M}\mathrm{od}}
\newcommand{\sltcat}[1][\qpar]{\catstuff{Q}_{#1}}
\newcommand{\slqmod}[1][e]{\catstuff{Q}_{#1}}

\newcommand{\slqmodgop}[1][e]{\twocatstuff{Q}^{\Seset}_{#1}}

\newcommand{\sltcatgop}[1][\qpar]{\twocatstuff{Q}^{\Seset}_{#1}}

\newcommand{\fuf}[2]{{}_{#2}\morstuff{X}_{#1}}
\newcommand{\fudf}[2]{{}_{#2}\morstuff{Y}_{#1}}

\newcommand{\qdim}[1][\qqpar]{#1\mathrm{dim}}

\newcommand{\sltwo}{\mathfrak{sl}_{2}}

\newcommand{\pxy}[2][\fu,\fud]{\mathsf{U}_{#2}(#1)}

\newcommand{\amod}[1][M]{\algstuff{#1}}

\newcommand{\M}{\algstuff{M}}
\newcommand{\Mt}[1][z]{\M_{#1}^{\mathrm{tot}}}
\newcommand{\Mg}[1][z]{\M_{#1}(\gc)}
\newcommand{\Mo}[1][z]{\M_{#1}(\oc)}
\newcommand{\Mp}[1][z]{\M_{#1}(\pc)}

\newcommand{\cM}{\twocatstuff{M}}

\newcommand{\varx}{\varstuff{x}}
\newcommand{\vary}{\varstuff{y}}

\newcommand{\deltoid}{\mathsf{d}}
\newcommand{\discoid}{\mathsf{d}_{3}}

\newcommand{\Idmatrix}{\mathrm{Id}}

\newcommand{\claspideal}[1][e]{\twocatstuff{I}_{#1}}

\newcommand{\Rbim}[1][\intqpar]{\algstuff{R}_{#1}}

\newcommand{\elfunctor}[1][\qpar]{\twomorstuff{S}_{#1}}

\newcommand{\ADiag}[1][\intqpar]{\boldsymbol{s}\twocatstuff{BS}_{#1}^{\ast}}
\newcommand{\Adiag}[1][\intqpar]{\boldsymbol{s}\twocatstuff{BS}_{#1}}

\newcommand{\adiag}[1][\intqpar]{\twocatstuff{BS}_{#1}}

\newcommand{\aDiag}[1][\intqpar]{\boldsymbol{m}\twocatstuff{BS}_{#1}}

\newcommand{\subcatquo}[1][{\infty,\intqpar}]{\twocatstuff{T}_{#1}}
\newcommand{\Subcatquo}[1][{\infty,\intqpar}]{\boldsymbol{m}\twocatstuff{T}_{#1}}

\newcommand{\Kar}[1]{\twocatstuff{K}\!\mathrm{ar}(#1)}

\newcommand{\polybox}[1]{\fcolorbox{black}{mygray!30}{\color{black}$#1$}}

\newcommand{\cRKLg}[1]{\twomorstuff{p}^{#1}_{\gc}}
\newcommand{\cRKLo}[1]{\twomorstuff{p}^{#1}_{\oc}}
\newcommand{\cRKLp}[1]{\twomorstuff{p}^{#1}_{\pc}}
\newcommand{\cRKLx}[1]{\twomorstuff{p}^{#1}_{\tduc}}

\newcommand{\CRKLg}[1]{\twomorstuff{c}^{#1}_{\gc}}

\newcommand{\CRKLx}[1]{\twomorstuff{c}^{#1}_{\tduc}}

\newcommand{\cKLg}[1]{{}^{#1}\hspace*{-.2cm}{}_{\gc}\twomorstuff{p}}
\newcommand{\cKLo}[1]{{}^{#1}\hspace*{-.2cm}{}_{\oc}\twomorstuff{p}}
\newcommand{\cKLp}[1]{{}^{#1}\hspace*{-.2cm}{}_{\pc}\twomorstuff{p}}
\newcommand{\cKLx}[1]{{}^{#1}\hspace*{-.2cm}{}_{\tduc}\twomorstuff{p}}

\newcommand{\CKLx}[1]{{}^{#1}\hspace*{-.2cm}{}_{\tduc}\twomorstuff{c}}

\newcommand{\adiagfield}{\adiag[\qpar]}
\newcommand{\Adiagfield}{\Adiag[\qpar]}
\newcommand{\aDiagfield}{\aDiag[\qpar]}
\newcommand{\subcatquofield}{\subcatquo[\infty]}

\newcommand{\inte}{\scalebox{.6}{$[e]$}}
\newcommand{\aforme}{\C_{\inte}}
\newcommand{\sltcatefield}{\twocatstuff{Q}_{\inte}^{\Seset}}
\newcommand{\adiagefield}{\adiag[\inte]}

\newcommand{\subcatquoefield}{\subcatquo[\inte]}
\newcommand{\Subcatquoefield}{\Subcatquo[\inte]}
\newcommand{\elfunctorefield}{\twomorstuff{S}_{\inte}}

\newcommand{\claspidealefielde}{\claspideal[e]}

\newcommand{\claspidealgope}{\twocatstuff{H}_{e}}

\newcommand{\fincat}{\twocatstuff{A}^{\mathrm{f}}_{\mathrm{gr}}}

\newcommand{\GG}[1]{[#1]_{\oplus}}
\newcommand{\GGc}[1]{[#1]_{\oplus}^{\C}}
\newcommand{\GGcv}[1]{[#1]_{\oplus}^{\Cv}}
\newcommand{\End}{\algstuff{E}\mathrm{nd}}
\newcommand{\Hom}{\algstuff{H}\mathrm{om}}

\newcommand{\twoEnd}{\twocatstuff{E}\mathrm{nd}}
\newcommand{\twoHom}{\twocatstuff{H}\mathrm{om}}

\newcommand{\vcomp}{\circ_{v}}
\newcommand{\hcomp}{\circ_{h}}

\theoremstyle{definition}
\newtheorem{theoremm}{Theorem}[section]

\declaretheorem[style=definition,name=Theorem,qed=$\square$,numberlike=theoremm]{theorem}
\declaretheorem[style=definition,name=Theorem,qed=$\blacksquare$,numberlike=theoremm]{theoremqed}
\declaretheorem[style=definition,name=Corollary,qed=$\blacksquare$,numberlike=theoremm]{corollary}
\declaretheorem[style=definition,name=Lemma,qed=$\square$,numberlike=theoremm]{lemma}
\declaretheorem[style=definition,name=Lemma,qed=$\blacksquare$,numberlike=theoremm]{lemmaqed}

\declaretheorem[style=definition,name=Proposition,qed=$\square$,numberlike=theoremm]{proposition}
\declaretheorem[style=definition,name=Proposition,qed=$\blacksquare$,numberlike=theoremm]{propositionqed}

\declaretheorem[style=definition,name=Remark about colors,qed=$\blacktriangle$,numbered=no]{remarkcolor}
\declaretheorem[style=definition,name=Quantum conventions,qed=$\blacktriangle$,numbered=no]{qconventions}

\declaretheorem[style=definition,name=Example,qed=$\blacktriangle$,numberlike=theorem]{example}
\declaretheorem[style=definition,name=Definition,qed=$\blacktriangle$,numberlike=theorem]{definition}
\declaretheorem[style=definition,name=Definition,numberlike=theorem]{definitionnoqed}
\declaretheorem[style=definition,name=Remark,qed=$\blacktriangle$,numberlike=theorem]{remark}
\declaretheorem[style=definition,name=Convention,qed=$\blacktriangle$,numberlike=theorem]{convention}

\declaretheorem[style=definition,name=The dihedral story,qed=$\blacktriangle$,numberlike=theorem]{dihedral}
\declaretheorem[style=definition,name=CP,qed=$\blacktriangle$,numberlike=theorem]{problem}

\def\notation#1#2#3{\rlap{\hyperref[#1]{{\color{orchid}#2}}}\hspace*{8.2mm} \hbox to 47mm{#3\hfill}}

\newcommand{\makeqedtri}{\hfill\ensuremath{\blacktriangle}}

\allowdisplaybreaks

\numberwithin{equation}{section}
\usepackage[hypertexnames=false]{hyperref}
\usepackage{cleveref,bookmark}
\hypersetup{
    pdftoolbar=true,        
    pdfmenubar=true,        
    pdffitwindow=false,     
    pdfstartview={FitH},    
    pdftitle={Trihedral Soergel bimodules},    
    pdfauthor={Marco Mackaay, Volodymyr Mazorchuk, Vanessa Miemietz and Daniel Tubbenhauer},     
    pdfsubject={},   
    pdfcreator={Marco Mackaay, Volodymyr Mazorchuk, Vanessa Miemietz and Daniel Tubbenhauer},   
    pdfproducer={Marco Mackaay, Volodymyr Mazorchuk, Vanessa Miemietz and Daniel Tubbenhauer}, 
    pdfkeywords={}, 
    pdfnewwindow=true,      
    colorlinks=true,       
    linkcolor=mydarkblue,          
    citecolor=teal,        
    filecolor=magenta,      
    urlcolor=orchid,          
    linkbordercolor=lava,
    citebordercolor=teal,
    urlbordercolor=orchid,  
    linktocpage=true
}
\let\fullref\autoref
\def\makeautorefname#1#2{\expandafter\def\csname#1autorefname\endcsname{#2}}
\makeautorefname{equation}{Equation}%
\makeautorefname{footnote}{footnote}%
\makeautorefname{item}{item}%
\makeautorefname{figure}{Figure}%
\makeautorefname{table}{Table}%
\makeautorefname{part}{Part}%
\makeautorefname{appendix}{Appendix}%
\makeautorefname{chapter}{Chapter}%
\makeautorefname{section}{Section}%
\makeautorefname{subsection}{Section}%
\makeautorefname{subsubsection}{Section}%
\makeautorefname{paragraph}{Paragraph}%
\makeautorefname{subparagraph}{Paragraph}%
\makeautorefname{theorem}{Theorem}%
\makeautorefname{theo}{Theorem}%
\makeautorefname{thm}{Theorem}%
\makeautorefname{addendum}{Addendum}%
\makeautorefname{addend}{Addendum}%
\makeautorefname{add}{Addendum}%
\makeautorefname{maintheorem}{Main theorem}%
\makeautorefname{mainthm}{Main theorem}%
\makeautorefname{corollary}{Corollary}%
\makeautorefname{corol}{Corollary}%
\makeautorefname{coro}{Corollary}%
\makeautorefname{cor}{Corollary}%
\makeautorefname{lemma}{Lemma}%
\makeautorefname{lemm}{Lemma}%
\makeautorefname{lem}{Lemma}%
\makeautorefname{sublemma}{Sublemma}%
\makeautorefname{sublem}{Sublemma}%
\makeautorefname{subl}{Sublemma}%
\makeautorefname{proposition}{Proposition}%
\makeautorefname{proposit}{Proposition}%
\makeautorefname{propos}{Proposition}%
\makeautorefname{propo}{Proposition}%
\makeautorefname{prop}{Proposition}%
\makeautorefname{property}{Property}
\makeautorefname{proper}{Property}
\makeautorefname{scholium}{Scholium}%
\makeautorefname{step}{Step}%
\makeautorefname{conjecture}{Conjecture}%
\makeautorefname{conject}{Conjecture}%
\makeautorefname{conj}{Conjecture}%
\makeautorefname{question}{Question}
\makeautorefname{questn}{Question}
\makeautorefname{quest}{Question}
\makeautorefname{ques}{Question}
\makeautorefname{qn}{Question}
\makeautorefname{definition}{Definition}%
\makeautorefname{defin}{Definition}%
\makeautorefname{defi}{Definition}%
\makeautorefname{def}{Definition}%
\makeautorefname{dfn}{Definition}%
\makeautorefname{notation}{Notation}
\makeautorefname{nota}{Notation}
\makeautorefname{notn}{Notation}
\makeautorefname{remark}{Remark}%
\makeautorefname{rema}{Remark}%
\makeautorefname{rem}{Remark}%
\makeautorefname{rmk}{Remark}%
\makeautorefname{rk}{Remark}%
\makeautorefname{remarks}{Remarks}%
\makeautorefname{rems}{Remarks}%
\makeautorefname{rmks}{Remarks}%
\makeautorefname{rks}{Remarks}%
\makeautorefname{example}{Example}%
\makeautorefname{examp}{Example}%
\makeautorefname{exmp}{Example}%
\makeautorefname{exam}{Example}%
\makeautorefname{exa}{Example}%
\makeautorefname{algorithm}{Algorith}%
\makeautorefname{algo}{Algorith}%
\makeautorefname{alg}{Algorith}%
\makeautorefname{axiom}{Axiom}%
\makeautorefname{axi}{Axiom}%
\makeautorefname{ax}{Axiom}%
\makeautorefname{case}{Case}%
\makeautorefname{claim}{Claim}%
\makeautorefname{clm}{Claim}%
\makeautorefname{assumption}{Assumption}%
\makeautorefname{assumpt}{Assumption}%
\makeautorefname{conclusion}{Conclusion}%
\makeautorefname{concl}{Conclusion}%
\makeautorefname{conc}{Conclusion}%
\makeautorefname{condition}{Condition}%
\makeautorefname{condit}{Condition}%
\makeautorefname{cond}{Condition}%
\makeautorefname{construction}{Construction}%
\makeautorefname{construct}{Construction}%
\makeautorefname{const}{Construction}%
\makeautorefname{cons}{Construction}%
\makeautorefname{criterion}{Criterion}%
\makeautorefname{criter}{Criterion}%
\makeautorefname{crit}{Criterion}%
\makeautorefname{exercise}{Exercise}%
\makeautorefname{exer}{Exercise}%
\makeautorefname{exe}{Exercise}%
\makeautorefname{problem}{Problem}%
\makeautorefname{problm}{Problem}%
\makeautorefname{probm}{Problem}%
\makeautorefname{prob}{Problem}%
\makeautorefname{solution}{Solution}%
\makeautorefname{soln}{Solution}%
\makeautorefname{sol}{Solution}%
\makeautorefname{summary}{Summary}%
\makeautorefname{summ}{Summary}%
\makeautorefname{sum}{Summary}%
\makeautorefname{operation}{Operation}%
\makeautorefname{oper}{Operation}%
\makeautorefname{observation}{Observation}%
\makeautorefname{observn}{Observation}%
\makeautorefname{obser}{Observation}%
\makeautorefname{obs}{Observation}%
\makeautorefname{ob}{Observation}%
\makeautorefname{convention}{Convention}%
\makeautorefname{convent}{Convention}%
\makeautorefname{conv}{Convention}%
\makeautorefname{cvn}{Convention}%
\makeautorefname{warning}{Warning}%
\makeautorefname{warn}{Warning}%
\makeautorefname{note}{Note}%
\makeautorefname{fact}{Fact}%

\makeautorefname{problem}{CP}%
\makeautorefname{gproblem}{GP}%
\makeautorefname{dihedral}{`The dihedral story'}%

\begin{document}
\vbadness=10001
\hbadness=10001
\title[Trihedral Soergel bimodules]{Trihedral Soergel bimodules}
\author[M. Mackaay, V. Mazorchuk, V. Miemietz and D. Tubbenhauer]{Marco Mackaay, Volodymyr Mazorchuk, \\ Vanessa Miemietz and Daniel Tubbenhauer}

\address{M.M.: Center for Mathematical Analysis, Geometry, and Dynamical Systems,
Departamento de Matem\'{a}tica, Instituto Superior T\'{e}cnico, 1049-001 Lisboa, Portugal \& Departamento de Matem\'{a}tica, FCT, Universidade do Algarve, Campus de Gambelas, 8005-139 Faro, Portugal}
\email{mmackaay@ualg.pt}

\address{Vo.Ma.: Department of Mathematics, Uppsala University, Box. 480, SE-75106, Uppsala, Sweden}
\email{mazor@math.uu.se}

\address{Va.Mi.: School of Mathematics, University of East Anglia, Norwich NR4 7TJ, United Kingdom}
\email{v.miemietz@uea.ac.uk}

\address{D.T.: Institut f\"ur Mathematik, Universit\"at Z\"urich, Winterthurerstrasse 190, Campus Irchel, Office Y27J32, CH-8057 Z\"urich, Switzerland, \href{www.dtubbenhauer.com}{www.dtubbenhauer.com}}
\email{daniel.tubbenhauer@math.uzh.ch}

\begin{abstract}
The quantum Satake correspondence relates dihedral Soergel bimodules to the semisimple 
quotient of the quantum $\mathfrak{sl}_2$ representation category. It also establishes a 
precise relation between the simple transitive $2$-representations of both monoidal categories, which 
are indexed by bicolored $\mathsf{ADE}$ Dynkin diagrams.

Using the quantum Satake correspondence 
between affine $\mathsf{A}_{2}$ Soergel bimodules 
and the semisimple quotient of the quantum $\mathfrak{sl}_3$ representation 
category, we introduce trihedral Hecke algebras and Soergel bimodules, 
generalizing dihedral Hecke algebras and Soergel bimodules.
These have their own Kazhdan--Lusztig combinatorics, simple transitive 
$2$-representations corresponding to tricolored generalized $\mathsf{ADE}$ Dynkin 
diagrams.
\end{abstract}

\maketitle

\tableofcontents

\renewcommand{\theequation}{\thesection-\arabic{equation}}
\addtocontents{toc}{\protect\setcounter{tocdepth}{1}}
\section{Introduction}\label{sec:intro}

\subsection*{Non-negative integral representation theory}\label{subsec:intro-first}

In pioneering work \cite{KaLu}, Kazhdan--Lusztig defined 
their celebrated bases of Hecke algebras 
for Coxeter groups. Crucially, on these bases the structure 
constants of the algebras belong to 
$\N=\Z_{\geq 0}$. This started a program to study $\N$-algebras, 
which have a fixed basis with non-negative 
integral structure constants, see e.g. 
\cite{Lu2}, \cite{EK1}, where these algebras are called $\Z_+$-rings.

As proposed by the work of Kazhdan--Lusztig, 
for $\N$-algebras it makes sense to study and classify
$\N$-representations, i.e. representations with a fixed basis on which the 
fixed bases elements of the algebra act 
by non-negative integral matrices, see e.g. \cite{EK1}. 
The first examples are the so-called cell representations, which 
were originally defined for Hecke algebras \cite{KaLu}, but can be defined for 
all $\N$-algebras (and even $\R_{\geq 0}$-algebras, see 
\cite{KM1}). As it turns out, $\N$-representations are interesting from various 
points of view, with applications and connections to e.g. graph theory,
conformal field theory, fusion/modular tensor categories and subfactor theory.
\medskip

Categorical analogs of $\N$-algebras are 
monoidal categories, which we consider as one-object $2$-categories, 
or $2$-categories. These decategorify 
to $\N$-algebras, because the isomorphism classes 
of the indecomposable $1$-morphisms 
form naturally a $\N$-basis. For example, 
Hecke algebras of Coxeter groups are categorified by Soergel 
bimodules \cite{So0} such that indecomposable bimodules decategorify to the 
Kazhdan--Lusztig basis elements \cite{EW}. 

The categorical incarnation of 
$\N$-representation theory is $2$-re\-presentation theory. 
Any $2$-re\-presentation decategorifies 
naturally to a $\N$-representation, with the $\N$-basis given by 
the isomorphism classes of the indecomposable $1$-morphisms. However, not 
all $\N$-representations can be obtained in this way.  

In $2$-representation theory, the 
simple transitive $2$-representations play the role of the simple 
representations \cite{MM5}. Although 
their decategorifications need not be simple 
as complex representations, they are the ``simplest'' $2$-representations, 
as attested e.g. by the categorical 
Jordan--H{\"o}lder theorem \cite{MM5}. 
This naturally motivates the problem of classification of
simple transitive $2$-representations of $2$-categories. Just as the cell representations 
form a natural class of $\N$-representations of any
$\N$-algebra, cell $2$-representations form a 
natural class of simple transitive $2$-representations of any 
finitary $2$-category (i.e. $2$-categories 
with certain finiteness conditions \cite{MM1}). 
A crucial difference is that cell $2$-representations are always simple 
transitive, while cell representations are usually not simple. 
\medskip

In this paper, we restrict our attention 
to certain subquotients of the Hecke 
algebra of affine type $\typeA_2$, 
which we call \textit{trihedral Hecke algebras}, and their categorification by 
subquotients of Soergel bimodules of affine 
type $\typeA_2$, which we call 
\textit{trihedral Soergel bimodules}. 
These should have $2$-representations indexed by 
\textit{tricolored generalized 
$\ADE$ Dynkin diagrams} with \textit{trihedral zigzag algebras} making their appearance.
As we explain below, we think of these as 
rank three analogs of dihedral 
Hecke algebras, dihedral Soergel bimodules and zigzag algebras, respectively.
Finally, \cite{AT1} established a relation between dihedral
Soergel bimodules and the non-semisimple category of tilting modules
of quantum $\mathfrak{sl}_2$ at roots of unity. Based on that result and on
\cite{RiWi-tilting-p-canonical}, we expect there to be an interesting
relation between trihedral Soergel bimodules and a non-semisimple,
full subcategory of tilting modules of $\mathfrak{sl}_3$ at roots of
unity (or in prime characteristic).

\subsection*{The dihedral story}\label{subsec:intro-second}

For finite Coxeter types, the classification 
of the simple transitive $2$-re\-presentations of Soergel bimodules 
is only partially known, see e.g. \cite{KMMZ}, \cite{MMMZ}, \cite{Zi1}. 
There are two exceptions: 
For Coxeter type $\typeA$, the cell $2$-representations exhaust the 
simple transitive $2$-representations of Soergel bimodules \cite{MM5}, 
so the classification problem has been solved. 
For Coxeter type $\typei$, which is the type of the dihedral group 
with $2(e+2)$ elements, there also 
exists a complete classification of simple 
transitive $2$-representations \cite{KMMZ}, \cite{MT1} 
(for $e=10,16$ or $28$ the classification
is only known under the additional assumption of gradeability), 
which is completely different from the one for type $\typeA$. 
In this case, the simple transitive $2$-representations of rank greater than one 
are classified by bicolored $\ADE$ Dynkin diagrams, 
with the cell $2$-representations 
being the ones corresponding to Dynkin diagrams 
of type $\typeA$. The others, corresponding to Dynkin types $\typeD$ 
and $\typeE$, are not equivalent to cell 
$2$-representations and revealed interesting new features 
in $2$-representation theory, e.g. the two bicolorings of type $\typeE_7$ 
give non-equivalent $2$-representations which categorify the same $\N$-representation 
\cite{MT1}. For completeness, we remark that 
there are precisely two rank-one $2$-representations 
corresponding to the highest and the lowest two-sided cells, 
which categorify the trivial and the sign representation of the Hecke algebra. 
We note that going to the small quotient $\twocatstuff{D}_{e}$ by annihilating the highest cell 
avoids that we have to worry about the categorical analog of the trivial representation.
\medskip

This case is particularly interesting because 
of Elias' quantum Satake correspondence \cite{El2}, \cite{El1} between 
$\slqmod(\sltwo)$ and $\twocatstuff{D}_{e}$. 
Here $\slqmod(\sltwo)$ denotes the semisimple quotient of the monoidal category of 
finite-dimensional quantum $\sltwo$-modules 
(of type $1$), where the quantum parameter $\qqpar$ is a primitive, complex
$2(e+2)^{\mathrm{th}}$ root of unity. This correspondence is given by 
a nice, but slightly technical $2$-functor, 
so we omit further details at this stage.

Note that, when $\qpar$ is generic, the quantum 
Satake correspondence also exists, 
but between the whole category of finite-dimensional 
quantum $\sltwo$-modules $\slqmod[\qpar](\sltwo)$
and Soergel bimodules of the infinite 
dihedral type $\typei[\infty]$, which coincides with affine 
type $\typeA_{1}$.
 
One consequence of Elias' Satake 
correspondence is 
a precise relation between the simple 
transitive $2$-representations of 
$\slqmod(\sltwo)$ and $\twocatstuff{D}_{e}$.  
However, the corresponding $2$-representations
are not equivalent, because $\slqmod(\sltwo)$ is semisimple 
while $\twocatstuff{D}_{e}$ is not. 
\medskip

Let us explain this in a bit more detail. 
Equivalence classes of simple transitive $2$-re\-presentations of 
finitary $2$-categories (or graded versions of them) correspond 
bijectively to Morita equivalence classes of simple 
algebra $1$-morphisms in the abelianizations of these $2$-categories. 
This was initially proved for 
semisimple tensor categories \cite{Os1} and later 
generalized to certain finitary 
$2$-categories with duality \cite{MMMT1}. 
Kirillov--Ostrik \cite{K-O} classified 
the simple algebra $1$-morphisms 
in $\slqmod(\sltwo)$ up to Morita equivalence, 
under some natural assumptions, in terms of 
$\ADE$ Dynkin diagrams. From their results, 
via the quantum Satake correspondence, we can get all 
indecomposable algebra $1$-morphisms, up to Morita 
equivalence, in $\twocatstuff{D}_{e}$. (The latter is additive but 
not abelian, which is why we get 
indecomposable instead of simple algebra $1$-morphisms.)

Given an $\ADE$ Dynkin diagram $\Gg$ and the 
corresponding algebra $1$-morphism $\morstuff{A}^{\Gg}$, 
the category underlying the $2$-representation of $\slqmod(\sltwo)$ is equivalent to 
the category of $\morstuff{A}^{\Gg}$-mo\-dules 
in $\slqmod(\sltwo)$. The quiver of this category 
is trivial: its vertices coincide with those of $\Gg$, but it has no edges 
because $\slqmod(\sltwo)$ is semisimple. However, 
the quiver underlying the corresponding simple transitive 
$2$-representation of $\twocatstuff{D}_{e}$ 
is the so-called doubled quiver of type $\Gg$, 
which has two oppositely oriented edges between each pair of adjacent vertices. 
Its quiver algebra, the zigzag algebra, was for example
studied by Huerfano--Khovanov \cite{HK1}. 
It has very nice properties and shows up in 
various mathematical contexts nowadays.   
\medskip

Kirillov--Ostrik's classification can be seen 
as a quantum version of the McKay correspondence between 
finite subgroups of $\mathrm{SU}(2)$ 
and $\ADE$ Dynkin diagrams. The vertices of such a Dynkin diagram $\Gg$ 
correspond to the simple $\morstuff{A}^{\Gg}$-modules 
in $\slqmod(\sltwo)$. These module categories decategorify 
to $\N$-representations of the 
Grothendieck group of $\slqmod(\sltwo)$, the 
so-called Verlinde algebra, 
which were classified by Etingof--Khovanov \cite{EK1}. The Verlinde algebra is 
isomorphic to a polynomial algebra in one 
variable quotient by 
the ideal generated by the $(e+1)^{\mathrm{th}}$ 
Chebyshev polynomial $\pxy[\fu]{e+1}$ 
(normalized and of the second kind). Thus, 
Etingof--Khovanov basically classified all non-negative integer matrices which are 
killed by $\pxy[\fu]{e+1}$. (Note that not all of them come from 
$2$-representations of $\slqmod(\sltwo)$, 
because some correspond to 
graphs which are not Dynkin diagrams of type $\ADE$.)

Similarly, the Hecke algebra $\hecke(\typei)$ of Coxeter type $\typei$ 
can be obtained as a quotient of 
the Hecke algebra $\hecke(\typeat{1})$ of affine type 
$\typeA_{1}$, where $\vpar$ is a generic parameter 
(the decategorification of the grading within the Soergel 
$2$-category). Let $\theta_s,\theta_t$ denote 
the Kazhdan--Lusztig generators corresponding to 
the simple reflections, in 
both $\hecke(\typei)$ and $\hecke(\typeat{1})$. Furthermore, 
let $\theta_{w_{0}}$ be the Kazhdan--Lusztig 
basis element in $\hecke(\typei)$ for the longest 
word in the dihedral group. Then there are two ways to write 
$\theta_{w_{0}}$ as a linear combination 
of alternating products of $\theta_{s}$ and 
$\theta_{t}$, which only differ by the choice of the 
fixed final Kazhdan--Lusztig generator in each product. 
The coefficients in both linear 
combinations are precisely the coefficients of $\pxy[\fu]{e+1}$. 
(This observation is implicit in \cite{Lu1}.) 
Then $\hecke(\typei)$ is obtained from $\hecke(\typeat{1})$ 
by declaring both these linear combinations 
to be equal to each other. By 
declaring them to be equal to zero, 
we obtain the small quotient $\algstuff{D}_{e}$ 
of $\hecke(\typei)$, which is precisely the 
algebra that corresponds to the Verlinde algebra under 
the quantum Satake correspondence. 
Moreover, one can show that these algebras 
have a very similar $\N$-representation theory.

To conclude, one could say that Elias' quantum 
Satake correspondence \cite{El2}, \cite{El1} categorifies the relation 
between the Verlinde algebra and the small dihedral
quotient, while the results from 
\cite{KMMZ}, \cite{MT1}, \cite{MMMT1} 
categorify the relations between their $\N$-representations.

\subsection*{The trihedral story}\label{subsec:intro-third}

Now let us get to the topic of this paper. Elias 
also defined a quantum Satake correspondence between 
$\slqmod[e]=\slqmod[e](\mathfrak{sl}_3)$ and 
the $2$-category of Soergel bimodules of affine type $\typeA_{2}$ \cite{El1}. 
In this paper, we study certain 
subquotients of these Soergel bimodules, depending on a choice 
of a primitive, complex $2(e+3)^{\mathrm{th}}$ 
root of unity $\qqpar$, and their $2$-representation theory. 
Our construction uses the quantum Satake 
correspondence with $\slqmod[e]$, whose 
Grothendieck group is isomorphic to a polynomial algebra in 
two variables quotient by the ideal generated by a set 
of polynomials $\pxy{m,n}$, for $m+n=e+1, m,n\in\N$.
These polynomials were introduced to the field of orthogonal polynomials by 
Koornwinder \cite{Ko} and they generalize the Chebyshev polynomials.
To the best of our knowledge, these subquotients are new 
and have not been studied before.
\medskip

In fact, even their decategorifications, 
which are certain subquotients of the Hecke algebra $\hecke(\typeat{2})$
of affine type $\typeA_{2}$, seem to be 
new. For each $e\in\N$, we call the corresponding 
subquotient the trihedral Hecke algebra of level $e$ 
and denoted it by $\subquo[e]$. 
These algebras have their own Kazhdan--Lusztig 
combinatorics and interesting $\N$-representations. 
We see the trihedral Hecke algebras as rank 
three analogs of the small quotients of the 
dihedral Hecke algebras. 
There are many similarities, but also some differences. 
For example, as far as we can tell, 
the trihedral Hecke algebras are not deformations of any 
group algebra. But they are 
semisimple algebras and the classification of their 
irreducible representations 
runs in parallel to the analogous classification for 
dihedral Hecke algebras, and their $\N$-representation theory 
has also a very similar behavior.
\medskip

Now to the categorified story: In the trihedral case, the quantum Satake correspondence 
for $\qpar$ being generic only gives a 
$2$-subcategory of the affine type $\typeA_{2}$ Soergel $2$-category. We call this 
the $2$-category of trihedral Soergel bimodules of 
level $\infty$ and denote it by $\subcatquo[\infty]$. 
The $2$-category $\subcatquo[\infty]$ 
admits quotients 
$\subcatquo[e]$, 
the trihedral Soergel bimodules of 
level $e$, which via the 
quantum Satake correspondence for $\qqpar$ is related to $\slqmod[e]$.
The corresponding decategorifications are 
the trihedral Hecke algebras $\subquo[\infty]$ 
and $\subquo[e]$.
\medskip

Coming back to representation theory, people have studied the 
$\N$-representations of the Grothendieck group of $\slqmod[e]$, 
as they arise in conformal field theory and 
the study of fusion/modular tensor 
categories, see e.g. \cite{Ga1}, \cite{EP}, 
\cite{Sch} and related works. This time, four 
families of graphs play an important role and, by analogy with 
the $\sltwo$ case, their types are called 
$\typeA$, conjugate $\typeA$, $\typeD$ and $\typeE$, although they are 
not Dynkin diagrams. Their adjacency matrices, which 
are non-negative integral matrices, 
are annihilated by Koornwinder's polynomials, 
just as in \cite{EK1}. 
Furthermore, the type $\typeA$ graphs can be seen 
as a cut-off of the positive Weyl chamber of $\slt$, just as the usual 
type $\typeA$ Dynkin diagrams can be seen as cut-offs for $\sltwo$.
Finally, the type $\typeD$ graphs for 
$\slt$ come from a $\zeethree$-symmetry of 
these cut-offs, just as the type $\typeD$ Dynkin 
diagrams come from a $\zeetwo$-symmetry.

Simple algebra $1$-morphisms in $\slqmod[e]$ 
and the corresponding simple transitive $2$-representations 
have also been studied e.g. in \cite{Sch} 
and are closely related to these $\ADE$ 
type graphs. Via the quantum Satake correspondence, we therefore 
get indecomposable algebra $1$-mor\-phisms in $\subquo[e]$ 
and the corresponding simple transitive 
$2$-representations of the trihedral Soergel bimodules. 
Since we are not familiar with some of the 
ingredients in the construction of algebra $1$-morphisms in \cite{Sch}, we 
have given an alternative construction, using the symmetric 
$\slt$-web calculus, 
as in \cite{RT1}, \cite{TVW1}.
For this reason, our construction so 
far only works for types $\typeA$ and $\typeD$, so 
we restrict our attention to those two types. Almost by construction, the 
cell $2$-representations are equivalent to the simple transitive 
$2$-representations of type $\typeA$. The ones 
of type $\typeD$ have a different rank and are therefore inequivalent. 
For the other types, we have no 
conjectures at all, and we are not even sure whether they 
correspond to $2$-representations. 
\medskip

Computing the quiver algebras explicitly proved to be much 
harder this time. We define type $\typeA$ quiver algebras which, up to scalars, 
are the ones underlying the cell $2$-representations of $\subcatquo[e]$. 
These algebras are the trihedral analogs of the zigzag algebras
of type $\typeA$, e.g. the 
endomorphisms algebras of their vertices 
are the cohomology rings of the full flag variety of flags in
$\C^3$, instead of the flags in $\C^2$ as in the dihedral case. 
For this reason, we call them trihedral zigzag algebras. The type $\typeD$ 
trihedral zigzag algebras can be obtained 
from these by using the $\zeethree$-symmetry, just as the 
$\typeD$ dihedral zigzag algebras can be obtained from the type $\typeA$ 
via a $\zeetwo$-symmetry,
but we have not worked out the details.

Finally, let us stress that our trihedral 
zigzag algebras are different from Grant's 
\cite{Gr1} higher zigzag 
algebras, which are only subalgebras of the 
trihedral zigzag algebras of type $\typeA$, although both 
underlying graphs come from a cut-off of the positive Weyl chamber of $\slt$.

\subsection*{The Nhedral story}\label{subsec:intro-fourth}

We expect that our story generalizes to 
$\mathfrak{sl}_N$ for arbitrary $N\geq 2$: 
the Soergel bimodules of affine 
type $\typeA_{N-1}$ are known, the quantum 
Satake correspondence is conjectured to exist, 
the analogs of Koornwinder's 
Chebyshev polynomials are also known, and the corresponding 
generalized $\ADE$ type graphs appear in the mathematical 
physics literature on fusion algebras or 
the classification of subgroups of quantum $\mathrm{SU}(N)$, see e.g. \cite{DFZ}, \cite{Oc}.
We expect that there exist \textit{$N$hedral 
algebras} and \textit{$N$hedral Soergel bimodules} of level $e$ 
(where $\qqpar$ would be a primitive, complex $2(e+N)^{\mathrm{th}}$ root of unity),
and \textit{$N$hedral zigzag algebras} of $\ADE$-type quivers 
such that the endomorphism algebra 
of every vertex is isomorphic to the cohomology ring 
of the full flag variety of $\C^N$.

\begin{remarkcolor}\label{remark:colors}
We use colors in this paper (we recommend to read the paper in color), 
and the colors which we need are blue $\bluebox$, 
red $\redbox$, yellow $\yellowbox$, green $\greenbox$, orange $\orangebox$ 
and purple $\purplebox$ which will appear as indicated by the 
preceding boxes.
\end{remarkcolor}

\begin{qconventions}\label{remark:qparamters}
The notation $\vpar$ will mean a generic parameter which plays the role of 
the decategorification of the grading which we will meet in \fullref{sec:A2-diagrams}.

In contrast, $\qpar$ will also denote a generic parameter, 
but it will turn up on the categorified 
level as our quantum parameter.
Moreover, $\qqpar$ will be a primitive, complex $2(e+3)^{\mathrm{th}}$ 
root of unity $\qqpar^{2(e+3)}=1$ 
which is a specialization of $\qpar$, but never of $\vpar$. 
Here $e\in\N=\Z_{\geq 0}$ will usually be arbitrary, but fixed, and is called the level.

The ground field will always be $\Cv=\C(\vpar)$, 
$\Cq=\C(\qpar)$ or $\C=\C(\qqpar)$, if not stated otherwise. 
Sometimes, instead of working over a ground field, 
we will work over rings as e.g. $\aformv=\Z[\vpar,\vpar^{-1}]$ or 
semirings as e.g. $\aformvN=\N[\vpar,\vpar^{-1}]$ and their quantum counterparts.
(It will be clear from the notation whether we work 
with $\vpar$ or $\qpar$. Moreover, a subscript $[\placeholder]$ 
will always indicate that we are in the case of (semi)rings rather than fields.)
In this context, we use the 
$\vpar$-numbers, factorials and binomials, where $s\in\Z,t\in\Z_{\geq 1}$
\begin{gather}\label{eq:qnumbers-typeAD}
\vnumber{s} 
= 
\tfrac{\vpar^s - \vpar^{-s}}{\vpar^{\phantom{1}} - \vpar^{-1}},
\quad
\vfrac{t} 
= 
\vnumber{t} \vnumber{t{-}1} \dots \vnumber{1},
\quad
\vbinn{s}{t} 
= 
\tfrac{\vnumber{s} \vnumber{s{-}1} 
\dots 
\vnumber{s{-}t{+}1}}{\vnumber{t} \vnumber{t{-}1} \dots \vnumber{1}},
\end{gather}
all of which are in $\aformv$.
By convention, $\vnumber{0}!=1=\vbin{s}{0}$. 
Note that $\vnumber{0}=0=\vbin{0\leq s<t}{t}$ 
and $\vnumber{-s}=-\vnumber{s}$. Similarly with $\qpar$ or $\qqpar$ instead of $\vpar$.
\end{qconventions}
\addtocontents{toc}{\protect\setcounter{tocdepth}{2}}
\noindent \textbf{Acknowledgments.} We thank Ben Elias for
generously sharing his ideas and his insights into the field of 
Soergel bimodules. We also thank Rostyslav Kozhan for his help with
orthogonal polynomials, Michael Ehrig, Joseph Grant and Mikhail Khovanov for some helpful discussions, 
and the referee for an extremely careful reading of the manuscript.
D.T. likes to thank SAGE for disproving all of his conjectures related to this project.

A part of this paper was written while Va.Mi. was visiting the 
Max Planck Institute in Bonn, whose hospitality and financial support is 
gratefully acknowledged.
Another part of this paper was
written while D.T. participated in the Junior
Trimester Program ``Symplectic Geometry and 
Representation Theory'' of the Hausdorff Research Institute 
for mathematics (HIM). 
The hospitality of the HIM during this 
period is gratefully acknowledged.

M.M. was partially supported by FCT/Portugal through project UID/MAT/04459/2013,  
Vo.Ma. was supported
by the Swedish Research Council and
G{\"o}ran Gustafssons Stiftelse during this project, and 
D.T. is grateful to NCCR SwissMAP for generous support.
%
\section{Some \texorpdfstring{$\slt$}{sl3} combinatorics}\label{section:sl3-stuff}

This section is mostly a collection of known 
results, formulated in our notation.

\subsection{Quantum \texorpdfstring{$\slt$}{sl3}}\label{subsec:qslthree}

Throughout, $m,n,k,l$ will denote non-negative integers.

\subsubsection{Some conventions}\label{subsec:weights-etc}

We always use the following conventions when working with $\slt$.
Denote by 
$\varepsilon_1,\varepsilon_2,\varepsilon_3$ 
the standard basis vectors of $\R^3$. We endow 
$\R^3$ with the usual symmetric bilinear form 
$(\varepsilon_i,\varepsilon_j)=\delta_{ij}$ and 
let $E=\{(x_1,x_2,x_3)\in\R^3\mid x_1+x_2+x_3=0\}$ be the 
Euclidean subspace of $\R^3$ (with induced 
symmetric bilinear form).
We also fix two simple roots $\alpha_1=\varepsilon_1-\varepsilon_2$ 
and $\alpha_2=\varepsilon_2-\varepsilon_3$, 
and coroots $\alpha_1^\vee$ and $\alpha_2^\vee$ 
such that 
$\langle\alpha_i,\alpha_j^{\vee}
\rangle=(\alpha_i,\alpha_j)=a_{ij}$, for $i,j=1,2$,
are the entries of 
the (usual) Cartan matrix 
$\begin{psmallmatrix}
2 & -1 \\
-1 & 2
\end{psmallmatrix}$ of $\slt$.
The (integral) weights 
are $X=\{\lambda \in E\mid 
\langle\lambda,\alpha_1^{\vee}\rangle\in\Z\text{ and }
\langle\lambda,\alpha_2^{\vee}\rangle\in\Z\}$. The 
dominant (integral) weights are 
$X^+=\{\lambda \in E\mid 
\langle\lambda,\alpha_1^{\vee}\rangle\in\N\text{ and }
\langle\lambda,\alpha_2^{\vee}\rangle\in\N\}$.
We identify $X=\Z^2$ and $X^+=\N^2$, cf. \eqref{eq:weight-picture}, 
with $X^+$ also called the
positive Weyl chamber. We also use the 
fundamental weights $\omega_{1},\omega_{2}\in E$
(which are characterized by 
$\langle\omega_{i},\alpha_j^{\vee}\rangle=\delta_{i,j}$), 
and $\lambda=(m,n)\in X^+$ for us 
means $\lambda=m\omega_{1}+n\omega_{2}$.

The following picture summarizes our root and weight conventions for $e=3$:
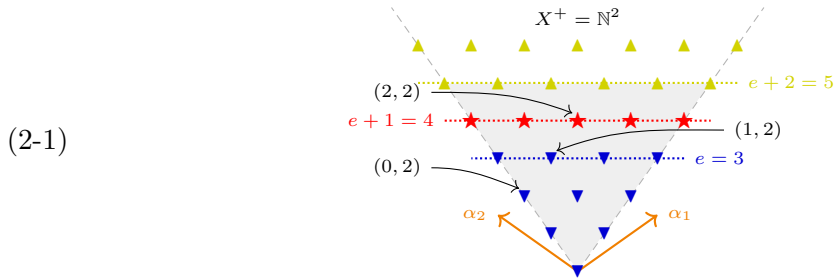
\begin{gather}\label{eq:weight-picture}
\begin{tikzpicture}[anchorbase, xscale=.35, yscale=.5]
	\draw[thin, white, fill=mygray, opacity=.2] (0,0) to (5,5) to (-5,5) to (0,0);
	\draw[thin, densely dashed, mygray] (0,0) to (7,7);
	\draw[thin, densely dashed, mygray] (0,0) to (-7,7);
	\node at (0,6.75) {\text{{\tiny $X^+=\N^{2}$}}};
	\draw[thick, densely dotted, myblue] (4,3) node [right] {\text{{\tiny $e=3$}}} to (-4,3);
	\draw[thick, densely dotted, myred] (-5,4) node [left] {\text{{\tiny $e+1=4$}}} to (5,4);
	\draw[thick, densely dotted, myyellow] (6,5) node [right] {\text{{\tiny $e+2=5$}}} to (-6,5);
	\draw[thick, myorange, ->] (0,0) to (3,1.5) node [right] {\text{{\tiny $\alpha_1$}}};
	\draw[thick, myorange, ->] (0,0) to (-3,1.5) node [left] {\text{{\tiny $\alpha_2$}}};
	\node at (0,0) {$\bulletb$};
	\node at (0,2) {$\bulletb$};
	\node at (3,3) {$\bulletb$};
	\node at (-3,3) {$\bulletb$};
	\node at (0,4) {$\bulletr$};
	\node at (1,1) {$\bulletb$};
	\node at (-2,2) {$\bulletb$};
	\node at (1,3) {$\bulletb$};
	\node at (-2,4) {$\bulletr$};
	\node at (4,4) {$\bulletr$};
	\node at (2,2) {$\bulletb$};
	\node at (-1,1) {$\bulletb$};
	\node at (-1,3) {$\bulletb$};
	\node at (2,4) {$\bulletr$};
	\node at (-4,4) {$\bulletr$};
	\node at (-5,5) {$\bullety$};
	\node at (-3,5) {$\bullety$};
	\node at (-1,5) {$\bullety$};
	\node at (1,5) {$\bullety$};
	\node at (3,5) {$\bullety$};
	\node at (5,5) {$\bullety$};
	\node at (-6,6) {$\bullety$};
	\node at (-4,6) {$\bullety$};
	\node at (-2,6) {$\bullety$};
	\node at (0,6) {$\bullety$};
	\node at (2,6) {$\bullety$};
	\node at (4,6) {$\bullety$};
	\node at (6,6) {$\bullety$};
	\draw[thin, ->] (-5.5,4.75) node [left] {\text{{\tiny $(2,2)$}}} to [out=0, in=165] (-.2,4.2);
	\draw[thin, ->] (5.5,3.75) node [right] {\text{{\tiny $(1,2)$}}} to [out=180, in=15] (-.8,3.2);
	\draw[thin, ->] (-5.5,2.75) node [left] {\text{{\tiny $(0,2)$}}} to [out=0, in=165] (-2.2,2.2);
\end{tikzpicture}
\end{gather}
We have also indicated an example of a cut-off, denoting its weights by $\bulletb$, which 
depend on the level $e$, i.e. the integral 
points $X^+(e)=\{\lambda \in X^+\mid 
\langle\lambda,\alpha^{\vee}_{1}\rangle\leq 
e\text{ and }\langle\lambda,\alpha^{\vee}_{2}\rangle\leq 
e\text{ and }\langle\lambda,\alpha^{\vee}_{1}+\alpha^{\vee}_{2}\rangle\leq e\}$. Such 
cut-offs play an important role in our paper. Moreover, we 
usually quotient by data associated to 
the line $e+1$ as illustrate by the symbols $\bulletr$ in \eqref{eq:weight-picture}.

\subsubsection{Generic quantum \texorpdfstring{$\slt$}{sl3}}\label{subsec:generic}

Let $\Uslt$ denote the quantum enveloping ($\Cq$-)algebra 
associated to $\slt$. We refer the reader to \cite[Chapters 4-7]{Ja} 
(whose conventions we silently adopt using 
the root and weight setting from above)
for details. 
We denote by $\sltcat=\sltcatpre$ the category 
of finite-dimensional (left) 
$\Uslt$-modules (of type $1$, cf. \cite[Section 5.2]{Ja}). 
Recall that $\sltcat$ is semisimple with a 
complete set of pairwise non-isomorphic, irreducible 
$\Uslt$-modules parametrized by (the integral part of) the
positive Weyl chamber 
\[
\left\{
\Ll_{m,n}\mid (m,n)\in X^+
\right\}.
\]
The subscripts $m,n$ indicate the highest weight 
$m\omega_1+n\omega_2$ of the 
irreducible module, by which it is uniquely determined. Moreover, 
$\Uslt$ is a Hopf algebra, so we can tensor 
$\Uslt$-modules and take duals. Thus, if $\GG{\placeholder}$ denotes the (additive)
Grothendieck group, then 
\[
\left\{
[\Ll_{m,n}]=\GG{\Ll_{m,n}}\in\GG{\sltcat}\mid (m,n)\in X^+
\right\}
\]
is a $\Z$-basis of $\GG{\sltcat}$, and 
$\GG{\sltcat}$ is a ring. Extending the scalars to $\C$, we get a $\C$-algebra:
\[
\GGc{\sltcat}=\GG{\sltcat}\otimes_\Z\C.
\] 
Throughout the paper, we will use 
notations similar to $\GGc{\placeholder}$, indicating scalar extensions.

\begin{remark}\label{remark:generic-quantum-group}
Since $\qpar$ is generic, we can 
identify $\GG{\sltcat}$ 
with the corresponding Grothendieck ring of the 
category of complex, finite-dimensional representations
of $\slt$, cf. \cite[Theorems 5.15 and 5.17]{Ja}. 
This means that all our calculations 
below follow from standard results in the representation theory of $\slt$.
\end{remark}

The two $\Uslt$-modules
\begin{gather}\label{eq:the-variables}
\fu=\Ll_{1,0},
\quad
\quad
\fud=\Ll_{0,1}
\end{gather}
are called the fundamental representations 
of $\slt$. Note that they are dual, i.e. $\fu^*\cong\fud$ as 
$\Uslt$-modules. More generally, we have 
$(\Ll_{m,n})^*\cong\Ll_{n,m}$ for all $m,n\in\N$.

In the following we write $\fu^k=\fu^{\otimes k}$, 
$\fud^l=\fud^{\otimes l}$ and $\fu\fud=\fu\otimes\fud$ for short, and below
we will consider these as variables in some polynomial ring.

\begin{remark}\label{remark:sl3-cat-GG}
Every $\Ll_{m,n}$ appears  
as a direct summand of a suitable tensor product 
of $\fu$ and $\fud$. Moreover, 
$\sltcat$ is braided monoidal, so $\fu\fud\cong\fud\fu$ 
as $\Uslt$-modules.
Hence, the Grothendieck group of $\sltcat$ is a commutative ring and 
\[
\left\{
[\fu^k\fud^l]
\mid
(k,l)\in X^+
\right\}
\]
is an alternative basis of $\GG{\sltcat}$ and $\GGc{\sltcat}$.
\end{remark}

Using the above, in particular \fullref{remark:sl3-cat-GG},
we define $d^{k,l}_{m,n}\in\Z$ as follows:
\begin{gather}\label{eq:L-vs-L}
[\Ll_{m,n}]
=
{\textstyle\sum_{k,l}}\,
d^{k,l}_{m,n}\cdot
[\fu^{k}\fud^{l}].
\end{gather}
Clearly, $d^{k,l}_{m,n}=0$ unless $k+l\leq m+n$. 
Thus, the sum in \eqref{eq:L-vs-L} is finite. 

Note that the $d^{k,l}_{m,n}$ can be computed inductively, 
cf. \fullref{example:sl3-polys}.
Moreover, we have $d^{k,l}_{m,n}=d^{l,k}_{n,m}$ and 
$d_{m,n}^{m,n}=1=d_{k,l}^{k,l}$.

\begin{definition}\label{definition:color-code}
For later usage, let us define colors associated 
to the $\Ll$'s:
\begin{gather}\label{eq:color-code}
\chi_c(\Ll_{m,n})
=
\begin{cases}
\gc, & \text{if }m+2n\equiv 0 \bmod 3,\\
\oc, & \text{if }m+2n\equiv 1 \bmod 3,\\
\pc, & \text{if }m+2n\equiv 2 \bmod 3.\\ 
\end{cases}
\end{gather}
We call $\chi_c(\Ll_{m,n})$ 
the central character of $\Ll_{m,n}$.
\end{definition}

(To explain our choice of name: The center of $\mathrm{SU}_3$ is 
$\zeethree$. The generator of $\zeethree$ 
can act on any irreducible $\mathrm{SU}_3$-module 
by multiplication with a primitive, complex, third root of unity. This is what is 
encoded by $\chi_c$.)

Observe that $\chi_c(\fu)=\oc$ and $\chi_c(\fud)=\pc$, while
the representation theory of $\slt$ immediately gives that tensoring with $\fu$ 
changes the central character by adding $1 \bmod 3$, while tensoring 
with $\fud$ adds $2 \bmod 3$. Thus:

\begin{lemmaqed}\label{lemma:central-character}
All irreducible summands of $\fu^k\fud^l$ have 
central character $\chi_c(\Ll_{k,l})$.
\end{lemmaqed}

\subsubsection{The semisimplified root of unity case}\label{subsec:non-generic}

Let $\Uqslt$ be the specialization 
of (the integral form of) $\Uslt$ obtained by 
specializing $\qpar$ to $\qqpar$, see e.g. \cite{Lu3}, \cite{APW} for details.

Its category $\sltcat[\qqpar]=\sltcatpre[\qqpar]$ of 
finite-dimensional (left) $\sltcat[\qqpar]$-modules
(of type $1$) is far from being semisimple. However, 
it has a semisimple quotient $\slqmod$, which is roughly obtained 
by killing the so-called tilting modules 
of quantum dimension zero. We refer 
to \cite{AP} for details, but all the reader needs to know 
for our purposes is that all $\Uslt$-modules $\Ll_{m,n}$ with 
$0\leq m+n\leq e$ can also be regarded as irreducible $\Uqslt$-modules. 
Moreover,
\[
\left\{
[\Ll_{m,n}]\phantom{\fu^{k}}\!\!\!\mid 0\leq m+n\leq e
\right\},
\quad
\left\{
[\fu^{k}\fud^{l}]\mid 0\leq k+l\leq e
\right\},
\]
are bases of $\GG{\slqmod}$ and $\GGc{\slqmod}$, and 
the quantum fusion product endows
$\slqmod$ with the structure of a monoidal category, so 
$\GG{\slqmod}$ is a ring and $\GGc{\slqmod}$ is an algebra.

Note also that the rank of $\GG{\slqmod}$, respectively the dimension of $\GGc{\slqmod}$, is equal to 
the triangular number 
\[
t_e=\tfrac{(e+1)(e+2)}{2},
\] 
which follows from the fact that $\slqmod$ is only supported 
on a triangular cut-off of the positive Weyl chamber of $\slt$, 
cf. \eqref{eq:weight-picture}. 
Said otherwise, since $\slqmod$ is semisimple, we have  
\[
\GGc{\slqmod}\cong\C^{t_e}
\] 
as vector spaces.

\subsection{Chebyshev-like polynomials for \texorpdfstring{$\slt$}{sl3}}\label{subsec:opolys}

We now recall certain polynomials introduced in the context of orthogonal polynomials 
by Koornwinder \cite{Ko}, but phrased in a more convenient way for our purposes.

\subsubsection{The $\slt$-polynomials}\label{subsec:def-poly}

Consider the polynomial ring $\Z[\fu,\fud]$, 
in which $\fu$ and $\fud$ from \eqref{eq:the-variables} are treated 
as formal variables.

\begin{definition}\label{definition:sl3-polys}
For each $m,n$ we define 
$\pxy{m,n}\in\Z[\fu,\fud]$ by
\begin{gather}\label{eq:the-c-poly}
\pxy{m,n}=
{\textstyle\sum_{k,l}}\,
d_{m,n}^{k,l}
\cdot
\fu^k\fud^l,
\end{gather}
with $d_{m,n}^{k,l}\in\Z$ as in \eqref{eq:L-vs-L}.
\end{definition}

For fixed $e$, we often consider all the polynomials 
$\pxy{m,n}$ with $m+n=e+1$ together, cf. \fullref{example:sl3-polys}.

\begin{example}\label{example:sl3-polys}
The first few of these polynomials are $\pxy{0,0}=1$ and:
\[
\begin{tikzpicture}[baseline=(current bounding box.center)]
  \matrix (m) [matrix of math nodes, row sep=.0cm, column
  sep=.0cm, text height=1.5ex, text depth=0.25ex, ampersand replacement=\&, font=\scriptsize,
  nodes={anchor=south},
  row 3/.style={nodes={minimum height=1cm}},
  row 4/.style={nodes={minimum height=1.5cm}},
  row 5/.style={nodes={minimum height=1.5cm}},
  ] {
e=0 \& 
\pxy{1,0}=\fu,
\;\;
\pxy{0,1}=\fud, 
\\
e=1 \&
\pxy{2,0}
=
\fu^{2}-\fud,
\;\;
\pxy{1,1}
=
\fu\fud-1,
\;\;
\pxy{0,2}
=
\fud^{2}-\fu,
\\
e=2 \&
\begin{gathered}
\pxy{3,0}
=
\fu^{3}-2\fu\fud+1,
\;\;
\pxy{2,1}
=
\fu^{2}\fud-\fud^{2}-\fu,
\\
\pxy{1,2}
=
\fu\fud^{2}-\fu^{2}-\fud,
\;\;
\pxy{0,3}
=
\fud^{3}-2\fu\fud+1,
\end{gathered}
\\
e=3 \&
\begin{gathered}
\pxy{4,0}
=
\fu^{4}-3\fu^{2}\fud+\fud^{2}+2\fu,
\;\;
\pxy{3,1}
=
\fu^{3}\fud-2\fu\fud^{2}-\fu^{2}+2\fud,
\\
\pxy{2,2}
=
\fu^{2}\fud^{2}-\fu^{3}-\fud^{3},
\\
\pxy{1,3}
=
\fu\fud^{3}-2\fu^{2}\fud-\fud^{2}+2\fu,
\;\;
\pxy{0,4}
=
\fud^{4}-3\fu\fud^{2}+\fu^{2}+2\fud,
\end{gathered}
\\
e=4 \&
\begin{gathered}
\pxy{5,0}
=
\fu^{5}-4\fu^{3}\fud+3\fu\fud^{2}+3\fu^{2}-2\fud,
\;\;
\pxy{4,1}
=
\fu^{4}\fud-3\fu^{2}\fud^{2}-\fu^{3}+\fud^{3}+4\fu\fud-1,
\\
\pxy{3,2}
=
\fu^{3}\fud^{2}-\fu^{4}-2\fu\fud^{3}
+\fu^{2}\fud+2\fud^{2}-\fu,
\;\;
\pxy{2,3}
=
\fu^{2}\fud^{3}-\fud^{4}-2\fu^{3}\fud
+\fu\fud^{2}+2\fu^{2}-\fud,
\\
\pxy{1,4}
=
\fu\fud^{4}-3\fu^{2}\fud^{2}-\fud^{3}+\fu^{3}+4\fu\fud-1,
\;\;
\pxy{0,5}
=
\fud^{5}-4\fu\fud^{3}+3\fu^{2}\fud+3\fud^{2}-2\fu.
\end{gathered}
\\
};
\draw[thin, densely dotted] ($(m-1-1)+(-.5,-.25)$) to ($(m-1-1)+(13.1,-.25)$);
\draw[thin, densely dotted] ($(m-1-1)+(-.5,-.8)$) to ($(m-1-1)+(13.1,-.8)$);
\draw[thin, densely dotted] ($(m-1-1)+(-.5,-1.8)$) to ($(m-1-1)+(13.1,-1.8)$);
\draw[thin, densely dotted] ($(m-1-1)+(-.5,-3.35)$) to ($(m-1-1)+(13.1,-3.35)$);
\draw[thin, densely dotted] ($(m-1-1)+(.5,.25)$) to ($(m-1-1)+(.5,-4.8)$);
\end{tikzpicture}
\]
Note that the ones with $m+n=e+1$ correspond to the $e+1$-line in \eqref{eq:weight-picture}. 
\end{example}

By convention, $\pxy{m,n}$ and $\Ll_{m,n}$ with negative subscripts $m$ or $n$ are zero.

\begin{lemma}\label{lemma:recursion}
We have the following Chebyshev-like recursion relations 
\begin{gather*}
\pxy{m,n}=\pxy[\fud,\fu]{n,m},
\\
\fu\pxy{m,n}
=
\pxy{m+1,n}+\pxy{m-1,n+1}+\pxy{m,n-1},
\\
\fud\pxy{m,n}=
\pxy{m,n+1}+\pxy{m+1,n-1}+\pxy{m-1,n}.
\end{gather*}
Together with the starting conditions for $e=0,1$ as in \fullref{example:sl3-polys}, these 
recursion relations determine the polynomials 
$\pxy{m,n}$ for all $m,n$.
\end{lemma}

\begin{proof}
The relation $\pxy{m,n}=\pxy[\fud,\fu]{n,m}$ 
boils down to $\fu\cong\fud^*$. Moreover, 
by standard results in the representation theory of $\slt$, we obtain
\begin{align}
\fu\otimes\Ll_{m,n}
&\cong
\Ll_{m+1,n}\oplus\Ll_{m-1,n+1}\oplus\Ll_{m,n-1},
\label{eq:sl3-thingy-a}
\\
\fud\otimes\Ll_{m,n}
&\cong
\Ll_{m,n+1}\oplus\Ll_{m+1,n-1}\oplus\Ll_{m-1,n},
\label{eq:sl3-thingy-b}
\end{align}
which proves the two recursions.
\end{proof}

\begin{lemma}\label{lemma:no-const-term}
The polynomial $\pxy{m,n}$ has a non-zero 
constant term if and only if $m\equiv n\bmod 3$ and $m\not\equiv 2\bmod 3$. 
This constant term is equal to 
$1$ if $m\equiv n\equiv 0\bmod 3$, and equal to $-1$ if $m\equiv n\equiv 1\bmod 3$.
\end{lemma}

\begin{proof}
The claim follows inductively from \fullref{example:sl3-polys} 
and \fullref{lemma:recursion}.
\end{proof}

\subsubsection{Their complex roots}\label{subsec:roots-poly}

The following definition will be crucial for us.
\begin{definition}\label{definition:the-main-ideal}
For fixed level $e$, let $\vanideal{e}$ be the ideal generated by
\[
\left\{
\pxy{m,n}\mid m+n=e+1
\right\}
\subset\Z[\fu,\fud].
\]
We call $\vanideal{e}$ the vanishing ideal of level $e$. Associated 
to it is the vanishing set of level $e$
\[
\vanset{e}=
\left\{
(\alpha,\beta)\in\C^2\mid 
p(\alpha,\beta)=0\;\text{for all}\; p\in \vanideal{e}
\right\}\subset\C^2
\]
which we consider as a complex variety.
\end{definition}

Since 
$\fu$ and $\fud$ generate 
$\slqmod$, we have 
\begin{gather}\label{eq:roots1}
\GGc{\slqmod}\cong\C[\fu,\fud]/\vanideal{e}\cong\C^{t_e},
\end{gather}
as vector spaces, where $t_e=\tfrac{(e+1)(e+2)}{2}$ denotes the triangular number. 
Note that the left isomorphism in \eqref{eq:roots1} is actually an isomorphism of algebras, 
which follows from the explicit form of the 
fusion rules for $\slqmod$ (which can be deduced from e.g. \cite[Corollary 8]{Saw} 
or \cite[Proposition 3.2.2]{Sch}).

Using this, we can compute $\#\vanset{e}$, the number of points in $\vanset{e}$, i.e. 
the number of common roots of the polynomials in $\vanideal{e}$.

\begin{lemma}\label{lemma:level-vanishing}
We have $\#\vanset{e}=t_e$.
\end{lemma}

Before we prove \fullref{lemma:level-vanishing}, let us fix some notation for 
complex numbers: $\iunit$ denotes $\sqrt{-1}$ 
(in the positive upper half-plane), 
$\zetaroot=\exp(2\pi\iunit\neatfrac{1}{3})$ and $\overline{z}$ 
will denote the complex conjugate of a complex number $z\in\C$.

\begin{proof}
By \eqref{eq:roots1} and a corollary of Hilbert's Nullstellensatz
\cite[Corollary I.7.4]{Fu}, we immediately see that
$\#\vanset{e}\leq t_e$, 

To see the equality, consider the following functions, due to \cite{Ko}:
\begin{align*}
Z\colon\C^2\to\C,\quad
Z(\sigma,\tau)=
&\exp(\iunit\sigma)+\exp(-\iunit\tau)+\exp(\iunit(-\sigma+\tau)),
\\
E^-_{a,b}\colon\C^2\to\C,\quad
E^-_{a,b}(\sigma,\tau)=
&\exp(\iunit(a\sigma+b\tau))-\exp(\iunit((a+b)\sigma-b\tau))
\\
&+\exp(\iunit(-(a+b)\sigma+a\tau))-\exp(\iunit(-b\sigma-a\tau))
\\
&+\exp(\iunit(b\sigma-(a+b)\tau))-\exp(\iunit(-a\sigma+(a+b)\tau)),
\end{align*}
where $a,b\in\N$.

The functions $Z$ and 
$E^-_{a,b}$ are clearly $2\pi$-periodic 
in both variables, i.e. they define 
functions on a $2$-torus $\mathrm{T}^2$. As one easily checks, $Z$ is invariant 
and $E^-_{a,b}$ is antiinvariant 
under the reflections $(\sigma,\tau)\mapsto (-\sigma+\tau,\tau)$
and $(\sigma,\tau)\mapsto (\sigma, \sigma-\tau)$, which generate 
the symmetric group $S_3$. The fundamental domain of the quotient 
$\mathrm{T}^2/S_3$ is equal to
\[
D=
\left\{
(\sigma,\tau)\mid 0\leq \sigma+\tau\leq 2\pi,\;
\neatfrac{\sigma}{2}\leq \tau\leq 2\sigma
\right\}.
\] 
Note that all zeros of $E^-_{1,1}$ lie on the boundary of $D$. Therefore  
$Z$ and 
$E^-_{m+1,n+1}/E^-_{1,1}$
define functions on the interior of $D$. 

As explained in \cite{Ko}, $Z$ and 
its complex conjugate $\overline{Z}$ map $D$ bijectively onto 
the ($3$-cusps) discoid $\discoid=\{z=(x,y)\in\C\mid -z^2\overline{z}^2+4z^3+\overline{z}^3-18z\overline{z}+27\geq 0\}$ 
bounded by the deltoid curve $\deltoid=\{z=2\exp(\iunit t)+\exp(-2\iunit t)\mid t\in[0,2\pi[\}$
(also called Steiner's hypocycloid):
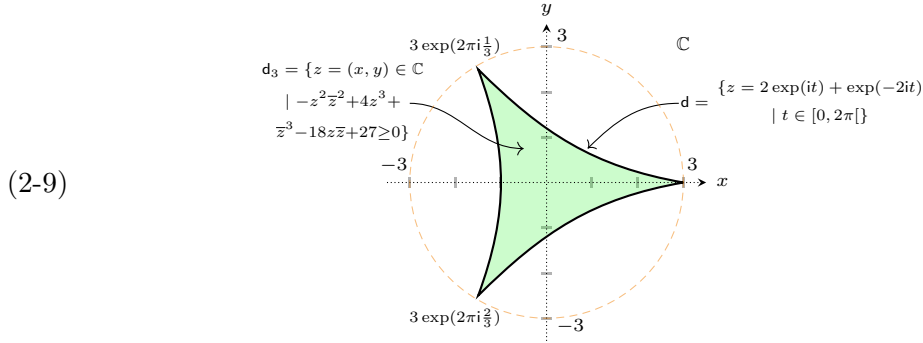
\begin{gather}\label{eq:deltoid}
\begin{tikzpicture}[anchorbase, scale=.6, tinynodes]
\draw[thin, marked=.0, marked=.166, marked=.333, marked=.666, marked=.833, marked=1.0, white] (0,-3) to (0,3);
\draw[thin, marked=.0, marked=.166, marked=.333, marked=.666, marked=.833, marked=1.0, white] (-3,0) to (3,0);
\draw[thick, white, fill=mygreen, opacity=.2] (3,0) to [out=170, in=315] (-1.5,2.5) to [out=290, in=70] (-1.5,-2.5) to [out=45, in=190] (3,0);
\draw[thin, densely dotted, ->, >=stealth] (-3.5,0) 
to (-3.35,0) node [above] {$-3$}
to (3.2,0) node [above] {$3$}
to (3.5,0) node[right] {$x$};
\draw[thin, densely dotted, ->, >=stealth] (0,-3.5) 
to (0,-3.2) node [right] {$-3$}
to (0,3.2) node [right] {$3$}
to (0,3.5) node[above] {$y$};
\draw[thick] (3,0) to [out=170, in=315] (-1.5,2.5) to [out=290, in=70] (-1.5,-2.5) to [out=45, in=190] (3,0);
\node at (-2,3) {\scalebox{.85}{$3\exp(2\pi\iunit\neatfrac{1}{3})$}};
\node at (-2,-3) {\scalebox{.85}{$3\exp(2\pi\iunit\neatfrac{2}{3})$}};
\node at (3,3) {$\C$};
\node at (5.6,1.75) {\scalebox{.85}{$\deltoid= \begin{gathered} \{z=2\exp(\iunit t)+\exp(-2\iunit t)
\\
\mid t\in[0,2\pi[
\} \end{gathered}$}};
\draw[thin, ->] (2.9,1.75) to [out=180, in=45] (.9,.75);
\node at (-4.5,1.75) {\scalebox{.85}{$\begin{gathered}
\discoid=\{z=(x,y)\in\C \\
\mid -z^2\overline{z}^2{+}4z^3{+}\\
\overline{z}^3{-}18z\overline{z}{+}27{\geq} 0\}
\end{gathered}$}};
\draw[thin, ->] (-3,1.75) to [out=0, in=180] (-.5,.75);
\draw[thin, densely dashed, opacity=.5, myorange] (0,0) circle (3cm);
\node at (0,-4) {The disciod $\discoid=\discoid(\slt)$ bounded by the deltoid curve $\deltoid$};
\end{tikzpicture}
\end{gather}
The discoid $\discoid$ has a $\zeethree$-symmetry,  
given by $(z,\overline{z})\mapsto (\zetaroot^{\pm 1}z,\zetaroot^{\mp 1}\overline{z})$, and 
its singularities are the primitive, complex third roots of unity multiplied by $3$.

For any $a,b\in\N$, the zeros of $E^-_{a,b}$ are 
known, cf. \cite[Section 7.1]{EP}. However, let us give an independent proof.    
\newline

\noindent\textit{\setword{`\fullref{lemma:level-vanishing}.Claim'}{claim-section-sl3}.}
Let $a,b\in\N,a+b=s\geq 2$. 
Then $E^-_{a,b}(\sigma,\tau)=0$ if
\begin{equation}\label{eq:zeros}
\left(\sigma,\tau\right)=
\left(\neatfrac{2\pi(2c+d+3)}{3s},\neatfrac{2\pi (c+2d+3)}{3s}\right), 
\quad\text{with}\; c,d\in\N.
\end{equation}

\noindent\textit{Proof of \ref{claim-section-sl3}.} 
We have
\[
\zetaroot^{\sneatfrac{a(2c+d)+b(c+2d)}{s}}
=
\zetaroot^{\sneatfrac{a(2c+d)+b(c+2d)-3(a+b)(c+d)}{s}}
=
\zetaroot^{\sneatfrac{-b(2c+d)-a(c+2d)}{s}}, 
\]
using that $a+b=s$, $(2c+d)+(c+2d)=3(c+d)$ and $\zetaroot^3=1$. Similarly, we obtain
\begin{gather*}
\zetaroot^{\sneatfrac{(a+b)(2c+d+3)-b(c+2d+3)}{s}}
=\zetaroot^{\sneatfrac{(a+b)(2c+d+3)-b(c+2d+3)+3(a+b)(c+d+2)}{s}}
\\
=\zetaroot^{\sneatfrac{2(a+b)(2c+d+3)+a(c+2d+3)}{s}}
=\zetaroot^{\sneatfrac{-(a+b)(2c+d+3)+a(c+2d+3)}{s}}.
\end{gather*}
\begin{gather*}
\zetaroot^{\sneatfrac{b(2c+d+3)-(a+b)(c+2d+3)}{s}}
=\zetaroot^{\sneatfrac{b(2c+d+3)-(a+b)(c+2d+3)-3(a+b)(c+d+2)}{s}}
\\
=\zetaroot^{\sneatfrac{-a(2c+d+3)-2(a+b)(c+2d+3)}{s}}
=\zetaroot^{\sneatfrac{-a(2c+d+3)+(a+b)(c+2d+3)}{s}}.
\end{gather*}
This gives $E^-_{a,b}(\sigma,\tau)=0$ for 
$(\sigma,\tau)$ as in \eqref{eq:zeros}, and completes the proof of \ref{claim-section-sl3}.
\medskip

Next, for any $m,n$, we have   
\[
\pxy[Z(\sigma,\tau),\overline{Z(\sigma,\tau)}]{m,n}
=
E^-_{m+1,n+1}(\sigma,\tau)/E^-_{1,1}(\sigma,\tau). 
\]
Let $(\sigma,\tau)$ be as 
in \eqref{eq:zeros} with $a=m+1$ and 
$b=n+1$, and assume $(\sigma,\tau)$ is in 
the interior of $D$. 
Then we have $\pxy[Z(\sigma,\tau),\overline{Z(\sigma,\tau)}]{m,n}=0$ 
by \ref{claim-section-sl3}.

To make the connection with our notation from before, 
let $m+n=e+1$ and $k=c,l=d$. By the above, for all  
\begin{equation}\label{eq:zeros2}
\left(\sigma,\tau\right)=
\left(\neatfrac{2\pi(2k+l+3)}{3(e+3)},\neatfrac{2\pi (k+2l+3)}{3(e+3)}\right) 
\quad\text{with}\; 0\leq k+l\leq e,
\end{equation} 
we have $\pxy[Z(\sigma,\tau),\overline{Z(\sigma,\tau)}]{m,n}=0$. 

Thus, $\#\vanset{e}\geq\#\left\{(k,l)\in X^+\mid 0\leq k+l\leq e \right\} = t_e$.
Since we already know that $\#\vanset{e}\leq t_e$, equality must hold.
\end{proof}

\begin{remark}\label{remark:level-vanishing}
Applying $Z$ to \eqref{eq:zeros2} gives the 
precise form of the elements of $\vanset{e}$:
\[
\vanset{e}=
\left\{
(\alpha,\beta)\in\C^2
\mid
\alpha=Z(\sigma,\tau),
\beta=\overline{Z(\sigma,\tau)}
\right\}
\]
for $(\sigma,\tau)$ as in \eqref{eq:zeros2}. For a fixed level, the common 
roots of the polynomials $\pxy{m,n}$ all lie in the interior of the 
discoid from \eqref{eq:deltoid}.
\end{remark}

\begin{example}\label{example:plot-zeros}
The polynomials for $e=1,2,3$
are given in \fullref{example:sl3-polys}.
The first (or $\fu$) entries of their common zeros are
\begin{gather*}
e=1\colon
\{
\text{roots of }
(X-1)(X^2+X+1)
\}
,
\\
e=2\colon
\{ 
\text{roots of }
(X^2-X-1)(X^4+X^3+2X^2-X+1)
\}
,
\\
e=3\colon
\{
\text{roots of }
X(X-2)(X^2 + 2X + 4)(X^6 - X^3 + 1)
\}
.
\end{gather*}
The second (or $\fud$)
entries are the complex conjugates.
Plotted to $\C$ one gets
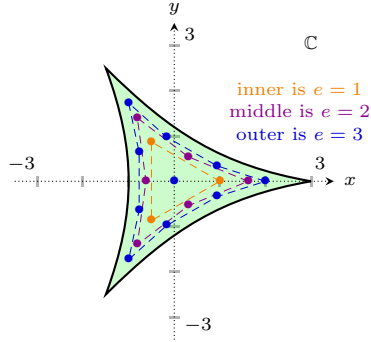
\[
\begin{tikzpicture}[anchorbase, scale=.6, tinynodes]
\draw[thin, marked=.0, marked=.166, marked=.333, marked=.666, marked=.833, marked=1.0, white] (0,-3) to (0,3);
\draw[thin, marked=.0, marked=.166, marked=.333, marked=.666, marked=.833, marked=1.0, white] (-3,0) to (3,0);
\draw[thick, white, fill=mygreen, opacity=.2] (3,0) to [out=170, in=315] (-1.5,2.5) to [out=290, in=70] (-1.5,-2.5) to [out=45, in=190] (3,0);
\draw[thin, densely dotted, ->, >=stealth] (-3.5,0) 
to (-3.35,0) node [above] {$-3$}
to (3.2,0) node [above] {$3$}
to (3.5,0) node[right] {$x$};
\draw[thin, densely dotted, ->, >=stealth] (0,-3.5) 
to (0,-3.2) node [right] {$-3$}
to (0,3.2) node [right] {$3$}
to (0,3.5) node[above] {$y$};
\draw[thick] (3,0) to [out=170, in=315] (-1.5,2.5) to [out=290, in=70] (-1.5,-2.5) to [out=45, in=190] (3,0);
\node at (3,3) {$\C$};
\node[myorange] at (1,0) {$\bullet$};
\node[myorange] at (-.5,.87) {$\bullet$};
\node[myorange] at (-.5,-.87) {$\bullet$};
\draw[very thin, densely dashed, myorange] (1,0) to (-.5,.87) to (-.5,-.87) to (1,0);
\node[mypurple] at (1.62,0) {$\bullet$};
\node[mypurple] at (-.62,0) {$\bullet$};
\node[mypurple] at (-.81,1.4) {$\bullet$};
\node[mypurple] at (-.81,-1.4) {$\bullet$};
\node[mypurple] at (.31,.54) {$\bullet$};
\node[mypurple] at (.31,-.54) {$\bullet$};
\draw[very thin, densely dashed, mypurple] (1.62,0) to (.31,.54) to (-.81,1.4) to (-.62,0) to (-.81,-1.4) to (.31,-.54) to (1.62,0);
\node[myblue] at (0,0) {$\bullet$};
\node[myblue] at (2,0) {$\bullet$};
\node[myblue] at (-1,1.73) {$\bullet$};
\node[myblue] at (-1,-1.73) {$\bullet$};
\node[myblue] at (-.77,.64) {$\bullet$};
\node[myblue] at (-.77,-.64) {$\bullet$};
\node[myblue] at (-.17,.98) {$\bullet$};
\node[myblue] at (-.17,-.98) {$\bullet$};
\node[myblue] at (.94,.34) {$\bullet$};
\node[myblue] at (.94,-.34) {$\bullet$};
\draw[very thin, densely dashed, myblue] (2,0) to (.94,.34) to (-.17,.98) to (-1,1.73) to (-.77,.64) to (-.77,-.64) to (-1,-1.73) to (-.17,-.98) to (.94,-.34) to (2,0);
\node[myorange] at (2.75,2) {inner is $e=1$};
\node[mypurple] at (2.75,1.5) {middle is $e=2$};
\node[myblue] at (2.75,1) {outer is $e=3$};
\end{tikzpicture}
\]
Letting $e\gg 0$, these approximate the deltoid curve 
$\deltoid$ (layer-wise).
\end{example}
%
\section{Trihedral Hecke algebras}\label{section:funny-algebra}

As before, $k,l,m,n$ etc. will be non-negative integers, 
and $e$ will denote the level.
We are now going to introduce the trihedral Hecke algebras. 
The reader possibly spots the 
analogies with the dihedral Hecke algebra right away, but, for 
completeness, we have also listed some of
them in \fullref{subsec:dihedral-group}. 

\subsection{Some color conventions}\label{subsec:our-color-code}

Throughout we will use the set of primary colors 
$\Prset=\{\bc,\rc,\yc\}$, the elements of which are {\color{myblue}blue} $\bc$, 
{\color{myred}red} $\rc$ and {\color{myyellow}yellow} $\yc$, the set of 
secondary colors $\Seset=\{\gc,\oc,\pc\}$, the elements of which are 
{\color{mygreen}green} $\gc=\{\bc,\yc\}$, {\color{myorange}orange} $\oc=\{\yc,\rc\}$ and 
{\color{mypurple}purple} $\pc=\{\bc,\rc\}$, and the color {\color{mycream}white} $\wc$. 
We also use dummy colors  $\tduc,\tdudc\in\Seset$, and from now on $\tduc,\tdudc$, etc. will 
always denote arbitrary elements in $\Seset$.

Moreover, we fix a cyclic ordering, and its inverse, of the secondary colors:
\begin{gather}\label{eq:color-tensor}
\xy
(0,0)*{
\raisebox{.1cm}{$\begin{tikzpicture}[baseline=(current bounding box.center),yscale=0.6]
  \matrix (m) [matrix of math nodes, row sep=.2cm, column
  sep=.1cm, text height=1.5ex, text depth=0.25ex, ampersand replacement=\&] {
\pc \&  \& \oc \\
    \& \gc \&  \\};
  \path[thick, myyellow, ->] ($(m-2-2)+(.1,.15)$) edge ($(m-1-3)+(-.1,-.3)$);
  \path[thick, densely dashed, myred, ->] (m-1-3) edge (m-1-1);
  \path[thick, densely dotted, myblue, ->] ($(m-1-1)+(.1,-.3)$) edge ($(m-2-2)+(-.1,.15)$);
\end{tikzpicture}$}};
(0,-7)*{\text{{\tiny$\rho^{\phantom{-}}\!\!\!\!\colon \gc\mapsfrom\pc\mapsfrom\oc\mapsfrom\gc$}}};
\endxy
,\quad\quad
\xy
(0,0)*{
\raisebox{.1cm}{$\begin{tikzpicture}[baseline=(current bounding box.center),yscale=0.6]
  \matrix (m) [matrix of math nodes, row sep=.2cm, column
  sep=.1cm, text height=1.5ex, text depth=0.25ex, ampersand replacement=\&] {
\pc \&  \& \oc \\
    \& \gc \&  \\};
  \path[thick, myyellow, <-] ($(m-2-2)+(.1,.15)$) edge ($(m-1-3)+(-.1,-.3)$);
  \path[thick, densely dashed, myred, <-] (m-1-3) edge (m-1-1);
  \path[thick, densely dotted, myblue, <-] ($(m-1-1)+(.1,-.3)$) edge ($(m-2-2)+(-.1,.15)$);
\end{tikzpicture}$}};
(0,-7)*{\text{{\tiny$\rho^{-1}\colon\gc\mapsfrom\oc\mapsfrom\pc\mapsfrom\gc$}}};
\endxy
\end{gather}
Note that we usually read from right to left, i.e. 
we use the operator notation. 

The action of $\rho$ on $\Seset$ can be read off 
from \eqref{eq:color-tensor}: $\rho(\gc)=\oc$, $\rho(\oc)=\pc$ 
and $\rho(\pc)=\gc$, and $\rho^{k-l}$ only depends on $(k-l)\bmod 3$, 
for any $k,l$.

\subsection{The trihedral Hecke algebra of level \texorpdfstring{$\infty$}{infty}}\label{subsec:definition}

In this section and in \fullref{subsec:quotient-algebra}, we work over $\Cv=\C(\vpar)$,
with $\vpar$ being a generic parameter.

\subsubsection{The underlying Coxeter group}\label{subsec:weyl-group}

Let $\Wgroup$ be the Coxeter group of affine type $\typea{2}$, generated by 
three reflections that we denote by $\bc,\rc,\yc$, i.e.
\begin{gather*}
\begin{tikzpicture}[anchorbase, xscale=.4, yscale=.55]
	\draw [thick] (0,0) to (.4,.4) node[right] {\text{{\tiny$3$}}} to (1,1);
	\draw [thick] (0,0) to (-.4,.4) node[left] {\text{{\tiny$3$}}} to (-1,1);
	\draw [thick] (1,1) to (0,1) node[above] {\text{{\tiny$3$}}} to (-1,1);
	\node at (0,0) {$\bullet$};
	\node at (0,-.3) {$\text{{\tiny$\bc$}}$};
	\node at (1,1) {$\bullet$};
	\node at (1.25,1.25) {$\text{{\tiny$\rc$}}$};
	\node at (-1,1) {$\bullet$};
	\node at (-1.25,1.25) {$\text{{\tiny$\yc$}}$};
	\node at (0,.675) {\text{{\tiny$\typeat{2}$}}};
\end{tikzpicture}
\rightsquigarrow
\Wgroup
=
\left\langle
\bc,\rc,\yc\mid
\bc^2=\rc^2=\yc^2=1,
\;
\bc\rc\bc=\rc\bc\rc,
\,
\bc\yc\bc=\yc\bc\yc,
\,
\rc\yc\rc=\yc\rc\yc
\right\rangle.
\end{gather*}
In order to simplify the notation, 
we identify the vertices in the Coxeter diagram of $\Wgroup$ with 
the corresponding reflections. 

Moreover, let $\gc$, $\oc$ and $\pc$ be the maximal proper parabolic subsets, 
and let $\Wgroup_{\gc}, \Wgroup_{\oc}$ and $\Wgroup_{\pc}$ 
be the corresponding 
standard parabolic subgroups of $\Wgroup$, which 
are all isomorphic to the (finite) type $\typeA_2$ Weyl group. 
Furthermore, let
\begin{gather}\label{eq:sec-Weyl}
w_{\gc}=\bc\yc\bc=\yc\bc\yc\in\Wgroup_{\gc},
\quad
\quad 
w_{\oc}=\rc\yc\rc=\yc\rc\yc\in\Wgroup_{\oc},
\quad
\quad
w_{\pc}=\bc\rc\bc=\rc\bc\rc\in\Wgroup_{\pc}
\end{gather}
denote the longest elements 
in these parabolic 
subgroups.

\subsubsection{The trihedral Hecke algebra}\label{subsec:def-sl3alg-1}

We define now the trihedral Hecke algebra 
of level $\infty$.  

\begin{definition}\label{definition:funny-alg1}
Let $\subquo$ be the associative, unital ($\Cv$-)algebra generated by three elements
$\theta_{\gc}$, $\theta_{\oc}$, $\theta_{\pc}$
subject to the following relations.
\begin{gather}\label{eq:first-rel}
\theta_{\gc}^2=\vfrac{3}\theta_{\gc},
\quad\quad
\theta_{\oc}^2=\vfrac{3}\theta_{\oc},
\quad\quad
\theta_{\pc}^2=\vfrac{3}\theta_{\pc},
\end{gather}
\begin{gather}\label{eq:second-rel}
\theta_{\gc}\theta_{\oc}\theta_{\gc}=\theta_{\gc}\theta_{\pc}\theta_{\gc},
\quad\quad
\theta_{\oc}\theta_{\gc}\theta_{\oc}=\theta_{\oc}\theta_{\pc}\theta_{\oc},
\quad\quad
\theta_{\pc}\theta_{\gc}\theta_{\pc}=\theta_{\pc}\theta_{\oc}\theta_{\pc}.
\end{gather}
Here, $\vfrac{3}$ is the $\vpar$-factorial from \eqref{eq:qnumbers-typeAD}.
\end{definition}

Let $\hecke=\hecke(\typeat{2})$ denote the Hecke algebra 
of affine type $\typea{2}$, see e.g. \cite[Section 2]{So}. 
Recall that $\hecke$ can be defined as  
the associative, unital ($\Cv$)-algebra
generated by $\theta_{\bc},\theta_{\rc}$ and $\theta_{\yc}$ 
subject to
\begin{gather}\label{eq:quadratic}
\theta_{\bc}^2=\vnumber{2}\theta_{\bc},
\quad\quad
\theta_{\yc}^2
=\vnumber{2}\theta_{\yc},
\quad\quad
\theta_{\rc}^2=\vnumber{2}\theta_{\rc},
\\
\label{eq:cubic}
\begin{aligned}
(\theta_{w_{\gc}}
=)&\theta_{\bc}\theta_{\yc}\theta_{\bc}-\theta_{\bc}\\
=&\theta_{\yc}\theta_{\bc}\theta_{\yc}-\theta_{\yc},
\end{aligned}
\quad\quad
\begin{aligned}
(\theta_{w_{\oc}}
=)&\theta_{\rc}\theta_{\yc}\theta_{\rc}-\theta_{\rc}\\
=&\theta_{\yc}\theta_{\rc}\theta_{\yc}-\theta_{\yc},
\end{aligned}
\quad\quad
\begin{aligned}
(\theta_{w_{\pc}}
=)&\theta_{\bc}\theta_{\rc}\theta_{\bc}-\theta_{\bc}\\
=&\theta_{\rc}\theta_{\bc}\theta_{\rc}-\theta_{\rc}.
\end{aligned}
\end{gather}
For any $w\in\Wgroup$, 
let $\theta_w$ be the 
corresponding Kazhdan--Lusztig (KL for short) basis element of 
$\hecke$, e.g. the expression $\theta_{w_{\tduc}}$ 
in \eqref{eq:cubic}. 
(Note that $\theta_w$ is denoted $\underline{H}_w$ in \cite[Section 2]{So},
while the standard basis is denoted $H_w$ therein.)

\begin{lemma}\label{lemma:quotient-of-affine}
The algebra homomorphism given by 
\[
\theta_{\gc}\mapsto\theta_{w_{\gc}},
\quad\quad
\theta_{\oc}\mapsto\theta_{w_{\oc}},
\quad\quad
\theta_{\pc}\mapsto\theta_{w_{\pc}},
\]
defines an embedding $\subquo\hookrightarrow\hecke$ of algebras.
\end{lemma}

\begin{proof}
By \eqref{eq:quadratic}, \eqref{eq:cubic}
and the identity $\vnumber{2}^3-\vnumber{2}=\vnumber{2}\vnumber{3}$, we obtain  
\[
\theta_{w_{\gc}}^2=\vfrac{3}\theta_{w_{\gc}}, 
\quad\quad
\theta_{w_{\oc}}^2=\vfrac{3}\theta_{w_{\oc}}, 
\quad\quad
\theta_{w_{\pc}}^2=\vfrac{3}\theta_{w_{\pc}}.
\]
This shows that \eqref{eq:first-rel} holds in $\hecke$.

Proving \eqref{eq:second-rel} is harder. 
Let us indicate how to prove 
$\theta_{w_{\gc}}\theta_{w_{\oc}}\theta_{w_{\gc}}
=\theta_{w_{\gc}}\theta_{w_{\pc}}\theta_{w_{\gc}}$. 
(The other two follow by exchanging colors.)
By \eqref{eq:cubic}, 
this is equivalent to proving
\begin{gather}\label{eq:triple}
\left(\theta_{\bc}\theta_{\yc}\theta_{\bc}-\theta_{\bc}\right)
\left(\theta_{\bc}\theta_{\rc}\theta_{\bc}-\theta_{\bc}\right)
\left(\theta_{\bc}\theta_{\yc}\theta_{\bc}-\theta_{\bc}\right)
=
\left(\theta_{\yc}\theta_{\bc}\theta_{\yc}-\theta_{\yc}\right)
\left(\theta_{\yc}\theta_{\rc}\theta_{\yc}-\theta_{\yc}\right)
\left(\theta_{\yc}\theta_{\bc}\theta_{\yc}-\theta_{\yc}\right).
\end{gather}
By \eqref{eq:quadratic}, the right-hand side in \eqref{eq:triple} is equal to 
\begin{gather}\label{eq:triple-lhs-rewritten}
\begin{gathered}
\vnumber{2}^2(\theta_{\yc}\theta_{\bc}\theta_{\yc}\theta_{\rc}
\theta_{\yc}\theta_{\bc}\theta_{\yc}
-\theta_{\yc}\theta_{\bc}\theta_{\yc}\theta_{\rc}\theta_{\yc}
-\theta_{\yc}\theta_{\rc}\theta_{\yc}\theta_{\bc}\theta_{\yc}
-\theta_{\yc}\theta_{\bc}\theta_{\yc}\theta_{\bc}\theta_{\yc}
+2\theta_{\yc}\theta_{\bc}\theta_{\yc} + 
\theta_{\yc}\theta_{\rc}\theta_{\yc}-\theta_{\yc})
\\
\stackrel{\eqref{eq:cubic}}{=}
\vnumber{2}^2\left(\theta_{\yc}\theta_{\bc}\theta_{\yc}
\theta_{\rc}\theta_{\yc}\theta_{\bc}\theta_{\yc}
-\theta_{\yc}\theta_{\bc}\theta_{\yc}\theta_{\rc}\theta_{\yc}
-\theta_{\yc}\theta_{\rc}\theta_{\yc}\theta_{\bc}\theta_{\yc}
-\vnumber{3}\theta_{w_{\gc}}
+\theta_{\yc}\theta_{\rc}\theta_{\yc}\right).
\end{gathered}
\end{gather}
Similarly, the left-hand side in \eqref{eq:triple} is equal to 
\begin{gather}\label{eq:triple-rhs-rewritten}
\vnumber{2}^2\left(\theta_{\bc}\theta_{\yc}\theta_{\bc}\theta_{\rc}
\theta_{\bc}\theta_{\yc}\theta_{\bc}
-\theta_{\bc}\theta_{\yc}\theta_{\bc}\theta_{\rc}\theta_{\bc}
-\theta_{\bc}\theta_{\rc}\theta_{\bc}\theta_{\yc}\theta_{\bc}
-\vnumber{3}\theta_{w_{\gc}}
+\theta_{\bc}\theta_{\rc}\theta_{\bc}\right).
\end{gather}

One can obtain \eqref{eq:triple-rhs-rewritten} 
from \eqref{eq:triple-lhs-rewritten} by systematically 
using \eqref{eq:cubic} 
and the fact that $w_{\gc},w_{\oc},w_{\pc}$ have two equivalent expressions each. 
For example, by \eqref{eq:cubic}, 
we have $\theta_{\yc}\theta_{\bc}\theta_{\yc}=
\theta_{\bc}\theta_{\yc}\theta_{\bc}+\theta_{\bc}-\theta_{\yc}$. 
Using this to rewrite the first term in \eqref{eq:triple-lhs-rewritten} and 
carefully continuing in this way yields the 
claimed equality.

Finally, using an appropriate integral form, $\hecke$ specializes 
to $\C[\Wgroup]$ for $\vpar=1$. Moreover, recall that $\C[\Wgroup]$ has a faithful
representation $\algstuff{P}_{1}$, which is induced by the regular $\Wgroup$-action on the 
set of alcoves obtained from the hyperplane 
arrangement associated to $\Wgroup$, and
that $\algstuff{P}_{1}$ can be 
$\vpar$-deformed to $\algstuff{P}_{\vpar}$, cf. \cite[Section 4 and Lemma 4.1]{So}. The 
$\vpar$-deformation $\algstuff{P}_{\vpar}$ 
stays faithful: Each standard basis element $H_w\in\C[\Wgroup]$ 
is mapped to a different $\C$-linear operator by $\algstuff{P}_{1}$, 
so each KL basis element $\theta_w\in\hecke$ 
is mapped to a different $\Cv$-linear operator by 
$\algstuff{P}_{\vpar}$, due to the particular form of the change 
of basis 
\[
\theta_w\in H_w+
{\scriptstyle\sum_{w^{\prime}\leq_{\mathrm{B}} w}}\,
\vpar\Z[\vpar]H_{w^{\prime}}.
\] 
Here $\leq_{\mathrm{B}}$ is the Bruhat order, see e.g. \cite[Claim 2.3]{So}. 
By pulling back $\algstuff{P}_{\vpar}$ to $\subquo$ along the 
algebra homomorphism in this lemma, injectivity of the latter follows from the 
faithfulness of the representation.
\end{proof}

\subsubsection{The trihedral Kazhdan--Lusztig combinatorics}\label{subsec:def-sl3alg-2}

We are going to define a quotient of $\subquo$. 
In order to do that, we first have to introduce certain elements. 
For any $k,l,\tduc$, let 
\begin{gather}\label{eq:KLelement}
\rklx{k,l}=\rklx{k,l}(\theta)=\theta_{\tduc_{k+l}}\cdots\theta_{\tduc_1}\theta_{\tduc_0},
\end{gather}
where $\tduc_i$ for all 
$0\leq i\leq k+l$ is given by $\tduc_0=\tduc$, $\tduc_{i+1}=\rho(\tduc_i)$ 
for (any) $k$ values of 
$i$, and $\tduc_{i+1}=\rho^{-1}(\tduc_i)$ for the remaining values 
of $i$. Note that 
\[
\rklx{0,0}=\theta_\tduc\quad\text{for any}\; \tduc.
\]
Moreover, by convention, $\rklx{k,l}=0$ in case $k$ or $l$ are negative.
We call $\tduc$ the (right) starting color of $\rklx{k,l}$. 
The fact that $\rklx{k,l}$ is well-defined is established by the following lemma.

\begin{lemma}\label{lemma:well-defined-paths}
For any $k,l,\tduc$, 
the element $\rklx{k,l}$ only depends on 
$k$ and $l$, not on the chosen sequence $\tduc_{k+l},\cdots,\tduc_{1},\tduc_{0}=\tduc$.
\end{lemma}

\begin{proof}
We claim that there is a normal form, i.e. any 
word representing $\rklx{k,l}$ is 
equivalent to the word $\theta_{\tduc_{k+l}}\cdots\theta_{\tduc_{0}}$
such that $\tduc_0=\tduc$ and
\begin{gather}\label{eq:normal-form}
\tduc_r=\rho(\tduc_{r-1}),\text{ for all }1\leq r\leq k,
\quad\quad 
\tduc_r=\rho^{-1}(\tduc_{r-1}),\text{ for all }k+1\leq r\leq k+l,
\end{gather}
which is clear if $l=0$.
Otherwise, any word representing $\rklx{k,l}$ involves $k$ 
counterclockwise rotations and $l$ clockwise rotations of $\Seset$. 
Hence, if such a word is not in normal form, then 
we will find a subsequence of the form
\begin{gather*}
\theta_{\tduc_{i}}\theta_{\rho^{-1}(\tduc_{i})}\theta_{\tduc_{i}}
\stackrel{\eqref{eq:second-rel}}{=}
\theta_{\tduc_{i}}\theta_{\rho(\tduc_{i})}\theta_{\tduc_{i}},
\end{gather*}
which we rewrite as above.
We can then continue recursively until we get \eqref{eq:normal-form}.
\end{proof}

Similarly, we can define 
\begin{gather}\label{eq:left-right-relations}
\kly{k,l}=\rklx{k,l}\quad\text{such that}\; \tdudc=\rho^{k-l}(\tduc),
\end{gather}
for $k,l,\tdudc$. 
\fullref{lemma:well-defined-paths} implies, 
mutatis mutandis, that
$\kly{k,l}$ is also independent of the chosen sequence 
$\tdudc=\tduc_{k+l},\dots,\tduc_0$.

\begin{remark}\label{remark-KL-combinatorics}
We can view $\fu$ and $\fud$ as acting via counterclockwise 
respectively clockwise rotation of \eqref{eq:color-tensor}. 
By \fullref{lemma:well-defined-paths},
we can view the elements $\rklx{k,l}$ as being associated to 
$\fu^k\fud^l$ (after fixing a starting color $\tduc$), because its 
definition involves $k$ times the application of $\rho$ and 
$l$ times that of $\rho^{-1}$. \fullref{lemma:well-defined-paths} 
then translates into the equality 
$\fu\fud=\fud\fu$. 
\end{remark}

\begin{example}\label{example-KL-combinatorics}
Let us fix $\gc$ as a starting color. Then
\begin{gather*}
\,
\xy
(0,2.5)*{\rklg{2,0}=\theta_\pc\theta_\oc\theta_\gc};
(0,-2.5)*{\text{$\leftrightsquigarrow\fu^2\fud^0$}};
\endxy
,
\quad
\xy
(0,2.5)*{\rklg{1,1}=\theta_\gc\theta_\pc\theta_\gc
=\theta_\gc\theta_\oc\theta_\gc
};
(0,-2.5)*{\text{$\leftrightsquigarrow\fu^1\fud^1=\fud^1\fu^1$}};
\endxy
,
\quad
\xy
(0,2.5)*{\rklg{0,2}=\theta_\oc\theta_\pc\theta_\gc};
(0,-2.5)*{\text{$\leftrightsquigarrow\fu^0\fud^2$}};
\endxy
,
\end{gather*}
where we think of the color changes 
$\gc\mapsfrom\pc\mapsfrom\oc\mapsfrom\gc$ as corresponding to 
multiplication by $\fu$, and the 
color changes $\gc\mapsfrom\oc\mapsfrom\pc\mapsfrom\gc$ as corresponding to 
multiplication by $\fud$.
\end{example}

Recall that $d^{k,l}_{m,n}$ denote the numbers from
\fullref{section:sl3-stuff}, coming from the 
representation theory of $\slt$. For each pair $m,n$, we define 
three colored KL basis elements:
\begin{gather}\label{eq:the-expressions}
\begin{gathered}
\RKLg{m,n}=
{\textstyle\sum_{k,l}}\,
\vnumber{2}^{-k-l}\,d^{k,l}_{m,n}\,\rklg{k,l},
\quad\quad
\RKLo{m,n}=
{\textstyle\sum_{k,l}}\,
\vnumber{2}^{-k-l}\,d^{k,l}_{m,n}\,\rklo{k,l},
\\
\RKLp{m,n}=
{\textstyle\sum_{k,l}}\,
\vnumber{2}^{-k-l}\,d^{k,l}_{m,n}\,\rklp{k,l}.
\end{gathered}
\end{gather}
Note that the three sums are finite, 
because $d^{k,l}_{m,n}=0$ unless $k+l\leq m+n$, as mentioned after \eqref{eq:L-vs-L}. 
Moreover, by convention, $\RKLx{k,l}=0$ in case $k$ or $l$ are negative.

Furthermore, by \eqref{eq:L-vs-L} and \fullref{lemma:central-character},
$d^{k,l}_{m,n}=0$ if $k-l \not\equiv m-n\bmod 3$. 
This implies that, for any $m,n,\tduc$, 
all terms $\rklx{k,l}$ of $\RKLx{m,n}$ in \eqref{eq:the-expressions} 
have the same left-most factor $\theta_\tdudc$, where 
$\tdudc=\rho^{m-n}(\tduc)$, by \eqref{eq:left-right-relations}. Therefore, we can also define  
\begin{gather}\label{eq:left-right-KL}
\KLy{m,n}=\RKLx{m,n} \quad\text{such that}\; \tdudc=\rho^{m-n}(\tduc).
\end{gather}
We call $\RKLg{m,n}, \RKLo{m,n}$ and $\RKLp{m,n}$ 
the (right) colored KL elements. 
As before, 
\[
\RKLx{0,0}=\theta_\tduc\quad\text{for any}\; \tduc.
\]

\begin{example}\label{example-KL-combinatorics-2}
For a fixed $\tduc$, the element $\RKLx{m,n}$ (or alternatively $\KLx{m,n}$) is associated 
to the orthogonal polynomial $\pxy{m,n}$ 
from \fullref{subsec:opolys}, cf. \fullref{example-KL-combinatorics-2}. 
For example, fix $\gc$ as a starting color. Then
\begin{gather*}
\,
\xy
(0,2.5)*{\RKLg{2,0}=\vnumber{2}^{-2}\theta_\pc\theta_\oc\theta_\gc-\vnumber{2}^{-1}\,\theta_\pc\theta_\gc};
(0,-2.5)*{\text{$\leftrightsquigarrow\pxy{2,0}=\fu^2-\fud$}};
\endxy
,
\;\;
\xy
(0,2.5)*{\RKLg{1,1}=\vnumber{2}^{-2}\theta_\gc\theta_\pc\theta_\gc
-\,\theta_\gc
};
(0,-2.5)*{\text{$\leftrightsquigarrow\pxy{1,1}=\fu\fud-1$}};
\endxy
,
\;\;
\xy
(0,2.5)*{\RKLg{0,2}=\vnumber{2}^{-2}\theta_\oc\theta_\pc\theta_\gc-\vnumber{2}^{-1}\,\theta_\oc\theta_\gc};
(0,-2.5)*{\text{$\leftrightsquigarrow\pxy{0,2}=\fud^2-\fu$}};
\endxy
.
\end{gather*}
Similarly for the other colors.
\end{example}

As we will see, \fullref{proposition:cat-the-algebra} identifies the colored KL elements with the Grothendieck 
classes of the indecomposables in a certain 
$2$-full $2$-subcategory of singular Soergel bimodules. In particular, the 
next lemma and proposition need some notions from categorification 
which we only recall in \fullref{sec:A2-diagrams}. Consequently, 
we postpone their proofs until the end of \fullref{sec:A2-diagrams}.

\begin{lemma}\label{lemma:multiplication}
For all $m,n,\tduc,\tdudc$, we have
\begin{gather}\label{eq:multiplication}
\theta_{\tduc}\RKLy{m,n}=
\begin{cases}
\vfrac{3}\RKLy{m,n}, &\text{if } \rho^{m-n}(\tduc)=\tdudc,\\
\vnumber{2}\left(\RKLy{m+1,n}+ \RKLy{m-1,n+1} + \RKLy{m,n-1}\right), &\text{if } \rho^{m+1-n}(\tduc)=\tdudc,
\\
\vnumber{2}\left(\RKLy{m,n+1}+ \RKLy{m+1,n-1} + \RKLy{m-1,n}\right), &\text{if } \rho^{m-(n+1)}(\tduc)=\tdudc,
\end{cases}
\end{gather}
where terms with negative indices are zero. 
\end{lemma}

By \eqref{eq:left-right-KL} and \fullref{lemma:multiplication}, we also have 
\begin{gather}\label{eq:multiplication2}
\RKLy{m,n}\theta_{\tduc}=
\begin{cases}
\vfrac{3}\RKLx{m,n}, &\text{if } \tduc=\tdudc,\\
\vnumber{2}\left(\RKLx{m+1,n}+ \RKLx{m-1,n+1} + \RKLx{m,n-1}\right), &\text{if }\rho(\tduc)=\tdudc,\\
\vnumber{2}\left(\RKLx{m,n+1}+ \RKLx{m+1,n-1} + \RKLx{m-1,n}\right), &\text{if }\rho^{-1}(\tduc)=\tdudc,
\end{cases}
\end{gather}
where again terms with negative indices are zero. Moreover, 
there are also the evident versions of \eqref{eq:multiplication} 
and \eqref{eq:multiplication2} using $\KLy{m,n}$ instead of $\RKLy{m,n}$.

\begin{example}\label{example-KL-combinatorics-3}
The reader should compare \eqref{eq:multiplication} 
and \eqref{eq:multiplication2} with the recursion formulas 
from \fullref{lemma:recursion}. This is no coincidence, keeping 
\fullref{remark-KL-combinatorics} and \fullref{example-KL-combinatorics-2} in mind. For example, 
one can easily check directly that
\begin{gather*}
\RKLo{0,1}\theta_\gc=
(\vnumber{2}^{-1}\theta_\gc\theta_\oc)\theta_\gc=
\vnumber{2}(\vnumber{2}^{-2}\theta_\gc\theta_\oc\theta_\gc-\theta_\gc)
+\vnumber{2}\theta_\gc
=
\vnumber{2}(\RKLg{1,1}+\underbrace{\RKLg{-1,1}}_{=0}+\RKLg{0,0}),
\\
\theta_\gc\RKLo{0,1}=
\theta_\gc(\vnumber{2}^{-1}\theta_\gc\theta_\oc)
=
\vfrac{3}(\vnumber{2}^{-1}\theta_\gc\theta_\oc)
=
\vfrac{3}\RKLo{0,1},
\end{gather*}
and similarly for right or left multiplication by $\theta_\oc$ or $\theta_\pc$.
\end{example}

\begin{proposition}\label{proposition:two-bases} 
Each of the four sets 
\begin{gather*}
\basisH=
\{1\}
\cup
\{\rklx{k,l}\mid (k,l)\in X^+,\, \tduc\in\Seset\},
\;\;
\basisC=
\{1\}
\cup\{\RKLx{m,n}\mid (m,n)\in X^+,\, \tduc\in\Seset\}
\\
\Hbasis=
\{1\}\cup
\{\klx{k,l}\mid (k,l)\in X^+,\, \tduc\in\Seset\},
\;\;
\Cbasis=
\{1\}\cup
\{\KLx{m,n}\mid (m,n)\in X^+,\, \tduc\in\Seset\}
\end{gather*} 
is a basis of $\subquo$.
\end{proposition}

As for \fullref{lemma:multiplication}, 
the proof of \fullref{proposition:two-bases} is postponed until \fullref{sec:A2-diagrams}.
As we will see, the bases $\basisH$ and $\Hbasis$ could be called 
Bott--Samelson bases.

Following \cite{KaLu}, we can define 
left, right and two-sided 
cells for $\subquo$. We have chosen to work with the basis $\basisC$.

\begin{definition}\label{definition:cells-first}
We define a left preorder on $\basisC$ 
by declaring $\RKLx{m,n}\lgeq\RKLy{m^{\prime},n^{\prime}}$ 
if there exists an element 
$\algstuff{Z}\in\basisC$ such that $\RKLx{m,n}$ appears as a summand of 
$\algstuff{Z}\RKLy{m^{\prime},n^{\prime}}$, when the latter is written as a linear combination 
of elements in $\basisC$. 

This preorder gives rise to 
an equivalence relation by declaring 
$\RKLx{m,n}\lsim\RKLy{m^{\prime},n^{\prime}}$ whenever 
$\RKLx{m,n}\lgeq\RKLy{m^{\prime},n^{\prime}}$ and 
$\RKLy{m^{\prime},n^{\prime}}\lgeq\RKLx{m,n}$. 
The equivalence classes of $\lsim$ are called left cells.

Similarly, right multiplication gives rise to a 
right preorder $\rgeq$, a right equivalence relation $\rsim$ and 
right cells $\rcell$. Multiplication on both 
sides, gives rise to a two-sided preorder $\tgeq$, 
a two-sided equivalence relation $\tsim$ and 
two-sided cells $\tcell$.
\end{definition}

Clearly, $1\in\subquo$ forms 
a cell $\{1\}$ on its own, which is left, right and two-sided at once, and always the lowest cell. 
We call $\{1\}$ the trivial cell. The other cells are as follows.

\begin{proposition}\label{proposition:cells}
The non-trivial cells for the algebra $\subquo$ are
\begin{gather*}
\lcell_{\tduc}=\left\{\RKLx{m,n}\mid (m,n)\in X^+\right\},
\quad\quad
{}_{\tduc}\rcell=\left\{\KLx{m,n}\mid (m,n)\in X^+\right\},
\quad\quad \hbox{ for $\tduc\in \Seset$,}
\\
\tcell=
\left\{\RKLx{m,n}\mid (m,n)\in X^+,\tduc\in\Seset\right\}
=
\left\{\KLx{m,n}\mid (m,n)\in X^+,\tduc\in\Seset\right\},
\end{gather*}
where $\lcell_{\tduc}$, ${}_{\tduc}\rcell$ and $\tcell$ 
are left, right and two-sided cells
respectively.
\end{proposition}

\begin{proof}
Fix $\gc$ as a starting color. Applying \eqref{eq:multiplication2} to 
$\theta_{\tduc}\RKLg{m,n}$, yields $\RKLg{m+1,n}\lgeq\RKLg{m,n}$ for 
$\tduc$ being chosen such we can apply the middle cases. 
We also obtain
$\RKLg{m,n}\lgeq\RKLg{m+1,n}$, by 
applying \eqref{eq:multiplication2} to $\theta_{\tdudc}\RKLg{m+1,n}$ 
for appropriate $\tdudc$. Thus, we have
$\RKLg{m,n}\lsim\RKLg{m+1,n}$. 
Similarly, we deduce $\RKLg{m,n}\lsim\RKLg{m,n-1}$, 
$\RKLg{m,n}\lsim\RKLg{m,n+1}$ and $\RKLg{m,n}\lsim\RKLg{m-1,n}$.
Thus, for fixed $m$ we get that all $\RKLg{m,\placeholder}$ are in the same left cell, 
and similarly for fixed $n$ we get that all $\RKLg{\placeholder,n}$ are in the same left cell. 
We can also deduce that $\RKLg{m,n}\lsim\RKLg{m-1,n+1}$ and $\RKLg{m,n}\lsim\RKLg{m+1,n-1}$.
In summary, all $\RKLg{m,n}$ belong to the same 
left cell. 
Since left multiplication will never change the rightmost 
color of a word, we conclude
that $\lcell_{\gc}$ is indeed a left cell.

Analogously, one can show that $\lcell_{\oc}$ and 
$\lcell_{\pc}$ are left cells, and,
mutatis mutandis, that
${}_{\tduc}\rcell$ is a right cell, for $\tduc$.
Finally, the statement about two-sided cell follows from \eqref{eq:multiplication}.
\end{proof}

\subsection{The quotient of level \texorpdfstring{$e$}{e}}\label{subsec:quotient-algebra}

\subsubsection{Its definition}\label{subsec:quotient-algebra-1}

We are now ready to define interesting, 
finite-dimensional quotients of $\subquo$, which are compatible 
with the cell structure.

\begin{definition}\label{definition:the-quotient-defined}
For fixed level $e$, let $\killideal{e}$ 
be the two-sided ideal in $\subquo$ generated by
\begin{gather*}
\left\{
\RKLx{m,n} \mid m+n=e+1,\; \tduc\in\Seset
\right\}
=
\left\{
\phantom{\RKLx{m,n}}\hspace*{-.75cm}\KLx{m,n} \mid m+n=e+1,\; \tduc\in\Seset
\right\}.
\end{gather*}
We define the the trihedral Hecke algebra of level $e$ as
\[
\subquo[e]=\subquo/\killideal{e}
\]
and we call $\killideal{e}$ the vanishing ideal of level $e$.
\end{definition}

\begin{remark}\label{remark:small-quotient}
We point out that $\subquo[e]$ is the trihedral analog of the 
so called small quotient in the dihedral case, cf. \fullref{remark:dihedral-group2}.
\end{remark}

\begin{proposition}\label{proposition:dimension}
Each of the two sets 
\begin{gather*}\label{eq:cell-basis-2}
\begin{aligned}
\basisC[e]=
\{1\}\cup
&
\left
\{\RKLx{m,n} \mid 0\leq m+n\leq e,\; \tduc\in\Seset
\right\} 
,
\\
\Cbasis[e]=
\{1\}\cup
&
\left
\{\KLx{m,n} \mid 0\leq m+n\leq e,\; \tduc\in\Seset
\right\},
\end{aligned}
\end{gather*}
is a basis of $\subquo[e]$. 
Thus, we have 
$\dim_{\Cv}\subquo[e] = 3\tfrac{(e+1)(e+2)}{2}+1=3t_e+1$.
\end{proposition}

\begin{proof}
By \fullref{lemma:multiplication}, $\subquo$ is an $\N$-filtered 
algebra with $\subquo\cong {\textstyle\bigcup_{i\in\N}}(\subquo)_i$
such that $(\subquo)_0=\{1\}$ and, for any $i\in\Z_{\geq 1}$, we have  
\[
(\subquo)_{i}=\Cv\left\{
\RKLx{m,n} \mid 0\leq m+n\leq i-1,\; \tduc\in\Seset 
\right\}.
\]
Note that $\RKLx{0,0}=\theta_{\tduc}$ has filtration degree $1$, so 
the multiplication rule in 
\fullref{lemma:multiplication} is compatible with this filtration. 

Since $\killideal{e}$ is generated by 
homogeneous elements, $\subquo[e]$ is also an $\N$-filtered algebra. 
In order to prove finite-dimensionality, consider the associated $\N$-graded algebra 
\[
\catstuff{E}(\subquo[e])=
{\textstyle\bigoplus_{i\in\N}}\,
(\subquo[e])_i/(\subquo[e])_{i-1},
\]
where $(\subquo[e])_{-1}=\{0\}$, by convention. Note that 
\[
\RKLx{m,n}\equiv\rklx{m,n}\bmod(\subquo[e])_{m+n},
\] 
for all $m,n$. We have $\catstuff{E}(\subquo[e])_{e+2}=\{0\}$,
by \fullref{lemma:multiplication}, and 
$\catstuff{E}(\subquo[e])_i=\{0\}$ for
all $i\geq e+3$, also by \fullref{lemma:multiplication}. 

The first statement follows, since
$\left\{\RKLx{m,n}\mid m+n=i-1,\; \tduc\in\Seset \right\}$ is, 
by \fullref{proposition:two-bases}, a 
basis of $\catstuff{E}(\subquo[e])_i$, for all $1\leq i\leq e+1$. 
The dimension 
formula is then clear.   

The version with $\KLx{m,n}$ can be shown verbatim.
\end{proof}

\begin{corollary}\label{corollary:cells}
The non-trivial cells for the algebra $\subquo[e]$ are
as in \fullref{proposition:cells}, but intersected with the 
bases from \fullref{proposition:dimension}. 

In particular, the non-trivial left and right cells have each 
cardinality $t_e=\tfrac{(e+1)(e+2)}{2}$. The non-trivial 
two-sided cell is the disjoint 
union of them all, so it has cardinality $3t_e$. 
\end{corollary}

\begin{example}\label{example:cells}
The left cells correspond to the generalized type $\typeA$ Dynkin diagrams 
$\graphA{e}$ in \fullref{subsec:gen-D-list}, which are 
cut-offs of the positive Weyl chamber as in \eqref{eq:weight-picture}, such that 
the basis elements of the left cells correspond to the vertices of the diagram. 

The prototypical examples to keep in mind are
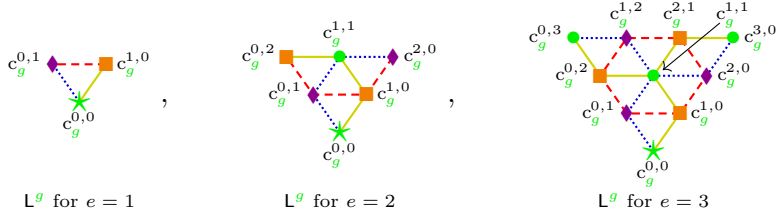
\[
\xy
(0,0)*{
\begin{tikzpicture}[anchorbase, xscale=.35, yscale=.5]
	\draw [thick, myyellow] (0,0) node[below, black] {$\text{\tiny$\RKLg{0,0}$}$} to (1,1) node[right, black] {$\text{\tiny$\RKLg{1,0}$}$};
	\draw [thick, densely dotted, myblue] (0,0) to (-1,1) node[left, black] {$\text{\tiny$\RKLg{0,1}$}$};
	\draw [thick, densely dashed, myred] (1,1) to (-1,1);
	\node at (0,0) {$\bulletgstart$};
	\node at (1,1) {$\bulleto$};
	\node at (-1,1) {$\bulletp$};
\end{tikzpicture}};
(0,-14)*{\text{{\tiny $\lcell^{\gc}$ for $e=1$}}};
\endxy
,\quad\quad
\xy
(0,0)*{
\begin{tikzpicture}[anchorbase, xscale=.35, yscale=.5]
	\draw [thick, myyellow] (0,0) node[below, black] {$\text{\tiny$\RKLg{0,0}$}$} to (1,1) node[right, black] {$\text{\tiny$\RKLg{1,0}$}$} to (0,2);
	\draw [thick, myyellow] (0,2) to (-2,2) node[left, black] {$\text{\tiny$\RKLg{0,2}$}$};
	\draw [thick, densely dotted, myblue] (0,0) to (-1,1) node[left, black] {$\text{\tiny$\RKLg{0,1}$}$} to (0,2);
	\draw [thick, densely dotted, myblue] (0,2) node[above, black] {$\text{\tiny$\RKLg{1,1}$}$} to (2,2) node[right, black] {$\text{\tiny$\RKLg{2,0}$}$};
	\draw [thick, densely dashed, myred] (2,2) to (1,1) to (-1,1) to (-2,2);
	\node at (0,0) {$\bulletgstart$};
	\node at (0,2) {$\bulletg$};
	\node at (1,1) {$\bulleto$};
	\node at (-2,2) {$\bulleto$};
	\node at (2,2) {$\bulletp$};
	\node at (-1,1) {$\bulletp$};
\end{tikzpicture}};
(0,-14)*{\text{{\tiny $\lcell^{\gc}$ for $e=2$}}};
\endxy
,\quad\quad
\xy
(0,0)*{
\begin{tikzpicture}[anchorbase, xscale=.35, yscale=.5]
	\draw [thick, myyellow] (0,0) node[below, black] {$\text{\tiny$\RKLg{0,0}$}$} to (1,1) node[right, black] {$\text{\tiny$\RKLg{1,0}$}$} to (0,2) to (1,3) node[above, black] {$\text{\tiny$\RKLg{2,1}$}$} to (3,3) node[right, black] {$\text{\tiny$\RKLg{3,0}$}$} node[above, black] {$\text{\tiny$\RKLg{1,1}$}$};
	\draw [thick, myyellow] (0,2) to (-2,2) node[left, black] {$\text{\tiny$\RKLg{0,2}$}$} to (-3,3) node[left, black] {$\text{\tiny$\RKLg{0,3}$}$};
	\draw [thick, densely dotted, myblue] (0,0) to (-1,1) node[left, black] {$\text{\tiny$\RKLg{0,1}$}$} to (0,2) to (-1,3) to (-3,3);
	\draw [thick, densely dotted, myblue] (0,2) to (2,2) node[right, black] {$\text{\tiny$\RKLg{2,0}$}$} to (3,3);
	\draw [thick, densely dashed, myred] (1,1) to (-1,1) to (-2,2) to (-1,3) node[above, black] {$\text{\tiny$\RKLg{1,2}$}$} to (1,3) to (2,2) to (1,1);
	\node at (0,0) {$\bulletgstart$};
	\node at (0,2) {$\bulletg$};
	\node at (3,3) {$\bulletg$};
	\node at (-3,3) {$\bulletg$};
	\node at (1,1) {$\bulleto$};
	\node at (-2,2) {$\bulleto$};
	\node at (1,3) {$\bulleto$};
	\node at (2,2) {$\bulletp$};
	\node at (-1,1) {$\bulletp$};
	\node at (-1,3) {$\bulletp$};
	\draw [<-] (.35,2.15) to (2.5,3.4);
\end{tikzpicture}};
(0,-14)*{\text{{\tiny $\lcell^{\gc}$ for $e=3$}}};
\endxy
\]
where we also display the associated colored KL basis elements. 
The starting (rightmost) color is indicated by $\bulletgstart$.  
The color of any vertex is the color of the 
leftmost $\theta_{\tduc}$ in any of the terms of the corresponding 
colored KL basis element. The arrows of the same type, emanating from a given vertex, indicate 
the terms which appear on the right-hand 
side of the multiplication rule in \eqref{eq:multiplication2}.
\end{example}

\subsubsection{Trihedral simples}\label{subsec:semisimplicity}

Next, we classify all simple representations of $\subquo[e]$ on 
($\Cv$-)vector spaces, 
cf. \eqref{eq:the-simples}. To this end, note that 
the ideal $\killideal{e}$ 
defining $\subquo$ is built such that we can use 
Koornwinder's Chebyshev polynomials and their roots 
as in \fullref{subsec:roots-poly} below.

First, the one-dimensional representations of $\subquo[e]$ are easy 
to define, since they correspond to characters. Each such character 
\[
\M_{\lambda_\gc,\lambda_\oc,\lambda_\pc}
\colon\subquo[e]\to\Cv
\]
is completely determined by its value on the generators
\[
\theta_\gc\mapsto \lambda_\gc,\;
\theta_\oc\mapsto \lambda_\oc,\;
\theta_\pc\mapsto \lambda_\pc.
\]
Therefore, we can identify $\M_{\lambda_\gc,\lambda_\oc,\lambda_\pc}$ with a triple 
$(\lambda_\gc,\lambda_\oc,\lambda_\pc)\in\Cv^{3}$.

\begin{lemma}\label{lemma:further-restrictions} 
The following table 
\begin{gather}\label{eq:the-one-dims}
\begin{tikzpicture}[baseline=(current bounding box.center),yscale=0.6]
  \matrix (m) [matrix of math nodes, row sep={.85cm,between origins}, column
  sep={4.0cm,between origins}, text height=1.5ex, text depth=0.25ex, ampersand replacement=\&] {
e\equiv 0\bmod 3 \&  e\not\equiv 0\bmod 3 \\
  \begin{gathered}\M_{0,0,0},\,\M_{\vfrac{3},0,0},
  \\
  \M_{0,\vfrac{3},0},\,\M_{0,0,\vfrac{3}}\end{gathered} \& \M_{0,0,0}  \\};
  \draw[densely dashed] ($(m-1-1.south west)+ (-.9,0)$) to (m-1-2.south east);
  \draw[densely dashed] ($(m-1-2.north west) + (-.3,0)$) to ($(m-1-2.north west) + (-.3,-2.75)$);
\end{tikzpicture}
\end{gather}
gives a complete, irredundant list of one-dimensional 
$\subquo[e]$-representations.
\end{lemma}

\begin{proof}
Let us first check which 
triples $(\lambda_\gc,\lambda_\oc,\lambda_\pc)$ give a well-defined character of $\subquo[e]$: 
by \eqref{eq:first-rel}, we see that $\lambda_\tduc$ has to be zero or $\vfrac{3}$. 
Moreover, \eqref{eq:second-rel} implies that either 
all $\theta_\tduc$ act by zero, precisely one 
of them acts by $\vfrac{3}$, or 
all of them act by $\vfrac{3}$. 
Further restrictions are imposed by requiring 
the representation to vanish
on $\vanideal{e}$.
 
Let us now give the details. The representation $\M_{0,0,0}$ vanishes on $\vanideal{e}$, 
because, by definition, $\RKLg{m,n}$, $\RKLo{m,n}$, $\RKLp{m,n}$ have no constant term for all $m,n$, 
since their starting color is always $\theta_\tduc$.

The representation $\M_{\vfrac{3},\vfrac{3},\vfrac{3}}$ does not vanish 
on $\vanideal{e}$, since all polynomials $\pxy{m,n}$ have a unique term 
of highest degree. This follows 
from the representation theory of of $\slt$, since $\fu^{m}\fud^{n}$ 
has a unique highest summand. This coefficient of this term contributes a maximal power of $\vpar$ when 
evaluated, which cannot be canceled by the coefficients of other terms, 
e.g.
\[
\fu^{2}-\fud
\leftrightsquigarrow
\vnumber{2}^{-2}\theta_\pc\theta_\oc\theta_\gc-\vnumber{2}^{-1}\theta_\pc\theta_\gc
\mapsto
\vnumber{2}^{-2}\vfrac{3}\vfrac{3}\vfrac{3}-\vnumber{2}^{-1}\vfrac{3}\vfrac{3}\neq 0.
\]
Thus, $\M_{\vfrac{3},\vfrac{3},\vfrac{3}}$ is not a representation 
of $\subquo[e]$.

When $e\equiv 0\bmod 3$, there are three more characters, namely 
\[
\M_{\vfrac{3},0,0},\quad\M_{0,\vfrac{3},0},\quad\M_{0,0,\vfrac{3}}.
\]
To see this note that, 
for $m+n=e+1$ and $e\equiv 0\bmod 3$, we have $m+n\equiv 1 \bmod 3$. 
Hence, $m\equiv n\equiv 0,1\bmod 3$ is impossible in this case. 
By \fullref{lemma:no-const-term}, this means that $\pxy{m,n}$ does not have a non-zero constant term.
It follows that all terms in $\RKLg{m,n}$, $\RKLo{m,n}$, $\RKLp{m,n}$ 
contain a factor $\theta_\tdudc\theta_\tduc$ for some $\tduc\neq\tdudc$. 
For any of the three 
$\M_{\vfrac{3},0,0}$, $\M_{0,\vfrac{3},0}$ or $\M_{0,0,\vfrac{3}}$, we therefore have 
$\theta_\tdudc\theta_\tduc\mapsto 0$. This shows that 
$\RKLg{m,n},\RKLo{m,n},\RKLp{m,n}\mapsto 0$.

The three corresponding one-dimensional representations are clearly non-isomorphic.
\end{proof}

Let us now study the simple representations of dimension three, which 
depend on a complex number $z\in\C$.
To this end, we define three matrices
\begin{gather}\label{eq:three-dims}
\begin{gathered}
\Mg=
\vnumber{2}
{\scriptstyle
\begin{pmatrix}
\vnumber{3} & \overline{z}& z\\
0 & 0 & 0\\
0& 0 & 0
\end{pmatrix}
}
,
\quad
\Mo=
\vnumber{2}
{\scriptstyle
\begin{pmatrix}
0 & 0 & 0\\
z& \vnumber{3} & \overline{z}\\ 
0& 0 & 0
\end{pmatrix}
}
,
\\
\Mp=
\vnumber{2}
{\scriptstyle
\begin{pmatrix}
0 & 0 & 0\\
0 & 0 & 0\\
\overline{z} & z & \vnumber{3}
\end{pmatrix}
}
.
\end{gathered}
\end{gather}
Let $\Mt=\Mg+\Mo+\Mp$.

Next, we use the explicit description 
of the elements in $\vanset{e}$, cf. \fullref{remark:level-vanishing}.

\begin{lemma}\label{lemma:three-dim-rep1}
The matrices $\Mg,\Mo,\Mp$ define a representation $\M_z$ of $\subquo[e]$ 
on $\Cv^3$, such that 
\[
\theta_{\gc} \mapsto \Mg,\quad \theta_{\oc} \mapsto \Mo,\quad \theta_{\pc} \mapsto \Mp,
\]
if and only if $(z,\overline{z})\in\vanset{e}$. 
\end{lemma}

\begin{proof} 
Two short calculations show that 
$\M_z(\tduc)$ respects the 
relations \eqref{eq:first-rel} and \eqref{eq:second-rel}. 
The fact that $\M_z$ vanishes on $\vanideal{e}$ 
if and only if $(z,\overline{z})\in\vanset{e}$,  
follows by the proof of \fullref{lemma:level-vanishing}, as we defined 
$\RKLg{m,n}$, $\RKLo{m,n}$ and $\RKLp{m,n}$ 
in terms of $\pxy{m,n}$. Note that in the calculation of $\M_z(\RKLx{m,n})$ the positive powers of $\vnumber{2}$, 
due to \eqref{eq:three-dims}, cancel against the negative powers of $\vnumber{2}$, which 
appear in \eqref{eq:the-expressions}, up to an overall factor $\vnumber{2}$.
\end{proof}

Recall that $\zetaroot=\exp(2\pi\iunit\neatfrac{1}{3})$.
 
\begin{lemma}\label{lemma:three-dim-rep2} 
Let 
$(z,\overline{z}),(z^{\prime},
\overline{z}^{\prime})\in\vanset{e}$, $z\neq z^\prime$.
\smallskip 
\begin{enumerate}

\setlength\itemsep{.15cm}

\renewcommand{\theenumi}{(\ref{lemma:three-dim-rep2}.a)}
\renewcommand{\labelenumi}{\theenumi}

\item\label{lemma:three-dim-rep2-a} $\M_z\cong\M_{z^{\prime}}$ as representations 
of $\subquo[e]$ if and only if 
$z^{\prime}=\zetaroot^{\pm 1}z$.

\renewcommand{\theenumi}{(\ref{lemma:three-dim-rep2}.b)}
\renewcommand{\labelenumi}{\theenumi}

\item\label{lemma:three-dim-rep2-b} $\M_z$ is simple if and only if $z\neq 0$.\qedhere
\end{enumerate}
\end{lemma}

\begin{proof}
\textit{\ref{lemma:three-dim-rep2-a}.}
Suppose that $z^{\prime}=\zetaroot^{\pm 1}z$. Then we have 
the following base change between $\Mt$ and $\Mt[z^{\prime}]$:
\[
\vnumber{2}
{\scriptstyle
\begin{pmatrix}
\vnumber{3} & \overline{z}^{\prime}& z^{\prime}\\
z^{\prime} & \vnumber{3}  & \overline{z}^{\prime}\\
\overline{z}^{\prime}& z^{\prime} & \vnumber{3} 
\end{pmatrix}
}
=
\vnumber{2}
{\scriptstyle
\begin{pmatrix}
\zetaroot^{\mp 1} & 0 & 0\\
0 & 1 & 0 \\
0 & 0 & \zetaroot^{\pm 1}
\end{pmatrix}
}
\!
{\scriptstyle
\begin{pmatrix}
\vnumber{3} & \overline{z}& z\\
z & \vnumber{3}  & \overline{z}\\
\overline{z} & z & \vnumber{3} 
\end{pmatrix}
}
\!
{\scriptstyle
\begin{pmatrix}
\zetaroot^{\pm 1} & 0 & 0\\
0 & 1 & 0 \\
0 & 0 & \zetaroot^{\mp 1}
\end{pmatrix}
}
.
\]
This shows that $\M_z\cong\M_{z^{\prime}}$ as $\subquo[e]$-representations. 

To see the converse, we compute the 
eigenvalues and eigenvectors of $\Mt$:
\begin{gather}\label{eq:eigen-things}
\begin{tikzpicture}[baseline=(current bounding box.center),yscale=0.6]
  \matrix (m) [matrix of math nodes, row sep=.5em, column
  sep=1em, text height=1.0ex, text depth=0.25ex, ampersand replacement=\&] {
\vnumber{2}\left(z+\overline{z}+\vnumber{3}\right) 
\&  
\vnumber{2}\left(\zetaroot^{-1} z+\zetaroot \overline{z}+\vnumber{3}\right)
\&
\vnumber{2}\left(\zetaroot z+\zetaroot^{-1}\overline{z}+\vnumber{3}\right)
\\
(1,1,1)\in\Cv^3
\& 
(1, \zetaroot,\zetaroot^{-1})\in\Cv^3 
\& 
(1,\zetaroot^{-1}, \zetaroot)\in\Cv^3  
\\};
  \draw[densely dotted] ($(m-1-2.north west) + (-.15,0)$) to ($(m-1-2.north west) + (-.15,-1.75)$);
  \draw[densely dotted] ($(m-1-3.north west) + (-.15,0)$) to ($(m-1-3.north west) + (-.15,-1.75)$);
\end{tikzpicture}
\end{gather}
Since $\vpar$ is generic, these are three non-zero eigenvalues with
three linearly independent eigenvectors showing that 
$\Mt$ can be diagonalized.

Now suppose $\M_z\cong\M_{z^{\prime}}$. Then $\Mt$ 
and $\Mt[z^{\prime}]$ must have the same eigenvalues, so 
the above implies that one of the following three triples of equations must hold.
\begin{gather*}
z+\overline{z}=z^{\prime}+\overline{z}^{\prime}
\quad\text{and}\quad
\zetaroot z + \zetaroot^{-1}\overline{z}=\zetaroot z^{\prime} + \zetaroot^{-1}\overline{z}^{\prime}
\quad\text{and}\quad
\zetaroot^{-1}z + \zetaroot \overline{z}=\zetaroot^{-1} z^{\prime} + \zetaroot \overline{z}^{\prime},
\\
z+\overline{z}=\zetaroot z^{\prime}+\zetaroot^{-1}\overline{z}^{\prime}
\quad\text{and}\quad
\zetaroot z + \zetaroot^{-1}\overline{z}=\zetaroot^{-1} z^{\prime} + \zetaroot\overline{z}^{\prime}
\quad\text{and}\quad
\zetaroot^{-1}z + \zetaroot \overline{z}=z^{\prime} +\overline{z}^{\prime},
\\
z+\overline{z}=\zetaroot^{-1} z^{\prime}+\zetaroot\overline{z}^{\prime}
\quad\text{and}\quad
\zetaroot z + \zetaroot^{-1}\overline{z}=z^{\prime} + \overline{z}^{\prime}
\quad\text{and}\quad
\zetaroot^{-1}z + \zetaroot \overline{z}=\zetaroot z^{\prime} + \zetaroot^{-1} \overline{z}^{\prime}.
\end{gather*}
One easily checks that these are satisfied 
if and only if 
$z=z^{\prime}$ (top triple), 
$z=\zetaroot z^{\prime}$ (middle triple)
or $z=\zetaroot^{-1} z^{\prime}$ (bottom triple).
\newline

\noindent\textit{\ref{lemma:three-dim-rep2-b}.}
In case $z=0$, one clearly has
\[
\M_0\cong\M_{\vfrac{3},0,0}\oplus\M_{0,\vfrac{3},0}\oplus\M_{0,0,\vfrac{3}}, 
\] 
where the one-dimensional representations were defined in \eqref{eq:the-one-dims}.

Now suppose that $z\neq 0$ and that $\M_z$ is reducible. 
Then it must have a subrepresentation of dimension one or two. 
The explicit description of 
the eigenvalues and eigenvectors of $\Mt$ 
from \eqref{eq:eigen-things} shows that this is impossible.

To see this, first note that
the restriction of $\Mt$ 
to the vector space underlying the potential subrepresentation 
would be diagonalizable as well.

Secondly, in case the eigenvalues in \eqref{eq:eigen-things} 
are all distinct, at least one eigenvector therein is also an eigenvector 
for the restriction.
However, applying $\Mg$, $\Mo$ and $\Mp$ to 
any of the three eigenvectors in \eqref{eq:eigen-things} gives three linear independent vectors, which 
shows that no subrepresentation can exist in case of distinct eigenvalues.

Thirdly, assume that 
two of the three eigenvalues in \eqref{eq:eigen-things} coincide. 
Then there must exist a linear combination of the corresponding two eigenvectors which is an eigenvector 
for the restriction. Applying $\Mg$, $\Mo$ and 
$\Mp$ to that eigenvector would give 
three linear independent vectors, as can easily be checked. We get a contradiction again.

Finally, since $z\neq 0$, not all eigenvalues 
in \eqref{eq:eigen-things} can be equal, so we are done.
\end{proof}

Recall from the proof 
of \fullref{lemma:level-vanishing} the functions $Z$ 
and $\overline{Z}$, which map $D$ bijectively onto 
the discoid $\discoid$, and which determine $\vanset{e}$.
If $e\not\equiv 0\bmod 3$, then 
$Z(\sigma,\tau)\neq 0$ for all $(\sigma,\tau)$ as in \eqref{eq:zeros2}. 
By \fullref{lemma:three-dim-rep2}, this implies 
that the total number of pairwise 
non-isomorphic $\M_z$ is 
equal to $\neatfrac{t_e}{3}$.
If $e\equiv 0\bmod 3$, then that number 
is equal to $\neatfrac{(t_e-1)}{3}$, because 
$Z(\sigma,\tau)=0$ if and only if $2k+l=e=k+2l$ if and only if
$k=l=\neatfrac{e}{3}$.

Summarized, we have the following 
non-isomorphic, simple 
$\subquo[e]$-representations:
\begin{gather}\label{eq:the-simples}
\begin{tikzpicture}[baseline=(current bounding box.center),yscale=0.6]
  \matrix (m) [matrix of math nodes, row sep=1em, column
  sep=1em, text height=1.5ex, text depth=0.25ex, ampersand replacement=\&] {
\phantom{a}
\& 
e\equiv 0\bmod 3 
\&  
e\not\equiv 0\bmod 3 
\\
\text{one-dim.}
\&
\begin{gathered}
\M_{0,0,0},\,\M_{\vfrac{3},0,0},
  \\
\M_{0,\vfrac{3},0},\,\M_{0,0,\vfrac{3}}
\end{gathered} 
\& 
\M_{0,0,0}  
\\
\text{quantity}
\&
4
\& 
1
\\
\text{three-dim.}
\&
\M_z,
\;
(z,\overline{z})\in\vanset{e}^{\zetaroot}-\{(0,0)\}
\& 
\M_z,
\;
(z,\overline{z})\in\vanset{e}^{\zetaroot}
\\
\text{quantity}
\&
\neatfrac{(t_e-1)}{3}
\& 
\neatfrac{t_e}{3}
\\};
  \draw[densely dashed] ($(m-1-1.south west)+ (-.75,0)$) to ($(m-1-3.south east)+ (.4,0)$);
  \draw[densely dashed] ($(m-3-1.south west)+ (-.25,-.25)$) to ($(m-3-3.south east)+ (1.25,-.25)$);
  \draw[densely dashed] ($(m-1-3.north west) + (-.6,0)$) to ($(m-1-3.north west) + (-.6,-7.2)$);
  \draw[densely dashed] ($(m-1-2.north west) + (-1.4,0)$) to ($(m-1-2.north west) + (-1.4,-7.2)$);
\end{tikzpicture}
\end{gather}
Here $\vanset{e}^{\zetaroot}$ denotes the set of $\zeethree$-orbits in $\vanset{e}$ 
under the action $(z,\overline{z})\mapsto (\zetaroot z,\zetaroot^{-1}\overline{z})$.

\begin{example}\label{example:three-dims}
By \eqref{eq:the-simples}, the three-dimensional simple
representation of $\subquo[e]$ are indexed by the $\zeethree$-orbits of points 
in the interior of $\discoid$ (cf. \fullref{example:plot-zeros}.), e.g.:
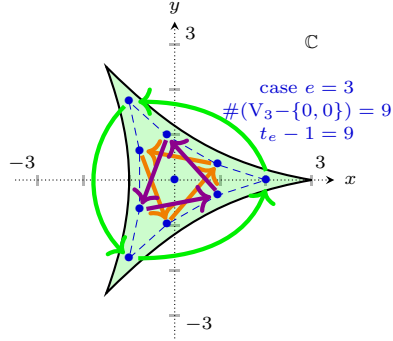
\[
\begin{tikzpicture}[anchorbase, scale=.6, tinynodes]
\draw[thin, marked=.0, marked=.166, marked=.333, marked=.666, marked=.833, marked=1.0, white] (0,-3) to (0,3);
\draw[thin, marked=.0, marked=.166, marked=.333, marked=.666, marked=.833, marked=1.0, white] (-3,0) to (3,0);
\draw[thick, white, fill=mygreen, opacity=.2] (3,0) to [out=170, in=315] (-1.5,2.5) to [out=290, in=70] (-1.5,-2.5) to [out=45, in=190] (3,0);
\draw[thin, densely dotted, ->, >=stealth] (-3.5,0) 
to (-3.35,0) node [above] {$-3$}
to (3.2,0) node [above] {$3$}
to (3.5,0) node[right] {$x$};
\draw[thin, densely dotted, ->, >=stealth] (0,-3.5) 
to (0,-3.2) node [right] {$-3$}
to (0,3.2) node [right] {$3$}
to (0,3.5) node[above] {$y$};
\draw[thick] (3,0) to [out=170, in=315] (-1.5,2.5) to [out=290, in=70] (-1.5,-2.5) to [out=45, in=190] (3,0);
\node at (3,3) {$\C$};
\node[myblue] at (0,0) {$\bullet$};
\node[myblue] at (2,0) {$\bullet$};
\node[myblue] at (-1,1.73) {$\bullet$};
\node[myblue] at (-1,-1.73) {$\bullet$};
\node[myblue] at (-.77,.64) {$\bullet$};
\node[myblue] at (-.77,-.64) {$\bullet$};
\node[myblue] at (-.17,.98) {$\bullet$};
\node[myblue] at (-.17,-.98) {$\bullet$};
\node[myblue] at (.94,.34) {$\bullet$};
\node[myblue] at (.94,-.34) {$\bullet$};
\draw[very thin, densely dashed, myblue] (2,0) to (.94,.34) to (-.17,.98) to (-1,1.73) to (-.77,.64) to (-.77,-.64) to (-1,-1.73) to (-.17,-.98) to (.94,-.34) to (2,0);
\node[myblue] at (2.9,2) {case $e=3$};
\node[myblue] at (2.9,1.5) {$\#(\vanset{3}{-}\{0,0\})=9$};
\node[myblue] at (2.9,1) {$t_e-1=9$};
\draw[ultra thick, mygreen, ->] (2,.2) to [out=110, in=0] (-.8,1.73);
\draw[ultra thick, mygreen, ->] (-1.1,1.63) to [out=225, in=135] (-1.1,-1.63);
\draw[ultra thick, mygreen, ->] (-.8,-1.73) to [out=0, in=250] (2,-.2);
\draw[ultra thick, myorange, ->] (.8,.37) to (-.62,.64);
\draw[ultra thick, myorange, ->] (-.72,.54) to (-.19,-.88);
\draw[ultra thick, myorange, ->] (-.1,-.88) to (.89,.28);
\draw[ultra thick, mypurple, ->] (-.19,.88) to (-.72,-.54);
\draw[ultra thick, mypurple, ->] (-.62,-.64) to (.8,-.37);
\draw[ultra thick, mypurple, ->] (.9,-.23) to (-.1,.88);
\end{tikzpicture}
\]
Here the arrows indicate the $\zeethree$-symmetry. 
\end{example}

We are now ready to provide a classification of simple 
$\subquo[e]$-representations.

\begin{theorem}\label{theorem:classification-simples}
The table \eqref{eq:the-simples} gives a complete, irredundant list of 
simple $\subquo[e]$-representa\-tions. 
Furthermore, the algebra $\subquo[e]$ is semisimple. 
\end{theorem}

\begin{proof}
By \fullref{proposition:dimension} 
the algebra $\subquo[e]$ is of dimension 
$3\tfrac{(e+1)(e+2)}{2}+1$. 
From the representation theory of finite-dimensional algebras we thus have
\begin{gather}\label{eq:ineq-semisimple}
\dim_{\Cv}\subquo[e]
=
3\tfrac{(e+1)(e+2)}{2}+1
=
3t_{e}+1 
\geq 
{\textstyle\sum_{\M}}\,(\dim_{\Cv}\M)^2,
\end{gather}
where the sum is taken over any set 
of pairwise non-isomorphic, simple 
$\subquo[e]$-representations $\M$. If equality holds in 
\eqref{eq:ineq-semisimple} for such a set, then that set 
is complete and $\subquo[e]$ is semisimple. 
\newline

\noindent\textit{Case $e\not\equiv 0\bmod 3$.} We 
use the data 
from \eqref{eq:the-simples} in 
\eqref{eq:ineq-semisimple} and obtain
\[
{\textstyle\sum_{\M}}\,(\dim_{\Cv}\M)^2
=
\neatfrac{t_e}{3}
\cdot
3^2
+
1\cdot 
1^2
=
3t_{e}+1
,
\] 
which shows both statements.
\newline

\noindent\textit{Case $e\equiv 0\bmod 3$.} 
Similarly, we compute
\[
{\textstyle\sum_{\M}}\,(\dim_{\Cv}\M)^2
=
\neatfrac{(t_e-1)}{3}
\cdot 3^2+4\cdot 1^2
=
3(t_e-1)+4=
3t_{e}+1
,
\] 
which again shows both statements.
\end{proof}

\subsection{Generalizing dihedral Hecke algebras}\label{subsec:dihedral-group}

We finish this section by listing some analogies to the dihedral case. 
The crucial link between the dihedral and the trihedral case is the following:
The $\sltwo$-version of the polynomial $\pxy{m,n}$ from \fullref{subsec:opolys}
is the Chebyshev polynomial $\pxy[\fu]{m}$ (normalized and of the second kind). 
Using the convention that the $\pxy[\fu]{m}$ are zero for negative subscripts, 
they satisfy the recursion relation  
\[
\pxy[\fu]{0}=1,
\quad
\pxy[\fu]{1}=\fu,
\quad
\fu\pxy[\fu]{m}=
\pxy[\fu]{m+1}+\pxy[\fu]{m-1}.
\] 
Here, $\fu$ corresponds to the 
fundamental representation of $\sltwo$. 
The analog of the discoid $\discoid$ from \eqref{eq:deltoid} 
is the interval $\deltoid_{2}=[-2,2]$, whose boundary is the 
pair of primitive, complex second roots of unity, multiplied by $2$. (Note the evident 
$\zeetwo$-symmetry of $\deltoid_{2}$.)

\begin{dihedral}\label{remark:dihedral-group1}
Let $\dihquo=\hecke(\typei[\infty])=\hecke(\typeat{1})$ denote the dihedral 
Hecke algebra of the infinite dihedral group, i.e. 
the Weyl group of affine type $\typea{1}$, 
and $\hecke(\typei)$ the dihedral Hecke algebra of dimension $2(e+2)$, 
which is of finite Coxeter type $\typei$.
The first analogy of our story to the dihedral case 
is provided by \fullref{lemma:quotient-of-affine}, the difference 
being that the trihedral Hecke algebra 
is a proper subalgebra of $\hecke$. 
The entries of the change-of-basis matrix from the 
(colored) KL basis to the Bott--Samelson basis of $\dihquo$ are precisely the 
coefficients of the polynomials $\pxy[\fu]{m}$, see for example \cite[Section 2.2]{El2}.  
\end{dihedral}

\begin{dihedral}\label{remark:dihedral-group-cells}
By \fullref{proposition:cells},
all non-trivial cells of $\subquo$ are infinite, and there 
are three non-trivial left and right 
cells, one for each $\tduc\in\Seset$, whose disjoint union 
forms the only non-trivial two-sided cell. This is
another analogy to the dihedral case: 
the algebra $\dihquo$ has two non-trivial left and right cells, one for each of its
Coxeter generators, whose disjoint union forms 
the only non-trivial two-sided cell.
\end{dihedral}

\begin{dihedral}\label{remark:dihedral-group2} 
Let $\dihquo[e]$ denote the small quotient of $\hecke(\typei)$, 
obtained by killing the top cell. \fullref{subsec:quotient-algebra-1} provides the 
third analogy: $\dihquo[e]$ can be obtained as a quotient of $\dihquo[\infty]$ 
by the ideal generated by the two elements related to the irreducible 
$\sltwo$-module $\Ll_{e+1}$ under the quantum Satake correspondence; the non-trivial 
left cells of $\dihquo[e]$ have order $e+1$ and $\dim_{\Cv}\dihquo[e]=2(e+2)-1=2e+3$.
\end{dihedral}

\begin{dihedral}\label{remark:dihedral-group3}
\fullref{theorem:classification-simples} 
provides another analogy to the dihedral case: 
$\dihquo[e]$ is semisimple over $\C$, and all of its 
simples are either one- or two-dimensional, 
with the number of their isomorphism classes 
depending on whether $e\equiv 0\bmod 2$ or 
$e\equiv 1\bmod 2$. 
Analogously to \eqref{eq:three-dims}, the two-dimensional simples 
can be defined by matrices whose off-diagonal, non-zero entries are the roots of the 
Chebyshev polynomials $\pxy[\fu]{e+1}$, i.e. 
its (colored) KL generators are send to
\scalebox{.8}{$\begin{psmallmatrix}
\vnumber{2} & z\\
0 & 0
\end{psmallmatrix}$} and 
\scalebox{.8}{$\begin{psmallmatrix}
0 & 0\\
\overline{z} & \vnumber{2}
\end{psmallmatrix}$},
where $z=\overline{z}$ is a root of $\pxy[\fu]{e+1}$.
\end{dihedral}
%
\section{Trihedral Soergel bimodules}\label{sec:A2-diagrams}

The purpose of this section is to categorify the 
trihedral Hecke algebras $\subquo$ and $\subquo[e]$ 
from \fullref{section:funny-algebra}, where $e$ still denotes the level. 
As before, we have collected some analogies to the 
dihedral case at the end of the section, cf. \fullref{subsec:dihedral-SB}

\subsection{Bott--Samelson bimodules for affine \texorpdfstring{$\typea{2}$}{A2}}\label{subsec:sbim}

First, we recall the diagrammatic $2$-category 
$\Adiag$ from \cite[Section 3.3]{El1}. We 
call it the ($2$-category of) 
singular Bott--Samelson bimodules of affine type $\typea{2}$.

\subsubsection{\texorpdfstring{$2$}{2}-categorical conventions}\label{subsubsec:2-cat-conv}

For generalities and terminology on $2$-categories, we 
refer for example to \cite{Le1} or \cite{McL1}.

\begin{convention}\label{convention:generated}
We use $2$-categories given by generators and relations.
This means that 
$1$-morphisms are obtained by 
compositions $\circ$ of the generating $1$-morphisms, and 
$2$-morphisms are obtained by horizontal $\hcomp$ and vertical 
$\vcomp$ compositions of the $2$-generators whenever this makes sense.
(In particular, the interchange law leads to additional relations in our 
$2$-categories, called height relations.)
Relations are supposed to hold between $2$-morphisms.
Details about such $2$-categories can be found e.g. in \cite[Section 2.2]{Ro1}.
\end{convention}

\begin{convention}\label{convention:reading}
We read $1$-morphisms from right to left, using the operator-notation, and 
$2$-morphisms from bottom to top and right to left. 
These conventions are illustrated in \fullref{definition:ssbim-free} below.
Note that we usually omit the $1$-morphisms in the pictures, 
and we will simplify diagrams by drawing them in a more 
topological fashion, using e.g. \fullref{example:more-gens}.
\end{convention}

\begin{convention}\label{convention:grading}
A ($\Z$-)graded $2$-category for us is a $2$-category 
whose $2$-hom spaces are ($\Z$-)graded, meaning that the 
$2$-generators have a given degree, the relations are homogeneous 
and the degree is additive under horizontal and vertical 
composition. Moreover, $1$-morphisms are formal shifts 
of generating $1$-morphisms, indicated by $\{a\}$ for $a\in\Z$, so 
there is a formal $\Z$-action on $1$-morphisms
such that 
$\{k\}(\morstuff{F}\{a\})=\morstuff{F}\{a+k\}$ for all $k\in\Z$. Finally, 
a $2$-morphism 
$\twomorstuff{f}\colon\morstuff{F}\{0\}\to\morstuff{G}\{0\}$, 
homogeneous of degree $d$, is of degree $d-a+b$ seen as 
a $2$-morphism $\twomorstuff{f}\colon\morstuff{F}\{a\}\to\morstuff{G}\{b\}$. 
For more information on 
such $2$-categories, see e.g. \cite[Section 5.1]{Lau}.
\end{convention}

\subsubsection{The definition of \texorpdfstring{$\Adiag$}{singSbim}}\label{subsubsec:def-adiag}

Let $\aformq=\C[\qpar,\qpar^{-1}]$ and $\Rbim=\aformq[\rootb,\rootr,\rooty]$, 
where $\rootb,\rootr,\rooty$ are formal variables.
We define an action of the affine Weyl group $\Wgroup$ from 
\fullref{subsec:weyl-group} on $\Rbim$:
\begin{gather}\label{eq:sl3-exotic-action}
\begin{tikzpicture}[baseline=(current bounding box.center)]
  \matrix (m) [matrix of math nodes, nodes in empty cells, row sep={0.5cm,between origins}, column
   sep={2.25cm,between origins}, text height=1.6ex, text depth=0.25ex, ampersand replacement=\&] {
\phantom{\bc}    \& \rootb \& \rootr \& \rooty\\
\bc \& -\rootb \& \rootb+\rootr \& \qpar^{-1}\rootb+\rooty\\
\rc \& \rootb+\rootr \& -\rootr \& \qpar\rootr+\rooty\\
\yc \& \rootb+\qpar\rooty \& \rootr+\qpar^{-1}\rooty \& -\rooty\\
    };
\draw[densely dashed] ($(m-2-1.north)+(-.15,0)$) edge ($(m-2-4.north)+(1.05,0)$);
\draw[densely dashed] ($(m-3-1.north)+(-.15,0)$) edge ($(m-3-4.north)+(1.05,0)$);
\draw[densely dashed] ($(m-4-1.north)+(-.15,0)$) edge ($(m-4-4.north)+(1.05,0)$);
\draw[densely dashed] ($(m-1-1.east)+(1.1,.2)$) edge ($(m-4-1.east)+(1.1,-.3)$);
\draw[densely dashed] ($(m-1-1.east)+(3,.2)$) edge ($(m-4-1.east)+(3,-.3)$);
\draw[densely dashed] ($(m-1-1.east)+(5.4,.2)$) edge ($(m-4-1.east)+(5.4,-.3)$);
\end{tikzpicture}
\end{gather}
One easily checks that \eqref{eq:sl3-exotic-action} is well-defined. This also 
gives rise to an action of the secondary colors on $\Rbim$
by using \eqref{eq:sec-Weyl} (recalling that e.g. $\gc=\{\bc,\yc\}$). Thus, we can define:

\begin{definition}\label{definition:thin-invariant-subrings}
For any $\duc\in\Prset$ and $\tduc\in\Seset$, let $\Rbim^{\duc}$ and $\Rbim^{\tduc}$ 
denote the subrings of $\Rbim$ consisting of all $\duc$-invariant and $\tduc$-invariant elements,
respectively.
\end{definition}

Recall that we always use $\tduc,\tdudc\in\Seset$ as secondary dummy colors, 
and we also use the primary dummy colors $\duc,\dudc\in\Prset$ from now on. 
Moreover, identifying our colors with proper subsets of $\Prset$, 
including the empty subset, we say that 
two of them are compatible if one is a subset of the 
other, e.g. as the colors connected by an edge below. 
\begin{gather}\label{eq:color-compatible}
\begin{tikzpicture}[baseline=(current bounding box.center), tinynodes]
  \matrix (m) [matrix of math nodes, nodes in empty cells, row sep=.02cm, column
  sep=.01cm, text height=1.6ex, text depth=0.25ex, ampersand replacement=\&] {
{\color{mygray}\scalebox{.99}{$p$}} \& \scalebox{.99}{$\gc$}\, \& {\color{mygray}\scalebox{.99}{$o$}} \\
\,\scalebox{.99}{$\bc$} \& {\color{mygray}\scalebox{.99}{$r$}} \& \scalebox{.99}{$\yc$} \\
\& \scalebox{.99}{$\wc$} \&  \\
    };
	\draw[thin] ($(m-2-1.north)+(0,-.2cm)$) to ($(m-1-2.south)+(-.05,.1cm)$);
	\draw[thin] ($(m-2-3.north)+(0,-.2cm)$) to ($(m-1-2.south)+(.05,.1cm)$);
	\draw[thin] ($(m-3-2.north)+(-.05,-.2cm)$) to ($(m-2-1.south)+(0,.1cm)$);
	\draw[thin] ($(m-3-2.north)+(.05,-.2cm)$) to ($(m-2-3.south)+(0,.1cm)$);
\end{tikzpicture}
,\quad\quad
\begin{tikzpicture}[baseline=(current bounding box.center), tinynodes]
  \matrix (m) [matrix of math nodes, nodes in empty cells, row sep=.02cm, column
  sep=.01cm, text height=1.6ex, text depth=0.25ex, ampersand replacement=\&] {
{\color{mygray}\scalebox{.99}{$p$}} \& {\color{mygray}\scalebox{.99}{$g$}} \& \scalebox{.99}{$\oc$} \\
{\color{mygray}\scalebox{.99}{$b$}} \& \scalebox{.99}{$\rc$} \& \scalebox{.99}{$\yc$} \\
\& \scalebox{.99}{$\wc$}\&  \\
    };
	\draw[thin] ($(m-2-3.north)+(.035,-.2cm)$) to ($(m-1-3.south)+(.035,.1cm)$);
	\draw[thin] ($(m-3-2.north)+(0,-.2cm)$) to ($(m-2-2.south)+(0,.1cm)$);
	\draw[thin] ($(m-3-2.north)+(.05,-.2cm)$) to ($(m-2-3.south)+(0,.1cm)$);
	\draw[thin] ($(m-2-2.north)+(.05,-.2cm)$) to ($(m-1-3.south)+(0,.1cm)$);
\end{tikzpicture}
,\quad\quad
\begin{tikzpicture}[baseline=(current bounding box.center), tinynodes]
  \matrix (m) [matrix of math nodes, nodes in empty cells, row sep=.02cm, column
  sep=.01cm, text height=1.6ex, text depth=0.25ex, ampersand replacement=\&] {
\scalebox{.99}{$\pc$} \& {\color{mygray}\scalebox{.99}{$g$}} \& {\color{mygray}\scalebox{.99}{$o$}} \\
\scalebox{.99}{$\bc$} \& \scalebox{.99}{$\rc$} \& {\color{mygray}\scalebox{.99}{$y$}} \\
\& \scalebox{.99}{$\wc$} \&  \\
    };
	\draw[thin] ($(m-2-1.north)+(-.035,-.2cm)$) to ($(m-1-1.south)+(-.035,.1cm)$);
	\draw[thin] ($(m-3-2.north)+(-.05,-.2cm)$) to ($(m-2-1.south)+(0,.1cm)$);
	\draw[thin] ($(m-3-2.north)+(0,-.2cm)$) to ($(m-2-2.south)+(0,.1cm)$);
	\draw[thin] ($(m-2-2.north)+(-.05,-.2cm)$) to ($(m-1-1.south)+(0,.1cm)$);
\end{tikzpicture}
\end{gather}

\begin{example}\label{example:colors-compatible}
The color $\bc$ is compatible with $\wc$, $\gc$ and $\pc$, but not with $\rc$, $\yc$ or $\oc$.
\end{example}

We will define the $2$-category of singular Soergel bimodules 
as a quotient of the following $2$-category, which we view as a free version 
of it.

\begin{definition}\label{definition:ssbim-free}
Let $\ADiag$ be the $2$-category defined as follows.
\medskip

\noindent\textit{\setword{`Objects of $\ADiag$'}{sbim-objects}.}
The objects are proper subsets of $\Prset=\{\bc,\yc,\rc\}$, 
including the empty subset $\wc$.  
The one-element subsets are identified with 
$\bc,\yc,\rc$, the two-element subsets are 
identified with $\gc,\oc,\pc$, using 
the color conventions from \fullref{subsec:our-color-code}.  
\medskip

\noindent\textit{\setword{`$1$-morphisms of $\ADiag$'}{sbim-morphisms}.}
By definition, there is one generating $1$-morphism 
for each pair of distinct compatible colors.
Namely, including all other 
compatible variations using the conventions from \eqref{eq:color-compatible} 
and writing e.g. $\bc\wc=\bc\circ\wc$ for short:
\begin{gather*}
\xy
(0,0)*{
\wc\bc\colon\wc\leftarrow\bc,
\quad
\bc\wc\colon\bc\leftarrow\wc,
\quad
\bc\gc\colon\bc\leftarrow\gc,
\quad
\yc\gc\colon\yc\leftarrow\gc,
\quad
\gc\bc\colon\gc\leftarrow\bc,
\quad
\gc\yc\colon\gc\leftarrow\yc,
\quad
\text{etc.}};
(0,-4.5)*{\text{{\tiny compatible as in \eqref{eq:color-compatible}}}};
\endxy
\end{gather*}
\smallskip

\noindent\textit{\setword{`$2$-morphisms of $\ADiag$'}{sbim-two-morphisms}.}
The $2$-morphisms are generated by 
two kinds of $2$-generators.
The first kind are
cups, caps and crossings given as follows.
\begin{gather}\label{eq:gens-sbim-1}
\xy
(0,0)*{
\begin{tikzpicture}[anchorbase, scale=.4, tinynodes]
	\draw[very thin, densely dotted, fill=white] (0,0) to [out=90, in=180] (1,1) to [out=0, in=90] (2,0) to (3,0) to (3,2) to (-1,2) to (-1,0) to (0,0);
	\fill[myblue, opacity=0.3] (0,0) to [out=90, in=180] (1,1) to [out=0, in=90] (2,0) to (0,0);
	\draw[bstrand, directed=.999] (0,0) to [out=90, in=180] (1,1) to [out=0, in=90] (2,0);
	\node at (-.5,-.375) {$\wc$};
	\node at (1,2.3) {$\wc$};
	\node at (1,-.375) {$\bc$};
	\node at (2.5,-.375) {$\wc$};
\end{tikzpicture}
\colon
\begin{matrix}
\wc
\\
\Uparrow
\\
\wc\bc\wc
\end{matrix}};
(-3.6,-8)*{\text{{\tiny degree $1$}}};
\endxy
,\quad\quad
\xy
(0,0)*{
\begin{tikzpicture}[anchorbase, scale=.4, tinynodes]
	\draw[very thin, densely dotted, fill=white] (0,0) to [out=90, in=180] (1,1) to [out=0, in=90] (2,0) to (0,0);
	\fill[myblue, opacity=0.3] (0,0) to [out=90, in=180] (1,1) to [out=0, in=90] (2,0) to (3,0) to (3,2) to (-1,2) to (-1,0) to (0,0);
	\draw[bstrand, directed=.999] (2,0) to [out=90, in=0] (1,1) to [out=180, in=90] (0,0);
	\node at (-.5,-.375) {$\bc$};
	\node at (1,2.3) {$\bc$};
	\node at (1,-.375) {$\wc$};
	\node at (2.5,-.375) {$\bc$};
\end{tikzpicture}};
(0,-8)*{\text{{\tiny degree $-1$}}};
\endxy
,\quad\quad
\xy
(0,0)*{
\begin{tikzpicture}[anchorbase, scale=.4, tinynodes]
	\fill[myorange, opacity=0.8] (0,2) to [out=270, in=180] (1,1) to [out=0, in=270] (2,2) to (0,2);
	\fill[myyellow, opacity=0.3] (0,2) to [out=270, in=180] (1,1) to [out=0, in=270] (2,2) to (3,2) to (3,0) to (-1,0) to (-1,2) to (0,2);
	\draw[rstrand, directed=.999] (2,2) to [out=270, in=0] (1,1) to [out=180, in=270] (0,2);
	\node at (-.5,2.3) {$\yc$};
	\node at (1,-.375) {$\yc$};
	\node at (1,2.3) {$\oc$};
	\node at (2.5,2.3) {$\yc$};
\end{tikzpicture}};
(0,-8)*{\text{{\tiny degree $2$}}};
\endxy
,\quad\quad
\xy
(0,0)*{
\begin{tikzpicture}[anchorbase, scale=.4, tinynodes]
	\fill[myorange, opacity=0.8] (0,2) to [out=270, in=180] (1,1) to [out=0, in=270] (2,2) to (3,2) to (3,0) to (-1,0) to (-1,2) to (0,2);
	\fill[myyellow, opacity=0.3] (0,2) to [out=270, in=180] (1,1) to [out=0, in=270] (2,2) to (0,2);
	\draw[rstrand, directed=.999] (0,2) to [out=270, in=180] (1,1) to [out=0, in=270] (2,2);
	\node at (-.5,2.3) {$\oc$};
	\node at (1,-.375) {$\oc$};
	\node at (1,2.3) {$\yc$};
	\node at (2.5,2.3) {$\oc$};
\end{tikzpicture}};
(0,-8)*{\text{{\tiny degree $-2$}}};
\endxy
,\quad\quad
\xy
(0,0)*{
\begin{tikzpicture}[anchorbase, scale=.4, tinynodes]
	\fill[mypurple, opacity=0.8] (3,2) to (2,2) to [out=270, in=0] (1,1) to [out=0, in=90] (2,0) to (3,0) to (3,2);
	\draw[very thin, densely dotted, fill=white] (-1,2) to (0,2) to [out=270, in=180] (1,1) to [out=180, in=90] (0,0) to (-1,0) to (-1,2);
	\fill[myblue, opacity=0.3] (1,1) to [out=0, in=90] (2,0) to (0,0) to [out=90, in=180] (1,1);
	\fill[myred, opacity=0.3] (1,1) to [out=0, in=270] (2,2) to (0,2) to [out=270, in=180] (1,1);
	\draw[rstrand, directed=.999] (2,0) to [out=90, in=0] (1,1) to [out=180, in=270] (0,2);
	\draw[bstrand, directed=.999] (0,0) to [out=90, in=180] (1,1) to [out=0, in=270] (2,2);
	\node at (-.5,2.3) {$\wc$};
	\node at (-.5,-.375) {$\wc$};
	\node at (1,2.3) {$\rc$};
	\node at (1,-.375) {$\bc$};
	\node at (2.5,2.3) {$\pc$};
	\node at (2.5,-.375) {$\pc$};
\end{tikzpicture}};
(0,-8)*{\text{{\tiny degree $0$}}};
\endxy
\end{gather}
(We frame $\wc$-colored regions for readability.)
The generators displayed in \eqref{eq:gens-sbim-1} are all generators up to
colored variations: each strand separates 
two regions colored by subsets of $\Prset$ that differ by a primary color, 
which is used to color that strand. 
The strands are oriented such that the region colored by 
the smaller subset of $\Prset$ lies to their left.

The second kind of $2$-generators are decorations 
of the regions by polynomials in $\Rbim$ that are invariant under the parabolic subgroup corresponding 
to the color of the region, i.e. 
\begin{gather}\label{eq:gens-sbim-2}
\begin{tikzpicture}[anchorbase, scale=.4, tinynodes]
	\draw[very thin, densely dotted, fill=white] (0,0) to (2,0) to (2,2) to (0,2) to (0,0);
	\node at (1,1) {$\polybox{p}$};
	\node at (1,-.5) {$p{\in}\Rbim^{\wc}{=}\Rbim$};
\end{tikzpicture}
,\quad
\begin{tikzpicture}[anchorbase, scale=.4, tinynodes]
	\fill[myblue, opacity=0.3] (0,0) to (2,0) to (2,2) to (0,2) to (0,0);
	\node at (1,1) {$\polybox{p}$};
	\node at (1,-.5) {$p\in\Rbim^{\bc}$};
\end{tikzpicture}
,\quad
\begin{tikzpicture}[anchorbase, scale=.4, tinynodes]
	\fill[myred, opacity=0.3] (0,0) to (2,0) to (2,2) to (0,2) to (0,0);
	\node at (1,1) {$\polybox{p}$};
	\node at (1,-.5) {$p\in\Rbim^{\rc}$};
\end{tikzpicture}
,\quad
\begin{tikzpicture}[anchorbase, scale=.4, tinynodes]
	\fill[myyellow, opacity=0.3] (0,0) to (2,0) to (2,2) to (0,2) to (0,0);
	\node at (1,1) {$\polybox{p}$};
	\node at (1,-.5) {$p\in\Rbim^{\yc}$};
\end{tikzpicture}
,\quad
\begin{tikzpicture}[anchorbase, scale=.4, tinynodes]
	\fill[mygreen, opacity=0.8] (0,0) to (2,0) to (2,2) to (0,2) to (0,0);
	\node at (1,1) {$\polybox{p}$};
	\node at (1,-.5) {$p\in\Rbim^{\gc}$};
\end{tikzpicture}
,\quad
\begin{tikzpicture}[anchorbase, scale=.4, tinynodes]
	\fill[myorange, opacity=0.8] (0,0) to (2,0) to (2,2) to (0,2) to (0,0);
	\node at (1,1) {$\polybox{p}$};
	\node at (1,-.5) {$p\in\Rbim^{\oc}$};
\end{tikzpicture}
,\quad
\begin{tikzpicture}[anchorbase, scale=.4, tinynodes]
	\fill[mypurple, opacity=0.8] (0,0) to (2,0) to (2,2) to (0,2) to (0,0);
	\node at (1,1) {$\polybox{p}$};
	\node at (1,-.5) {$p\in\Rbim^{\pc}$};
\end{tikzpicture}
\end{gather}
The polynomials are allowed to move around as long as they do not cross any strand.
\medskip

\noindent\textit{\setword{`Grading on $\ADiag$'}{sbim-grading}.}
We endow $\ADiag$ with the structure of a graded $2$-category 
by giving the generators from \eqref{eq:gens-sbim-1} and 
\eqref{eq:gens-sbim-2}
the following degree.

\begin{enumerate}[label=$\blacktriangleright$]

\setlength\itemsep{.15cm}

\item Clockwise cups and caps between $\wc$ and $\duc$ have degree $1$, while 
their anticlockwise counterparts have degree $-1$.

\item Clockwise cups and caps between $\duc$ and a compatible $\tduc$ have degree $2$, while 
their anticlockwise counterparts have degree $-2$.

\item Crossings are of degree $0$.

\item Homogeneous polynomials are graded by twice their polynomial degree, i.e. the 
formal variables $\rootb,\rootr,\rooty$ are of degree $2$.

\end{enumerate}
We have indicate some of these in \eqref{eq:gens-sbim-1}. 
\end{definition}

\begin{example}\label{example:2-cat-con-a}
In general, a $1$-morphism is a finite 
string of generating $1$-morphisms, which are 
indicated by their source and target, e.g. 
$\yc\oc\rc\pc\bc\wc\colon\yc\leftarrow\oc\leftarrow\rc\leftarrow\pc\leftarrow\bc\leftarrow\wc$. 
(By convention, we identify the objects $\duc,\tduc$ with the 
identity $1$-morphisms on them.)
Furthermore, 
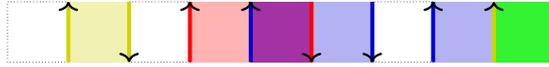
\[
\begin{tikzpicture}[anchorbase, scale=.4, tinynodes]
	\draw[very thin, densely dotted, fill=white] (-6,0) to (-8,0) to (-8,2) to (-6,2) to (-6,0);
	\draw[very thin, densely dotted, fill=white] (-2,0) to (-4,0) to (-4,2) to (-2,2) to (-2,0);
	\draw[very thin, densely dotted, fill=white] (4,0) to (6,0) to (6,2) to (4,2) to (4,0);
	\fill[myyellow, opacity=0.3] (-4,0) to (-4,2) to (-6,2) to (-6,0) to (0-4,0);
	\fill[myred, opacity=0.3] (0,0) to (0,2) to (-2,2) to (-2,0) to (0,0);
	\fill[myblue, opacity=0.3] (2,0) to (4,0) to (4,2) to (2,2) to (2,0);
	\fill[myblue, opacity=0.3] (6,0) to (8,0) to (8,2) to (6,2) to (6,0);
	\fill[mygreen, opacity=0.8] (8,0) to (8,2) to (10,2) to (10,0) to (8,0);
	\fill[mypurple, opacity=0.8] (0,0) to (0,2) to (2,2) to (2,0) to (0,0);
	\draw[ystrand, directed=.999] (-6,0) to (-6,2);
	\draw[ystrand, directed=.999] (-4,2) to (-4,0);
	\draw[ystrand, directed=.999] (8,0) to (8,2);
	\draw[rstrand, directed=.999] (-2,0) to (-2,2);
	\draw[rstrand, directed=.999] (2,2) to (2,0);
	\draw[bstrand, directed=.999] (0,0) to (0,2);
	\draw[bstrand, directed=.999] (4,2) to (4,0);
	\draw[bstrand, directed=.999] (6,0) to (6,2);
\end{tikzpicture}
\]
is an example of the 
coloring of facets 
and strands.
\end{example}

\begin{example}\label{example:more-gens}
As usual, one can define sideways crossings, e.g.
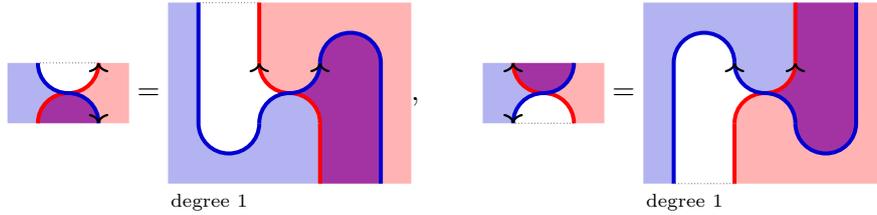
\[
\xy
(0,0)*{
\begin{tikzpicture}[anchorbase, scale=.4, tinynodes]
	\draw[very thin, densely dotted, fill=white] (1,1) to [out=0, in=270] (2,2) to (0,2) to [out=270, in=180] (1,1);
	\fill[myred, opacity=0.3] (3,2) to (2,2) to [out=270, in=0] (1,1) to [out=0, in=90] (2,0) to (3,0) to (3,2);
	\fill[myblue, opacity=0.3] (-1,2) to (0,2) to [out=270, in=180] (1,1) to [out=180, in=90] (0,0) to (-1,0) to (-1,2);
	\fill[mypurple, opacity=0.8] (1,1) to [out=0, in=90] (2,0) to (0,0) to [out=90, in=180] (1,1);
	\draw[rstrand, directed=.999] (0,0) to [out=90, in=180] (1,1) to [out=0, in=270] (2,2);
	\draw[bstrand, directed=.999] (0,2) to [out=270, in=180] (1,1) to [out=0, in=90] (2,0);
\end{tikzpicture}
=
\begin{tikzpicture}[anchorbase, scale=.4, tinynodes]
	\draw[very thin, densely dotted, fill=white] (-2,4) to (-2,0) to [out=270, in=180] (-1,-1) to [out=0, in=270] (0,0) to [out=90, in=180] (1,1) to [out=180, in=270] (0,2) to (0,4) to (-2,4);
	\fill[myred, opacity=0.3] (5,-2) to (5,4) to (0,4) to (0,2) to [out=270, in=180] (1,1) to [out=0, in=270] (2,2) to [out=90, in=180] (3,3) to [out=0, in=90] (4,2) to (4,-2) to (5,-2);
	\fill[myblue, opacity=0.3] (-3,4) to (-2,4) to (-2,0) to [out=270, in=180] (-1,-1) to [out=0, in=270] (0,0) to [out=90, in=180] (1,1) to [out=0, in=90] (2,-2) to (-3,-2) to (-3,4);
	\fill[mypurple, opacity=0.8] (2,-2) to (2,0) to [out=90, in=0] (1,1) to [out=0, in=270] (2,2) to [out=90, in=180] (3,3) to [out=0, in=90] (4,2) to (4,-2) to (2,-2);
	\draw[rstrand, directed=.999] (2,0) to [out=90, in=0] (1,1) to [out=180, in=270] (0,2);
	\draw[rstrand] (2,0) to (2,-2);
	\draw[rstrand] (0,2) to (0,4);
	\draw[bstrand, directed=.999] (0,0) to [out=90, in=180] (1,1) to [out=0, in=270] (2,2);
	\draw[bstrand] (2,2) to [out=90, in=180] (3,3) to [out=0, in=90] (4,2) to (4,-2);
	\draw[bstrand] (-2,4) to (-2,0) to [out=270, in=180] (-1,-1) to [out=0, in=270] (0,0);
\end{tikzpicture}};
(0,-14.5)*{\text{{\tiny degree $1$}}};
\endxy
,\quad\quad
\xy
(0,0)*{
\begin{tikzpicture}[anchorbase, scale=.4, tinynodes]
	\draw[very thin, densely dotted, fill=white] (1,1) to [out=0, in=90] (2,0) to (0,0) to [out=90, in=180] (1,1);
	\fill[myred, opacity=0.3] (3,2) to (2,2) to [out=270, in=0] (1,1) to [out=0, in=90] (2,0) to (3,0) to (3,2);
	\fill[myblue, opacity=0.3] (-1,2) to (0,2) to [out=270, in=180] (1,1) to [out=180, in=90] (0,0) to (-1,0) to (-1,2);
	\fill[mypurple, opacity=0.8] (1,1) to [out=0, in=270] (2,2) to (0,2) to [out=270, in=180] (1,1);
	\draw[rstrand, directed=.999] (2,0) to [out=90, in=0] (1,1) to [out=180, in=270] (0,2);
	\draw[bstrand, directed=.999] (2,2) to [out=270, in=0] (1,1) to [out=180, in=90] (0,0);
\end{tikzpicture}
=
\begin{tikzpicture}[anchorbase, scale=.4, tinynodes]
	\draw[very thin, densely dotted, fill=white] (-2,-2) to (-2,0) to [out=90, in=180] (-1,1) to [out=180, in=270] (-2,2) to [out=90, in=0] (-3,3) to [out=180, in=90] (-4,2) to (-4,-2) to (-2,-2);
	\fill[myred, opacity=0.3] (3,4) to (2,4) to (2,0) to [out=270, in=0] (1,-1) to [out=180, in=270] (0,0) to [out=90, in=0] (-1,1) to [out=180, in=90] (-2,-2) to (3,-2) to (3,4);
	\fill[myblue, opacity=0.3] (-5,-2) to (-5,4) to (0,4) to (0,2) to [out=270, in=0] (-1,1) to [out=180, in=270] (-2,2) to [out=90, in=0] (-3,3) to [out=180, in=90] (-4,2) to (-4,-2) to (-5,-2);
	\fill[mypurple, opacity=0.8] (2,4) to (2,0) to [out=270, in=0] (1,-1) to [out=180, in=270] (0,0) to [out=90, in=0] (-1,1) to [out=0, in=270] (0,2) to (0,4) to (2,4);
	\draw[rstrand, directed=.999] (-2,0) to [out=90, in=180] (-1,1) to [out=0, in=270] (0,2);
	\draw[rstrand] (-2,0) to (-2,-2);
	\draw[rstrand] (0,2) to (0,4);
	\draw[bstrand, directed=.999] (0,0) to [out=90, in=0] (-1,1) to [out=180, in=270] (-2,2);
	\draw[bstrand] (-2,2) to [out=90, in=0] (-3,3) to [out=180, in=90] (-4,2) to (-4,-2);
	\draw[bstrand] (2,4) to (2,0) to [out=270, in=0] (1,-1) to [out=180, in=270] (0,0);
\end{tikzpicture}};
(0,-14.5)*{\text{{\tiny degree $1$}}};
\endxy
\]
Note that these are of degree $1$.
\end{example}

\begin{remark}\label{remark:q-sneaks-in}
The $2$-category $\ADiag$ depends on $\qpar$, since the quantum parameter is 
in the definition of the rings $\Rbim$, cf. \eqref{eq:sl3-exotic-action}.
\end{remark}

Before we can go on, we need some algebraic notions.

\subsubsection{An interlude on Frobenius extensions}\label{subsubsec:frob}

The relations of $\Adiag$ actually come from 
a cube of Frobenius extensions. 
(For details on Frobenius extensions see e.g. \cite{ESW}.)

\begin{definition}\label{definition:frob-extensions}
A (commutative) Frobenius extension is an extension of
commutative rings $\K^{\prime}\subset\K$ with 
$\K$ being a 
free $\K^{\prime}$-bimodule of finite rank, 
together with a $\K^{\prime}$-bilinear trace map
$\qDema\colon\K\to\K^{\prime}$ which  
gives rise to a non-degenerate bilinear pairing 
\[
\langle\cdot,\cdot\rangle
\colon
\somebasis\times\somebasis^{\fdual}
\to
\K^{\prime}.
\]
Moreover, for a Frobenius extension there exist
two $\K^{\prime}$-bases 
$\somebasis,\somebasis^{\fdual}$ of $\K$, 
such that for any $\algstuff{x}\in\somebasis$ there is precisely one 
element $\algstuff{x}^{\fdual}\in\somebasis^{\fdual}$ satisfying
\[
\langle \algstuff{x},\algstuff{x}^{\prime}\rangle=\qDema(\algstuff{x}\algstuff{x}^{\prime})
=
\delta_{\algstuff{x}^{\prime},\algstuff{x}^{\fdual}}.
\]
The elements $\algstuff{x}$ and 
$\algstuff{x}^{\fdual}$, respectively the bases $\somebasis$ and $\somebasis^{\fdual}$ 
are called dual to each other. 

The number of elements $\#\somebasis=\#\somebasis^{\fdual}$ is called the 
rank. 

Such an extension is called graded if $\K,\K^{\prime}$ are 
graded rings, $\K$ is graded as an  
$\K^{\prime}$-bimodule, $\somebasis,\somebasis^{\fdual}$ consist 
of homogeneous elements, and $\qDema$ is a homogeneous map.
\end{definition}

Note that the dual elements $\algstuff{x},\algstuff{x}^{\fdual}$ satisfy 
$\deg(\algstuff{x})+\deg(\algstuff{x}^{\fdual})=-\deg(\qDema)$.

\begin{definition}\label{definition:thin-demazure-action}
We let $\qDema_{\duc}\colon\Rbim\to\Rbim^{\duc}$ 
be defined via the formula
$\qDema_{\duc}(f)=\sneatfrac{f-\duc(f)}{\rootdu}$.
We call these the primary Demazure operators.
Similarly, we define 
\begin{gather*}
\begin{gathered}
\qDema_{\gc}^{\bc}=\qpar\qDema_{\bc}\qDema_{\yc}
\colon\Rbim^{\bc}\to\Rbim^{\gc},
\quad\quad
\qDema_{\gc}^{\yc}=\qDema_{\yc}\qDema_{\bc}
\colon\Rbim^{\yc}\to\Rbim^{\gc},
\quad\quad
\qDema_{\oc}^{\rc}=\qpar^{-1}\qDema_{\rc}\qDema_{\yc}
\colon\Rbim^{\rc}\to\Rbim^{\oc},
\\
\qDema_{\oc}^{\yc}=\qDema_{\yc}\qDema_{\rc}
\colon\Rbim^{\yc}\to\Rbim^{\oc},
\quad\quad
\qDema_{\pc}^{\bc}=\qDema_{\bc}\qDema_{\rc}
\colon\Rbim^{\bc}\to\Rbim^{\pc},
\quad\quad
\qDema_{\pc}^{\rc}=\qDema_{\rc}\qDema_{\bc}
\colon\Rbim^{\rc}\to\Rbim^{\pc},
\end{gathered}
\end{gather*}
which we call the mixed Demazure operators.
Finally, we define
\begin{gather*}
\begin{gathered}
\qDema_{\gc}=\qpar\qDema_{\bc}\qDema_{\yc}\qDema_{\bc}
=\qDema_{\yc}\qDema_{\bc}\qDema_{\yc}
\colon\Rbim\to\Rbim^{\gc},
\quad
\qDema_{\oc}=\qpar^{-1}\qDema_{\rc}\qDema_{\yc}\qDema_{\rc}
=\qDema_{\yc}\qDema_{\rc}\qDema_{\yc}
\colon\Rbim\to\Rbim^{\oc},
\\
\qDema_{\pc}=\qDema_{\bc}\qDema_{\rc}\qDema_{\bc}
=\qDema_{\rc}\qDema_{\bc}\qDema_{\rc}
\colon\Rbim\to\Rbim^{\pc},
\end{gathered}
\end{gather*}
which we call the secondary Demazure operators.
\end{definition}

Note that the action on the linear terms
determines the whole action since 
we have the twisted Leibniz rule 
$\qDema_{\duc}(fg)=\qDema_{\duc}(f)g+\duc(f)\qDema_{\duc}(g)$.
Moreover, a straightforward calculation (cf. \cite[(3.9)]{El1}) yields
\begin{gather*}
\qpar\qDema_{\bc}\qDema_{\yc}\qDema_{\bc}
=\qDema_{\yc}\qDema_{\bc}\qDema_{\yc},
\quad\quad
\qpar^{-1}\qDema_{\rc}\qDema_{\yc}\qDema_{\rc}
=\qDema_{\yc}\qDema_{\rc}\qDema_{\yc},
\quad\quad
\qDema_{\bc}\qDema_{\rc}\qDema_{\bc}
=\qDema_{\rc}\qDema_{\bc}\qDema_{\rc},
\end{gather*}
showing that the mixed Demazure operators are well-defined. 
(The careful reader might additionally want 
to check that the primary Demazure operators 
are well-defined by checking that $\qDema_{\duc}(f)$ is a 
$\duc$-invariant polynomial.)

\begin{remark}\label{remark:grading}
Recalling that the root variables are of degree $2$, 
one easily observes that the 
primary, mixed and secondary Demazure operators 
are homogeneous of degree $-2,-4$ and $-6$, respectively.
\end{remark}

\begin{lemma}\label{lemma:frob-extensions}
We have Frobenius extensions
\[
\qDema_{\duc}\colon\Rbim\to\Rbim^{\duc},
\quad\quad
\qDema_{\tduc}^{\duc}\colon\Rbim^{\duc}\to\Rbim^{\tduc},
\quad\quad
\qDema_{\tduc}
\colon\Rbim\to\Rbim^{\tduc},
\]
of rank $2,3$ and $6$, respectively,
which are compatible in the sense that $\qDema_{\duc}=\qDema_{\tduc}^{\duc}\qDema_{\tduc}$.
\end{lemma}

\begin{proof}
One can prove this lemma by computing explicit dual bases. 
(Note that this requires $2$ and $3$ to be invertible.)
We do not need them here and omit the calculations.
\end{proof}

\begin{definition}\label{definition:comult-elements}
Choose any pairs of dual bases $\dubasis,\dubasisd$ of 
$\qDema_{\duc}\colon \Rbim\to \Rbim^{\duc}$, $\dubcbasis,\dubcbasisd$ 
of $\qDema^{\duc}_{\tduc}\colon \Rbim^{\duc}\to \Rbim^{\tduc}$ and $\dutbasis,\dutbasisd$ of $\qDema^{\tduc} 
\colon\Rbim\to\Rbim^{\tduc}$. Let
\begin{gather}\label{eq:comult-elements}
\Frobel{\duc}{\phantom{\duc}}={\textstyle \sum_{a\in \dubasis}} a\otimes a^{\fdual},
\quad\quad
\Frobel{\tduc}{\duc}={\textstyle \sum_{a\in \dubcbasis}} a\otimes a^{\fdual},
\quad\quad
\Frobel{\tduc}{\phantom{\duc}}={\textstyle \sum_{a\in \dutbasis}} a\otimes a^{\fdual},
\end{gather}
where $a^{\fdual}$ denotes the basis element dual to $a$.
\end{definition}

Note that the elements $\Frobel{\placeholder}{\placeholder}$ are 
well-defined, i.e. do not depend on the choice 
of dual bases (see e.g. \cite[Section 2.4]{El1}).

\begin{definition}\label{definition:frob-elements}
We define the following elements $\frobel{\placeholder}{\placeholder}$ in $\Rbim$.
\begin{gather}\label{eq:frob-elements}
\begin{gathered}
\begin{tikzpicture}[baseline=(current bounding box.center)]
  \matrix (m) [matrix of math nodes, nodes in empty cells, row sep={0.5cm,between origins}, column
   sep={2.25cm,between origins}, text height=1.6ex, text depth=0.25ex, ampersand replacement=\&] {
\phantom{\bc}    \& \bc,\yc \& \bc \& \yc \& \wc\\
\gc \& \qpar^{-1}\rootb+\rooty \& \rooty\frobel{\gc}{\bc,\yc} \& \rootb\frobel{\gc}{\bc,\yc} 
\& \rootb\rooty\frobel{\gc}{\bc,\yc}
\\
    };
\draw[densely dashed] ($(m-2-1.north)+(-.15,0)$) edge ($(m-2-5.north)+(1.05,0)$);
\draw[densely dashed] ($(m-1-1.east)+(1.0,.2)$) edge ($(m-2-1.east)+(1.0,-.3)$);
\draw[densely dashed] ($(m-1-1.east)+(3.2,.2)$) edge ($(m-2-1.east)+(3.2,-.3)$);
\draw[densely dashed] ($(m-1-1.east)+(5.3,.2)$) edge ($(m-2-1.east)+(5.3,-.3)$);
\draw[densely dashed] ($(m-1-1.east)+(7.7,.2)$) edge ($(m-2-1.east)+(7.7,-.3)$);
\end{tikzpicture}
\\
\begin{tikzpicture}[baseline=(current bounding box.center)]
  \matrix (m) [matrix of math nodes, nodes in empty cells, row sep={0.5cm,between origins}, column
   sep={2.25cm,between origins}, text height=1.6ex, text depth=0.25ex, ampersand replacement=\&] {
\phantom{\bc}    \& \rc,\yc \& \rc \& \yc \& \wc\\
\oc \& \qpar\rootr+\rooty \& \rooty\frobel{\oc}{\rc,\yc} \& \rootr\frobel{\oc}{\rc,\yc} 
\& \rootb\rooty\frobel{\oc}{\rc,\yc}
\\
    };
\draw[densely dashed] ($(m-2-1.north)+(-.15,0)$) edge ($(m-2-5.north)+(1.05,0)$);
\draw[densely dashed] ($(m-1-1.east)+(1.0,.2)$) edge ($(m-2-1.east)+(1.0,-.3)$);
\draw[densely dashed] ($(m-1-1.east)+(3.2,.2)$) edge ($(m-2-1.east)+(3.2,-.3)$);
\draw[densely dashed] ($(m-1-1.east)+(5.3,.2)$) edge ($(m-2-1.east)+(5.3,-.3)$);
\draw[densely dashed] ($(m-1-1.east)+(7.7,.2)$) edge ($(m-2-1.east)+(7.7,-.3)$);
\end{tikzpicture}
\\
\begin{tikzpicture}[baseline=(current bounding box.center)]
  \matrix (m) [matrix of math nodes, nodes in empty cells, row sep={0.5cm,between origins}, column
   sep={2.25cm,between origins}, text height=1.6ex, text depth=0.25ex, ampersand replacement=\&] {
\phantom{\bc}    \& \bc,\rc \& \bc \& \rc \& \wc\\
\pc \& \rootb+\rootr \& \rootr\frobel{\pc}{\bc,\rc} \& \rootb\frobel{\pc}{\bc,\rc} 
\& \rootb\rootr\frobel{\pc}{\bc,\rc}
\\
    };
\draw[densely dashed] ($(m-2-1.north)+(-.15,0)$) edge ($(m-2-5.north)+(1.05,0)$);
\draw[densely dashed] ($(m-1-1.east)+(1.0,.2)$) edge ($(m-2-1.east)+(1.0,-.3)$);
\draw[densely dashed] ($(m-1-1.east)+(3.2,.2)$) edge ($(m-2-1.east)+(3.2,-.3)$);
\draw[densely dashed] ($(m-1-1.east)+(5.3,.2)$) edge ($(m-2-1.east)+(5.3,-.3)$);
\draw[densely dashed] ($(m-1-1.east)+(7.7,.2)$) edge ($(m-2-1.east)+(7.7,-.3)$);
\end{tikzpicture}
\end{gathered}
\end{gather}
This is to be read as e.g. 
$\frobel{\gc}{\bc,\yc}=\qpar^{-1}\rootb+\rooty$
and 
$\frobel{\gc}{\wc}=\frobel{\gc}{\phantom{\wc}}=\rootb\rooty(\qpar^{-1}\rootb+\rooty)$ etc.
\end{definition}

\subsubsection{The continuation of definition of \texorpdfstring{$\Adiag$}{singSbim}}\label{subsubsec:def-adiag-second}

\begin{definition}\label{definition:ssbim}
Let $\Adiag$ be the $2$-quotient of the additive, $\aformq$-linear 
closure of $\ADiag$ defined as follows.
\medskip

\noindent\textit{\setword{`Relations of $\Adiag$'}{sbim-rel}.}
(We only give the relations for one choice of compatible colors and comment 
on the others choices, where `Var.: comp. color.' means that the analogous relation 
holds for other compatible colorings in the sense of \eqref{eq:color-compatible}.)

First, polynomial multiplication, i.e. polynomial decorations 
on a facet multiply, and isotopy relations:
\begin{gather}\label{eq:iso-rel}
\xy
(0,0)*{
\begin{tikzpicture}[anchorbase, scale=.4, tinynodes]
	\fill[myyellow, opacity=0.3] (-1,0) to (-1,1.5) to [out=90, in=180] (-.5,2) to [out=0, in=90] (0,1.5) to [out=270, in=180] (.5,1) to [out=0, in=270] (1,1.5) to (1,3) to (-2,3) to (-2,0) to (-1,0);
	\fill[mygreen, opacity=0.8] (-1,0) to (-1,1.5) to [out=90, in=180] (-.5,2) to [out=0, in=90] (0,1.5) to [out=270, in=180] (.5,1) to [out=0, in=270] (1,1.5) to (1,3) to (2,3) to (2,0) to (-1,0);
	\draw[bstrand] (-1,0) to (-1,1.5) to [out=90, in=180] (-.5,2) to [out=0, in=90] (0,1.5) to [out=270, in=180] (.5,1) to [out=0, in=270] (1,1.5) to (1,3);
	\draw[bstrand, directed=.5] (0,1.4) to (0,1.39);
\end{tikzpicture}
=
\begin{tikzpicture}[anchorbase, scale=.4, tinynodes]
	\fill[myyellow, opacity=0.3] (0,0) to (0,3) to (-1,3) to (-1,0) to (0,0);
	\fill[mygreen, opacity=0.8] (0,0) to (0,3) to (1,3) to (1,0) to (0,0);
	\draw[bstrand, directed=.55] (0,0) to (0,3);
\end{tikzpicture}
=
\begin{tikzpicture}[anchorbase, scale=.4, tinynodes]
	\fill[myyellow, opacity=0.3] (1,0) to (1,1.5) to [out=90, in=0] (.5,2) to [out=180, in=90] (0,1.5) to [out=270, in=0] (-.5,1) to [out=180, in=270] (-1,1.5) to (-1,3) to (-2,3) to (-2,0) to (1,0);
	\fill[mygreen, opacity=0.8] (1,0) to (1,1.5) to [out=90, in=0] (.5,2) to [out=180, in=90] (0,1.5) to [out=270, in=0] (-.5,1) to [out=180, in=270] (-1,1.5) to (-1,3) to (2,3) to (2,0) to (1,0);
	\draw[bstrand] (1,0) to (1,1.5) to [out=90, in=0] (.5,2) to [out=180, in=90] (0,1.5) to [out=270, in=0] (-.5,1) to [out=180, in=270] (-1,1.5) to (-1,3);
	\draw[bstrand, directed=.5] (0,1.4) to (0,1.39);
\end{tikzpicture}
,\;\;\;\;
\begin{tikzpicture}[anchorbase, scale=.4, tinynodes]
	\draw[very thin, densely dotted, fill=white] (-2,2) to [out=90, in=180] (0,3.5) to [out=0, in=90] (2,2) to (2,-2) to (2.5,-2) to (2.5,4) to (-3,4) to (-3,0) to [out=270, in=180] (-2.5,-.5) to [out=0, in=270] (-2,0) to [out=90, in=180] (-1,1) to [out=180, in=270] (-2,2);
	\fill[myorange, opacity=0.8] (0,0) to [out=270, in=0] (-2,-1.5) to [out=180, in=270] (-4,0) to (-4,4) to (-4.5,4) to (-4.5,-2) to (1,-2) to (1,2) to [out=90, in=0] (.5,2.5) to [out=180, in=90] (0,2) to [out=270, in=0] (-1,1) to [out=0, in=90] (0,0);
	\fill[myyellow, opacity=0.3] (0,0) to [out=270, in=0] (-2,-1.5) to [out=180, in=270] (-4,0) to (-4,4) to (-3,4) to (-3,0) to [out=270, in=180] (-2.5,-.5) to [out=0, in=270] (-2,0) to [out=90, in=180] (-1,1) to [out=0, in=90] (0,0);
	\fill[myred, opacity=0.3] (-2,2) to [out=90, in=180] (0,3.5) to [out=0, in=90] (2,2) to (2,-2) to (1,-2) to (1,2) to [out=90, in=0] (.5,2.5) to [out=180, in=90] (0,2) to [out=270, in=0] (-1,1) to [out=180, in=270] (-2,2);
	\draw[ystrand, directed=.999] (-2,0) to [out=90, in=180] (-1,1) to [out=0, in=270] (0,2);
	\draw[ystrand] (-2,0) to [out=270, in=0] (-2.5,-.5) to [out=180, in=270] (-3,0) to (-3,4);
	\draw[ystrand] (0,2) to [out=90, in=180] (.5,2.5) to [out=0, in=90] (1,2) to (1,-2);
	\draw[rstrand, directed=.999] (0,0) to [out=90, in=0] (-1,1) to [out=180, in=270] (-2,2);
	\draw[rstrand] (0,0) to [out=270, in=0] (-2,-1.5) to [out=180, in=270] (-4,0) to (-4,4);
	\draw[rstrand] (-2,2) to [out=90, in=180] (0,3.5) to [out=0, in=90] (2,2) to (2,-2);
\end{tikzpicture}
=
\begin{tikzpicture}[anchorbase, scale=.4, tinynodes]
	\draw[very thin, densely dotted, fill=white] (3,2) to (2,2) to [out=270, in=0] (1,1) to [out=0, in=90] (2,0) to (3,0) to (3,2);
	\fill[myorange, opacity=0.8] (-1,2) to (0,2) to [out=270, in=180] (1,1) to [out=180, in=90] (0,0) to (-1,0) to (-1,2);
	\fill[myyellow, opacity=0.3] (1,1) to [out=0, in=270] (2,2) to (0,2) to [out=270, in=180] (1,1);
	\fill[myred, opacity=0.3] (1,1) to [out=0, in=90] (2,0) to (0,0) to [out=90, in=180] (1,1);
	\draw[ystrand, directed=.999] (2,2) to [out=270, in=0] (1,1) to [out=180, in=90] (0,0);
	\draw[rstrand, directed=.999] (0,2) to [out=270, in=180] (1,1) to [out=0, in=90] (2,0);
\end{tikzpicture}
=
\begin{tikzpicture}[anchorbase, scale=.4, tinynodes]
	\draw[very thin, densely dotted, fill=white] (0,0) to [out=270, in=180] (2,-1.5) to [out=0, in=270] (4,0) to (4,4) to (4.5,4) to (4.5,-2) to (-1,-2) to (-1,2) to [out=90, in=180] (-.5,2.5) to [out=0, in=90] (0,2) to [out=270, in=180] (1,1) to [out=180, in=90] (0,0);
	\fill[myorange, opacity=0.8] (2,2) to [out=90, in=0] (0,3.5) to [out=180, in=90] (-2,2) to (-2,-2) to (-2.5,-2) to (-2.5,4) to (3,4) to (3,0) to [out=270, in=0] (2.5,-.5) to [out=180, in=270] (2,0) to [out=90, in=0] (1,1) to [out=0, in=270] (2,2);
	\fill[myyellow, opacity=0.3] (0,0) to [out=270, in=180] (2,-1.5) to [out=0, in=270] (4,0) to (4,4) to (3,4) to (3,0) to [out=270, in=0] (2.5,-.5) to [out=180, in=270] (2,0) to [out=90, in=0] (1,1) to [out=180, in=90] (0,0);
	\fill[myred, opacity=0.3] (2,2) to [out=90, in=0] (0,3.5) to [out=180, in=90] (-2,2) to (-2,-2) to (-1,-2) to (-1,2) to [out=90, in=180] (-.5,2.5) to [out=0, in=90] (0,2) to [out=270, in=180] (1,1) to [out=0, in=270] (2,2);
	\draw[ystrand, directed=.999] (0,0) to [out=90, in=180] (1,1) to [out=0, in=270] (2,2);
	\draw[ystrand] (0,0) to [out=270, in=180] (2,-1.5) to [out=0, in=270] (4,0) to (4,4);
	\draw[ystrand] (2,2) to [out=90, in=0] (0,3.5) to [out=180, in=90] (-2,2) to (-2,-2);
	\draw[rstrand, directed=.999] (2,0) to [out=90, in=0] (1,1) to [out=180, in=270] (0,2);
	\draw[rstrand] (2,0) to [out=270, in=180] (2.5,-.5) to [out=0, in=270] (3,0) to (3,4);
	\draw[rstrand] (0,2) to [out=90, in=0] (-.5,2.5) to [out=180, in=90] (-1,2) to (-1,-2);
\end{tikzpicture}
};
(-10,-14.5)*{\text{{\tiny Var.: comp. color.}}};
\endxy
\end{gather}

Then various relations involving circles, called 
circle removals:
\\
\noindent\begin{tabularx}{0.99\textwidth}{XXX}
\begin{equation}\hspace{-9.5cm}\label{eq:circle-first}
\xy
(0,0)*{
\begin{tikzpicture}[anchorbase, scale=.4, tinynodes]
	\draw[very thin, densely dotted, fill=white] (-.5,-1.5) to (-.5,1.5) to (2.5,1.5) to (2.5,-1.5) to (-.5,-1.5);
	\fill[myblue, opacity=0.3] (0,0) to [out=90, in=180] (1,1) to [out=0, in=90] (2,0) to [out=270, in=0] (1,-1) to [out=180, in=270] (0,0);
	\draw[bstrand, directed=.999] (0,0) to [out=90, in=180] (1,1) to [out=0, in=90] (2,0);
	\draw[bstrand] (0,0) to [out=270, in=180] (1,-1) to [out=0, in=270] (2,0);
\end{tikzpicture}
=
\begin{tikzpicture}[anchorbase, scale=.4, tinynodes]
	\draw[very thin, densely dotted, fill=white] (-.5,-1.5) to (-.5,1.5) to (2.5,1.5) to (2.5,-1.5) to (-.5,-1.5);
	\node at (1,0) {$\polybox{\rootb}$};
\end{tikzpicture}};
(0,-8)*{\text{{\tiny Var.: comp. color.,}}};
(0,-10.5)*{\text{{\tiny using 
$\rootr$ or $\rooty$.}}};
\endxy
\end{equation} &
\begin{equation}\hspace{-9.5cm}\label{eq:circle-secondary}
\xy
(0,0)*{
\begin{tikzpicture}[anchorbase, scale=.4, tinynodes]
	\fill[myblue, opacity=0.3] (-.5,-1.5) to (-.5,1.5) to (2.5,1.5) to (2.5,-1.5) to (-.5,-1.5);
	\fill[white] (0,0) to [out=90, in=180] (1,1) to [out=0, in=90] (2,0) to [out=270, in=0] (1,-1) to [out=180, in=270] (0,0);
	\fill[mypurple, opacity=0.8] (0,0) to [out=90, in=180] (1,1) to [out=0, in=90] (2,0) to [out=270, in=0] (1,-1) to [out=180, in=270] (0,0);
	\draw[rstrand, directed=.999] (0,0) to [out=90, in=180] (1,1) to [out=0, in=90] (2,0);
	\draw[rstrand] (0,0) to [out=270, in=180] (1,-1) to [out=0, in=270] (2,0);
\end{tikzpicture}
=
\begin{tikzpicture}[anchorbase, scale=.4, tinynodes]
	\fill[myblue, opacity=0.3] (-.5,-1.5) to (-.5,1.5) to (2.5,1.5) to (2.5,-1.5) to (-.5,-1.5);
	\node at (1,0) {$\polybox{\frobel{\pc}{\bc}}$};
\end{tikzpicture}};
(0,-8)*{\text{{\tiny Var.: comp. color.,}}};
(0,-10.5)*{\text{{\tiny using $\frobel{\tduc}{\duc}$.}}};
\endxy
\end{equation} &
\begin{equation}\hspace{-9.5cm}\label{eq:circle-primary}
\xy
(0,0)*{
\begin{tikzpicture}[anchorbase, scale=.4, tinynodes]
	\fill[mygreen, opacity=0.8] (-.5,-1.5) to (-.5,1.5) to (2.5,1.5) to (2.5,-1.5) to (-.5,-1.5);
	\fill[white] (0,0) to [out=90, in=180] (1,1) to [out=0, in=90] (2,0) to [out=270, in=0] (1,-1) to [out=180, in=270] (0,0);
	\fill[myyellow, opacity=0.3] (0,0) to [out=90, in=180] (1,1) to [out=0, in=90] (2,0) to [out=270, in=0] (1,-1) to [out=180, in=270] (0,0);
	\draw[bstrand, rdirected=.999] (0,0) to [out=90, in=180] (1,1) to [out=0, in=90] (2,0);
	\draw[bstrand] (0,0) to [out=270, in=180] (1,-1) to [out=0, in=270] (2,0);
	\node at (1,0) {$\polybox{p}$};
\end{tikzpicture}
=
\begin{tikzpicture}[anchorbase, scale=.4, tinynodes]
	\fill[mygreen, opacity=0.8] (-.5,-1.5) to (-.5,1.5) to (2.5,1.5) to (2.5,-1.5) to (-.5,-1.5);
	\node at (1,0) {$\polybox{\qDema_{\gc}^{\yc}(p)}$};
\end{tikzpicture}};
(0,-8)*{\text{{\tiny Var.: comp. color.,}}};
(0,-10.5)*{\text{{\tiny using 
$\qDema_{\duc}^{\phantom{\duc}}$ or $\qDema_{\tduc}^{\duc}$.}}};
\endxy
\end{equation}
\end{tabularx}\\
(Note that there is also a variation of \eqref{eq:circle-primary} 
with a circular $\wc$-region in the middle bounded by a primary colored region outside.)

Moreover, we have polynomial sliding and neck cutting relations, i.e.
\\
\noindent\begin{tabularx}{0.99\textwidth}{XX}
\begin{equation}\hspace{-7cm}\label{eq:poly-move}
\xy
(0,0)*{
\begin{tikzpicture}[anchorbase, scale=.4, tinynodes]
	\fill[myyellow, opacity=0.3] (0,0) to (0,3) to (-2,3) to (-2,0) to (0,0);
	\fill[myorange, opacity=0.8] (0,0) to (0,3) to (2,3) to (2,0) to (0,0);
	\draw[rstrand, directed=.55] (0,0) to (0,3);
	\node at (-1,1.5) {$\polybox{p}$};
\end{tikzpicture}
=
\begin{tikzpicture}[anchorbase, scale=.4, tinynodes]
	\fill[myyellow, opacity=0.3] (0,0) to (0,3) to (-2,3) to (-2,0) to (0,0);
	\fill[myorange, opacity=0.8] (0,0) to (0,3) to (2,3) to (2,0) to (0,0);
	\draw[rstrand, directed=.55] (0,0) to (0,3);
	\node at (1,1.5) {$\polybox{p}$};
\end{tikzpicture}
,\,p\in\Rbim^{\oc}
};
(0,-8)*{\text{{\tiny Var.: comp. color.,}}};
(0,-10.5)*{\text{{\tiny for 
$p\in\Rbim^{\duc}$ or $p\in\Rbim^{\tduc}$.}}};
\endxy
\end{equation} &
\begin{equation}\hspace{-8cm}\label{eq:neck-cut}
\xy
(0,0)*{
\begin{tikzpicture}[anchorbase, scale=.4, tinynodes]
	\fill[myred, opacity=0.3] (-1,0) to (3,0) to (3,3) to (-1,3) to (-1,0);
	\fill[white] (.5,0) to [out=90, in=180] (1,.75) to [out=0, in=90] (1.5,0) to (.5,0);
	\fill[white] (.5,3) to [out=270, in=180] (1,2.25) to [out=0, in=270] (1.5,3) to (.5,3);
	\fill[mypurple, opacity=0.8] (.5,0) to [out=90, in=180] (1,.75) to [out=0, in=90] (1.5,0) to (.5,0);
	\fill[mypurple, opacity=0.8] (.5,3) to [out=270, in=180] (1,2.25) to [out=0, in=270] (1.5,3) to (.5,3);
	\draw[bstrand] (.5,0) to [out=90, in=180] (1,.75) to [out=0, in=90] (1.5,0);
	\draw[bstrand] (1.5,3) to [out=270, in=0] (1,2.25) to [out=180, in=270] (.5,3);
	\draw[bstrand, directed=.5] (1.1,.75) to (1.11,.75);
	\draw[bstrand, directed=.5] (.91,2.25) to (.9,2.25);
\end{tikzpicture}
=
\begin{tikzpicture}[anchorbase, scale=.4, tinynodes]
	\fill[myred, opacity=0.3] (-1,0) to (.5,0) to (.5,3) to (-1,3) to (-1,0);
	\fill[myred, opacity=0.3] (3,0) to (1.5,0) to (1.5,3) to (3,3) to (3,0);
	\fill[mypurple, opacity=0.8] (.5,0) to (.5,3) to (1.5,3) to (1.5,0) to (.5,0);
	\draw[bstrand, directed=.2, directed=.9] (.5,0) to (.5,3);
	\draw[bstrand, directed=.2, directed=.9] (1.5,3) to (1.5,0);
	\node at (1,1.5) {$\polybox{\Frobel{\pc}{\rc}}$};
\end{tikzpicture}};
(0,-8)*{\text{{\tiny Var.: comp. color.,}}};
(0,-10.5)*{\text{{\tiny using 
$\Frobel{\duc}{\phantom{\duc}}$ or $\Frobel{\tduc}{\duc}$.}}};
\endxy
\end{equation}
\end{tabularx}\\
The notation in the neck cutting relations \eqref{eq:neck-cut} indicates that one has put the tensor factors 
of the various summands of the $\Frobel{\placeholder}{\placeholder}$ in the 
corresponding regions (i.e. left tensor factors in the leftmost region and 
right tensor factors in the rightmost region), cf. \fullref{example:more-relations}.

Next, Reidemeister-like relations:
\\
\noindent\begin{tabularx}{0.99\textwidth}{XXX}
\begin{equation}\hspace{-9.0cm}\label{eq:rm-first}
\xy
(0,0)*{
\begin{tikzpicture}[anchorbase, scale=.4, tinynodes]
	\draw[very thin, densely dotted, fill=white] (-1,0) to [out=90, in=225] (-.5,.75) to [out=135, in=270] (-1,1.5) to [out=90, in=225] (-.5,2.25) to [out=135, in=270] (-1,3) to (-1.75,3) to (-1.75,0) to (-1,0);
	\fill[myblue, opacity=0.3] (-1,0) to [out=90, in=225] (-.5,.75) to [out=315, in=90] (0,0) to (-1,0);
	\fill[myblue, opacity=0.3] (-1,3) to [out=270, in=135] (-.5,2.25) to [out=45, in=270] (0,3) to (-1,3);
	\fill[myyellow, opacity=0.3] (-.5,.75) to [out=135, in=270] (-1,1.5) to [out=90, in=225] (-.5,2.25) to [out=315, in=90] (0,1.5) to [out=270, in=45] (-.5,.75);
	\fill[mygreen, opacity=0.8] (0,0) to [out=90, in=315] (-.5,.75) to [out=45, in=270] (0,1.5) to [out=90, in=315] (-.5,2.25) to [out=45, in=270] (0,3) to (.75,3) to (.75,0) to (0,0);
	\draw[ystrand, directed=.55] (0,0) to [out=90, in=270] (-1,1.5) to [out=90, in=270] (0,3);
	\draw[bstrand, directed=.55] (-1,0) to [out=90, in=270] (0,1.5) to [out=90, in=270] (-1,3);
\end{tikzpicture}
=
\begin{tikzpicture}[anchorbase, scale=.4, tinynodes]
	\draw[very thin, densely dotted, fill=white] (-1.75,0) to (-1.75,3) to (-1,3) to (-1,0) to (-1.75,0);
	\fill[myblue, opacity=0.3] (-1,0) to (-1,3) to (0,3) to (0,0) to (-1,0);
	\fill[mygreen, opacity=0.8] (0,0) to (0,3) to (.75,3) to (.75,0) to (0,0);
	\draw[ystrand, directed=.55] (0,0) to (0,3);
	\draw[bstrand, directed=.55] (-1,0) to (-1,3);
\end{tikzpicture}
};
(0,-8)*{\text{{\tiny Var.: comp. color.}}};
(0,-10.5)*{\text{{\tiny $\phantom{\frobel{\tduc}{\duc}}$}}};
\endxy
\end{equation} &
\begin{equation}\hspace{-9.0cm}\label{eq:rm-second}
\xy
(0,0)*{
\begin{tikzpicture}[anchorbase, scale=.4, tinynodes]
	\fill[myyellow, opacity=0.3] (-1,0) to [out=90, in=225] (-.5,.75) to [out=135, in=270] (-1,1.5) to [out=90, in=225] (-.5,2.25) to [out=135, in=270] (-1,3) to (-1.75,3) to (-1.75,0) to (-1,0);
	\fill[myorange, opacity=0.8] (-1,0) to [out=90, in=225] (-.5,.75) to [out=315, in=90] (0,0) to (-1,0);
	\fill[myorange, opacity=0.8] (-1,3) to [out=270, in=135] (-.5,2.25) to [out=45, in=270] (0,3) to (-1,3);
	\fill[white] (-.5,.75) to [out=135, in=270] (-1,1.5) to [out=90, in=225] (-.5,2.25) to [out=315, in=90] (0,1.5) to [out=270, in=45] (-.5,.75);
	\fill[myred, opacity=0.3] (0,0) to [out=90, in=315] (-.5,.75) to [out=45, in=270] (0,1.5) to [out=90, in=315] (-.5,2.25) to [out=45, in=270] (0,3) to (.75,3) to (.75,0) to (0,0);
	\draw[ystrand, directed=.55] (0,3) to [out=270, in=90] (-1,1.5) to [out=270, in=90] (0,0);
	\draw[rstrand, directed=.55] (-1,0) to [out=90, in=270] (0,1.5) to [out=90, in=270] (-1,3);
\end{tikzpicture}
=
\begin{tikzpicture}[anchorbase, scale=.4, tinynodes]
	\fill[myyellow, opacity=0.3] (-1.75,0) to (-1.75,3) to (-1,3) to (-1,0) to (-1.75,0);
	\fill[myorange, opacity=0.8] (-1,0) to (-1,3) to (0,3) to (0,0) to (-1,0);
	\fill[myred, opacity=0.3] (0,0) to (0,3) to (.75,3) to (.75,0) to (0,0);
	\draw[ystrand, directed=.2, directed=.9] (0,3) to (0,0);
	\draw[rstrand, directed=.2, directed=.9] (-1,0) to (-1,3);
	\node at (-.5,1.5) {$\polybox{\qDema\Frobel{\oc}{\phantom{.}}}$};
\end{tikzpicture}};
(0,-8)*{\text{{\tiny Var.: comp. color.,}}};
(0,-10.5)*{\text{{\tiny using $\qDema\Frobel{\tduc}{\phantom{.}}$.}}};
\endxy
\end{equation} &
\begin{equation}\hspace{-9.0cm}\label{eq:rm-third}
\xy
(0,0)*{
\begin{tikzpicture}[anchorbase, scale=.4, tinynodes]
	\fill[myred, opacity=0.3] (-1,0) to [out=90, in=225] (-.5,.75) to [out=135, in=270] (-1,1.5) to [out=90, in=225] (-.5,2.25) to [out=135, in=270] (-1,3) to (-1.75,3) to (-1.75,0) to (-1,0);
	\draw[very thin, densely dotted, fill=white] (-1,0) to [out=90, in=225] (-.5,.75) to [out=315, in=90] (0,0) to (-1,0);
	\draw[very thin, densely dotted, fill=white] (-1,3) to [out=270, in=135] (-.5,2.25) to [out=45, in=270] (0,3) to (-1,3);
	\fill[mypurple, opacity=0.8] (-.5,.75) to [out=135, in=270] (-1,1.5) to [out=90, in=225] (-.5,2.25) to [out=315, in=90] (0,1.5) to [out=270, in=45] (-.5,.75);
	\fill[myblue, opacity=0.3] (0,0) to [out=90, in=315] (-.5,.75) to [out=45, in=270] (0,1.5) to [out=90, in=315] (-.5,2.25) to [out=45, in=270] (0,3) to (.75,3) to (.75,0) to (0,0);
	\draw[rstrand, directed=.55] (-1,3) to [out=270, in=90] (0,1.5) to [out=270, in=90] (-1,0);
	\draw[bstrand, directed=.55] (0,0) to [out=90, in=270] (-1,1.5) to [out=90, in=270] (0,3);
\end{tikzpicture}
=
\begin{tikzpicture}[anchorbase, scale=.4, tinynodes]
	\fill[myred, opacity=0.3] (-1.75,0) to (-1.75,3) to (-1,3) to (-1,0) to (-1.75,0);
	\fill[myblue, opacity=0.3] (0,0) to (0,3) to (.75,3) to (.75,0) to (0,0);
	\draw[very thin, densely dotted, fill=white] (-1,0) to [out=90, in=270] (-1.32,1.5) to [out=90, in=270] (-1,3) to (0,3) to [out=270, in=90] (.32,1.5) to [out=270, in=90] (0,0) to (-1,0);
	\draw[rstrand, directed=.55] (-1,3) to [out=270, in=90] (-1.32,1.5) to [out=270, in=90] (-1,0);
	\draw[bstrand, directed=.55] (0,0) to [out=90, in=270] (.32,1.5) to [out=90, in=270] (0,3);
	\node at (-.475,1.5) {$\frobel{\pc}{\bc,\rc}$};
\end{tikzpicture}};
(0,-8)*{\text{{\tiny Var.: comp. color.,}}};
(0,-10.5)*{\text{{\tiny using $\frobel{\tduc}{\duc,\dudc}$.}}};
\endxy
\end{equation}
\end{tabularx}\\
where the notation $\qDema\Frobel{\oc}{\phantom{.}}$ in \eqref{eq:rm-second} means 
\begin{gather*}
\qDema\Frobel{\oc}{\phantom{.}}=
\qDema_{\yc}(\Frobel{\oc}{\rc}(1))\otimes\Frobel{\oc}{\rc}(2)
=
\Frobel{\oc}{\yc}(1)\otimes\qDema_{\rc}(\Frobel{\oc}{\yc}(2))
\end{gather*}
which is to be read again in the corresponding regions.

Finally, the square relations, which we exemplify 
by the case in which $\yc\wc\bc$ is at the bottom and $\bc\wc\yc$ 
is at the top:
\begin{gather}\label{eq:square}
\xy
(0,0)*{
\begin{tikzpicture}[anchorbase, scale=.4, tinynodes]
	\draw[very thin, densely dotted] (-3,0) to (-3,3);
	\draw[very thin, densely dotted] (-1,3) to (1,3);
	\draw[very thin, densely dotted] (-1,0) to (1,0);
	\draw[very thin, densely dotted] (3,0) to (3,3);
	\fill[myyellow, opacity=.3] (3,3) to (1,3) to (0,2.25) to (1,1.5) to (3,3);
	\fill[myyellow, opacity=.3] (-3,0) to (-1,0) to (0,.75) to (-1,1.5) to (-3,0);
	\fill[myblue, opacity=.3] (-3,3) to (-1,3) to (0,2.25) to (-1,1.5) to (-3,3);
	\fill[myblue, opacity=.3] (3,0) to (1,0) to (0,.75) to (1,1.5) to (3,0);
	\fill[mygreen, opacity=.8] (0,2.25) to (-1,1.5) to (0,.75) to (1,1.5) to (0,2.25);
	\draw[ystrand, directed=.65] (-3,0) to (1,3);
	\draw[ystrand, rdirected=.4] (-1,0) to (3,3);
	\draw[bstrand, directed=.4] (1,0) to (-3,3);
	\draw[bstrand, rdirected=.65] (3,0) to (-1,3);
\end{tikzpicture}
=
\begin{tikzpicture}[anchorbase, scale=.4, tinynodes]
	\draw[very thin, densely dotted] (-3,0) to (-3,3);
	\draw[very thin, densely dotted] (-1,3) to (1,3);
	\draw[very thin, densely dotted] (-1,0) to (1,0);
	\draw[very thin, densely dotted] (3,0) to (3,3);
	\fill[myyellow, opacity=.3] (-3,0) to (-2.75,.25) to [out=45, in=45] (-.75,.25) to (-1,0) to (-3,0);
	\fill[myyellow, opacity=.3] (1,3) to (.75,2.75) to [out=225, in=225] (2.75,2.75) to (3,3) to (1,3);
	\fill[myblue, opacity=.3] (-3,3) to (-1,3) to (3,0) to (1,0) to (-3,3);
	\draw[ystrand, directed=.55] (-3,0) to (-2.75,.25) to [out=45, in=45] (-.75,.25) to (-1,0);
	\draw[ystrand, rdirected=.55] (1,3) to (.75,2.75) to [out=225, in=225] (2.75,2.75) to (3,3);
	\draw[bstrand, directed=.4] (1,0) to (-3,3);
	\draw[bstrand, rdirected=.65] (3,0) to (-1,3);
\end{tikzpicture}
+
\qpar^{-1}
\begin{tikzpicture}[anchorbase, scale=.4, tinynodes]
	\draw[very thin, densely dotted] (-3,0) to (-3,3);
	\draw[very thin, densely dotted] (-1,3) to (1,3);
	\draw[very thin, densely dotted] (-1,0) to (1,0);
	\draw[very thin, densely dotted] (3,0) to (3,3);
	\fill[myyellow, opacity=.3] (3,3) to (1,3) to (-3,0) to (-1,0) to (3,3);
	\fill[myblue, opacity=.3] (1,0) to (.75,.25) to [out=135, in=135] (2.75,.25) to (3,0) to (1,0);
	\fill[myblue, opacity=.3] (-3,3) to (-2.75,2.75) to [out=315, in=315] (-.75,2.75) to (-1,3) to (-3,3);
	\draw[ystrand, directed=.65] (-3,0) to (1,3);
	\draw[ystrand, rdirected=.4] (-1,0) to (3,3);
	\draw[bstrand, directed=.45] (1,0) to (.75,.25) to [out=135, in=135] (2.75,.25) to (3,0);
	\draw[bstrand, rdirected=.65] (-3,3) to (-2.75,2.75) to [out=315, in=315] (-.75,2.75) to (-1,3);
\end{tikzpicture}};
(0,-8.5)*{\text{{\tiny Var.: comp. color.; in case $\yc\wc\rc$ replace $\qpar^{-1}\rightsquigarrow\qpar$; 
in case $\pc\wc\bc$ replace $\qpar^{-1}\rightsquigarrow 1$.}}};
\endxy
\end{gather}
(We stress that \eqref{eq:square} is not invariant under color change.)
\end{definition}

\begin{definition}\label{definition:regular}
The $2$-category of regular Bott--Samelson bimodules is defined as 
\[
\adiag=\Adiag(\wc,\wc),
\] 
i.e. the $2$-full $2$-subcategory of $\Adiag$ generated by diagrams 
whose left- and rightmost color is $\wc$. Note that $\adiag$ 
has only one object, namely $\wc$.

The $2$-category of maximally singular Bott--Samelson bimodules is defined as 
\[
\aDiag={\textstyle\bigoplus_{\tduc,\tdudc\in\Seset}}\,\Adiag(\tduc,\tdudc),
\] 
i.e. the $2$-full $2$-subcategory of $\Adiag$ generated by diagrams whose 
left- and rightmost colors are secondary.
\end{definition}

Note that we can always extend scalars to 
e.g. $\Cq=\C(\qpar)$ and we indicate this by changing the 
subscript $[\qpar]$ to $\qpar$.

\begin{remark}\label{remark:cat-affine-a2-scalars}
$\Adiag$ is an additive, $\aformq$-linear, graded $2$-category, which is, 
however, not idempotent closed. 
This is remedied by considering its Karoubi envelope $\Kar{\Adiag}$, which we 
take as the definition of the $2$-category of 
singular Soergel bimodules of affine type $\typea{2}$.
\end{remark}

Thus, we have:
\smallskip
\begin{enumerate}[label=$\blacktriangleright$]

\setlength\itemsep{.15cm}

\item The $2$-category of singular 
Bott--Samelson bimodules, whose notation contains an $\boldsymbol{s}$.

\item The $2$-category of regular 
Bott--Samelson bimodules, whose notation has no $\boldsymbol{s}$. 

\item The $2$-category of maximally 
singular Bott--Samelson bimodules, 
whose notation contains an $\boldsymbol{m}$. 
As we will see, the degree-zero 
part of this $2$-subcategory, 
for a fixed choice of shifts of the $1$-morphisms, is semisimple.

\item The corresponding $2$-categories 
of singular, regular and maximally singular Soergel bimodules are the Karoubi envelopes of these, by definition.

\item Various scalar extensions of these, indicated by subscripts.

\end{enumerate}

We use similar notations throughout, e.g. for scalar extensions 
of $2$-functors.

\begin{remark}\label{remark:cat-affine-a2} 
By \cite[Theorem A.1]{El1}, the decategorification 
of $\Kar{\Adiagfield}$, via the split Grothendieck group, is isomorphic to the 
affine $\typea{2}$ Hecke algebroid. As explained for example in \cite[Section 2.3]{Wi-sing-soergel} 
(under the name Schur algebroid), this is a multi object version of the affine 
Hecke algebra $\hecke$ from \fullref{subsec:def-sl3alg-1}. Moreover, 
the $2$-full $2$-subcategory $\Kar{\adiagfield}$ decategorifies to 
$\hecke$, see e.g. \cite[Theorem 3.17]{ElWi-soergel-calculus}.
\end{remark}

\subsubsection{Examples and further comments}\label{subsubsec:further-comments}

\begin{example}\label{example:q-demazure-action}
In accordance with \eqref{eq:color-compatible}, we have the following Frobenius extensions 
\[
\begin{tikzpicture}[baseline=(current bounding box.center), tinynodes]
  \matrix (m) [matrix of math nodes, nodes in empty cells, row sep=.02cm, column
  sep=.01cm, text height=1.6ex, text depth=0.25ex, ampersand replacement=\&] {
{\color{mygray}\scalebox{.8}{$\Rbim^{p}$}} \& \scalebox{.8}{$\Rbim^{\gc}$}\, \& {\color{mygray}\scalebox{.8}{$\Rbim^{o}$}} \\
\,\scalebox{.8}{$\Rbim^{\bc}$} \& {\color{mygray}\scalebox{.8}{$\Rbim^{r}$}} \& \scalebox{.8}{$\Rbim^{\yc}$} \\
\& \scalebox{.8}{$\Rbim$} \&  \\
    };
	\draw[thin] ($(m-2-1.north)+(0,-.225cm)$) to ($(m-1-2.south)+(-.05,.075cm)$);
	\draw[thin] ($(m-2-3.north)+(0,-.225cm)$) to ($(m-1-2.south)+(.05,.075cm)$);
	\draw[thin] ($(m-3-2.north)+(-.05,-.225cm)$) to ($(m-2-1.south)+(0,.075cm)$);
	\draw[thin] ($(m-3-2.north)+(.05,-.225cm)$) to ($(m-2-3.south)+(0,.075cm)$);
\end{tikzpicture}
\quad\quad
\begin{tikzpicture}[baseline=(current bounding box.center), tinynodes]
  \matrix (m) [matrix of math nodes, nodes in empty cells, row sep=.02cm, column
  sep=.01cm, text height=1.6ex, text depth=0.25ex, ampersand replacement=\&] {
{\color{mygray}\scalebox{.8}{$\Rbim^{p}$}} \& {\color{mygray}\scalebox{.8}{$\Rbim^{g}$}} \& \scalebox{.8}{$\Rbim^{\oc}$} \\
{\color{mygray}\scalebox{.8}{$\Rbim^{b}$}} \& \scalebox{.8}{$\Rbim^{\rc}$} \& \scalebox{.8}{$\Rbim^{\yc}$} \\
\& \scalebox{.8}{$\Rbim$}\&  \\
    };
	\draw[thin] ($(m-2-3.north)+(.035,-.225cm)$) to ($(m-1-3.south)+(.035,.075cm)$);
	\draw[thin] ($(m-3-2.north)+(0,-.225cm)$) to ($(m-2-2.south)+(0,.075cm)$);
	\draw[thin] ($(m-3-2.north)+(.05,-.225cm)$) to ($(m-2-3.south)+(0,.075cm)$);
	\draw[thin] ($(m-2-2.north)+(.05,-.225cm)$) to ($(m-1-3.south)+(0,.075cm)$);
\end{tikzpicture}
\quad\quad
\begin{tikzpicture}[baseline=(current bounding box.center), tinynodes]
  \matrix (m) [matrix of math nodes, nodes in empty cells, row sep=.02cm, column
  sep=.01cm, text height=1.6ex, text depth=0.25ex, ampersand replacement=\&] {
\scalebox{.8}{$\Rbim^{\pc}$} \& {\color{mygray}\scalebox{.8}{$\Rbim^{g}$}} \& {\color{mygray}\scalebox{.8}{$\Rbim^{o}$}} \\
\scalebox{.8}{$\Rbim^{\bc}$} \& \scalebox{.8}{$\Rbim^{\rc}$} \& {\color{mygray}\scalebox{.8}{$\Rbim^{y}$}} \\
\& \scalebox{.8}{$\Rbim$} \&  \\
    };
	\draw[thin] ($(m-2-1.north)+(-.035,-.225cm)$) to ($(m-1-1.south)+(-.035,.075cm)$);
	\draw[thin] ($(m-3-2.north)+(-.05,-.225cm)$) to ($(m-2-1.south)+(0,.075cm)$);
	\draw[thin] ($(m-3-2.north)+(0,-.225cm)$) to ($(m-2-2.south)+(0,.075cm)$);
	\draw[thin] ($(m-2-2.north)+(-.05,-.225cm)$) to ($(m-1-1.south)+(0,.075cm)$);
\end{tikzpicture}
\]
with the corresponding trace maps going upwards. Moreover,
\begin{gather}\label{eq:ranks}
\qDema(\rootb)=2,
\quad\quad
\qDema_{\gc}^{\bc}(\frobel{\gc}{\bc})=3,
\quad\quad
\qDema_{\gc}(\frobel{\gc}{})=6,
\end{gather}
as an easy calculation shows. Similar results hold for other colors.
 
Note that the numbers in \eqref{eq:ranks}, which follow from 
\eqref{eq:circle-first}, \eqref{eq:circle-secondary} and \eqref{eq:circle-primary}, are precisely the ranks of the corresponding Frobenius extensions.
\end{example}

\begin{example}\label{example:some-relations}
When working with $\Adiag$, it is important to remember 
that the polynomial $2$-generators 
of a given facet are invariant under the action of the parabolic subgroup 
which corresponds to the color of that region. For example, $\frobel{\gc}{\yc}$ is an element of 
$\Rbim^{\yc}$, and applying $\qDema_{\gc}^{\yc}$ to it 
will make it additionally $\bc$-invariant. In fact,
\[
\begin{tikzpicture}[anchorbase, scale=.4, tinynodes]
	\fill[mygreen, opacity=0.8] (-1,-2) to (-1,2) to (3,2) to (3,-2) to (-1,-2);
	\fill[white] (-.5,0) to [out=90, in=180] (1,1.5) to [out=0, in=90] (2.5,0) to [out=270, in=0] (1,-1.5) to [out=180, in=270] (-.5,0);
	\fill[myyellow, opacity=0.3] (-.5,0) to [out=90, in=180] (1,1.5) to [out=0, in=90] (2.5,0) to [out=270, in=0] (1,-1.5) to [out=180, in=270] (-.5,0);
	\fill[white] (.25,0) to [out=90, in=180] (1,.75) to [out=0, in=90] (1.75,0) to [out=270, in=0] (1,-.75) to [out=180, in=270] (.25,0);
	\fill[mygreen, opacity=0.8] (.25,0) to [out=90, in=180] (1,.75) to [out=0, in=90] (1.75,0) to [out=270, in=0] (1,-.75) to [out=180, in=270] (.25,0);
	\draw[bstrand, directed=.999] (.25,0) to [out=90, in=180] (1,.75) to [out=0, in=90] (1.75,0);
	\draw[bstrand] (.25,0) to [out=270, in=180] (1,-.75) to [out=0, in=270] (1.75,0);
	\draw[bstrand, rdirected=.999] (-.5,0) to [out=90, in=180] (1,1.5) to [out=0, in=90] (2.5,0);
	\draw[bstrand] (-.5,0) to [out=270, in=180] (1,-1.5) to [out=0, in=270] (2.5,0);
\end{tikzpicture}
\stackrel{\eqref{eq:circle-secondary}}{=}
\begin{tikzpicture}[anchorbase, scale=.4, tinynodes]
	\fill[mygreen, opacity=0.8] (-1,-2) to (-1,2) to (3,2) to (3,-2) to (-1,-2);
	\fill[white] (-.5,0) to [out=90, in=180] (1,1.5) to [out=0, in=90] (2.5,0) to [out=270, in=0] (1,-1.5) to [out=180, in=270] (-.5,0);
	\fill[myyellow, opacity=0.3] (-.5,0) to [out=90, in=180] (1,1.5) to [out=0, in=90] (2.5,0) to [out=270, in=0] (1,-1.5) to [out=180, in=270] (-.5,0);
	\draw[bstrand, rdirected=.999] (-.5,0) to [out=90, in=180] (1,1.5) to [out=0, in=90] (2.5,0);
	\draw[bstrand] (-.5,0) to [out=270, in=180] (1,-1.5) to [out=0, in=270] (2.5,0);
	\node at (1,0) {$\polybox{\frobel{\gc}{\yc}}$};
\end{tikzpicture}
\stackrel{\eqref{eq:circle-primary}}{=}
\begin{tikzpicture}[anchorbase, scale=.4, tinynodes]
	\fill[mygreen, opacity=0.8] (-1,-2) to (-1,2) to (3,2) to (3,-2) to (-1,-2);
	\node at (1,0) {$\polybox{\qDema_{\gc}^{\yc}(\frobel{\gc}{\yc})}$};
\end{tikzpicture}
=
3\,
\begin{tikzpicture}[anchorbase, scale=.4, tinynodes]
	\fill[mygreen, opacity=0.8] (-1,-2) to (-1,2) to (3,2) to (3,-2) to (-1,-2);
\end{tikzpicture}
\]
We also get
\[
\begin{tikzpicture}[anchorbase, scale=.4, tinynodes]
	\fill[mygreen, opacity=0.8] (-1,-2) to (-1,2) to (3,2) to (3,-2) to (-1,-2);
	\fill[white] (-.5,0) to [out=90, in=180] (1,1.5) to [out=0, in=90] (2.5,0) to [out=270, in=0] (1,-1.5) to [out=180, in=270] (-.5,0);
	\fill[myyellow, opacity=0.3] (-.5,0) to [out=90, in=180] (1,1.5) to [out=0, in=90] (2.5,0) to [out=270, in=0] (1,-1.5) to [out=180, in=270] (-.5,0);
	\fill[white] (.25,0) to [out=90, in=180] (1,.75) to [out=0, in=90] (1.75,0) to [out=270, in=0] (1,-.75) to [out=180, in=270] (.25,0);
	\fill[myorange, opacity=0.8] (.25,0) to [out=90, in=180] (1,.75) to [out=0, in=90] (1.75,0) to [out=270, in=0] (1,-.75) to [out=180, in=270] (.25,0);
	\draw[bstrand, rdirected=.999] (-.5,0) to [out=90, in=180] (1,1.5) to [out=0, in=90] (2.5,0);
	\draw[bstrand] (-.5,0) to [out=270, in=180] (1,-1.5) to [out=0, in=270] (2.5,0);
	\draw[rstrand, directed=.999] (.25,0) to [out=90, in=180] (1,.75) to [out=0, in=90] (1.75,0);
	\draw[rstrand] (.25,0) to [out=270, in=180] (1,-.75) to [out=0, in=270] (1.75,0);
\end{tikzpicture}
\stackrel{\eqref{eq:circle-secondary}}{=}
\begin{tikzpicture}[anchorbase, scale=.4, tinynodes]
	\fill[mygreen, opacity=0.8] (-1,-2) to (-1,2) to (3,2) to (3,-2) to (-1,-2);
	\fill[white] (-.5,0) to [out=90, in=180] (1,1.5) to [out=0, in=90] (2.5,0) to [out=270, in=0] (1,-1.5) to [out=180, in=270] (-.5,0);
	\fill[myyellow, opacity=0.3] (-.5,0) to [out=90, in=180] (1,1.5) to [out=0, in=90] (2.5,0) to [out=270, in=0] (1,-1.5) to [out=180, in=270] (-.5,0);
	\draw[bstrand, rdirected=.999] (-.5,0) to [out=90, in=180] (1,1.5) to [out=0, in=90] (2.5,0);
	\draw[bstrand] (-.5,0) to [out=270, in=180] (1,-1.5) to [out=0, in=270] (2.5,0);
	\node at (1,0) {$\polybox{\frobel{\oc}{\yc}}$};
\end{tikzpicture}
\stackrel{\eqref{eq:circle-primary}}{=}
\begin{tikzpicture}[anchorbase, scale=.4, tinynodes]
	\fill[mygreen, opacity=0.8] (-1,-2) to (-1,2) to (3,2) to (3,-2) to (-1,-2);
	\node at (1,0) {$\polybox{\qDema_{\gc}^{\yc}(\frobel{\oc}{\yc})}$};
\end{tikzpicture}
=
\qnumber{3}\,
\begin{tikzpicture}[anchorbase, scale=.4, tinynodes]
	\fill[mygreen, opacity=0.8] (-1,-2) to (-1,2) to (3,2) to (3,-2) to (-1,-2);
\end{tikzpicture}
\]
which we will need below.
\end{example}

\begin{example}\label{example:more-relations}
We have
\[
\begin{tikzpicture}[anchorbase, scale=.4, tinynodes]
	\fill[myyellow, opacity=0.3] (-1,0) to [out=90, in=225] (-.5,.75) to [out=135, in=270] (-1,1.5) to [out=90, in=225] (-.5,2.25) to [out=135, in=270] (-1,3) to (-1.75,3) to (-1.75,0) to (-1,0);
	\fill[mygreen, opacity=0.8] (-1,0) to [out=90, in=225] (-.5,.75) to [out=315, in=90] (0,0) to (-1,0);
	\fill[mygreen, opacity=0.8] (-1,3) to [out=270, in=135] (-.5,2.25) to [out=45, in=270] (0,3) to (-1,3);
	\fill[white] (-.5,.75) to [out=135, in=270] (-1,1.5) to [out=90, in=225] (-.5,2.25) to [out=315, in=90] (0,1.5) to [out=270, in=45] (-.5,.75);
	\fill[myblue, opacity=0.3] (0,0) to [out=90, in=315] (-.5,.75) to [out=45, in=270] (0,1.5) to [out=90, in=315] (-.5,2.25) to [out=45, in=270] (0,3) to (.75,3) to (.75,0) to (0,0);
	\draw[ystrand, directed=.55] (0,3) to [out=270, in=90] (-1,1.5) to [out=270, in=90] (0,0);
	\draw[bstrand, directed=.55] (-1,0) to [out=90, in=270] (0,1.5) to [out=90, in=270] (-1,3);
\end{tikzpicture}
\stackrel{\eqref{eq:rm-second}}{=}
\begin{tikzpicture}[anchorbase, scale=.4, tinynodes]
	\fill[myyellow, opacity=0.3] (-1.75,0) to (-1.75,3) to (-1,3) to (-1,0) to (-1.75,0);
	\fill[mygreen, opacity=0.8] (-1,0) to (-1,3) to (0,3) to (0,0) to (-1,0);
	\fill[myblue, opacity=0.3] (0,0) to (0,3) to (.75,3) to (.75,0) to (0,0);
	\draw[ystrand, directed=.2, directed=.9] (0,3) to (0,0);
	\draw[bstrand, directed=.2, directed=.9] (-1,0) to (-1,3);
	\node at (-.5,1.5) {$\polybox{\qDema\Frobel{\gc}{\phantom{.}}}$};
\end{tikzpicture}
=
{\textstyle \sum_{\algstuff{x}\in\somebasis_{\gc}^{\bc}}}\,
\begin{tikzpicture}[anchorbase, scale=.4, tinynodes]
	\fill[myyellow, opacity=0.3] (-4,0) to (-4,3) to (-1,3) to (-1,0) to (-4,0);
	\fill[mygreen, opacity=0.8] (-1,0) to (-1,3) to (0,3) to (0,0) to (-1,0);
	\fill[myblue, opacity=0.3] (0,0) to (0,3) to (3,3) to (3,0) to (0,0);
	\draw[ystrand, directed=.55] (0,3) to (0,0);
	\draw[bstrand, directed=.55] (-1,0) to (-1,3);
	\node at (-2.65,1.5) {$\polybox{\qDema_{\yc}(\algstuff{x})}$};
	\node at (1.5,1.5) {$\polybox{\algstuff{x}^{\fdual}}$};
\end{tikzpicture}
=
{\textstyle \sum_{\algstuff{y}\in\somebasis_{\gc}^{\yc}}}\,
\begin{tikzpicture}[anchorbase, scale=.4, tinynodes]
	\fill[myyellow, opacity=0.3] (-4,0) to (-4,3) to (-1,3) to (-1,0) to (-4,0);
	\fill[mygreen, opacity=0.8] (-1,0) to (-1,3) to (0,3) to (0,0) to (-1,0);
	\fill[myblue, opacity=0.3] (0,0) to (0,3) to (3,3) to (3,0) to (0,0);
	\draw[ystrand, directed=.55] (0,3) to (0,0);
	\draw[bstrand, directed=.55] (-1,0) to (-1,3);
	\node at (-2.5,1.5) {$\polybox{\algstuff{y}}$};
	\node at (1.65,1.5) {$\polybox{\qDema_{\bc}(\algstuff{y}^{\fdual})}$};
\end{tikzpicture}
\]

More generally, cf. \cite[(3.15d)]{El1}, the 
relations in
\fullref{definition:ssbim} imply 
\[
\begin{tikzpicture}[anchorbase, scale=.4, tinynodes]
	\fill[myyellow, opacity=0.3] (-1,0) to [out=90, in=225] (-.5,.75) to [out=135, in=270] (-1,1.5) to [out=90, in=225] (-.5,2.25) to [out=135, in=270] (-1,3) to (-1.75,3) to (-1.75,0) to (-1,0);
	\fill[mygreen, opacity=0.8] (-1,0) to [out=90, in=225] (-.5,.75) to [out=315, in=90] (0,0) to (-1,0);
	\fill[mygreen, opacity=0.8] (-1,3) to [out=270, in=135] (-.5,2.25) to [out=45, in=270] (0,3) to (-1,3);
	\fill[white] (-.5,.75) to [out=135, in=270] (-1,1.5) to [out=90, in=225] (-.5,2.25) to [out=315, in=90] (0,1.5) to [out=270, in=45] (-.5,.75);
	\fill[myblue, opacity=0.3] (0,0) to [out=90, in=315] (-.5,.75) to [out=45, in=270] (0,1.5) to [out=90, in=315] (-.5,2.25) to [out=45, in=270] (0,3) to (.75,3) to (.75,0) to (0,0);
	\draw[ystrand, directed=.55] (0,3) to [out=270, in=90] (-1,1.5) to [out=270, in=90] (0,0);
	\draw[bstrand, directed=.55] (-1,0) to [out=90, in=270] (0,1.5) to [out=90, in=270] (-1,3);
	\node at (-.5,1.5) {\text{{\tiny$p$}}};
\end{tikzpicture}
=
{\textstyle \sum_{\algstuff{x}\in\somebasis_{\gc}^{\bc}}}\,
\begin{tikzpicture}[anchorbase, scale=.4, tinynodes]
	\fill[myyellow, opacity=0.3] (-4,0) to (-4,3) to (-1,3) to (-1,0) to (-4,0);
	\fill[mygreen, opacity=0.8] (-1,0) to (-1,3) to (0,3) to (0,0) to (-1,0);
	\fill[myblue, opacity=0.3] (0,0) to (0,3) to (3,3) to (3,0) to (0,0);
	\draw[ystrand, directed=.55] (0,3) to (0,0);
	\draw[bstrand, directed=.55] (-1,0) to (-1,3);
	\node at (-2.65,1.5) {$\polybox{\qDema_{\yc}(p\algstuff{x})}$};
	\node at (1.5,1.5) {$\polybox{\algstuff{x}^{\fdual}}$};
\end{tikzpicture}
=
{\textstyle \sum_{\algstuff{y}\in\somebasis_{\gc}^{\yc}}}\,
\begin{tikzpicture}[anchorbase, scale=.4, tinynodes]
	\fill[myyellow, opacity=0.3] (-4,0) to (-4,3) to (-1,3) to (-1,0) to (-4,0);
	\fill[mygreen, opacity=0.8] (-1,0) to (-1,3) to (0,3) to (0,0) to (-1,0);
	\fill[myblue, opacity=0.3] (0,0) to (0,3) to (3,3) to (3,0) to (0,0);
	\draw[ystrand, directed=.55] (0,3) to (0,0);
	\draw[bstrand, directed=.55] (-1,0) to (-1,3);
	\node at (-2.5,1.5) {$\polybox{\algstuff{y}}$};
	\node at (1.65,1.5) {$\polybox{\!\qDema_{\bc}(p\algstuff{y}^{\fdual})\!\!}$};
\end{tikzpicture},
p\in\Rbim.
\]
As usual, similar relations hold for other colors.
\end{example}

\subsection{The trihedral Soergel bimodules of level \texorpdfstring{$\infty$}{infty}}\label{subsec:ubsbim}

\subsubsection{The definition}\label{subsec:cat-thealgebra-def}

We first consider a $2$-subcategory categorifying 
$\subquo$.

\begin{definition}\label{definition:the-subquo-cat}
Let $\subcatquo$ be 
the additive closure of 
the $2$-full $2$-subcategory of $\adiag$, 
whose $1$-morphisms are generated by the color strings that have at least one secondary color and 
have $\wc$ as the left- and rightmost color but nowhere else in the string. 

Its scalar extension is denoted by $\subcatquofield=\subcatquo[\infty,\qpar]$. 
\end{definition}

\begin{example}\label{example:the-subquo-cat}
The prototypical $1$-morphisms of $\subcatquo$ are $\wc$ and
all compatible color variation of
\[
\wc\bc\gc\bc\wc,
\quad
\wc\yc\gc\yc\wc,
\quad
\wc\yc\gc\bc\wc,
\quad
\wc\bc\gc\yc\wc,
\quad
\wc\bc\gc\yc\oc\yc\wc,
\quad
\wc\bc\gc\yc\oc\rc\wc,
\quad
\text{etc.}
\] 
All other $1$-morphism in $\subcatquo$ are direct sums of these, e.g. 
$\wc\yc\gc\yc\oc\yc\wc\oplus\wc\yc\oc\yc\wc$.
\end{example}

\subsubsection{Some lemmas}\label{subsec:cat-thealgebra-lemmas}

We note the following lemma, which follows 
directly from \eqref{eq:rm-first}.

\begin{lemmaqed}\label{lemma:clasps-well-defined}
The following diagrams commute in $\Adiag$.
\begin{gather}\label{eq:clasps-well-defined}
\xymatrix@C+=1.3cm@L+=6pt{
\wc\bc\gc
\ar@<-4pt>@/_/[rr]|{\,\twomorstuff{id}_{\wc\bc\gc}\,}
\ar[r]^{
\begin{tikzpicture}[anchorbase, scale=.4, tinynodes]
	\draw[very thin, densely dotted, fill=white] (-1,1.5) to [out=90, in=225] (-.5,2.25) to [out=135, in=270] (-1,3) to (-2,3) to (-2,1.5) to (-1,1.5);
	\fill[myyellow, opacity=0.3] (-1,3) to [out=270, in=135] (-.5,2.25) to [out=45, in=270] (0,3) to (-1,3);
	\fill[myblue, opacity=0.3] (-1,1.5) to [out=90, in=225] (-.5,2.25) to [out=315, in=90] (0,1.5) to (-1,1.5);
	\fill[mygreen, opacity=0.8] (0,1.5) to [out=90, in=315] (-.5,2.25) to [out=45, in=270] (0,3) to (1,3) to (1,1.5) to (0,1.5);
	\draw[ystrand, directed=.999] (0,1.5) to [out=90, in=270] (-1,3);
	\draw[bstrand, directed=.999] (-1,1.5) to [out=90, in=270] (0,3);
\end{tikzpicture}}
&
\wc\yc\gc
\ar[r]^{
\begin{tikzpicture}[anchorbase, scale=.4, tinynodes]
	\draw[very thin, densely dotted, fill=white] (-1,1.5) to [out=90, in=225] (-.5,2.25) to [out=135, in=270] (-1,3) to (-2,3) to (-2,1.5) to (-1,1.5);
	\fill[myblue, opacity=0.3] (-1,3) to [out=270, in=135] (-.5,2.25) to [out=45, in=270] (0,3) to (-1,3);
	\fill[myyellow, opacity=0.3] (-1,1.5) to [out=90, in=225] (-.5,2.25) to [out=315, in=90] (0,1.5) to (-1,1.5);
	\fill[mygreen, opacity=0.8] (0,1.5) to [out=90, in=315] (-.5,2.25) to [out=45, in=270] (0,3) to (1,3) to (1,1.5) to (0,1.5);
	\draw[ystrand, directed=.999] (-1,1.5) to [out=90, in=270] (0,3);
	\draw[bstrand, directed=.999] (0,1.5) to [out=90, in=270] (-1,3);
\end{tikzpicture}}
&
\wc\bc\gc
}
,\quad\quad
\xymatrix@C+=1.3cm@L+=6pt{
\wc\yc\gc
\ar@<-4pt>@/_/[rr]|{\,\twomorstuff{id}_{\wc\yc\gc}\,}
\ar[r]^{
\begin{tikzpicture}[anchorbase, scale=.4, tinynodes]
	\draw[very thin, densely dotted, fill=white] (-1,1.5) to [out=90, in=225] (-.5,2.25) to [out=135, in=270] (-1,3) to (-2,3) to (-2,1.5) to (-1,1.5);
	\fill[myblue, opacity=0.3] (-1,3) to [out=270, in=135] (-.5,2.25) to [out=45, in=270] (0,3) to (-1,3);
	\fill[myyellow, opacity=0.3] (-1,1.5) to [out=90, in=225] (-.5,2.25) to [out=315, in=90] (0,1.5) to (-1,1.5);
	\fill[mygreen, opacity=0.8] (0,1.5) to [out=90, in=315] (-.5,2.25) to [out=45, in=270] (0,3) to (1,3) to (1,1.5) to (0,1.5);
	\draw[ystrand, directed=.999] (-1,1.5) to [out=90, in=270] (0,3);
	\draw[bstrand, directed=.999] (0,1.5) to [out=90, in=270] (-1,3);
\end{tikzpicture}}
&
\wc\bc\gc
\ar[r]^{
\begin{tikzpicture}[anchorbase, scale=.4, tinynodes]
	\draw[very thin, densely dotted, fill=white] (-1,1.5) to [out=90, in=225] (-.5,2.25) to [out=135, in=270] (-1,3) to (-2,3) to (-2,1.5) to (-1,1.5);
	\fill[myyellow, opacity=0.3] (-1,3) to [out=270, in=135] (-.5,2.25) to [out=45, in=270] (0,3) to (-1,3);
	\fill[myblue, opacity=0.3] (-1,1.5) to [out=90, in=225] (-.5,2.25) to [out=315, in=90] (0,1.5) to (-1,1.5);
	\fill[mygreen, opacity=0.8] (0,1.5) to [out=90, in=315] (-.5,2.25) to [out=45, in=270] (0,3) to (1,3) to (1,1.5) to (0,1.5);
	\draw[ystrand, directed=.999] (0,1.5) to [out=90, in=270] (-1,3);
	\draw[bstrand, directed=.999] (-1,1.5) to [out=90, in=270] (0,3);
\end{tikzpicture}}
&
\wc\yc\gc
}
\end{gather}
In particular, $\wc\bc\gc\cong\wc\yc\gc$.
The same holds for color variations with compatible colors.
\end{lemmaqed}

The following, where we 
silently use \fullref{lemma:clasps-well-defined}, should be 
compared to \fullref{lemma:quotient-of-affine}. 

\begin{lemma}\label{lemma:sts-decomp}
In $\Kar{\subcatquofield}$, 
the $1$-morphism $\wc\bc\gc\bc\wc\cong\wc\yc\gc\yc\wc$ is 
isomorphic to the indecomposable direct summand of 
$\wc\bc\wc\yc\wc\bc\wc$ or of $\wc\yc\wc\bc\wc\yc\wc$ which 
corresponds to the word $w_{\gc}=\bc\yc\bc=\yc\bc\yc\in\Wgroup_{\gc}$.
The same holds for all compatible color variations. 
\end{lemma}

Recall that $\adiagfield$ 
decategorifies to $\hecke(\typeat{2})$ 
(cf. \fullref{remark:cat-affine-a2}), such that 
the indecomposable $1$-morphisms in $\adiagfield$ decategorify 
to the KL basis elements of the affine type $\typea{2}$ Weyl group $\Wgroup$. 
Since $\subcatquofield$ is a $2$-full $2$-subcategory of $\adiagfield$, its indecomposable 
$1$-morphisms are also indecomposable as $1$-morphisms of the latter. Therefore, $\Kar{\subcatquofield}$ 
decategorifies to a subalgebra of $\hecke(\typeat{2})$, with a basis consisting of a 
particular subset of the KL basis elements.

\begin{proof}
By \eqref{eq:square}, we have 
\begin{gather}\label{eq:byb-decomp}
\begin{tikzpicture}[anchorbase, scale=.4, tinynodes]
	\draw[very thin, densely dotted, fill=white] (-3.5,2) to (-3.5,-2) to (-2.5,-2) to (-2.5,2) to (-3.5,2);
	\draw[very thin, densely dotted, fill=white] (-1.5,2) to (-1.5,-2) to (-.5,-2) to (-.5,2) to (-1.5,2);
	\draw[very thin, densely dotted, fill=white] (.5,2) to (.5,-2) to (1.5,-2) to (1.5,2) to (.5,2);
	\draw[very thin, densely dotted, fill=white] (2.5,2) to (2.5,-2) to (3.5,-2) to (3.5,2) to (2.5,2);
	\fill[myyellow, opacity=0.3] (-.5,2) to (-.5,-2) to (.5,-2) to (.5,2);
	\fill[myblue, opacity=0.3] (-2.5,2) to (-2.5,-2) to (-1.5,-2) to (-1.5,2);
	\fill[myblue, opacity=0.3] (1.5,2) to (1.5,-2) to (2.5,-2) to (2.5,2);
	\draw[ystrand, directed=.55] (-.5,-2) to (-.5,2);
	\draw[ystrand, directed=.55] (.5,2) to (.5,-2);
	\draw[bstrand, directed=.55] (-2.5,-2) to (-2.5,2);
	\draw[bstrand, directed=.55] (-1.5,2) to (-1.5,-2);
	\draw[bstrand, directed=.55] (1.5,-2) to (1.5,2);
	\draw[bstrand, directed=.55] (2.5,2) to (2.5,-2);
	\node at (-3,-.4) {$\wc$};
	\node at (-2,-.4) {$\bcblack$};
	\node at (-1,-.4) {$\wc$};
	\node at (0,-.4) {$\ycblack$};
	\node at (1,-.4) {$\wc$};
	\node at (2,-.4) {$\bcblack$};
	\node at (3,-.4) {$\wc$};
\end{tikzpicture}
=
\qpar\,
\begin{tikzpicture}[anchorbase, scale=.4, tinynodes]
	\draw[very thin, densely dotted, fill=white] (-3.5,2) to (-3.5,-2) to (-2.5,-2) to (-2.5,2) to (-3.5,2);
	\draw[very thin, densely dotted] (-1.5,-2) to (-.5,-2);
	\draw[very thin, densely dotted] (.5,-2) to (1.5,-2);
	\draw[very thin, densely dotted] (-1.5,2) to (-.5,2);
	\draw[very thin, densely dotted] (.5,2) to (1.5,2);
	\draw[very thin, densely dotted, fill=white] (2.5,2) to (2.5,-2) to (3.5,-2) to (3.5,2) to (2.5,2);
	\fill[mygreen, opacity=0.8] (-.5,-.8) to (-.5,.8) to (0,.75) to (.5,.8) to (.5,-.8);
	\fill[myyellow, opacity=0.3] (-.5,2) to (-.5,.8) to (0,.75) to (.5,.8) to (.5,2) to (-.5,2);
	\fill[myyellow, opacity=0.3] (-.5,-2) to (-.5,-.8) to (0,-.75) to (.5,-.8) to (.5,-2) to (-0.5,-2);
	\fill[myblue, opacity=0.3] (-2.5,-2) to (-2.5,2) to (-1.5,2) to [out=270, in=160] (-.5,.8) to (-0.5,-.8) to [out=200, in=90] (-1.5,-2) to (-2.5,-2);
	\fill[myblue, opacity=0.3] (2.5,-2) to  (1.5,-2) to [out=90, in=350] (.5,-.8) to (.5,.8) to [out=10, in=270] (1.5,2) to (2.5,2) to (2.5,-2);
	\draw[ystrand, directed=.55] (-.5,-2) to (-.5,2);
	\draw[ystrand, directed=.55] (.5,2) to (.5,-2);
	\draw[bstrand, directed=.55] (-2.5,-2) to (-2.5,2);
	\draw[bstrand, directed=.85] (-1.5,2) to [out=270, in=180] (0,.75) to [out=0, in=270] (1.5,2);
	\draw[bstrand, directed=.85] (1.5,-2) to [out=90, in=0] (0,-.75) to [out=180, in=90] (-1.5,-2);
	\draw[bstrand, directed=.55] (2.5,2) to (2.5,-2);
	\node at (-3,-.4) {$\wc$};
	\node at (-1.5,-.4) {$\bcblack$};
	\node at (0,-.4) {$\gcblack$};
	\node at (1.5,-.4) {$\bcblack$};
	\node at (3,-.4) {$\wc$};
\end{tikzpicture}
-
\qpar\,
\begin{tikzpicture}[anchorbase, scale=.4, tinynodes]
	\draw[very thin, densely dotted, fill=white] (-3.5,2) to (-3.5,-2) to (-2.5,-2) to (-2.5,2) to (-3.5,2);
	\draw[very thin, densely dotted] (-1.5,-2) to (-.5,-2);
	\draw[very thin, densely dotted] (.5,-2) to (1.5,-2);
	\draw[very thin, densely dotted] (-1.5,2) to (-.5,2);
	\draw[very thin, densely dotted] (.5,2) to (1.5,2);
	\draw[very thin, densely dotted, fill=white] (2.5,2) to (2.5,-2) to (3.5,-2) to (3.5,2) to (2.5,2);
	\fill[myyellow, opacity=0.3] (.5,2) to [out=270, in=0] (0,1.5) to [out=180, in=270] (-.5,2) to (.5,2);
	\fill[myyellow, opacity=0.3] (-.5,-2) to [out=90, in=180] (0,-1.5) to [out=0, in=90] (.5,-2) to (-.5,-2);
	\fill[myblue, opacity=0.3] (-2.5,-2) to (-2.5,2) to (-1.5,2) to [out=270, in=180] (0,.75) to [out=0, in=270] (1.5,2) to (2.5,2) to (2.5,-2)  to (1.5,-2) to [out=90, in=0] (0,-.75) to [out=180, in=90] (-1.5,-2) to (-2.5,-2);
	\draw[ystrand, directed=.85] (0.5,2) to [out=270, in=0] (0,1.5) to [out=180, in=270] (-.5,2);
	\draw[ystrand, directed=.85] (-0.5,-2) to [out=90, in=180] (0,-1.5) to [out=0, in=90] (.5,-2);
	\draw[bstrand, directed=.55] (-2.5,-2) to (-2.5,2);
	\draw[bstrand, directed=.85] (-1.5,2) to [out=270, in=180] (0,.75) to [out=0, in=270] (1.5,2);
	\draw[bstrand, directed=.85] (1.5,-2) to [out=90, in=0] (0,-.75) to [out=180, in=90] (-1.5,-2);
	\draw[bstrand, directed=.55] (2.5,2) to (2.5,-2);
	\node at (-3,-.4) {$\wc$};
	\node at (0,-.4) {$\bcblack$};
	\node at (3,-.4) {$\wc$};
\end{tikzpicture}
\end{gather}
It is not hard to check, using the relations in 
\fullref{definition:ssbim},
that \eqref{eq:byb-decomp} gives a 
decomposition into orthogonal idempotents. 

Note that the first idempotent on the right-hand side 
shows that $\wc\bc\gc\bc\wc$ is a direct summand of 
$\wc\bc\wc\yc\wc\bc\wc$, as indicated in \eqref{eq:byb-decomp}, i.e. 
$\wc\bc\wc\yc\wc\bc\wc\cong\wc\bc\gc\bc\wc\oplus\wc\bc\wc$. 
This decomposition decategorifies to  
\[
\theta_{w_{\gc}}
\stackrel{\eqref{eq:cubic}}{=}
\theta_{\bc}\theta_{\yc}\theta_{\bc}=\theta_{\bc\yc\bc} + \theta_{\bc},
\]
and it then follows from \cite[Theorem A.1]{El1} that the idempotents on 
the right-hand side of \eqref{eq:byb-decomp} are primitive.
This shows the lemma in 
case of $\bc$, $\yc$ and $\gc$. The other cases are analogous.
\end{proof}

\begin{lemma}\label{lemma:remove-white}
We have $\bc\wc\bc\cong\bc\{+1\}\oplus\bc\{-1\}$ in $\Adiag$. 
A similar result holds for all compatible color variations, keeping the color $\wc$.
\end{lemma}

Note that \fullref{lemma:remove-white} is only true for primary 
colors, since $\wc$ is never compatible to a secondary color.

\begin{proof}
This follows from the following diagram.
\begin{gather*}
\xymatrix@C+=2.5cm@L+=6pt{
 & \bc\{+1\} 
\ar[rd]|{\tfrac{1}{2}\,
\begin{tikzpicture}[anchorbase, scale=.4, tinynodes]
	\draw[very thin, densely dotted, fill=white] (-.5,2) to [out=270, in=180] (0,1) to [out=0, in=270] (.5,2) to (-.5,2);
	\fill[myblue, opacity=0.3] (-.5,2) to [out=270, in=180] (0,1) to [out=0, in=270] (.5,2) to (1.5,2) to (1.5,0) to (-1.5,0) to (-1.5,2) to (-.5,2);
	\draw[bstrand] (-.5,2) to [out=270, in=180] (0,1) to [out=0, in=270] (.5,2);
	\draw[bstrand, directed=.55] (.1,1) to (.11,1);
\end{tikzpicture}} 
 & 
\\
\bc\wc\bc \ar[ru]|{
\begin{tikzpicture}[anchorbase, scale=.4, tinynodes]
	\draw[very thin, densely dotted, fill=white] (-.5,0) to [out=90, in=180] (0,1) to [out=0, in=90] (.5,0) to (-.5,0);
	\fill[myblue, opacity=0.3] (-.5,0) to [out=90, in=180] (0,1) to [out=0, in=90] (.5,0) to (1.5,0) to (1.5,2) to (-1.5,2) to (-1.5,0) to (-.5,0);
	\draw[bstrand] (-.5,0) to [out=90, in=180] (0,1) to [out=0, in=90] (.5,0);
	\draw[bstrand, directed=.55] (-.1,1) to (-.11,1);
	\node at (0,.4) {$\rootb$};
\end{tikzpicture}}
\ar[rd]|{
\begin{tikzpicture}[anchorbase, scale=.4, tinynodes]
	\draw[very thin, densely dotted, fill=white] (-.5,0) to [out=90, in=180] (0,1) to [out=0, in=90] (.5,0) to (-.5,0);
	\fill[myblue, opacity=0.3] (-.5,0) to [out=90, in=180] (0,1) to [out=0, in=90] (.5,0) to (1.5,0) to (1.5,2) to (-1.5,2) to (-1.5,0) to (-.5,0);
	\draw[bstrand] (-.5,0) to [out=90, in=180] (0,1) to [out=0, in=90] (.5,0);
	\draw[bstrand, directed=.55] (-.1,1) to (-.11,1);
\end{tikzpicture}}
 & & \bc\wc\bc
\\
 & \bc\{-1\} 
\ar[ru]|{\tfrac{1}{2}\,
\begin{tikzpicture}[anchorbase, scale=.4, tinynodes]
	\draw[very thin, densely dotted, fill=white] (-.5,2) to [out=270, in=180] (0,1) to [out=0, in=270] (.5,2) to (-.5,2);
	\fill[myblue, opacity=0.3] (-.5,2) to [out=270, in=180] (0,1) to [out=0, in=270] (.5,2) to (1.5,2) to (1.5,0) to (-1.5,0) to (-1.5,2) to (-.5,2);
	\draw[bstrand] (-.5,2) to [out=270, in=180] (0,1) to [out=0, in=270] (.5,2);
	\draw[bstrand, directed=.55] (.1,1) to (.11,1);
	\node at (0,1.6) {$\rootb$};
\end{tikzpicture}}  
 & 
\\
}
\end{gather*}
Observing that $\qDema_{\bc}(\rootb)=2$, $\qDema_{\bc}(\rootb^2)=0$ and 
$\Frobel{\bc}{\phantom{.}}=\tfrac{1}{2}(\rootb\otimes 1+1\otimes\rootb)$, 
and using the relations in \fullref{definition:ssbim} , one can check that the left and right column in this diagram define 
mutual inverses: In this setup $2$-morphisms correspond to matrices of
diagrams, such that composition corresponds to matrix multiplication. The two $2$-morphisms 
above corresponding to the column matrices are
each other's inverses with respect to this composition. 
Similar arguments can be used for all other compatible color variations.
\end{proof}

\begin{remark}\label{remark:no-white-regions}
Using \fullref{lemma:remove-white} we simplify 
our diagrams and do not illustrate 
$\wc$ colored regions in the middle, if not necessary. For example, 
$\theta_\gc\theta_\gc$ should be thought 
of as corresponding to $\wc\bc\gc\bc\gc\bc\wc$ 
instead of $\wc\bc\gc\bc\wc\bc\gc\bc\wc$. 
However, the appearing grading shift in this simplification 
is exactly the reason for the scaling by powers of 
$\vnumber{2}^{-1}$ in \fullref{section:funny-algebra}.
\end{remark}

\begin{definitionnoqed}\label{definition:the-gog-gpg-iso}
Define the following $2$-morphisms in $\Adiag$ (and similar ones for other colors).
\begin{gather}\label{eq:o-and-u-cross}
\begin{gathered}
\begin{tikzpicture}[anchorbase, scale=.4, tinynodes]
	\fill[mygreen, opacity=0.8] (-1,1.5) to [out=90, in=225] (-.5,2.25) to [out=135, in=270] (-1,3) to (-1.5,3) to (-1.5,1.5) to (-1,1.5);
	\fill[mypurple, opacity=0.8] (-1,3) to [out=270, in=135] (-.5,2.25) to [out=45, in=270] (0,3) to (-1,3);
	\fill[myorange, opacity=0.8] (-1,1.5) to [out=90, in=225] (-.5,2.25) to [out=315, in=90] (0,1.5) to (-1,1.5);
	\fill[mygreen, opacity=0.8] (0,1.5) to [out=90, in=315] (-.5,2.25) to [out=45, in=270] (0,3) to (.5,3) to (.5,1.5) to (0,1.5);
	\draw[dstrand, Xmarked=.999] (0,1.5) to [out=90, in=270] (-1,3);
	\draw[dstrand, Xmarked=.999] (0,3) to [out=270, in=90] (-1,1.5);
\end{tikzpicture}
=
\begin{tikzpicture}[anchorbase, scale=.4, tinynodes]
	\fill[mygreen, opacity=0.8] (-3,0) to (-2,0) to [out=90, in=225] (-1.25,2) to [out=135, in=270] (-2,4) to (-3,4) to (-3,0);
	\fill[myyellow, opacity=0.3] (0,0) to (0,1.1) to (-1.25,2) to [out=225, in=90] (-2,0) to (0,0);
	\fill[myblue, opacity=0.3] (0,4) to (0,2.9) to (-1.25,2) to [out=135, in=270] (-2,4) to (0,4);
	\fill[myorange, opacity=0.8] (0,0) to (2,0) to (2,1.1) to (1,1) to (0,1.1) to (0,0);
	\fill[myred, opacity=0.3] (2,1.1) to (1,1) to (0,1.1) to (0,2.9) to (1,3) to (2,2.9) to (2,1.1);
	\fill[mypurple, opacity=0.8] (0,4) to (2,4) to (2,2.9) to (1,3) to (0,2.9) to (0,0);
	\fill[myyellow, opacity=0.3] (2,0) to (2,1.1) to (3.25,2) to [out=315, in=90] (4,0) to (2,0);
	\fill[myblue, opacity=0.3] (2,4) to (2,2.9) to (3.25,2) to [out=45, in=270] (4,4) to (2,4);
	\fill[mygreen, opacity=0.8] (5,0) to (4,0) to [out=90, in=315] (3.25,2) to [out=45, in=270] (4,4) to (5,4) to (5,0);
	\draw[ystrand, directed=.525] (-2,4) to [out=270, in=180] (1,1) to [out=0, in=270] (4,4);
	\draw[rstrand, directed=.55] (0,0) to (0,4);
	\draw[rstrand, rdirected=.55] (2,0) to (2,4);	
	\draw[bstrand, rdirected=.525] (-2,0) to [out=90, in=180] (1,3) to [out=0, in=90] (4,0);
\end{tikzpicture}
+\qpar^{-1}
\begin{tikzpicture}[anchorbase, scale=.4, tinynodes]
	\fill[myblue, opacity=0.3] (-2,4) to [out=270, in=180] (1,2.25) to [out=0, in=270] (4,4) to (2,4) to [out=270, in=0] (1,3) to [out=180, in=270] (0,4) to (-2,4);
	\fill[mypurple, opacity=0.8] (0,4) to (0,4) to [out=270, in=180] (1,3) to [out=0, in=270] (2,4);
	\fill[myorange, opacity=0.8] (0,0) to [out=90, in=180] (1,1) to [out=0, in=90] (2,0) to (0,0);
	\fill[myyellow, opacity=0.3] (-2,0) to [out=90, in=180] (1,1.75) to [out=0, in=90] (4,0) to (2,0) to [out=90, in=0] (1,1) to [out=180, in=90] (0,0) to (-2,0);
	\fill[mygreen, opacity=0.8] (-2,0) to [out=90, in=180] (1,1.75) to [out=0, in=90] (4,0) to (5,0) to (5,4) to (4,4) to [out=270, in=0] (1,2.25) to [out=180, in=270] (-2,4) to (-3,4) to (-3,0) to (-2,0);
	\draw[bstrand, rdirected=.525] (-2,0) to [out=90, in=180] (1,1.75) to [out=0, in=90] (4,0);
	\draw[rstrand, directed=.55] (0,0) to [out=90, in=180] (1,1) to [out=0, in=90] (2,0);
	\draw[rstrand, rdirected=.55] (0,4) to [out=270, in=180] (1,3) to [out=0, in=270] (2,4);	
	\draw[ystrand, directed=.525] (-2,4) to [out=270, in=180] (1,2.25) to [out=0, in=270] (4,4);
\end{tikzpicture}
\\
\begin{tikzpicture}[anchorbase, scale=.4, tinynodes]
	\fill[mygreen, opacity=0.8] (-1,1.5) to [out=90, in=225] (-.5,2.25) to [out=135, in=270] (-1,3) to (-1.5,3) to (-1.5,1.5) to (-1,1.5);
	\fill[myorange, opacity=0.8] (-1,3) to [out=270, in=135] (-.5,2.25) to [out=45, in=270] (0,3) to (-1,3);
	\fill[mypurple, opacity=0.8] (-1,1.5) to [out=90, in=225] (-.5,2.25) to [out=315, in=90] (0,1.5) to (-1,1.5);
	\fill[mygreen, opacity=0.8] (0,1.5) to [out=90, in=315] (-.5,2.25) to [out=45, in=270] (0,3) to (.5,3) to (.5,1.5) to (0,1.5);
	\draw[dstrand, Xmarked=.999] (-1,1.5) to [out=90, in=270] (0,3);
	\draw[dstrand, Xmarked=.999] (-1,3) to [out=270, in=90] (0,1.5);
\end{tikzpicture}
=
\begin{tikzpicture}[anchorbase, scale=.4, tinynodes]
	\fill[mygreen, opacity=0.8] (-3,0) to (-2,0) to [out=90, in=225] (-1.25,2) to [out=135, in=270] (-2,4) to (-3,4) to (-3,0);
	\fill[myblue, opacity=0.3] (0,0) to (0,1.1) to (-1.25,2) to [out=225, in=90] (-2,0) to (0,0);
	\fill[myyellow, opacity=0.3] (0,4) to (0,2.9) to (-1.25,2) to [out=135, in=270] (-2,4) to (0,4);
	\fill[mypurple, opacity=0.8] (0,0) to (2,0) to (2,1.1) to (1,1) to (0,1.1) to (0,0);
	\fill[myred, opacity=0.3] (2,1.1) to (1,1) to (0,1.1) to (0,2.9) to (1,3) to (2,2.9) to (2,1.1);
	\fill[myorange, opacity=0.8] (0,4) to (2,4) to (2,2.9) to (1,3) to (0,2.9) to (0,0);
	\fill[myblue, opacity=0.3] (2,0) to (2,1.1) to (3.25,2) to [out=315, in=90] (4,0) to (2,0);
	\fill[myyellow, opacity=0.3] (2,4) to (2,2.9) to (3.25,2) to [out=45, in=270] (4,4) to (2,4);
	\fill[mygreen, opacity=0.8] (5,0) to (4,0) to [out=90, in=315] (3.25,2) to [out=45, in=270] (4,4) to (5,4) to (5,0);
	\draw[ystrand, rdirected=.525] (-2,0) to [out=90, in=180] (1,3) to [out=0, in=90] (4,0);
	\draw[rstrand, directed=.55] (0,0) to (0,4);
	\draw[rstrand, rdirected=.55] (2,0) to (2,4);	
	\draw[bstrand, directed=.525] (-2,4) to [out=270, in=180] (1,1) to [out=0, in=270] (4,4);
\end{tikzpicture}
+\qpar^{\phantom{-1}}
\begin{tikzpicture}[anchorbase, scale=.4, tinynodes]
	\fill[myyellow, opacity=0.3] (-2,4) to [out=270, in=180] (1,2.25) to [out=0, in=270] (4,4) to (2,4) to [out=270, in=0] (1,3) to [out=180, in=270] (0,4) to (-2,4);
	\fill[myorange, opacity=0.8] (0,4) to (0,4) to [out=270, in=180] (1,3) to [out=0, in=270] (2,4);
	\fill[mypurple, opacity=0.8] (0,0) to [out=90, in=180] (1,1) to [out=0, in=90] (2,0) to (0,0);
	\fill[myblue, opacity=0.3] (-2,0) to [out=90, in=180] (1,1.75) to [out=0, in=90] (4,0) to (2,0) to [out=90, in=0] (1,1) to [out=180, in=90] (0,0) to (-2,0);
	\fill[mygreen, opacity=0.8] (-2,0) to [out=90, in=180] (1,1.75) to [out=0, in=90] (4,0) to (5,0) to (5,4) to (4,4) to [out=270, in=0] (1,2.25) to [out=180, in=270] (-2,4) to (-3,4) to (-3,0) to (-2,0);
	\draw[bstrand, directed=.525] (-2,4) to [out=270, in=180] (1,2.25) to [out=0, in=270] (4,4);
	\draw[rstrand, directed=.55] (0,0) to [out=90, in=180] (1,1) to [out=0, in=90] (2,0);
	\draw[rstrand, rdirected=.55] (0,4) to [out=270, in=180] (1,3) to [out=0, in=270] (2,4);	
	\draw[ystrand, rdirected=.525] (-2,0) to [out=90, in=180] (1,1.75) to [out=0, in=90] (4,0);
\end{tikzpicture}
\end{gathered}
\hspace{3.1cm}
\raisebox{-1.55cm}{\makeqedtri}
\hspace{-3.1cm}
\end{gather}
\end{definitionnoqed}

\begin{lemma}\label{lemma:thecrossings-2}
The following diagrams commute in $\Adiag$.
\begin{gather}\label{eq:clasps-well-defined-2}
\xymatrix@C+=1.3cm@L+=6pt{
\gc\bc\pc\bc\gc
\ar@<-4pt>@/_/[rr]|{\,\twomorstuff{id}_{\gc\bc\pc\bc\gc}\,}
\ar[r]^{
\begin{tikzpicture}[anchorbase, scale=.4, tinynodes]
	\fill[mygreen, opacity=0.8] (-1,1.5) to [out=90, in=225] (-.5,2.25) to [out=135, in=270] (-1,3) to (-1.5,3) to (-1.5,1.5) to (-1,1.5);
	\fill[mypurple, opacity=0.8] (-1,3) to [out=270, in=135] (-.5,2.25) to [out=45, in=270] (0,3) to (-1,3);
	\fill[myorange, opacity=0.8] (-1,1.5) to [out=90, in=225] (-.5,2.25) to [out=315, in=90] (0,1.5) to (-1,1.5);
	\fill[mygreen, opacity=0.8] (0,1.5) to [out=90, in=315] (-.5,2.25) to [out=45, in=270] (0,3) to (.5,3) to (.5,1.5) to (0,1.5);
	\draw[dstrand, Xmarked=.999] (0,1.5) to [out=90, in=270] (-1,3);
	\draw[dstrand, Xmarked=.999] (0,3) to [out=270, in=90] (-1,1.5);
\end{tikzpicture}}
&
\gc\yc\oc\yc\gc
\ar[r]^{
\begin{tikzpicture}[anchorbase, scale=.4, tinynodes]
	\fill[mygreen, opacity=0.8] (-1,1.5) to [out=90, in=225] (-.5,2.25) to [out=135, in=270] (-1,3) to (-1.5,3) to (-1.5,1.5) to (-1,1.5);
	\fill[myorange, opacity=0.8] (-1,3) to [out=270, in=135] (-.5,2.25) to [out=45, in=270] (0,3) to (-1,3);
	\fill[mypurple, opacity=0.8] (-1,1.5) to [out=90, in=225] (-.5,2.25) to [out=315, in=90] (0,1.5) to (-1,1.5);
	\fill[mygreen, opacity=0.8] (0,1.5) to [out=90, in=315] (-.5,2.25) to [out=45, in=270] (0,3) to (.5,3) to (.5,1.5) to (0,1.5);
	\draw[dstrand, Xmarked=.999] (-1,1.5) to [out=90, in=270] (0,3);
	\draw[dstrand, Xmarked=.999] (-1,3) to [out=270, in=90] (0,1.5);
\end{tikzpicture}}
&
\gc\bc\pc\bc\gc
}
,\quad\quad
\xymatrix@C+=1.3cm@L+=6pt{
\gc\bc\pc\bc\gc
\ar@<-4pt>@/_/[rr]|{\,\twomorstuff{id}_{\gc\bc\pc\bc\gc}\,}
\ar[r]^{
\begin{tikzpicture}[anchorbase, scale=.4, tinynodes]
	\fill[mygreen, opacity=0.8] (-1,1.5) to [out=90, in=225] (-.5,2.25) to [out=135, in=270] (-1,3) to (-1.5,3) to (-1.5,1.5) to (-1,1.5);
	\fill[myorange, opacity=0.8] (-1,3) to [out=270, in=135] (-.5,2.25) to [out=45, in=270] (0,3) to (-1,3);
	\fill[mypurple, opacity=0.8] (-1,1.5) to [out=90, in=225] (-.5,2.25) to [out=315, in=90] (0,1.5) to (-1,1.5);
	\fill[mygreen, opacity=0.8] (0,1.5) to [out=90, in=315] (-.5,2.25) to [out=45, in=270] (0,3) to (.5,3) to (.5,1.5) to (0,1.5);
	\draw[dstrand, Xmarked=.999] (-1,1.5) to [out=90, in=270] (0,3);
	\draw[dstrand, Xmarked=.999] (-1,3) to [out=270, in=90] (0,1.5);
\end{tikzpicture}}
&
\gc\yc\oc\yc\gc
\ar[r]^{
\begin{tikzpicture}[anchorbase, scale=.4, tinynodes]
	\fill[mygreen, opacity=0.8] (-1,1.5) to [out=90, in=225] (-.5,2.25) to [out=135, in=270] (-1,3) to (-1.5,3) to (-1.5,1.5) to (-1,1.5);
	\fill[mypurple, opacity=0.8] (-1,3) to [out=270, in=135] (-.5,2.25) to [out=45, in=270] (0,3) to (-1,3);
	\fill[myorange, opacity=0.8] (-1,1.5) to [out=90, in=225] (-.5,2.25) to [out=315, in=90] (0,1.5) to (-1,1.5);
	\fill[mygreen, opacity=0.8] (0,1.5) to [out=90, in=315] (-.5,2.25) to [out=45, in=270] (0,3) to (.5,3) to (.5,1.5) to (0,1.5);
	\draw[dstrand, Xmarked=.999] (0,1.5) to [out=90, in=270] (-1,3);
	\draw[dstrand, Xmarked=.999] (0,3) to [out=270, in=90] (-1,1.5);
\end{tikzpicture}}
&
\gc\bc\pc\bc\gc
}
\end{gather}
In particular, $\gc\yc\oc\yc\gc\cong\gc\bc\pc\bc\gc$.
A similar result holds for other colors.
\end{lemma}

\begin{proof}
We only prove that the left diagram commutes. To this end, 
we write $\twomorstuff{f}_1$ and 
$\qpar^{-1}\twomorstuff{f}_2$ for the two summands on the right-hand side of the top 
equality in \eqref{eq:o-and-u-cross}, and similarly $\twomorstuff{g}_1$ and $\qpar\twomorstuff{g}_2$ for the bottom 
equality. Using \eqref{eq:square}, followed by \eqref{eq:rm-first} and 
\cite[Claim 3.14]{El1}, we get $\twomorstuff{g}_1\vcomp\twomorstuff{f}_1
=\twomorstuff{id}_{\gc\bc\pc\bc\gc}+\twomorstuff{h}$.
Moreover, by first using \eqref{eq:rm-first} and \eqref{eq:rm-third} 
(and $\qDema_{\bc}\frobel{\oc}{\rc,\yc}=-\qnumber{2}^{2}$),
and then \eqref{eq:circle-primary} and 
\fullref{example:more-relations}, we get 
$\twomorstuff{g}_1\vcomp\qpar^{-1}\twomorstuff{f}_2+
\qpar\twomorstuff{g}_2\vcomp\twomorstuff{f}_1=-\qnumber{2}^{2}\twomorstuff{h}$. 
Finally, $\qpar\twomorstuff{g}_2\vcomp\qpar^{-1}\twomorstuff{f}_2=\qnumber{3}\twomorstuff{h}$ 
follows from \fullref{example:some-relations}, and we are done 
since $\qnumber{3}=\qnumber{2}^{2}-1$.
\end{proof}

\subsubsection{Categorifying \texorpdfstring{$\subquo$}{Tinfty}}\label{subsec:cat-thealgebra}

\begin{proposition}\label{proposition:cat-the-algebra}
The assignment given by 
\[
\theta_{\gc}\mapsto[\wc\bc\gc\bc\wc]=[\wc\yc\gc\yc\wc], 
\quad\quad 
\theta_{\oc}\mapsto[\wc\rc\oc\rc\wc]=[\wc\yc\oc\yc\wc],
\quad\quad  
\theta_{\pc}\mapsto[\wc\bc\pc\bc\wc]=[\wc\rc\pc\rc\wc],
\]
defines an isomorphism 
$\subquo\xrightarrow{\cong}\GGcv{\Kar{\subcatquofield}}$ 
of algebras.

Under this isomorphism, the elements 
of the basis $\basisC$ (or of $\Cbasis$) from \fullref{proposition:two-bases}
correspond to a complete set of indecomposables in 
$\Kar{\subcatquofield}$ (up to grading shift).
\end{proposition}

\makeautorefname{lemma}{Lemmas}

\begin{proof}
This follows now directly from \fullref{remark:cat-affine-a2}, and
\fullref{lemma:quotient-of-affine}, \ref{lemma:clasps-well-defined} 
and \ref{lemma:sts-decomp}.
\end{proof}

\makeautorefname{lemma}{Lemma}

\subsection{The \texorpdfstring{$2$}{2}-quotient of level \texorpdfstring{$e$}{e}}\label{subsec:quotient-category}

The quotient $\subquo[e]$ of $\subquo$ 
from \fullref{definition:the-quotient-defined} is defined 
by killing certain elements 
which correspond to the irreducibles 
$\Ll_{m,n}$ for $m+n=e+1$
in the representation 
category of $\slt$. We follow 
the same strategy on the categorified level. 

\subsubsection{From \texorpdfstring{$\slt$}{sl3} to singular bimodules: the generic case}\label{subsubsec:webs-to-soergel}

Recall from \fullref{subsec:generic} that $\sltcat$ 
denotes the category of finite-dimensional $\Uslt$-modules 
(with $\Cq$ being the ground field). 
The central character \eqref{eq:color-code} allows us to view 
$\sltcat$ as a $2$-category $\sltcatgop$:

\begin{definition}\label{definition:colored-webs-1}
For $\tduc$, let $\sltcat^{\tduc}$ denote the full subcategory 
of $\sltcat$ generated by the irreducibles with central character $\tduc$.
\end{definition}

Note that the subcategories 
$\sltcat^{\tduc}$ are not monoidal. However,
by \fullref{lemma:central-character}, tensoring with $\fu$ or $\fud$ defines functors between them 
\begin{gather}\label{eq:some-colored-endos}
\xy
(0,0)*{
\fuf{\tduc}{\rho(\tduc)}=\fu\otimes\placeholder
\colon
\sltcat^{\tduc}\to\sltcat^{\rho(\tduc)},
\quad\quad
\fudf{\tduc}{\rho^{-1}(\tduc)}=\fud\otimes\placeholder
\colon
\sltcat^{\tduc}\to\sltcat^{\rho^{-1}(\tduc)},};
\endxy
\end{gather}
which the reader should compare to \eqref{eq:color-tensor}. We will (reading 
right to left) depict them by
\[
\fuf{\gc}{\oc}
=
\begin{tikzpicture}[anchorbase, scale=.4, tinynodes]
	\fill[mygreen, opacity=0.8] (0,0) to (0,2) to (1,2) to (1,0) to (0,0);
	\fill[myorange, opacity=0.8] (0,0) to (0,2) to (-1,2) to (-1,0) to (0,0);
	\draw[dstrand, Xmarked=.55] (0,0) node [below] {$\fu$} to (0,2) node [above] {$\fu$};
\end{tikzpicture}
\colon
\sltcat^{\gc}\to\sltcat^{\oc},
\quad\quad
\fudf{\oc}{\gc}
=
\begin{tikzpicture}[anchorbase, scale=.4, tinynodes]
	\fill[myorange, opacity=0.8] (0,0) to (0,2) to (1,2) to (1,0) to (0,0);
	\fill[mygreen, opacity=0.8] (0,0) to (0,2) to (-1,2) to (-1,0) to (0,0);
	\draw[dstrand, Ymarked=.55] (0,0) node [below] {$\fud$} to (0,2) node [above] {$\fud$};
\end{tikzpicture}
\colon
\sltcat^{\oc}\to\sltcat^{\gc},
\quad\quad
\fudf{\oc}{\gc}\circ\fuf{\gc}{\oc}
=
\begin{tikzpicture}[anchorbase, scale=.4, tinynodes]
	\fill[mygreen, opacity=0.8] (0,0) to (0,2) to (1,2) to (1,0) to (0,0);
	\fill[mygreen, opacity=0.8] (-2,0) to (-2,2) to (-1,2) to (-1,0) to (-2,0);
	\fill[myorange, opacity=0.8] (0,0) to (0,2) to (-1,2) to (-1,0) to (0,0);
	\draw[dstrand, Xmarked=.55] (0,0) node [below] {$\fu$} to (0,2) node [above] {$\fu$};
	\draw[dstrand, Ymarked=.55] (-1,0) node [below] {$\fud$} to (-1,2) node [above] {$\fud$};
\end{tikzpicture}
\]
etc. The orientation is such that the color on the
left-hand side comes directly after the color on the right-hand side in the cyclic ordering
determined by $\rho$ in \eqref{eq:color-tensor}.
(We omit the $\fu$ and the $\fud$ 
in the pictures from now on.)

\begin{definition}\label{definition:colored-webs-2}
We define $\sltcatgop$ to be the additive, $\Cq$-linear closure of the
$2$-category whose objects are the categories $\sltcat^{\tduc}$, 
whose $1$-morphisms are composites of the functors in \eqref{eq:some-colored-endos}, 
and whose $2$-morphisms are natural transformations.
\end{definition} 
 
A natural transformation between composites of 
the functors from \eqref{eq:some-colored-endos} is the same as 
an $\Uslt$-equivariant map, see e.g. \cite[Proposition 2.5.4]{EGNO}. Therefore, we define
\begin{gather}\label{eq:basic-webs}
\begin{aligned}
&
\begin{tikzpicture}[anchorbase, scale=.4, tinynodes]
	\fill[mygreen, opacity=0.8] (0,0) to [out=270, in=180] (1,-1) to [out=0, in=270] (2,0) to (2.5,0) to (2.5,-2) to (-.5,-2) to (-.5,0) to (0,0);
	\fill[myorange, opacity=0.8] (0,0) to [out=270, in=180] (1,-1) to [out=0, in=270] (2,0) to (0,0);
	\draw[dstrand, Xmarked=.55] (0,0) to [out=270, in=180] (1,-1) to [out=0, in=270] (2,0);
\end{tikzpicture}
\colon
\begin{matrix}
\fud\fu
\\
\rotatebox{90}{$\hookrightarrow$}
\\
\Cq
\end{matrix},
\\
&
\begin{tikzpicture}[anchorbase, scale=.4, tinynodes]
	\fill[mygreen, opacity=0.8] (0,-1) to (0,0) to (1,1) to (1.5,1) to (1.5,-1) to (0,-1);
	\fill[myorange, opacity=0.8] (0,0) to (-1,1) to (1,1) to (0,0); 
	\fill[mypurple, opacity=0.8] (0,1) to (0,0) to (-1,1) to (-1.5,1) to (-1.5,-1) to (0,-1);
	\draw[dstrand, Ymarked=.55] (-1,1) to (0,0);
	\draw[dstrand, Ymarked=.55] (1,1) to (0,0);
	\draw[dstrand, Xmarked=.55] (0,0) to (0,-1);
\end{tikzpicture}
\colon
\begin{matrix}
\fu\fu
\\
\rotatebox{90}{$\hookrightarrow$}
\\
\fud
\end{matrix},
\end{aligned}
\quad\quad
\begin{aligned}
&
\begin{tikzpicture}[anchorbase, scale=.4, tinynodes]
	\fill[mygreen, opacity=0.8] (0,0) to [out=90, in=180] (1,1) to [out=0, in=90] (2,0) to (2.5,0) to (2.5,2) to (-.5,2) to (-.5,0) to (0,0);
	\fill[myorange, opacity=0.8] (0,0) to [out=90, in=180] (1,1) to [out=0, in=90] (2,0) to (0,0);
	\draw[dstrand, Xmarked=.55] (2,0) to [out=90, in=0] (1,1) to [out=180, in=90] (0,0);
\end{tikzpicture}
\colon
\begin{matrix}
\Cq
\\
\rotatebox{90}{$\twoheadrightarrow$}
\\
\fud\fu
\end{matrix},
\\
&
\begin{tikzpicture}[anchorbase, scale=.4, tinynodes]
	\fill[mygreen, opacity=0.8] (0,1) to (0,0) to (1,-1) to (1.5,-1) to (1.5,1) to (0,1);
	\fill[myorange, opacity=0.8] (0,0) to (-1,-1) to (1,-1) to (0,0); 
	\fill[mypurple, opacity=0.8] (0,1) to (0,0) to (-1,-1) to (-1.5,-1) to (-1.5,1) to (0,1);
	\draw[dstrand, Xmarked=.55] (-1,-1) to (0,0);
	\draw[dstrand, Xmarked=.55] (1,-1) to (0,0);
	\draw[dstrand, Ymarked=.55] (0,0) to (0,1);
\end{tikzpicture}
\colon
\begin{matrix}
\fud
\\
\rotatebox{90}{$\twoheadrightarrow$}
\\
\fu\fu
\end{matrix},
\end{aligned}
\quad\quad
\begin{aligned}
&
\begin{tikzpicture}[anchorbase, scale=.4, tinynodes]
	\fill[mygreen, opacity=0.8] (0,0) to [out=270, in=180] (1,-1) to [out=0, in=270] (2,0) to (2.5,0) to (2.5,-2) to (-.5,-2) to (-.5,0) to (0,0);
	\fill[mypurple, opacity=0.8] (0,0) to [out=270, in=180] (1,-1) to [out=0, in=270] (2,0) to (0,0);
	\draw[dstrand, Ymarked=.55] (0,0) to [out=270, in=180] (1,-1) to [out=0, in=270] (2,0);
\end{tikzpicture}
\colon
\begin{matrix}
\fu\fud
\\
\rotatebox{90}{$\hookrightarrow$}
\\
\Cq
\end{matrix},
\\
&
\begin{tikzpicture}[anchorbase, scale=.4, tinynodes]
	\fill[mygreen, opacity=0.8] (0,-1) to (0,0) to (1,1) to (1.5,1) to (1.5,-1) to (0,-1);
	\fill[myorange, opacity=0.8] (0,1) to (0,0) to (-1,1) to (-1.5,1) to (-1.5,-1) to (0,-1);
	\fill[mypurple, opacity=0.8] (0,0) to (-1,1) to (1,1) to (0,0); 
	\draw[dstrand, Xmarked=.55] (-1,1) to (0,0);
	\draw[dstrand, Xmarked=.55] (1,1) to (0,0);
	\draw[dstrand, Ymarked=.55] (0,0) to (0,-1);
\end{tikzpicture}
\colon
\begin{matrix}
\fud\fud
\\
\rotatebox{90}{$\hookrightarrow$}
\\
\fu
\end{matrix},
\end{aligned}
\quad\quad
\begin{aligned}
&
\begin{tikzpicture}[anchorbase, scale=.4, tinynodes]
	\fill[mygreen, opacity=0.8] (0,0) to [out=90, in=180] (1,1) to [out=0, in=90] (2,0) to (2.5,0) to (2.5,2) to (-.5,2) to (-.5,0) to (0,0);
	\fill[mypurple, opacity=0.8] (0,0) to [out=90, in=180] (1,1) to [out=0, in=90] (2,0) to (0,0);
	\draw[dstrand, Ymarked=.55] (2,0) to [out=90, in=0] (1,1) to [out=180, in=90] (0,0);
\end{tikzpicture}
\colon
\begin{matrix}
\Cq
\\
\rotatebox{90}{$\twoheadrightarrow$}
\\
\fu\fud
\end{matrix},
\\
&
\begin{tikzpicture}[anchorbase, scale=.4, tinynodes]
	\fill[mygreen, opacity=0.8] (0,1) to (0,0) to (1,-1) to (1.5,-1) to (1.5,1) to (0,1);
	\fill[myorange, opacity=0.8] (0,1) to (0,0) to (-1,-1) to (-1.5,-1) to (-1.5,1) to (0,1);
	\fill[mypurple, opacity=0.8] (0,0) to (-1,-1) to (1,-1) to (0,0);
	\draw[dstrand, Ymarked=.55] (-1,-1) to (0,0);
	\draw[dstrand, Ymarked=.55] (1,-1) to (0,0);
	\draw[dstrand, Xmarked=.55] (0,0) to (0,1);
\end{tikzpicture}
\colon
\begin{matrix}
\fu
\\
\rotatebox{90}{$\twoheadrightarrow$}
\\
\fud\fud
\end{matrix},
\end{aligned}
\end{gather}
to be the corresponding inclusions respectively projections, which 
are well-defined up to scalars. 
We do this in all color variations.

In this way we can view $\sltcatgop$ as being generated by the diagrams as in \eqref{eq:basic-webs}.

Fixing scalars appropriately (which we will do below), it is not hard to see that we get
\\
\noindent\begin{tabularx}{0.99\textwidth}{XXX}
\begin{equation}\hspace{-9.75cm}\label{eq:basic-webs-rels-isoa}
\begin{tikzpicture}[anchorbase, scale=.4, tinynodes]
	\fill[mygreen, opacity=0.8] (0,-1.5) to (0,0) to [out=90, in=180] (.5,1) to [out=0, in=90] (1,0) to [out=270, in=180] (1.5,-1) to [out=0, in=270] (2,0) to (2,1.5) to (2.5,1.5) to (2.5,-1.5) to (0,-1.5);
	\fill[myorange, opacity=0.8] (0,-1.5) to (0,0) to [out=90, in=180] (.5,1) to [out=0, in=90] (1,0) to [out=270, in=180] (1.5,-1) to [out=0, in=270] (2,0) to (2,1.5) to (-.5,1.5) to (-.5,-1.5) to (0,-1.5);
	\draw[dstrand, Xmarked=.55] (0,-1.5) to (0,0) to [out=90, in=180] (.5,1) to [out=0, in=90] (1,0) to [out=270, in=180] (1.5,-1) to [out=0, in=270] (2,0) to (2,1.5);
\end{tikzpicture}
{=}
\begin{tikzpicture}[anchorbase, scale=.4, tinynodes]
	\fill[mygreen, opacity=0.8] (0,0) to (0,3) to (.5,3) to (.5,0) to (0,0);
	\fill[myorange, opacity=0.8] (0,0) to (0,3) to (-.5,3) to (-.5,0) to (0,0);
	\draw[dstrand, Xmarked=.55] (0,0) to (0,3);
\end{tikzpicture}
{=}
\begin{tikzpicture}[anchorbase, scale=.4, tinynodes]
	\fill[mygreen, opacity=0.8] (0,-1.5) to (0,0) to [out=90, in=0] (-.5,1) to [out=180, in=90] (-1,0) to [out=270, in=0] (-1.5,-1) to [out=180, in=270] (-2,0) to (-2,1.5) to (.5,1.5) to (.5,-1.5) to (0,-1.5);
	\fill[myorange, opacity=0.8] (0,-1.5) to (0,0) to [out=90, in=0] (-.5,1) to [out=180, in=90] (-1,0) to [out=270, in=0] (-1.5,-1) to [out=180, in=270] (-2,0) to (-2,1.5) to (-2.5,1.5) to (-2.5,-1.5) to (0,-1.5);
	\draw[dstrand, Xmarked=.55] (0,-1.5) to (0,0) to [out=90, in=0] (-.5,1) to [out=180, in=90] (-1,0) to [out=270, in=0] (-1.5,-1) to [out=180, in=270] (-2,0) to (-2,1.5);
\end{tikzpicture}
\end{equation} &
\begin{equation}\hspace{-9.65cm}\label{eq:basic-webs-rels-isob}
\begin{tikzpicture}[anchorbase, scale=.4, tinynodes]
	\fill[myorange, opacity=0.8] (0,0) to [out=270, in=180] (1,-1) to [out=0, in=270] (2,0) to (2,1) to (1,1) to (0,0);
	\fill[mypurple, opacity=0.8] (0,0) to (-1,1) to (1,1) to (0,0); 
	\fill[mygreen, opacity=0.8] (-1.5,1) to (-1,1) to (0,0) to [out=270, in=180] (1,-1) to [out=0, in=270] (2,0) to (2,1) to (2.5,1) to (2.5,-2) to (-1.5,-2) to (-1.5,1);
	\draw[dstrand, Ymarked=.55] (-1,1) to (0,0);
	\draw[dstrand, Ymarked=.55] (1,1) to (0,0);
	\draw[dstrand, Xmarked=.55] (0,0) to [out=270, in=180] (1,-1) to [out=0, in=270] (2,0) to (2,1);
\end{tikzpicture}
{=}
\begin{tikzpicture}[anchorbase, scale=.4, tinynodes]
	\fill[myorange, opacity=0.8] (0,1) to (0,.5) to [out=180, in=90] (-.5,0) to [out=270, in=0] (-1,-.75) to [out=180, in=270] (-2,1) to (0,1);
	\fill[mypurple, opacity=0.8] (0,.5) to [out=0, in=90] (.5,0) to [out=270, in=0] (-1,-1.5) to [out=180, in=270] (-3,1) to (-2,1) to [out=270, in=180] (-1,-.75) to [out=0, in=270] (-.5,0) to [out=90, in=180] (0,.5);
	\fill[mygreen, opacity=0.8] (0,1) to (0,.5) to [out=0, in=90] (.5,0) to [out=270, in=0] (-1,-1.5) to [out=180, in=270] (-3,1) to (-3.5,1) to (-3.5,-2) to (1,-2) to (1,1) to (0,1);
	\draw[dstrand, Xmarked=.55] (0,.5) to [out=180, in=90] (-.5,0) to [out=270, in=0] (-1,-.75) to [out=180, in=270] (-2,1);
	\draw[dstrand, Xmarked=.55] (0,.5) to [out=0, in=90] (.5,0) to [out=270, in=0] (-1,-1.5) to [out=180, in=270] (-3,1);
	\draw[dstrand, Xmarked=.85] (0,.5) to (0,1);
\end{tikzpicture}
\end{equation}
&
\begin{equation}\hspace{-9.65cm}\label{eq:basic-webs-rels-isoc}
\begin{tikzpicture}[anchorbase, scale=.4, tinynodes]
	\fill[myorange, opacity=0.8] (0,0) to [out=90, in=180] (1,1) to [out=0, in=90] (2,0) to (2,-1) to (1,-1) to (0,0);
	\fill[mypurple, opacity=0.8] (0,0) to (-1,-1) to (1,-1) to (0,0); 
	\fill[mygreen, opacity=0.8] (-1.5,-1) to (-1,-1) to (0,0) to [out=90, in=180] (1,1) to [out=0, in=90] (2,0) to (2,-1) to (2.5,-1) to (2.5,2) to (-1.5,2) to (-1.5,-1);
	\draw[dstrand, Xmarked=.55] (-1,-1) to (0,0);
	\draw[dstrand, Xmarked=.55] (1,-1) to (0,0);
	\draw[dstrand, Ymarked=.55] (0,0) to [out=90, in=180] (1,1) to [out=0, in=90] (2,0) to (2,-1);
\end{tikzpicture}
{=}
\begin{tikzpicture}[anchorbase, scale=.4, tinynodes]
	\fill[myorange, opacity=0.8] (0,-1) to (0,-.5) to [out=180, in=270] (-.5,0) to [out=90, in=0] (-1,.75) to [out=180, in=90] (-2,-1) to (0,-1);
	\fill[mypurple, opacity=0.8] (0,-.5) to [out=0, in=270] (.5,0) to [out=90, in=0] (-1,1.5) to [out=180, in=90] (-3,-1) to (-2,-1) to [out=90, in=180] (-1,.75) to [out=0, in=90] (-.5,0) to [out=270, in=180] (0,-.5);
	\fill[mygreen, opacity=0.8] (0,-1) to (0,-.5) to [out=0, in=270] (.5,0) to [out=90, in=0] (-1,1.5) to [out=180, in=90] (-3,-1) to (-3.5,-1) to (-3.5,2) to (1,2) to (1,-1) to (0,-1);
	\draw[dstrand, Ymarked=.55] (0,-.5) to [out=180, in=270] (-.5,0) to [out=90, in=0] (-1,.75) to [out=180, in=90] (-2,-1);
	\draw[dstrand, Ymarked=.55] (0,-.5) to [out=0, in=270] (.5,0) to [out=90, in=0] (-1,1.5) to [out=180, in=90] (-3,-1);
	\draw[dstrand, Ymarked=.85] (0,-.5) to (0,-1);
\end{tikzpicture}
\end{equation}
\end{tabularx}
\\
\noindent\begin{tabularx}{0.99\textwidth}{XXX}
\begin{equation}\hspace{-10.0cm}\label{eq:basic-webs-rels-a}
\begin{tikzpicture}[anchorbase, scale=.4, tinynodes]
	\fill[mygreen, opacity=0.8] (0,0) to [out=270, in=180] (1,-1) to [out=0, in=270] (2,0) to (2.5,0) to (2.5,-2) to (-.5,-2) to (-.5,0) to (0,0);
	\fill[mygreen, opacity=0.8] (0,0) to [out=90, in=180] (1,1) to [out=0, in=90] (2,0) to (2.5,0) to (2.5,2) to (-.5,2) to (-.5,0) to (0,-0);
	\fill[myorange, opacity=0.8] (0,0) to [out=270, in=180] (1,-1) to [out=0, in=270] (2,0) to [out=90, in=0] (1,1) to [out=180, in=90] (0,0);
	\draw[dstrand, Xmarked=.55] (0,0) to [out=270, in=180] (1,-1) to [out=0, in=270] (2,0);
	\draw[dstrand, Xmarked=.55] (2,0) to [out=90, in=0] (1,1) to [out=180, in=90] (0,0);
\end{tikzpicture}
{=}
\qnumber{3}
\begin{tikzpicture}[anchorbase, scale=.4, tinynodes]
	\fill[mygreen, opacity=0.8] (2.5,-2) to (2.5,2) to (-.5,2) to (-.5,-2) to (2.5,-2);
\end{tikzpicture}
\end{equation} &
\begin{equation}\hspace{-9.7cm}\label{eq:basic-webs-rels-b}
\begin{tikzpicture}[anchorbase, scale=.4, tinynodes]
	\fill[mygreen, opacity=0.8] (0,0) to [out=45, in=270] (.5,1) to [out=90, in=315] (0,2) to (0,3) to (1.5,3) to (1.5,-1) to (0,-1) to (0,0);
	\fill[myorange, opacity=0.8] (0,0) to [out=135, in=270] (-.5,1) to [out=90, in=225] (0,2) to [out=315, in=90] (.5,1) to [out=270, in=45] (0,0);
	\fill[mypurple, opacity=0.8] (0,0) to [out=135, in=270] (-.5,1) to [out=90, in=225] (0,2) to (0,3) to (-1.5,3) to (-1.5,-1) to (0,-1) to (0,0);
	\draw[dstrand, Xmarked=.55] (0,0) to [out=135, in=270] (-.5,1) to [out=90, in=225] (0,2);
	\draw[dstrand, Xmarked=.55] (0,0) to [out=45, in=270] (.5,1) to [out=90, in=315] (0,2);
	\draw[dstrand, Ymarked=.55] (0,-1) to (0,0);
	\draw[dstrand, Ymarked=.55] (0,2) to (0,3);
\end{tikzpicture}
{=}
{-}\qnumber{2}
\begin{tikzpicture}[anchorbase, scale=.4, tinynodes]
	\fill[mygreen, opacity=0.8] (0,0) to (0,3) to (1.5,3) to (1.5,-1) to (0,-1) to (0,0);
	\fill[mypurple, opacity=0.8] (0,0) to (0,3) to (-1.5,3) to (-1.5,-1) to (0,-1) to (0,0);
	\draw[dstrand, Ymarked=.55] (0,-1) to (0,3);
\end{tikzpicture}
\end{equation} &
\begin{equation}\hspace{-9.15cm}\label{eq:basic-webs-rels-c}
\begin{tikzpicture}[anchorbase, scale=.4, tinynodes]
	\fill[mygreen, opacity=0.8] (1,-1) to (1,3) to (1.5,3) to (1.5,-1) to (1,-1);
	\fill[mygreen, opacity=0.8] (-1,-1) to (-1,3) to (-1.5,3) to (-1.5,-1) to (-1,-1);
	\fill[myorange, opacity=0.8] (-1,.2) to (1,.4) to (1,1.6) to (-1,1.8) to (-1,.2);
	\fill[mypurple, opacity=0.8] (-1,.2) to (1,.4) to (1,-1) to (-1,-1) to (-1,.2);
	\fill[mypurple, opacity=0.8] (1,1.6) to (1,3) to (-1,3) to (-1,1.8) to (1,1.6);
	\draw[dstrand, Ymarked=.55] (1,-1) to (1,0);
	\draw[dstrand, Xmarked=.55] (1,0) to (1,2);
	\draw[dstrand, Ymarked=.55] (1,2) to (1,3);
	\draw[dstrand, Xmarked=.55] (-1,-1) to (-1,0);
	\draw[dstrand, Ymarked=.55] (-1,0) to (-1,2);
	\draw[dstrand, Xmarked=.55] (-1,2) to (-1,3);
	\draw[dstrand, Ymarked=.55] (-1,.2) to (1,.4);
	\draw[dstrand, Xmarked=.55] (-1,1.8) to (1,1.6);
\end{tikzpicture}
{=}
\begin{tikzpicture}[anchorbase, scale=.4, tinynodes]
	\fill[mygreen, opacity=0.8] (1,-1) to (1,3) to (1.5,3) to (1.5,-1) to (1,-1);
	\fill[mygreen, opacity=0.8] (-1,-1) to (-1,3) to (-1.5,3) to (-1.5,-1) to (-1,-1);
	\fill[mypurple, opacity=0.8] (-1,-1) to (-1,3) to (1,3) to (1,-1) to (-1,-1);
	\draw[dstrand, Ymarked=.55] (1,-1) to (1,3);
	\draw[dstrand, Xmarked=.55] (-1,-1) to (-1,3);
\end{tikzpicture}
{+}
\begin{tikzpicture}[anchorbase, scale=.4, tinynodes]
	\fill[mygreen, opacity=0.8] (-1,3) to (-1.5,3) to (-1.5,-1) to (-1,-1) to [out=90, in=180] (0,0) to [out=0, in=90] (1,-1) to (1.5,-1) to (1.5,3) to (1,3) to [out=270, in=0] (0,2) to [out=180, in=270] (-1,3);
	\fill[mypurple, opacity=0.8] (-1,-1) to [out=90, in=180] (0,0) to [out=0, in=90] (1,-1) to (-1,-1);
	\fill[mypurple, opacity=0.8] (-1,3) to [out=270, in=180] (0,2) to [out=0, in=270] (1,3) to (-1,3);
	\draw[dstrand, Xmarked=.55] (-1,-1) to [out=90, in=180] (0,0) to [out=0, in=90] (1,-1);
	\draw[dstrand, Ymarked=.55] (-1,3) to [out=270, in=180] (0,2) to [out=0, in=270] (1,3);
\end{tikzpicture}
\end{equation}
\end{tabularx}\\
together with those obtained by varying the orientation and the colors, and the vertical 
mirrors of \eqref{eq:basic-webs-rels-isob} and \eqref{eq:basic-webs-rels-isoc}.

The following result is a consequence of \cite[Theorem 6.1]{Kup}.

\begin{lemmaqed}\label{lemma:webs-and-slt}
The $\Uslt$-equivariant maps/diagrams from \eqref{eq:basic-webs} 
together with the relations \eqref{eq:basic-webs-rels-isoa} to \eqref{eq:basic-webs-rels-c} 
give a generator-relation $2$-presentation of $\sltcatgop$.
\end{lemmaqed}
 
Following \cite[Section 3]{El1} we define a Satake $2$-functor.

\begin{definition}\label{definition:slt-soergel}
For $\tduc$ let $\elfunctor\colon\sltcatgop\to\aDiagfield$ be the 
$2$-functor defined as follows. On objects by $\elfunctor(\sltcat^{\tduc})=\tduc$, 
on $1$-morphisms by $\elfunctor(\fuf{\tduc}{\rho(\tduc)})=\rho(\tduc)\duc\tduc$ and 
$\elfunctor(\fudf{\tduc}{\rho^{-1}(\tduc)})=\rho^{-1}(\tduc)\duc\tduc$, and on 
$2$-morphisms by
\begin{gather}\label{eq:elfunctor}
\begin{aligned}
&
\begin{tikzpicture}[anchorbase, scale=.4, tinynodes]
	\fill[mygreen, opacity=0.8] (0,0) to [out=270, in=180] (1,-1) to [out=0, in=270] (2,0) to (2.5,0) to (2.5,-2) to (-.5,-2) to (-.5,0) to (0,0);
	\fill[myorange, opacity=0.8] (0,0) to [out=270, in=180] (1,-1) to [out=0, in=270] (2,0) to (0,0);
	\draw[dstrand, Xmarked=.55] (0,0) to [out=270, in=180] (1,-1) to [out=0, in=270] (2,0);
\end{tikzpicture}
\xmapsto{\elfunctor}
\begin{tikzpicture}[anchorbase, scale=.4, tinynodes]
	\fill[mygreen, opacity=0.8] (0,2) to [out=270, in=180] (1,1) to [out=0, in=270] (2,2) to (2.5,2) to (2.5,0) to (-.5,0) to (-.5,2) to (0,2);
	\fill[myorange, opacity=0.8] (1.5,2) to [out=270, in=0] (1,1.5) to [out=180, in=270] (.5,2) to (1.5,2);
	\fill[myyellow, opacity=0.3] (0,2) to [out=270, in=180] (1,1) to [out=0, in=270] (2,2) to (1.5,2) to [out=270, in=0] (1,1.5) to [out=180, in=270] (.5,2) to (0,2);
	\draw[rstrand, directed=.999] (1.5,2) to [out=270, in=0] (1,1.5) to [out=180, in=270] (.5,2);
	\draw[bstrand, directed=.999] (0,2) to [out=270, in=180] (1,1) to [out=0, in=270] (2,2);
\end{tikzpicture}
,
\\
&
\begin{tikzpicture}[anchorbase, scale=.4, tinynodes]
	\fill[mygreen, opacity=0.8] (0,-1) to (0,0) to (1,1) to (1.5,1) to (1.5,-1) to (0,-1);
	\fill[myorange, opacity=0.8] (0,0) to (-1,1) to (1,1) to (0,0); 
	\fill[mypurple, opacity=0.8] (0,1) to (0,0) to (-1,1) to (-1.5,1) to (-1.5,-1) to (0,-1);
	\draw[dstrand, Ymarked=.55] (-1,1) to (0,0);
	\draw[dstrand, Ymarked=.55] (1,1) to (0,0);
	\draw[dstrand, Xmarked=.55] (0,0) to (0,-1);
\end{tikzpicture}
\xmapsto{\elfunctor}
\begin{tikzpicture}[anchorbase, scale=.4, tinynodes]
	\fill[mygreen, opacity=0.8] (1,-1) to (1,-.6) to [out=90, in=315] (.55,0) to [out=45, in=270] (1,1) to (1.5,1) to (1.5,-1) to (1,-1);
	\fill[myorange, opacity=0.8] (-.5,1) to [out=270, in=135] (0,.3) to [out=45, in=270] (.5,1) to (-.5,1);
	\fill[mypurple, opacity=0.8] (-1,-1) to (-1,-.6) to [out=90, in=225] (-.55,0) to [out=135, in=270] (-1,1) to (-1.5,1) to (-1.5,-1) to (-1,-1);
	\fill[myblue, opacity=0.3] (1,-1) to (1,-.6) to [out=90, in=315] (.55,0) to [out=235, in=325] (-.55,0) to [out=225, in=90] (-1,-.6) to (-1,-1) to (1,-1);
	\fill[myred, opacity=0.3] (.5,1) to [out=270, in=45] (0,.3) to (-.55,0) to [out=135, in=270] (-1,1) to (.5,1);
	\fill[myyellow, opacity=0.3] (-.5,1) to [out=270, in=135] (0,.3) to (.55,0) to [out=45, in=270] (1,1) to (-.5,1);
	\draw[ystrand, directed=.35] (1,-1) to (1,-.6) to [out=90, in=270] (-.5,1);
	\draw[rstrand, rdirected=.35] (-1,-1) to (-1,-.6) to [out=90, in=270] (.5,1);
	\draw[bstrand, directed=.55] (-1,1) to [out=270, in=180] (0,-.2) to [out=0, in=270] (1,1);
\end{tikzpicture}
,
\end{aligned}
\;\;
\begin{aligned}
&
\begin{tikzpicture}[anchorbase, scale=.4, tinynodes]
	\fill[mygreen, opacity=0.8] (0,0) to [out=90, in=180] (1,1) to [out=0, in=90] (2,0) to (2.5,0) to (2.5,2) to (-.5,2) to (-.5,0) to (0,0);
	\fill[myorange, opacity=0.8] (0,0) to [out=90, in=180] (1,1) to [out=0, in=90] (2,0) to (0,0);
	\draw[dstrand, Xmarked=.55] (2,0) to [out=90, in=0] (1,1) to [out=180, in=90] (0,0);
\end{tikzpicture}
\xmapsto{\elfunctor}
\begin{tikzpicture}[anchorbase, scale=.4, tinynodes]
	\fill[mygreen, opacity=0.8] (0,-2) to [out=90, in=180] (1,-1) to [out=0, in=90] (2,-2) to (2.5,-2) to (2.5,0) to (-.5,0) to (-.5,-2) to (0,-2);
	\fill[myorange, opacity=0.8] (1.5,-2) to [out=90, in=0] (1,-1.5) to [out=180, in=90] (.5,-2) to (1.5,-2);
	\fill[myyellow, opacity=0.3] (0,-2) to [out=90, in=180] (1,-1) to [out=0, in=90] (2,-2) to (1.5,-2) to [out=90, in=0] (1,-1.5) to [out=180, in=90] (.5,-2) to (0,-2);
	\draw[rstrand, directed=.999] (.5,-2) to [out=90, in=180] (1,-1.5) to [out=0, in=90] (1.5,-2);
	\draw[bstrand, directed=.999] (2,-2) to [out=90, in=0] (1,-1) to [out=180, in=90] (0,-2);
\end{tikzpicture}
,
\\
&
\begin{tikzpicture}[anchorbase, scale=.4, tinynodes]
	\fill[mygreen, opacity=0.8] (0,1) to (0,0) to (1,-1) to (1.5,-1) to (1.5,1) to (0,1);
	\fill[myorange, opacity=0.8] (0,0) to (-1,-1) to (1,-1) to (0,0); 
	\fill[mypurple, opacity=0.8] (0,1) to (0,0) to (-1,-1) to (-1.5,-1) to (-1.5,1) to (0,1);
	\draw[dstrand, Xmarked=.55] (-1,-1) to (0,0);
	\draw[dstrand, Xmarked=.55] (1,-1) to (0,0);
	\draw[dstrand, Ymarked=.55] (0,0) to (0,1);
\end{tikzpicture}
\xmapsto{\elfunctor}
\begin{tikzpicture}[anchorbase, scale=.4, tinynodes]
	\fill[mygreen, opacity=0.8] (1,1) to (1,.6) to [out=270, in=45] (.55,0) to [out=315, in=90] (1,-1) to (1.5,-1) to (1.5,1) to (1,1);
	\fill[myorange, opacity=0.8] (-.5,-1) to [out=90, in=225] (0,-.3) to [out=315, in=90] (.5,-1) to (-.5,-1);
	\fill[mypurple, opacity=0.8] (-1,1) to (-1,.6) to [out=270, in=135] (-.55,0) to [out=225, in=90] (-1,-1) to (-1.5,-1) to (-1.5,1) to (-1,1);
	\fill[myblue, opacity=0.3] (1,1) to (1,.6) to [out=270, in=45] (.55,0) to [out=125, in=55] (-.55,0) to [out=135, in=270] (-1,.6) to (-1,1) to (1,1);
	\fill[myred, opacity=0.3] (.5,-1) to [out=90, in=315] (0,-.3) to (-.55,0) to [out=225, in=90] (-1,-1) to (.5,-1);
	\fill[myyellow, opacity=0.3] (-.5,-1) to [out=90, in=225] (0,-.3) to (.55,0) to [out=315, in=90] (1,-1) to (-.5,-1);
	\draw[ystrand, rdirected=.35] (1,1) to (1,.6) to [out=270, in=90] (-.5,-1);
	\draw[rstrand, directed=.35] (-1,1) to (-1,.6) to [out=270, in=90] (.5,-1);
	\draw[bstrand, rdirected=.55] (-1,-1) to [out=90, in=180] (0,.2) to [out=0, in=90] (1,-1);
\end{tikzpicture}
,
\end{aligned}
\;\;
\begin{aligned}
&
\begin{tikzpicture}[anchorbase, scale=.4, tinynodes]
	\fill[mygreen, opacity=0.8] (0,0) to [out=270, in=180] (1,-1) to [out=0, in=270] (2,0) to (2.5,0) to (2.5,-2) to (-.5,-2) to (-.5,0) to (0,0);
	\fill[mypurple, opacity=0.8] (0,0) to [out=270, in=180] (1,-1) to [out=0, in=270] (2,0) to (0,0);
	\draw[dstrand, Ymarked=.55] (0,0) to [out=270, in=180] (1,-1) to [out=0, in=270] (2,0);
\end{tikzpicture}
\xmapsto{\elfunctor}
\begin{tikzpicture}[anchorbase, scale=.4, tinynodes]
	\fill[mygreen, opacity=0.8] (0,2) to [out=270, in=180] (1,1) to [out=0, in=270] (2,2) to (2.5,2) to (2.5,0) to (-.5,0) to (-.5,2) to (0,2);
	\fill[mypurple, opacity=0.8] (1.5,2) to [out=270, in=0] (1,1.5) to [out=180, in=270] (.5,2) to (1.5,2);
	\fill[myblue, opacity=0.3] (0,2) to [out=270, in=180] (1,1) to [out=0, in=270] (2,2) to (1.5,2) to [out=270, in=0] (1,1.5) to [out=180, in=270] (.5,2) to (0,2);
	\draw[ystrand, directed=.999] (0,2) to [out=270, in=180] (1,1) to [out=0, in=270] (2,2);
	\draw[rstrand, directed=.999] (1.5,2) to [out=270, in=0] (1,1.5) to [out=180, in=270] (.5,2);
\end{tikzpicture}
,
\\
&
\begin{tikzpicture}[anchorbase, scale=.4, tinynodes]
	\fill[mygreen, opacity=0.8] (0,-1) to (0,0) to (1,1) to (1.5,1) to (1.5,-1) to (0,-1);
	\fill[mypurple, opacity=0.8] (0,0) to (-1,1) to (1,1) to (0,0); 
	\fill[myorange, opacity=0.8] (0,1) to (0,0) to (-1,1) to (-1.5,1) to (-1.5,-1) to (0,-1);
	\draw[dstrand, Xmarked=.55] (-1,1) to (0,0);
	\draw[dstrand, Xmarked=.55] (1,1) to (0,0);
	\draw[dstrand, Ymarked=.55] (0,0) to (0,-1);
\end{tikzpicture}
\xmapsto{\elfunctor}
\begin{tikzpicture}[anchorbase, scale=.4, tinynodes]
	\fill[mygreen, opacity=0.8] (1,-1) to (1,-.6) to [out=90, in=315] (.55,0) to [out=45, in=270] (1,1) to (1.5,1) to (1.5,-1) to (1,-1);
	\fill[mypurple, opacity=0.8] (-.5,1) to [out=270, in=135] (0,.3) to [out=45, in=270] (.5,1) to (-.5,1);
	\fill[myorange, opacity=0.8] (-1,-1) to (-1,-.6) to [out=90, in=225] (-.55,0) to [out=135, in=270] (-1,1) to (-1.5,1) to (-1.5,-1) to (-1,-1);
	\fill[myyellow, opacity=0.3] (1,-1) to (1,-.6) to [out=90, in=315] (.55,0) to [out=235, in=325] (-.55,0) to [out=225, in=90] (-1,-.6) to (-1,-1) to (1,-1);
	\fill[myred, opacity=0.3] (.5,1) to [out=270, in=45] (0,.3) to (-.55,0) to [out=135, in=270] (-1,1) to (.5,1);
	\fill[myblue, opacity=0.3] (-.5,1) to [out=270, in=135] (0,.3) to (.55,0) to [out=45, in=270] (1,1) to (-.5,1);
	\draw[ystrand, directed=.55] (-1,1) to [out=270, in=180] (0,-.2) to [out=0, in=270] (1,1);
	\draw[rstrand, rdirected=.35] (-1,-1) to (-1,-.6) to [out=90, in=270] (.5,1);
	\draw[bstrand, directed=.35] (1,-1) to (1,-.6) to [out=90, in=270] (-.5,1);
\end{tikzpicture}
,
\end{aligned}
\;\;
\begin{aligned}
&
\begin{tikzpicture}[anchorbase, scale=.4, tinynodes]
	\fill[mygreen, opacity=0.8] (0,0) to [out=90, in=180] (1,1) to [out=0, in=90] (2,0) to (2.5,0) to (2.5,2) to (-.5,2) to (-.5,0) to (0,0);
	\fill[mypurple, opacity=0.8] (0,0) to [out=90, in=180] (1,1) to [out=0, in=90] (2,0) to (0,0);
	\draw[dstrand, Ymarked=.55] (2,0) to [out=90, in=0] (1,1) to [out=180, in=90] (0,0);
\end{tikzpicture}
\xmapsto{\elfunctor}
\begin{tikzpicture}[anchorbase, scale=.4, tinynodes]
	\fill[mygreen, opacity=0.8] (0,-2) to [out=90, in=180] (1,-1) to [out=0, in=90] (2,-2) to (2.5,-2) to (2.5,0) to (-.5,0) to (-.5,-2) to (0,-2);
	\fill[mypurple, opacity=0.8] (1.5,-2) to [out=90, in=0] (1,-1.5) to [out=180, in=90] (.5,-2) to (1.5,-2);
	\fill[myblue, opacity=0.3] (0,-2) to [out=90, in=180] (1,-1) to [out=0, in=90] (2,-2) to (1.5,-2) to [out=90, in=0] (1,-1.5) to [out=180, in=90] (.5,-2) to (0,-2);
	\draw[ystrand, directed=.999] (2,-2) to [out=90, in=0] (1,-1) to [out=180, in=90] (0,-2);
	\draw[rstrand, directed=.999] (.5,-2) to [out=90, in=180] (1,-1.5) to [out=0, in=90] (1.5,-2);
\end{tikzpicture}
,
\\
&
\begin{tikzpicture}[anchorbase, scale=.4, tinynodes]
	\fill[mygreen, opacity=0.8] (0,1) to (0,0) to (1,-1) to (1.5,-1) to (1.5,1) to (0,1);
	\fill[mypurple, opacity=0.8] (0,0) to (-1,-1) to (1,-1) to (0,0); 
	\fill[myorange, opacity=0.8] (0,1) to (0,0) to (-1,-1) to (-1.5,-1) to (-1.5,1) to (0,1);
	\draw[dstrand, Ymarked=.55] (-1,-1) to (0,0);
	\draw[dstrand, Ymarked=.55] (1,-1) to (0,0);
	\draw[dstrand, Xmarked=.55] (0,0) to (0,1);
\end{tikzpicture}
\xmapsto{\elfunctor}
\begin{tikzpicture}[anchorbase, scale=.4, tinynodes]
	\fill[mygreen, opacity=0.8] (1,1) to (1,.6) to [out=270, in=45] (.55,0) to [out=315, in=90] (1,-1) to (1.5,-1) to (1.5,1) to (1,1);
	\fill[mypurple, opacity=0.8] (-.5,-1) to [out=90, in=225] (0,-.3) to [out=315, in=90] (.5,-1) to (-.5,-1);
	\fill[myorange, opacity=0.8] (-1,1) to (-1,.6) to [out=270, in=135] (-.55,0) to [out=225, in=90] (-1,-1) to (-1.5,-1) to (-1.5,1) to (-1,1);
	\fill[myyellow, opacity=0.3] (1,1) to (1,.6) to [out=270, in=45] (.55,0) to [out=125, in=55] (-.55,0) to [out=135, in=270] (-1,.6) to (-1,1) to (1,1);
	\fill[myred, opacity=0.3] (.5,-1) to [out=90, in=315] (0,-.3) to (-.55,0) to [out=225, in=90] (-1,-1) to (.5,-1);
	\fill[myblue, opacity=0.3] (-.5,-1) to [out=90, in=225] (0,-.3) to (.55,0) to [out=315, in=90] (1,-1) to (-.5,-1);
	\draw[ystrand, rdirected=.55] (-1,-1) to [out=90, in=180] (0,.2) to [out=0, in=90] (1,-1);
	\draw[rstrand, directed=.35] (-1,1) to (-1,.6) to [out=270, in=90] (.5,-1);
	\draw[bstrand, rdirected=.35] (1,1) to (1,.6) to [out=270, in=90] (-.5,-1);
\end{tikzpicture}
,
\end{aligned}
\end{gather}
together with similar assignments for the other generators.
\end{definition}

The following lemma recalls \cite[Claim 3.19]{El1}. We sketch its proof for the convenience of the reader and refer 
to \cite[Proof of Claim 3.19]{El1} for more details.

\begin{lemma}\label{lemma:el-well-def}
The $2$-functor $\elfunctor$ is well-defined.
\end{lemma}

\begin{proof}
We only need to show that \eqref{eq:basic-webs-rels-isoa}--\eqref{eq:basic-webs-rels-c} 
hold in the image of $\elfunctor$. 
The isotopies \eqref{eq:basic-webs-rels-isoa}--\eqref{eq:basic-webs-rels-isoc}
are clearly preserved. For \eqref{eq:basic-webs-rels-a} 
we have already verified this in \fullref{example:some-relations}, while 
\eqref{eq:basic-webs-rels-b} follows from \eqref{eq:rm-third}, \eqref{eq:rm-first} 
and \eqref{eq:circle-primary} 
together with $\qDema_{\bc}(\frobel{\oc}{\rc,\yc})=-\qnumber{2}$. 
The relation \eqref{eq:basic-webs-rels-c} is a bit more involved 
(but not hard), 
and can be proved by using \eqref{eq:square} on the $\elfunctor$-image of the square.
\end{proof}

We say that a $2$-functor from an ungraded $2$-category 
(whose $2$-morphisms are all of degree zero, by convention) to a graded $2$-category 
is a degree-zero $2$-equivalence, if it is a bijection on objects,
essentially surjective on $1$-morphisms, 
faithful on $2$-morphisms, and full onto degree-zero $2$-morphisms. 
Using this notion, the quantum Satake correspondence can be formulated as in \cite[Theorem 3.21]{El1}:

\begin{theoremqed}\label{theorem:q-satake}
The $2$-functor $\elfunctor$ is a degree-zero $2$-equivalence.
\end{theoremqed}

\begin{remark}\label{remark:q-satake}
Elias actually proves \fullref{theorem:q-satake} in 
much more generality. For us the important 
case is over the ring $\aformq=\C[\qpar,\qpar^{-1}]$, which then 
implies that \fullref{theorem:q-satake} holds 
over any ground ring we are going to use.
\end{remark}

Let $\cRKLg{m,n}=\cRKLg{m,n}(\fu^m\fud^n)$ denote the (unique) projection 
of $\fu^m\fud^n$ onto the irreducible direct summand $\Ll_{m,n}$, regarded as a 
$2$-morphism in $\sltcatgop$ with rightmost color $\gc$. 
We call $\cRKLg{m,n}$ the (right-green) $\slt$-clasp. 
Similarly, we define $\cRKLo{m,n}$, $\cRKLp{m,n}$ and 
$\cKLg{m,n}$, $\cKLo{m,n}$, $\cKLp{m,n}$. Note that there is actually a different clasp for each   
product of $m$ factors $\fu$ and $n$ factors $\fud$, but these clasps are all closely related, as we will see in 
\fullref{lemma:choice-does-not-matter}. For now, it suffices to consider only the one for $\fu^m\fud^n$. 

\begin{remark}\label{remark:clasps-formulas}
The $\slt$-clasps have a diagrammatic incarnation, obtained 
by coloring the diagrammatic clasps from \cite[Theorem 3.3]{Kim} 
(which gives the $\slt$-clasps in terms of a recursion), 
or (using slightly 
different conventions) from \cite[(1.8) and Section 3.2]{El3}.
\end{remark}

The colored $\slt$-clasps are $2$-morphisms in 
$\aDiagfield$, but do not belong to $\subcatquofield$, since 
their left- and rightmost colors are always secondary colors. Thus, 
we need to `biinduce them up to $\wc$' in order to have 
their appropriate analogs in $\subcatquofield$:

\begin{definition}\label{definition:clasps-categorified-quotient}
The colored (right-$\tduc$) clasps $\CRKLx{m,n}$ are defined by
\[
\CRKLx{m,n}=
\twomorstuff{id}_{\wc\dudc\tdudc}
\hcomp\cRKLx{m,n}\hcomp
\twomorstuff{id}_{\tduc\duc\wc}
\]
with $\hcomp$ being the horizontal composition in $\Adiagfield$. 
Here $\cRKLx{m,n}$ has leftmost color $\tdudc$, and $\duc,\dudc$ 
are compatible colors where we prefer $\bc$ over $\rc$ over $\yc$. 
We define $\CKLx{m,n}$ similarly. 
\end{definition}

The colored clasps are idempotent $2$-morphisms 
in $\subcatquofield$, which depend on the same choices as the colored $\slt$-clasps 
and, additionally, on the choice of $\duc,\tduc$. Again, this dependence is not essential, as 
we will show in \fullref{lemma:choice-does-not-matter}, so we abuse notation.

\begin{example}\label{example:clasps}
Using \cite[Theorem 3.3]{Kim}, one can write down the colored $\slt$-clasps 
explicitly, e.g. in the case $e+1=2$:
\begin{gather*}
\cRKLg{2,0}=
\begin{tikzpicture}[anchorbase, scale=.4, tinynodes]
	\fill[mygreen, opacity=0.8] (0,0) to (0,2) to (1,2) to (1,0) to (0,0);
	\fill[myorange, opacity=0.8] (0,0) to (0,2) to (-1,2) to (-1,0) to (0,0);
	\fill[mypurple, opacity=0.8] (-1,0) to (-1,2) to (-2,2) to (-2,0) to (-1,0);
	\draw[dstrand, Xmarked=.55] (0,0) to (0,2);
	\draw[dstrand, Xmarked=.55] (-1,0) to (-1,2);
\end{tikzpicture}
+
\tfrac{1}{\qnumber{2}}\,
\begin{tikzpicture}[anchorbase, scale=.4, tinynodes]
	\fill[mygreen, opacity=0.8] (0,0) to (-.5,.75) to (-.5,1.25) to (0,2) to (1,2) to (1,0) to (0,0);
	\fill[myorange, opacity=0.8] (-1,2) to (-.5,1.25) to (0,2) to (-1,2);
	\fill[myorange, opacity=0.8] (-1,0) to (-.5,.75) to (0,0) to (-1,0);
	\fill[mypurple, opacity=0.8] (-1,0) to (-.5,.75) to (-.5,1.25) to (-1,2) to (-2,2) to (-2,0) to (-1,0);
	\draw[dstrand, Xmarked=.8] (-.5,1.25) to (-1,2);
	\draw[dstrand, Xmarked=.8] (-.5,1.25) to (0,2);
	\draw[dstrand, Xmarked=.75] (-.5,1.25) to (-.5,.75);
	\draw[dstrand, Xmarked=.6] (-1,0) to (-.5,.75);
	\draw[dstrand, Xmarked=.6] (0,0) to (-.5,.75);
\end{tikzpicture}
,\quad
\cRKLg{1,1}
=
\begin{tikzpicture}[anchorbase, scale=.4, tinynodes]
	\fill[mygreen, opacity=0.8] (0,0) to (0,2) to (1,2) to (1,0) to (0,0);
	\fill[mypurple, opacity=0.8] (0,0) to (0,2) to (-1,2) to (-1,0) to (0,0);
	\fill[mygreen, opacity=0.8] (-1,0) to (-1,2) to (-2,2) to (-2,0) to (-1,0);
	\draw[dstrand, Ymarked=.55] (0,0) to (0,2);
	\draw[dstrand, Xmarked=.55] (-1,0) to (-1,2);
\end{tikzpicture}
-
\tfrac{1}{\qnumber{3}}\,
\begin{tikzpicture}[anchorbase, scale=.4, tinynodes]
	\fill[mygreen, opacity=0.8] (1,0) to (0,0) to [out=90, in=00] (-.5,.75) to [out=180, in=90] (-1,0) to (-2,0) to (-2,2) to (-1,2) to [out=270, in=180] (-.5,1.25) to [out=0, in=270] (0,2) to (1,2) to (1,0);
	\fill[mypurple, opacity=0.8] (0,0) to [out=90, in=00] (-.5,.75) to [out=180, in=90] (-1,0) to (0,0);
	\fill[mypurple, opacity=0.8] (0,2) to [out=270, in=00] (-.5,1.25) to [out=180, in=270] (-1,2) to (0,2);
	\draw[dstrand, Xmarked=.85] (-1,0) to [out=90, in=180] (-.5,.75) to [out=0, in=90] (0,0);
	\draw[dstrand, Xmarked=.85] (0,2) to [out=270, in=0] (-.5,1.25) to [out=180, in=270] (-1,2);
\end{tikzpicture}
,\quad
\cRKLg{0,2}=\!
\begin{tikzpicture}[anchorbase, scale=.4, tinynodes]
	\fill[mygreen, opacity=0.8] (0,0) to (0,2) to (1,2) to (1,0) to (0,0);
	\fill[mypurple, opacity=0.8] (0,0) to (0,2) to (-1,2) to (-1,0) to (0,0);
	\fill[myorange, opacity=0.8] (-1,0) to (-1,2) to (-2,2) to (-2,0) to (-1,0);
	\draw[dstrand, Ymarked=.55] (0,0) to (0,2);
	\draw[dstrand, Ymarked=.55] (-1,0) to (-1,2);
\end{tikzpicture}
+
\tfrac{1}{\qnumber{2}}\,
\begin{tikzpicture}[anchorbase, scale=.4, tinynodes]
	\fill[mygreen, opacity=0.8] (0,0) to (-.5,.75) to (-.5,1.25) to (0,2) to (1,2) to (1,0) to (0,0);
	\fill[mypurple, opacity=0.8] (-1,2) to (-.5,1.25) to (0,2) to (-1,2);
	\fill[mypurple, opacity=0.8] (-1,0) to (-.5,.75) to (0,0) to (-1,0);
	\fill[myorange, opacity=0.8] (-1,0) to (-.5,.75) to (-.5,1.25) to (-1,2) to (-2,2) to (-2,0) to (-1,0);
	\draw[dstrand, Xmarked=.75] (-1,2) to (-.5,1.25);
	\draw[dstrand, Xmarked=.75] (0,2) to (-.5,1.25);
	\draw[dstrand, Xmarked=.75] (-.5,.75) to (-.5,1.25);
	\draw[dstrand, Xmarked=.8] (-.5,.75) to (-1,0);
	\draw[dstrand, Xmarked=.8] (-.5,.75) to (0,0);
\end{tikzpicture}
\end{gather*}
Using $\elfunctor$ and biinduction, we get for example:
\[
\CRKLg{1,1}
=
\begin{tikzpicture}[anchorbase, scale=.4, tinynodes]
	\draw[very thin, densely dotted, fill=white] (-4.5,0) to (-4.5,3) to (-3.5,3) to (-3.5,0) to (-4.5,0);
	\draw[very thin, densely dotted, fill=white] (4.5,0) to (4.5,3) to (3.5,3) to (3.5,0) to (4.5,0);
	\fill[myblue, opacity=0.3] (-3.5,0) to (-3.5,3) to (-2.5,3) to (-2.5,0) to (-3.5,0);
	\fill[myblue, opacity=0.3] (-1.5,0) to (-1.5,3) to (-.5,3) to (-.5,0) to (-1.5,0);
	\fill[myblue, opacity=0.3] (1.5,0) to (1.5,3) to (.5,3) to (.5,0) to (1.5,0);
	\fill[myblue, opacity=0.3] (3.5,0) to (3.5,3) to (2.5,3) to (2.5,0) to (3.5,0);
	\fill[mygreen, opacity=0.8] (2.5,0) to (2.5,3) to (1.5,3) to (1.5,0) to (2.5,0);
	\fill[mygreen, opacity=0.8] (-2.5,0) to (-2.5,3) to (-1.5,3) to (-1.5,0) to (-2.5,0);
	\fill[mypurple, opacity=0.8] (-.5,0) to (-.5,3) to (.5,3) to (.5,0) to (-.5,0);
	\draw[ystrand, directed=.55] (-2.5,0) to (-2.5,3);
	\draw[ystrand, directed=.55] (-1.5,3) to (-1.5,0);
	\draw[ystrand, directed=.55] (1.5,0) to (1.5,3);
	\draw[ystrand, directed=.55] (2.5,3) to (2.5,0);
	\draw[rstrand, directed=.55] (-.5,0) to (-.5,3);
	\draw[rstrand, directed=.55] (.5,3) to (.5,0);
	\draw[bstrand, directed=.55] (-3.5,0) to (-3.5,3);
	\draw[bstrand, directed=.55] (3.5,3) to (3.5,0);
\end{tikzpicture}
-
\tfrac{1}{\qnumber{3}}\,
\begin{tikzpicture}[anchorbase, scale=.4, tinynodes]
	\draw[very thin, densely dotted, fill=white] (-4.5,0) to (-4.5,3) to (-3.5,3) to (-3.5,0) to (-4.5,0);
	\draw[very thin, densely dotted, fill=white] (4.5,0) to (4.5,3) to (3.5,3) to (3.5,0) to (4.5,0);
	\fill[myblue, opacity=0.3] (-3.5,0) to (-3.5,3) to (-2.5,3) to (-2.5,0) to (-3.5,0);
	\fill[myblue, opacity=0.3] (1.5,0) to [out=90, in=0] (0,1.25) to [out=180, in=90] (-1.5,0) to (1.5,0);
	\fill[myblue, opacity=0.3] (-1.5,3) to [out=270, in=180] (0,1.75) to [out=0, in=270] (1.5,3) to (-1.5,3);
	\fill[myblue, opacity=0.3] (3.5,0) to (3.5,3) to (2.5,3) to (2.5,0) to (3.5,0);
	\fill[mygreen, opacity=0.8] (1.5,0) to [out=90, in=0] (0,1.25) to [out=180, in=90] (-1.5,0) to (-2.5,0) to (-2.5,3) to (-1.5,3) to [out=270, in=180] (0,1.75) to [out=0, in=270] (1.5,3) to (2.5,3) to (2.5,0) to (1.5,0);
	\fill[mypurple, opacity=0.8] (-.5,0) to [out=90, in=180] (0,.5) to [out=0, in=90] (.5,0) to (-.5,0);
	\fill[mypurple, opacity=0.8] (.5,3) to [out=270, in=0] (0,2.5) to [out=180, in=270] (-.5,3) to (.5,3);
	\draw[ystrand, directed=.55] (-2.5,0) to (-2.5,3);
	\draw[ystrand, directed=.75] (1.5,0) to [out=90, in=0] (0,1.25) to [out=180, in=90] (-1.5,0);
	\draw[ystrand, directed=.75] (-1.5,3) to [out=270, in=180] (0,1.75) to [out=0, in=270] (1.5,3);
	\draw[ystrand, directed=.55] (2.5,3) to (2.5,0);
	\draw[rstrand, directed=.75] (-.5,0) to [out=90, in=180] (0,.5) to [out=0, in=90] (.5,0);
	\draw[rstrand, directed=.75] (.5,3) to [out=270, in=0] (0,2.5) to [out=180, in=270] (-.5,3);
	\draw[bstrand, directed=.55] (-3.5,0) to (-3.5,3);
	\draw[bstrand, directed=.55] (3.5,3) to (3.5,0);
\end{tikzpicture}
\]
(The outer $\bc$-regions come from our 
choice in \fullref{definition:clasps-categorified-quotient}.)
\end{example}

The next lemma shows that the colored clasps are independent of the choices that we made in their definition:

\begin{lemma}\label{lemma:choice-does-not-matter}
Let $\CRKLx{m,n}$ be a colored clasps and let $(\CRKLx{m,n})^{\prime}$ be defined similarly, 
but with some difference in the involved choices. Then there exists an
invertible $2$-morphisms $\twomorstuff{f}$ in $\subcatquofield$ 
such that $\CRKLx{m,n}\vcomp\twomorstuff{f}=\twomorstuff{f}\vcomp(\CRKLx{m,n})^{\prime}$.
\end{lemma}

\begin{proof}
If the ordering of the factors $\fu$ and $\fud$ for $\CRKLx{m,n}$ and $(\CRKLx{m,n})^{\prime}$ differs 
by precisely one pair, then \fullref{lemma:thecrossings-2} shows the claim.
If $\CRKLx{m,n}$ and $(\CRKLx{m,n})^{\prime}$ differ by precisely one choice of compatible color 
for `biinduction', then \fullref{lemma:clasps-well-defined} shows the claim.
Then the general statement then follows by induction. 
\end{proof}

\begin{corollary}\label{corollary:clasps-ideal}
If a two sided $2$-ideal in $\subcatquofield$ contains 
a colored clasp $\CRKLx{m,n}$, then it contains 
all colored clasps $(\CRKLx{m,n})^{\prime}$ 
which differ from $\CRKLx{m,n}$ by some of the choices involved in their definition.
\end{corollary}

\subsubsection{From \texorpdfstring{$\slt$}{sl3} to singular bimodules: the root of unity case}\label{subsec:quotient-algebra-categorified}

From now on we work over $\C$ by specializing $\qpar$ to $\qqpar$ 
which, as usual, is a $2(e+3)^{\mathrm{th}}$ primitive, complex root of unity. 
Formally this is done by repeating the above for the $\aformq$-linear $2$-categories 
which are scalar extended to 
$\aforme=\C[\qpar,\qpar^{-1},\qnumber{2}^{-1},\dots,\qnumber{e{+}1}^{-1}]$. 
We denote these using 
$[e]$ as a subscript, and the specialization at $\qpar=\qqpar$ we denote by 
$\placeholder\otimes_{\aforme}\C$. We also exclude the case $e=0$, which 
is a bit special and can easily be dealt with later on.

First of all, all previous definitions and 
results in the generic case are still valid in this case, except \fullref{lemma:webs-and-slt} 
(which we do not need in the following) and the definition of 
the (various) clasps for $m+n > e+1$. 
In particular, \fullref{lemma:el-well-def} and \fullref{theorem:q-satake} still hold 
for the specialization at $\qpar=\qqpar$. 

\begin{lemma}\label{lemma:slt-soergel-clasps}
The colored clasps are well-defined in $\subcatquoefield$ 
(seen as a $2$-subcategory of $\adiagefield$)
for $0\leq m+n\leq e+1$, and uniquely determined 
up to conjugation by an invertible $2$-morphism.
\end{lemma}

\makeautorefname{lemma}{Lemmas}

\begin{proof}
Decomposing $\fu^m\fud^n$ generically, i.e. 
for $\Uslt$, works similarly as in the 
root of unity case as long as $m+n\leq e+1$, see e.g. \cite[Section 3]{AP}.
Thus, one can use the specializations of the projectors 
from the generic case in the root of unity case. 
(Alternatively, using \fullref{remark:clasps-formulas}, one checks 
that the coefficients of the colored $\slt$-clasps specialize properly.)
Finally, note that \fullref{lemma:el-well-def} 
and \ref{lemma:choice-does-not-matter} also hold 
for $\qpar$ being specialized to $\qqpar$.
\end{proof}

\makeautorefname{lemma}{Lemma}

Having all the above established, we 
can define the $2$-category of 
trihedral Soergel bimodules of level $e$:

\begin{definition}\label{definition:categorified-quotient}
Let $\claspideal[{e}]$ be the two-sided $2$-ideal, called vanishing $2$-ideal of level $e$, 
in $\subcatquoefield\otimes_{\aforme}\C$ generated by
\begin{gather*}
\left\{\CRKLx{m,n} \mid m+n=e+1,\; \tduc\in\Seset\right\}
=
\left\{\CKLx{m,n} \mid m+n=e+1,\; \tduc\in\Seset\right\},
\end{gather*}
where we write e.g. $\CRKLx{m,n}=\CRKLx{m,n}\otimes_{\aforme}1$ for simplicity. 
We define
\[
\subcatquo[e]=(\subcatquoefield\otimes_{\aforme}\C)/\claspideal[{e}],
\]
which we call the $2$-category of trihedral Soergel bimodules of level $e$.
\end{definition}

\begin{remark}\label{remark:before-after}
Note that we specialize before taking the quotient, as Andersen--Paradowski do in order to define 
$\sltcat[e]$ in \cite{AP}, where they take the quotient of the already specialized category $\sltcat[\qqpar]$ by the 
ideal of so-called negligible modules. (This is explicitly described in e.g. 
\cite[Section 3.3]{BK1}.) Similarly, we always specialize first throughout.
\end{remark}

Let $\claspidealgope$ 
be the two-sided $2$-ideal in $\sltcatefield\otimes_{\aforme}\C$ generated by
\begin{gather*}
\left\{\cRKLx{m,n} \mid m+n=e+1,\; \tduc\in\Seset\right\}
=
\left\{\cKLx{m,n} \mid m+n=e+1,\; \tduc\in\Seset\right\}.
\end{gather*}
The maximally singular version of $\subcatquo[e]$ is 
\[
\Subcatquo[e]=(\Subcatquoefield\otimes_{\aforme}\C)/\elfunctor[\qqpar]\left(\claspidealgope\right).
\]
Here we again specialize $\qpar$ to $\qqpar$ and use the same conventions as before. 
We state a non-trivial consequence of the 
quantum Satake correspondence from \fullref{theorem:q-satake} 
and \fullref{remark:q-satake}.

\begin{lemma}\label{lemma:Satake}
$\elfunctorefield$ gives rise to 
a degree-zero $2$-equivalence 
$\elfunctor[e]\colon\slqmodgop\to\Subcatquo[e]$.
\end{lemma}

\begin{proof}
Because $\elfunctorefield$ is a degree-zero $2$-equivalence 
before quotienting by any clasps, by \fullref{remark:q-satake}, it sends 
indecomposable $1$-morphisms in $\sltcatefield$
to indecomposable $1$-morphisms in $\Subcatquoefield$, and 
the cell structures of $\sltcatefield$ and $\Subcatquoefield$ are 
isomorphic under $\elfunctorefield$.

Clearly, $\elfunctorefield$ specializes to $\elfunctor[\qqpar]\colon \Subcatquoefield\otimes_{\aforme}\C\to 
\sltcatefield\otimes_{\aforme}\C$, which descends to the $2$-functor 
$\elfunctor[e]\colon\slqmodgop\to\Subcatquo[e]$. Both $\elfunctor[\qqpar]$ and $\elfunctor[e]$ 
are essentially surjective on $1$-morphisms and full onto degree-zero $2$-morphisms. 
What is not immediately clear, is that $\elfunctor[e]\colon\slqmodgop\to\Subcatquo[e]$ is 
also faithful on $2$-morphisms: since $\subcatquoefield$ 
has $2$-morphisms of negative degree, the 
degree-zero part of $\claspidealefielde$ could a priori be bigger 
than $\elfunctor[\qqpar]\left(\claspidealgope\right)$. To show that that is not the case, we 
use the following roundabout argument.

From \cite{AP} (cf. \fullref{remark:before-after}) 
we know that we have an equivalence of $2$-categories
\[
(\sltcatefield\otimes_{\aforme}\C)/\claspidealgope
\cong\slqmodgop,
\]  
Hence, the indecomposable $1$-morphisms $\morstuff{F}$ in $\sltcatefield\otimes_{\aforme}\C$ for which 
$\twomorstuff{id}_{\morstuff{F}}\in\claspidealgope$ are 
strictly greater than 
the ones for which $\twomorstuff{id}_{\morstuff{F}}\not\in\claspidealgope$, in 
the two-sided cell preorder. By the observations in the first paragraph, 
the same must hold for the indecomposable 
$1$-morphisms in $\subcatquoefield\otimes_{\aforme}\C$ with respect to the 
$2$-ideal $\claspidealefielde$, which 
is generated by $\elfunctor[\qqpar]\left(\claspidealgope\right)$. This shows 
that $\elfunctor[e]\colon\slqmodgop\to\Subcatquo[e]$ 
is faithful on $2$-morphisms, 
since $\slqmodgop$ is semisimple.
\end{proof}

\begin{proposition}\label{proposition:cat-the-quoalgebra}
The isomorphism from \fullref{proposition:cat-the-algebra} 
gives an isomorphism 
$\subquo[e]\xrightarrow{\cong}\GGcv{\Kar{\subcatquo[e]}}$ 
of algebras.
\end{proposition}

\begin{proof}
This follows from the discussion 
above: By \fullref{corollary:clasps-ideal} and \fullref{lemma:slt-soergel-clasps} the vanishing $2$-ideal of level $e$ 
contains all colored clasps of level $e+1$. 
By by \fullref{proposition:cat-the-algebra} 
these decategorify to the $\RKLx{m,n}$ in the definition 
of $\subquo[e]$, while \fullref{lemma:Satake} ensures that 
the Grothendieck classes of the remaining $\CRKLx{m,n}$ 
form a basis of $\GGcv{\Kar{\subcatquo[e]}}$.
\end{proof}

\subsection{Generalizing dihedral Soergel bimodules}\label{subsec:dihedral-SB}

As before, we list certain analogies to the dihedral case.

\begin{dihedral}\label{remark:dihedral-SB1} The Hecke algebra $\hecke(\typeat{1})$ of 
\fullref{remark:dihedral-group1} is categorified by Soergel bimodules of affine type $\typea{1}$.
Here the Hecke algebra $\hecke(\typeat{2})$ is categorified by Soergel bimodules of affine type $\typea{2}$.
The difference is that now biinduction of the maximally singular bimodules only gives a proper $2$-subcategory.
\end{dihedral}

\makeautorefname{proposition}{Propositions}

\begin{dihedral}\label{remark:dihedral-SB2}
The Satake $2$-functor from \eqref{eq:elfunctor} 
exists in the dihedral case as well, with a bicolored 
version of quantum $\mathfrak{sl}_2$-modules 
as the source $2$-category. This $2$-category 
has a diagrammatic incarnation in terms of 
a $2$-colored Temperley--Lieb calculus \cite[Section 4.3]{El2}. 
The Soergel bimodules of finite Coxeter type 
$\typei$ can then be defined by annihilating the ideal generated by the colored
Jones--Wenzl projectors (i.e. colored $\mathfrak{sl}_2$ clasps) of level $e+1$ in this $2$-colored 
Temperley--Lieb calculus. Moreover, while 
the colored $\slt$-clasps satisfy the recursion 
in \fullref{lemma:recursion}, their $\mathfrak{sl}_2$ counterparts satisfy 
the Chebyshev recursion from \fullref{remark:dihedral-group2}. 
Finally, the analogs of \fullref{proposition:cat-the-algebra} 
and \ref{proposition:cat-the-quoalgebra}
hold as well.
\end{dihedral}

\makeautorefname{proposition}{Proposition}

\noindent\textbf{Missing proofs from \texorpdfstring{\fullref{section:funny-algebra}}{\ref{section:funny-algebra}}.}

\begin{proof}[Proof of \fullref{lemma:multiplication}]
Recall that the $2$-category $\adiag[\qpar]$ categorifies 
$\hecke$, i.e.
\begin{gather}\label{eq:EW-cat}
\hecke\xrightarrow{\cong}\GGcv{\Kar{\adiag[\qpar]}},
\end{gather}
such that the KL elements are sent 
to the Grothendieck classes of the indecomposable $1$-morphisms (with a fixed choice of grading), 
c.f. \cite{El1}, \cite{EW} or \fullref{remark:cat-affine-a2}. Furthermore, 
by \fullref{lemma:quotient-of-affine}, the algebra 
$\subquo$ can be embedded into $\hecke$ by sending the colored KL elements in $\subquo$ to 
KL elements in $\hecke$. Thus, we can 
identify elements of $\subquo$ with 
Grothendieck classes in $\GGcv{\Kar{\adiag[\qpar]}}$:

The element 
$\rklx{k,l}=\theta_{\tduc_{k+l}}\cdots\theta_{\tduc_1}\theta_{\tduc_0}\in \hecke$, 
with $\tduc_0=\tduc$, corresponds to 
\begin{gather}\label{eq:KL-element}
[
\wc\duc_{k+l}\tduc_{k+l}\duc_{k+l}\wc
\cdots
\wc\duc_1\tduc_1\duc_1\wc\duc_0\tduc_0\duc_0\wc
]
\in\GGcv{\Kar{\adiag[\qpar]}},
\end{gather}
where we can chose any compatible primary colors by 
\fullref{lemma:clasps-well-defined}. In fact, with \fullref{lemma:clasps-well-defined} in mind, we will
 denote all compatible primary colors simply by $\duc$ from now on. 

Using \fullref{lemma:remove-white}, we see that the element 
in \eqref{eq:KL-element} is 
equal to 
\[
\vnumber{2}^{k+l}\,
[
\wc\duc\tduc_{k+l}\duc
\cdots\duc\tduc_1 \duc\tduc_0\duc\wc
]
\in\GGcv{\Kar{\adiag[\qpar]}}.
\] 

Next, we use $\elfunctor$ from \eqref{eq:elfunctor}. 
By definition,  
$[\elfunctor]$ maps $[\fu^k\fud^l]\in\GGcv{\sltcatgop}$ to 
\[
[
\wc\duc\tduc_{k+l}\duc
\cdots\duc\tduc_1 \duc\tduc_0\duc\wc
]
\in\GGcv{\Kar{\adiag[\qpar]}},
\]
and to similar expressions with different rightmost color. 
By \fullref{remark:sl3-cat-GG}, this implies that $[\elfunctor]$ 
maps $[\Ll_{m,n}]$ to
\begin{gather}\label{eq:KL-element2}
{\textstyle\sum_{k,l}}
\vnumber{2}^{-k-l}d^{k,l}_{m,n}
[
\wc\duc\tduc_{k+l}\duc\wc
\cdots
\wc\duc\tduc_1\duc\wc\duc\tduc_0\duc\wc,
]
\end{gather}
and again to similar expressions with different rightmost color.

By \fullref{theorem:q-satake}, $\elfunctor$ is a degree-zero $2$-equivalence. 
In particular, it sends the simple $1$-morphisms to indecomposable $1$-morphisms. 
This implies that the element 
in \eqref{eq:KL-element2} is the Grothendieck 
class of an indecomposable $1$-morphism of $\Adiag[\qpar]$. 
Biinduction preserves indecomposibility, so our element $\RKLx{m,n}$ corresponds to the Grothendieck 
class of an indecomposable $1$-morphism in $\GGcv{\adiag[\qpar]}$. 
By the categorification theorem from \eqref{eq:EW-cat}, 
we see that $\RKLx{m,n}$ corresponds to a 
KL basis element in $\hecke$, and $\rklx{k,l}$ to the Grothendieck class of 
a Bott--Samelson bimodule. 

From the above, we obtain the first 
equation in \eqref{eq:multiplication}, since 
\[
[
\wc\duc\tduc\duc\wc\duc\tduc\duc\wc
]
=\vnumber{2}
[
\wc\duc\tduc\duc\tduc\duc\wc
]
= 
\vfrac{3}
[
\wc\duc\tduc\duc\wc
],
\]
by \fullref{lemma:remove-white} and \fullref{example:some-relations}. 
Using $[\elfunctor]$ and \eqref{eq:sl3-thingy-a}, we deduce the 
second equation in \eqref{eq:multiplication}.
Similarly, one can prove the third 
equation in \eqref{eq:multiplication} using \eqref{eq:sl3-thingy-b}.
\end{proof}

\begin{proof}[Proof of \fullref{proposition:two-bases}]
Let $\tduc$ be fixed for now. Recalling 
the notation from \fullref{section:sl3-stuff},
by \fullref{lemma:multiplication} and its 
proof given above, there is 
a $\Cv$-linear isomorphism between the scalar extension
$\GGcv{\sltcatgop}$ and 
$\Cv\left\{\rklx{k,l}\mid (k,l)\in X^+\right\}$, defined by 
\begin{gather*}
[\fu^k\fud^l]
\mapsto 
\vnumber{2}^{-k-l}\,\rklx{k,l}.
\end{gather*}
This shows that the $\rklx{k,l}$ are all linearly independent, 
and they are also linearly 
independent of $1$, of course. 

Since $\tduc$ was arbitrary and there are no relations in $\subquo$ which 
allow us to change the rightmost color in a word, it 
follows that 
\[
\left\{1 \right\}\cup \left\{\rklx{k,l}\mid (k,l)\in X^+,\; \tduc\in\Seset\right\}
\]
is a basis of $\subquo$.

Because $d^{m,n}_{m,n}=1$, and $d^{k,l}_{m,n}=0$ if $k+l>m+n$,
the above immediately implies that 
\[
\left\{\RKLx{m,n}\mid (m,n)\in X^+,\; \tduc\in\Seset\right\}
\]
is also a basis, 
since the transformation between the two 
sets of elements defined by \eqref{eq:the-expressions} 
is triangular with diagonal factors $\vnumber{2}^{-m-n}\neq 0$.  
\end{proof}
%
\section{Trihedral \texorpdfstring{$2$}{2}-representation theory}\label{sec:2-reps}

Keeping all notations from the previous sections, 
we are now going to explain the $2$-re\-presentation theory of 
the trihedral Soergel bimodules. Again, we have collected the 
analogies to the dihedral case at the end in \fullref{subsec:dihedral-group-cat}.
\medskip

\noindent\textbf{Background.}
\medskip

Let us briefly recall some terminology and results from $2$-re\-presentation theory as in e.g.
\cite{MM3} or \cite{MM5}, where we also need the graded setup as in \cite[Section 3]{MT1}. 

\subsubsection{\texorpdfstring{$\N_{[\vpar]}$}{Nv}-representation theory}\label{subsec:decat-story-a}

We start with the decategorified picture. Recall that $\vpar$ 
denotes a generic parameter, $\aformvN=\N[\vpar,\vpar^{-1}]$, 
$\aformv=\Z[\vpar,\vpar^{-1}]$ and $\Cv=\C(\vpar)$.

Following various authors, see e.g. \cite[Section 1]{EK1}, \cite[Chapter 3]{EGNO} or \cite{KM1}
and the references therein, we define:

\begin{definition}\label{definition:two-bases-integral}
A pair $(\posalg,\posbasis)$ of an associative, unital 
($\Cv$-)algebra $\posalg$ 
and a finite basis $\posbasis$ with $1\in\posbasis$
is called a $\aformvN$-algebra 
if
\[
\algstuff{x}\algstuff{y}\in\aformvN\posbasis
\]
holds for all $\algstuff{x},\algstuff{y}\in\posbasis$.
\end{definition}

\begin{definition}\label{definition:pos-integral-modules}
Let $(\posmod,\posmodbasis)$ be a pair 
of a (left) $(\posalg,\posbasis)$-representation $\posmod$ and a choice of a finite basis 
$\posmodbasis$ for it. We call $(\posmod,\posmodbasis)$
a $\aformvN$-representation if
\[
\posmod(\algstuff{z})\algstuff{m}\in\aformvN\posmodbasis
\]
holds for all $\algstuff{z}\in\posbasis,\algstuff{m}\in\posmodbasis$.
\end{definition}

\begin{example}\label{example:transitive}
These $\aformvN$-algebras and $\aformvN$-representations 
arise naturally as the decategorification 
of $2$-categories and $2$-representations, which will be recalled in the next section. 
\end{example}

Abusing notation, we sometimes write $\posalg$ instead of 
$(\posalg,\posbasis)$ and $\posmod$ instead of $(\posmod,\posmodbasis)$.

\begin{definition}\label{definition:eq-integral-modules}
Two $\aformvN$-representations $\posmod,\posmod^{\prime}$ are 
$\aformvN$-equivalent, denoted by 
$\posmod\cong_{+}\posmod^{\prime}$, if there exist a bijection 
$\posmodbasis\to\posmodbasis[\posmod^{\prime}]$
such that the induced linear map $\posmod\to\posmod^{\prime}$ 
is an isomorphism of $\posalg$-representations.
\end{definition}

\begin{example}\label{example:eq-integral-modules}
$\posmod\cong_{+}\posmod^{\prime}$ implies 
$\posmod\cong\posmod^{\prime}$ (meaning that the are 
isomorphic as $\posalg$-representa\-tions over $\Cv$), 
but not vice versa: 

First of all, $\posmod$ might be isomorphic over 
$\Cv$ to a $\posalg$-representation $\posmod^{\prime}$ that is 
not a $\aformvN$-re\-presentation. For example, consider the 
($\Cv$-)group algebra of any finite group with 
its basis given by the group elements. 
Its regular representation is a $\aformvN$-representation on this basis, and over $\Cv$ this 
representation decomposes into simple modules. 
However, most simple modules are not $\aformvN$-representations 
and the decomposition can usually not be obtained 
via base change matrices with entries from $\aformvN$.

Secondly, even if $\posmod\cong\posmod^{\prime}$ are 
two isomorphic $\aformvN$-representations, 
they may not be $\aformvN$-equivalent. For example, the dihedral 
Hecke algebra of type $\typei[12]$ has two 
$\aformvN$-representations,  
associated to the type $\typeE_6$ Dynkin, which are 
isomorphic over $\Cv$ but not $\aformvN$-equivalent 
(c.f. \cite[Theorem II(iii)]{MT1}).  
\end{example}

\subsubsection{Cells}\label{subsec:background-decat}

For any $\aformvN$-algebra $\posalg$ one can define 
cell theory as in \fullref{definition:cells-first}, e.g. 
$\algstuff{x}\lgeq\algstuff{y}$ for $\algstuff{x},\algstuff{y}\in\posbasis$ 
if there exists an element 
$\algstuff{z}\in\posbasis$ such that $\algstuff{x}$ appears as a summand of 
$\algstuff{z}\algstuff{y}$, when the latter is written as a linear combination 
of elements in $\posbasis$. We 
hence obtain (left, right and two-sided) cells $\lcell$, 
$\rcell$ and $\tcell$, and we can write $\lcell^{\prime}\lgeq\lcell$ etc.
See also e.g. \cite{KM1}  
(incorporating $\vpar$)
for details. The same notions be can defined for any 
$\aformvN$-representation $\posmod$, e.g. 
$\algstuff{m}\lgeq\algstuff{n}$ for 
$\algstuff{m},\algstuff{n}\in\posmodbasis$ if there exists some $\algstuff{z}\in\posbasis$ such that
$\algstuff{m}$ appears in $\posmod(\algstuff{z})\algstuff{n}$ with non-zero coefficient 
when written in terms of $\posmodbasis$.

\begin{definition}\label{definition:two-bases-integral-modules}
We call a $\aformvN$-representation 
$\posmod$ transitive 
if all basis elements belong to the same $\lsim$ equivalence class.
\end{definition}

\begin{remark}\label{remark:transitive}
Consider the graph with vertices given by
$\posmodbasis$ and with an oriented edge from $\algstuff{n}$ to 
$\algstuff{m}$ whenever $\algstuff{m}\lgeq\algstuff{n}$. Transitivity of $\posmod$ means that this graph  
is strongly connected.
\end{remark}

Similarly, we can also define the notion of a transitive $\aformv$-representation associated to a strongly 
connected graph. (Note that $\algstuff{m}\lgeq\algstuff{n}$ also makes sense over $\aformv$.)

\begin{definition}\label{definition:cell-module}
Fix $\lcell$. Let $\amod[M](\lgeq)$, respectively 
$\amod[M](\lgeqs)$, be the $\aformvN$-representations spanned by 
all $\algstuff{x}\in\posbasis$ in the union of all left cells $\lcell^{\prime}\lgeq\lcell$, 
respectively $\lcell^{\prime}\lgeqs\lcell$. (These are well-defined by 
\cite[Proposition 1]{KM1}.) We call 
$\amod[C]_{\lcell}=\amod[M](\lgeq)/\amod[M](\lgeqs)$ 
the (left) cell module for $\lcell$.
\end{definition}

By definition, all cell modules are transitive $\aformvN$-representations.

\begin{example}\label{example:cell-module-sym-group}
Coming back to \fullref{example:eq-integral-modules}: There is only one 
left (right, two-sided) cell for the group algebra of a finite group. 
The associated cell module is the regular representation.

However, on a different basis this might change considerably:
The Hecke algebras for (finite) Coxeter groups 
are $\aformvN$-algebras, where the 
KL basis plays the role of the basis $\posbasis$, see \cite{KaLu}. Their 
cell modules are Kazhdan--Lusztig's original cell modules. 
In the case of the symmetric group, these cell modules are the simple modules, 
but in general cell modules are not simple (since most simples are 
not $\aformvN$-representations).
\end{example}

\begin{example}\label{example:cell-module}
Decategorifications of cell $2$-representations, which will be recalled 
below, are key examples of cell modules.
\end{example}
 
Given any cell module $\amod[C]_{\lcell}$, the results in \cite[Section 8]{KM1} show that there 
exists a unique, maximal two-sided cell, called apex, which does not annihilate $\amod[C]_{\lcell}$. 
The same is true for general transitive $\aformvN$-representations by \cite[Section 9.2]{KM1}. 
Thus, we can restrict the study of transitive $\aformvN$-representations to a given apex.

\subsubsection{\texorpdfstring{$2$}{2}-representations of finitary \texorpdfstring{$2$}{2}-categories}\label{subsec:background-1}

Let $\K$ be a ring. An additive, $\K$-linear, ($\Z$-)graded $2$-category $\twocatstuff{C}$ 
(with the same grading conventions as in \fullref{convention:grading}), 
which is idempotent complete and Krull--Schmidt, is called 
graded finitary if:
\smallskip
\begin{enumerate}[label=$\blacktriangleright$]

\setlength\itemsep{.15cm}

\item It has finitely many objects, and all identity $1$-morphisms are indecomposable.

\item The $2$-hom spaces are free of finite $\K$-rank in each degree, and their grading is bounded from below.

\item Consider the $2$-subcategory of $\twocatstuff{C}$
having the same objects and $1$-morphisms, but only degree-preserving $2$-morphisms. Its 
split Grothendieck group is a free $\aformv$-module, with $\vpar$ corresponding to 
the grading shift, which we assume to be of finite $\aformv$-rank.

\end{enumerate}
\smallskip

(Note that the last point above implies that a graded finitary $2$-category has only 
finitely many equivalence classes of indecomposable $1$-morphisms up to grading shift.)

Similarly, a graded locally finitary $2$-category is as above, but relaxing the 
condition on the Grothendieck group by requiring it to be of countable $\aformv$-rank.

We also use graded finitary categories (having graded 
hom-spaces which are free of finite $\K$-rank), which are 
the objects of a $2$-category $\fincat$ with 
$1$-morphisms being additive, $\K$-linear, degree-preserving functors and 
$2$-morphisms being homogeneous natural
transformations of degree-zero. 

Let $(\fincat)^{\star}$ denote the 
$2$-category obtained from $\fincat$ by adding formal shifts 
to the $1$-morphisms. Its $2$-hom spaces are given by
\[
\twoHom_{(\fincat)^{\star}}(\obstuff{i},\obstuff{j})
=
{\textstyle \bigoplus_{s\in\Z}}\,
\twoHom_{\fincat}(\obstuff{i}\{s\},\obstuff{j}).
\]

\begin{example}\label{example:graded-finitary}
All $2$-categories in \fullref{sec:A2-diagrams} become graded 
(locally) finitary after taking their Karoubi envelope.
\end{example}
 
\begin{example}\label{example:graded-finitary-2}
Let $\algstuff{B}$ be a graded $\K$-algebra which is 
free of finite $\K$-rank.
The category of free, finite $\K$-rank, graded (left) $\algstuff{B}$-representations 
is a prototypical object of $\fincat$. For example, 
the graded representation categories of the quiver algebras $\zig[e]$ in \fullref{subsec:quiver} below 
are objects of $\fincat$.
\end{example}

A graded finitary $2$-representation of $\twocatstuff{C}$ is an additive, $\K$-linear
$2$-functor
\[
\cM\colon
\twocatstuff{C}
\to
(\fincat)^{\star}
\]
which is degree-preserving and commutes with shifts as in \cite[Definition 3.4]{MT1}.

\begin{example}\label{example:graded-finitary-3}
The principal $2$-representation 
$\twocatstuff{P}_{\obstuff{i}}=\twocatstuff{C}(\obstuff{i},\placeholder)$, 
where $\obstuff{i}$ is an object of $\twocatstuff{C}$, is a graded finitary 
$2$-representation of $\twocatstuff{C}$.
\end{example}

Graded finitary $2$-representations of $\twocatstuff{C}$ form a
graded $2$-category (in the sense of \fullref{convention:grading}), 
see \cite{MM3} for details, which can be adapted to the graded setting. In particular, there exists 
a well-defined notion of equivalence between such $2$-representations.

For simplicity, we say \textit{$2$-representation} instead of 
\textit{graded finitary $2$-representation} etc. from now on, i.e. we omit the 
\textit{graded finitary}.

\subsubsection{\texorpdfstring{$2$}{2}-cells}\label{subsec:background-2}

As in the case of $\aformvN$-algebras, one can define cells and cell $2$-representations of 
finitary $2$-categories: Let $\morstuff{X}$ and $\morstuff{Y}$ be 
indecomposable $1$-morphisms in a finitary $2$-category $\twocatstuff{C}$.
Set $\morstuff{X}\lgeq\morstuff{Y}$ if $\morstuff{X}$ is isomorphic 
to a direct summand of $\morstuff{Z}\morstuff{Y}$, up to a degree shift, for some 
indecomposable $1$-morphism $\morstuff{Z}$. Similarly one defines $\rgeq$ and $\tgeq$. 
The equivalence classes for these are 
called the respective cells, denoted by $\lcell$, $\rcell$ or $\tcell$. 
All these notions can be defined in a similar way 
for $2$-representations as well.

A finitary $2$-representation $\cM$ is transitive (see \cite[Section 3.1]{MM5}, 
or \cite[Definition 3.6]{MT1} in the graded setup), if $\cM$ is 
supported on one $\obstuff{i}\in\twocatstuff{C}$, and if all indecomposable objects 
$\obstuff{O},\obstuff{P}\in\cM(\obstuff{i})$ are in the same $\lsim$-equivalence class. A 
transitive $2$-representation is simple transitive, see \cite[Section 3.5]{MM5} (or \cite[Definition 3.6]{MT1} 
in the graded setup), if it does not have any non-zero, proper $\twocatstuff{C}$-invariant ideals.

\begin{remark}\label{remark:graded-finitary-3}
By \cite[Section 4]{MM5}, any $2$-representation has a weak Jordan--H{\"o}lder series
with simple transitive subquotients, which are unique up to 
permutation and equivalence. Therefore, it is natural to ask for the classification 
of simple transitive $2$-representations. Moreover, by \cite[Section 3]{MM5}, any 
transitive $2$-representation has a unique maximal $\twocatstuff{C}$-stable 
ideal which one can quotient by to get 
a simple transitive $2$-representation, called the simple transitive quotient.
\end{remark}

Every (graded) finitary $2$-category comes with a natural class of simple transitive $2$-representations: 

\begin{definition}\label{definition:2-cell-module}
Fix $\lcell$.
Then there exists $\obstuff{i}\in\twocatstuff{C}$ such that all
$1$-morphisms in $\lcell$ start 
at $\obstuff{i}$. 
Let $\cM(\lgeq)$ be the $2$-representations of $\twocatstuff{C}$ spanned by 
the additive closure of all indecomposable $1$-morphisms $\morstuff{F}$, in 
$\coprod_{\obstuff{j}\in\twocatstuff{C}}\twocatstuff{P}_{\obstuff{i}}(\obstuff{j})$, which belong to the union of all 
left cells $\lcell^{\prime}\lgeq\lcell$. 
Let $\twocatstuff{Z}(\lgeq)$ be the unique, 
proper two-sided $2$-ideal in $\cM(\lgeq)$.
(All of this is well-defined by 
\cite[Section 3.3 and Lemma 3]{MM5}.)
We call $\twocatstuff{C}_{\lcell}=\cM(\lgeq)/\twocatstuff{Z}(\lgeq)$ 
the cell $2$-representation for $\lcell$.
\end{definition}

Note that cell $2$-representations are always simple transitive.

\begin{example}\label{example:simple-cell}
In case of Soergel bimodules for the symmetric group,  
these exhaust all simple transitive $2$-representations and categorify 
the simple modules \cite{MM5}. However, both these facts are false in general, as 
the example of dihedral Soergel bimodules shows, see e.g. \cite{KMMZ}, \cite{MT1}.
\end{example}

\begin{remark}\label{remark:might-be-bigger}
On the decategorified level, the cell representation 
is obtained as the quotient of $\M(\lgeq)$ by $\M(\lgeqs)$,
cf. \fullref{definition:cell-module}. On the level of $2$-representations, the proper maximal 
two-sided $2$-ideal $\twocatstuff{Z}(\lgeq)$ strictly 
contains the two-sided $2$-ideal generated by the $2$-subrepresentation $\cM(\lgeqs)$ in general.
\end{remark}

Again, there is a unique, maximal 
two-sided cell, called $2$-apex, 
which does not annihilate a given cell $2$-representation. 
The same works for general transitive $2$-representations. 
See \cite[Section 3.2]{CM1} for more details.

\subsubsection{(Co)algebra \texorpdfstring{$1$}{1}-morphisms}\label{subsec:background-3}

An algebra $1$-morphism in $\twocatstuff{C}$ 
is a triple $(\morstuff{A},\mu,\eta)$, where 
$\morstuff{A}$ is a $1$-morphism and $\mu\colon\morstuff{A}\morstuff{A}\to\morstuff{A}$ 
and $\eta\colon\mathbbm{1}\to\morstuff{A}$ are $2$-morphisms satisfying 
the usual axioms for the multiplication and unit of an 
algebra.

Furthermore, there are compatible notions of module $1$-morphism 
over an algebra $1$-morphism $\morstuff{A}$, and 
of $2$-homomorphism between these. In this way, we get the $2$-categories  
$\modtwocat{\twocatstuff{C}}(\morstuff{A})$ 
(or $(\morstuff{A})\modtwocat{\twocatstuff{C}}$) 
of right (or left) $\morstuff{A}$-module $1$-morphisms in $\twocatstuff{C}$.
By post-composition $\modtwocat{\twocatstuff{C}}(\morstuff{A})$  
becomes a left $2$-representation of $\twocatstuff{C}$. Similarly, by pre-composition 
$(\morstuff{A})\modtwocat{\twocatstuff{C}}$ 
becomes a right $2$-representation of $\twocatstuff{C}$.

One defines coalgebra $1$-morphisms $(\morstuff{C},\delta,\varepsilon)$ in $\twocatstuff{C}$ 
and their respective comodule $2$-categories, which are also $2$-representations of $\twocatstuff{C}$, dually.

Finally, there are also compatible notions of bimodule $1$-morphism over an algebra $1$-morphism and $2$-homomorphism 
between bimodule $1$-morphisms. By definition, 
a Frobenius $1$-morphism $\morstuff{F}$ in $\twocatstuff{C}$ is 
an algebra $1$-morphism which is also a coalgebra $1$-morphism, such that 
the comultiplication $2$-morphism is a $2$-homomorphism between $\morstuff{F}$--$\morstuff{F}$-bimodule $1$-morphisms.

We refer to \cite{MMMT1} or \cite[Chapter 7]{EGNO} for further details.

\begin{remark}\label{remark:alg-objects}
Suppose that $\twocatstuff{C}$ is additionally fiat (meaning that it 
has a certain involution \cite[Section 2.4]{MM1}). Then 
\cite[Theorem 9]{MMMT1} asserts that, for any simple transitive $2$-representation $\cM$ of $\twocatstuff{C}$, 
there exists a simple algebra $1$-morphism $\morstuff{A}$ in $\overline{\twocatstuff{C}}$ 
(the projective abelianization of $\twocatstuff{C}$, as introduced in \cite[Section 3.2]{MMMT1}) 
such that $\cM$ is equivalent (as a $2$-representation of $\twocatstuff{C}$) 
to the subcategory of projective objects of $\modtwocat{\overline{\twocatstuff{C}}}(\morstuff{A})$. 
Hence, the classification of simple transitive $2$-representations of $\twocatstuff{C}$ is equivalent to 
the classification of simple algebra $1$-morphisms in  $\overline{\twocatstuff{C}}$. 
Or, dually, to the classification of cosimple coalgebra $1$-morphisms in $\underline{\twocatstuff{C}}$, the 
injective abelianization of $\twocatstuff{C}$.

The fiat $2$-categories $\twocatstuff{C}$ in this paper are special, because they are closely related to 
semisimple $2$-categories by the quantum Satake correspondence, and
the simple algebra $1$-morphisms which we study
in this paper all belong to $\twocatstuff{C}$.
\end{remark}

\subsection{Decategorified story}\label{subsec:decat-story}

\subsubsection{Trihedral transitive \texorpdfstring{$\N_{[\vpar]}$}{Nv}-representations}\label{subsec:Z-reps}

\makeautorefname{subsection}{Sections}

From  
\fullref{subsec:definition} 
and \ref{subsec:quotient-algebra}, 
in particular 
the connection to the representation theory of $\slt$, 
the following is evident.

\makeautorefname{subsection}{Section}

\makeautorefname{proposition}{Propositions}

\begin{propositionqed}\label{proposition:two-bases-integral}
The trihedral Hecke algebras are $\aformvN$-algebras, i.e.
for the basis $\basisC$ and $\basisC[e]$ from \fullref{proposition:two-bases} 
an \ref{proposition:dimension} we have
\[
\algstuff{x}\algstuff{y}\in\aformvN\basisC
\quad\text{and}\quad
\algstuff{x}^{\prime}\algstuff{y}^{\prime}\in\aformvN\basisC[e]
\]
for all $\algstuff{x},\algstuff{y}\in\basisC$ and 
$\algstuff{x}^{\prime},\algstuff{y}^{\prime}\in\basisC[e]$.
The same holds for the left colored KL bases.  
\end{propositionqed}

\makeautorefname{proposition}{Proposition}

This is our starting point for studying $\aformvN$-representations of the trihedral 
Hecke algebras. From now on, we fix the right colored KL bases for $\subquo$ and $\subquo[e]$,  
as in \fullref{proposition:two-bases-integral}.

\begin{example}\label{example:int-valued-0}
Most of the three-dimensional $\subquo[e]$-representations 
in \eqref{eq:the-simples} are not $\aformvN$-re\-presentations 
(for any choice of basis). 

For $e>1$, the one-dimensional representations $\M_{\vfrac{3},0,0}$, $\M_{0,\vfrac{3},0}$ and $\M_{0,0,\vfrac{3}}$
are also not $\aformvN$-representations, e.g. 
by \fullref{example-KL-combinatorics-2}, the action of $\RKLg{1,1}$ on 
$\M_{\vfrac{3},0,0}$ is given by 
\[
\RKLg{1,1}=\vnumber{2}^{-2}\theta_\gc\theta_\pc\theta_\gc
-\,\theta_\gc
\mapsto
-\vfrac{3}.
\]
Thus, $\M_{\vfrac{3},0,0}$ is not an $\aformvN$-representation.
\end{example}

Our next goal is to define 
several families of 
$\aformvN$-representations of the trihedral 
Hecke algebras.
Recall from \fullref{proposition:cells} and \fullref{corollary:cells} that 
the trihedral algebras have one trivial and one non-trivial two-sided
cell, both of which can be the apex of a transitive $\aformvN$-representation.
For the trivial cell there is only one such representation:

\begin{example}\label{example:int-valued}
The simple $\M_{0,0,0}$ 
(cf. \eqref{eq:the-simples})
is a transitive $\aformvN$-re\-pre\-sen\-ta\-tion of $\subquo[\infty]$, which also descends to $\subquo[e]$ for any $e$,
and its apex is the trivial cell. By \eqref{eq:the-one-dims} and \fullref{example:int-valued-0}, there are 
no other transitive $\aformvN$-representations whose apex is the 
trivial cell.
\end{example}

From now on we will only consider transitive $\aformvN$-representations 
whose apex is the unique, non-trivial two-sided cell.
For this purpose, we consider tricolored graphs, denoted by $\Gg$ etc., fixing certain conventions as follows.

\subsubsection{Graph-theoretic recollections}\label{subsec:three-colored-graphs}

For us a graph $\Gg$ is an undirected, connected, finite graph without loops, 
but possibly with multiple edges. We will also need 
graphs with directed edges and we indicate these by adding the superscript ${\placeholder}^\fu$ 
or ${\placeholder}^\fud$.

We call $\Gg=(\Gg,V=\{\Gset,\Oset,\Pset\},E=\{\Bset,\Rset,\Yset\})$ 
tricolored, with colors $\gc,\oc,\pc$, if $V$ and $E$ can be partitioned into three disjoint sets 
$\Gset,\Oset,\Pset$ and $\Bset,\Rset,\Yset$ such that
\[
\left(
\begin{gathered}
\begin{tikzpicture}[anchorbase, xscale=.35, yscale=.5]
	\draw [thick, myyellow] (0,0) to (3,0);
	\node at (0,0) {$\bulletg$};
	\node at (3,0) {$\bulleto$};
\end{tikzpicture}
\in\Yset\Rightarrow
\\
\bulletg\in\Gset\text{ and }\bulleto\in\Oset
\end{gathered}
\right),
\quad\quad
\left(
\begin{gathered}
\begin{tikzpicture}[anchorbase, xscale=.35, yscale=.5]
	\draw [thick, densely dashed, myred] (0,0) to (3,0);
	\node at (0,0) {$\bulleto$};
	\node at (3,0) {$\bulletp$};
\end{tikzpicture}
\in\Rset\Rightarrow
\\
\bulleto\in\Oset\text{ and }\bulletp\in\Pset
\end{gathered}
\right),
\quad\quad
\left(
\begin{gathered}
\begin{tikzpicture}[anchorbase, xscale=.35, yscale=.5]
	\draw [thick, densely dotted, myblue] (0,0) to (3,0);
	\node at (0,0) {$\bulletp$};
	\node at (3,0) {$\bulletg$};
\end{tikzpicture}
\in\Bset\Rightarrow
\\
\bulletp\in\Pset\text{ and }\bulletg\in\Gset
\end{gathered}
\right).
\]
(We usually denote a tricolored graph simply by $\Gg$, suppressing the tricoloring.)

The vertices of any tricolored graph $\Gg$ can be ordered  
such that the adjacency matrix $A(\Gg)$ is of the following form.
\begin{gather}\label{eq:ad-matrix}
A(\Gg)=\!\!
\raisebox{.25cm}{$\begin{tikzpicture}[baseline=(current bounding box.center),yscale=0.6]
  \matrix (m) [matrix of math nodes, row sep=.5em, column
  sep=.5em, text height=.5ex, text depth=0.25ex, ampersand replacement=\&] {
\phantom{a}
\& \Gset   
\& \Oset  
\& \Pset 
\\
\Gset
\& 0   
\& \hspace*{.01cm} A^{\mathrm{T}}\hspace*{.01cm}  
\& C 
\\
\Oset
\& A   
\& 0  
\& B^{\mathrm{T}} 
\\
\Pset
\& \,C^{\mathrm{T}}   
\& \,\,B\,\,  
\& 0 
\\   
};
\draw[densely dotted] (m-2-2.south west) to (m-2-4.south east);
\draw[densely dotted] (m-3-2.south west) to ($(m-3-4.south east)+(-.05,0)$);
\draw[densely dotted] ($(m-2-3.north west)+(0,.15)$) to (m-4-3.south west);
\draw[densely dotted] ($(m-2-3.north east)+(0,.15)$) to (m-4-3.south east);
\draw[thick] ($(m-2-2.north west)+(-.075,.15)$) to [out=255, in=90] ($(m-2-2.north west)+(-.2,-1.05)$) to [out=270, in=105] ($(m-2-2.north west)+(-.075,-2.25)$);
\draw[thick] ($(m-2-4.north east)+(.05,.15)$) to [out=285, in=90] ($(m-2-4.north east)+(.175,-1.05)$) to [out=270, in=75] ($(m-2-4.north east)+(.05,-2.25)$);
\end{tikzpicture}$}
,
\quad\quad
A(\oGg^{\fu})=A(\oGg^{\fud})^{\mathrm{T}}=\!\!
\raisebox{.25cm}{$
\begin{tikzpicture}[baseline=(current bounding box.center),yscale=0.6]
  \matrix (m) [matrix of math nodes, row sep=.5em, column
  sep=.5em, text height=.5ex, text depth=0.25ex, ampersand replacement=\&] {
\phantom{a}
\& \Gset   
\& \Oset  
\& \Pset 
\\
\Gset
\& 0   
\& \,\,\,0\,\,\,  
\& C 
\\
\Oset
\& A   
\& 0  
\& 0
\\
\Pset
\& 0   
\& \,\,B\,\,  
\& 0 
\\   
};
\draw[densely dotted] (m-2-2.south west) to (m-2-4.south east);
\draw[densely dotted] (m-3-2.south west) to ($(m-3-4.south east)+(-.05,0)$);
\draw[densely dotted] ($(m-2-3.north west)+(0,.15)$) to (m-4-3.south west);
\draw[densely dotted] ($(m-2-3.north east)+(0,.15)$) to (m-4-3.south east);
\draw[thick] ($(m-2-2.north west)+(-.075,.15)$) to [out=255, in=90] ($(m-2-2.north west)+(-.2,-1.05)$) to [out=270, in=105] ($(m-2-2.north west)+(-.075,-2.25)$);
\draw[thick] ($(m-2-4.north east)+(.05,.15)$) to [out=285, in=90] ($(m-2-4.north east)+(.175,-1.05)$) to [out=270, in=75] ($(m-2-4.north east)+(.05,-2.25)$);
\end{tikzpicture}$}.
\end{gather}
Here $A,B,C$ are matrices with 
entries in $\N$, encoding the connections 
$\Gset\rightarrow\Oset$ (matrix $A$), 
$\Oset\rightarrow\Pset$ (matrix $B$), 
and $\Pset\rightarrow\Gset$ (matrix $C$). 
We will always consider vertex-orderings of this form.
Moreover, $\Gg$ has two associated 
directed graphs $\oGg^{\fu}$ and $\oGg^{\fud}$
whose adjacency matrices are $A(\oGg^{\fu})$ 
and $A(\oGg^{\fud})$ as in \eqref{eq:ad-matrix}. 
They have the same vertex sets as $\Gg$, 
but their edges are oriented according to \eqref{eq:color-tensor}.

We write $i\in\Gg$ ($i\in\Gset$ etc.) meaning that 
$i$ is a ($\gc$-colored etc.) vertex of $\Gg$.
Furthermore, we denote by $S_{\Gg}$ the spectrum 
of $\Gg$, i.e. the multiset of eigenvalues of $A(\Gg)$, and we use similar notations for $\oGg^{\fu}$ and $\oGg^{\fud}$.

\begin{example}\label{example:triangle0}
Our main examples of tricolored graphs are all displayed in \fullref{subsec:gen-D-list}. 
Their spectra play an important role for us.
\end{example}

\begin{example}\label{example:triangle1}
The simplest examples, which are, however, fundamental for this paper, are the 
generalized type $\typeA$ Dynkin diagrams, e.g.:
\begin{gather*}
\graphA{1}
=
\begin{tikzpicture}[anchorbase, xscale=.35, yscale=.5]
	\draw [thick, myyellow] (0,0) node[below, mygreen] {\tiny{\text 1}} 
	to (1,1) node[below, myorange] {\tiny{\text 1}};
	\draw [thick, densely dotted, myblue] (0,0) 
	to (-1,1) node[below, mypurple] {\tiny{\text 1}};
	\draw [thick, densely dashed, myred] (1,1) to (-1,1);
	\node at (0,0) {$\bulletg$};
	\node at (1,1) {$\bulleto$};
	\node at (-1,1) {$\bulletp$};
\end{tikzpicture}
,\quad\quad
\graphA{1}^{\fu}
=
\begin{tikzpicture}[anchorbase, xscale=.35, yscale=.5]
	\draw [thick, myyellow, directed=.55] (0,0) node[below, mygreen] {\tiny{\phantom{1}}} to (1,1);
	\draw [thick, densely dotted, myblue, rdirected=.55] (0,0) to (-1,1);
	\draw [thick, densely dashed, myred, directed=.55] (1,1) to (-1,1);
	\node at (0,0) {$\bulletg$};
	\node at (1,1) {$\bulleto$};
	\node at (-1,1) {$\bulletp$};
\end{tikzpicture}
,\quad\quad
\graphA{1}^{\fud}
=
\begin{tikzpicture}[anchorbase, xscale=.35, yscale=.5]
	\draw [thick, myyellow, rdirected=.55] (0,0) node[below, mygreen] {\tiny{\phantom{1}}} to (1,1);
	\draw [thick, densely dotted, myblue, directed=.55] (0,0) to (-1,1);
	\draw [thick, densely dashed, myred, rdirected=.55] (1,1) to (-1,1);
	\node at (0,0) {$\bulletg$};
	\node at (1,1) {$\bulleto$};
	\node at (-1,1) {$\bulletp$};
\end{tikzpicture}
\\
\begin{tikzpicture}[baseline=(current bounding box.center),yscale=0.6]
  \matrix (m) [matrix of math nodes, row sep=1.65em, column
  sep=1em, text height=1.5ex, text depth=0.25ex, ampersand replacement=\&,font=\scriptsize] {
A=
\begin{pmatrix}
1\\
\end{pmatrix}
\&
B=
\begin{pmatrix}
1\\
\end{pmatrix}
\&
C=
\begin{pmatrix}
1\\
\end{pmatrix}
\\};
\end{tikzpicture}
\end{gather*}

\begin{gather*}
\graphA{2}
=
\begin{tikzpicture}[anchorbase, xscale=.35, yscale=.5]
	\draw [thick, myyellow] (0,0) node[below, mygreen] {\tiny{\text 1}} 
	to (1,1) node[below, myorange] {\tiny{\text 1}}
	to (0,2) node[below, mygreen] {\tiny{\text 2}};
	\draw [thick, myyellow] (0,2) 
	to (-2,2) node[below, myorange] {\tiny{\text 2}};
	\draw [thick, densely dotted, myblue] (0,0) 
	to (-1,1) node[below, mypurple] {\tiny{\text 1}}
	to (0,2);
	\draw [thick, densely dotted, myblue] (0,2) 
	to (2,2) node[below, mypurple] {\tiny{\text 2}};
	\draw [thick, densely dashed, myred] (1,1) to (-1,1) to (-2,2);
	\draw [thick, densely dashed, myred] (2,2) to (1,1);
	\node at (0,0) {$\bulletg$};
	\node at (0,2) {$\bulletg$};
	\node at (1,1) {$\bulleto$};
	\node at (-2,2) {$\bulleto$};
	\node at (2,2) {$\bulletp$};
	\node at (-1,1) {$\bulletp$};
\end{tikzpicture}
,\quad\quad
\graphA{2}^{\fu}
=
\begin{tikzpicture}[anchorbase, xscale=.35, yscale=.5]
	\draw [thick, myyellow, directed=.55] (0,0) node[below, mygreen] {\tiny{\phantom{1}}} to (1,1);
	\draw [thick, myyellow, rdirected=.55] (1,1) to (0,2);
	\draw [thick, myyellow, directed=.55] (0,2) to (-2,2);
	\draw [thick, densely dotted, myblue, rdirected=.55] (0,0) to (-1,1);
	\draw [thick, densely dotted, myblue, directed=.55] (-1,1) to (0,2);
	\draw [thick, densely dotted, myblue, rdirected=.55] (0,2) to (2,2);
	\draw [thick, densely dashed, myred, directed=.55] (1,1) to (-1,1);
	\draw [thick, densely dashed, myred, rdirected=.55] (-1,1) to (-2,2);
	\draw [thick, densely dashed, myred, rdirected=.55] (2,2) to (1,1);
	\node at (0,0) {$\bulletg$};
	\node at (0,2) {$\bulletg$};
	\node at (1,1) {$\bulleto$};
	\node at (-2,2) {$\bulleto$};
	\node at (2,2) {$\bulletp$};
	\node at (-1,1) {$\bulletp$};
\end{tikzpicture}
,\quad\quad
\graphA{2}^{\fud}
=
\begin{tikzpicture}[anchorbase, xscale=.35, yscale=.5]
	\draw [thick, myyellow, rdirected=.55] (0,0) node[below, mygreen] {\tiny{\phantom{1}}} to (1,1);
	\draw [thick, myyellow, directed=.55] (1,1) to (0,2);
	\draw [thick, myyellow, rdirected=.55] (0,2) to (-2,2); 
	\draw [thick, densely dotted, myblue, directed=.55] (0,0) to (-1,1);
	\draw [thick, densely dotted, myblue, rdirected=.55] (-1,1) to (0,2);
	\draw [thick, densely dotted, myblue, directed=.55] (0,2) to (2,2);
	\draw [thick, densely dashed, myred, rdirected=.55] (1,1) to (-1,1);
	\draw [thick, densely dashed, myred, directed=.55] (-1,1) to (-2,2); 
	\draw [thick, densely dashed, myred, directed=.55] (2,2) to (1,1);
	\node at (0,0) {$\bulletg$};
	\node at (0,2) {$\bulletg$};
	\node at (1,1) {$\bulleto$};
	\node at (-2,2) {$\bulleto$};
	\node at (2,2) {$\bulletp$};
	\node at (-1,1) {$\bulletp$};
\end{tikzpicture}
\\
\begin{tikzpicture}[baseline=(current bounding box.center),yscale=0.6]
  \matrix (m) [matrix of math nodes, row sep=1.65em, column
  sep=1em, text height=1.5ex, text depth=0.25ex, ampersand replacement=\&,font=\scriptsize] {
A=
\begin{pmatrix}
1 & 1 \\
0 & 1 \\
\end{pmatrix}
\&
B=
\begin{pmatrix}
1 & 1\\
1 & 0\\
\end{pmatrix}
\&
C=
\begin{pmatrix}
1 & 0\\
1 & 1\\
\end{pmatrix}
\\};
\end{tikzpicture}
\end{gather*}

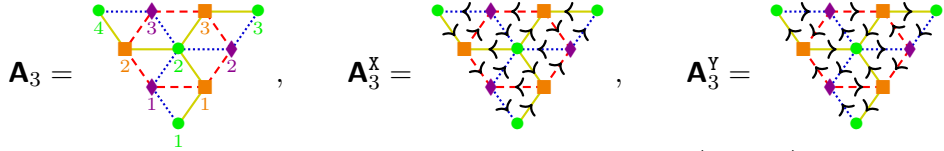
\begin{gather*}
\graphA{3}
=
\begin{tikzpicture}[anchorbase, xscale=.35, yscale=.5]
	\draw [thick, myyellow] (0,0) node[below, mygreen] {\tiny{\text 1}} 
	to (1,1) node[below, myorange] {\tiny{\text 1}}
	to (0,2) node[below, mygreen] {\tiny{\text 2}} 
	to (1,3) node[below, myorange] {\tiny{\text 3}}
	to (3,3) node[below, mygreen] {\tiny{\text 3}};
	\draw [thick, myyellow] (0,2) 
	to (-2,2) node[below, myorange] {\tiny{\text 2}}
	to (-3,3) node[below, mygreen] {\tiny{\text 4}};
	\draw [thick, densely dotted, myblue] (0,0) 
	to (-1,1) node[below, mypurple] {\tiny{\text 1}}
	to (0,2) 
	to (-1,3) node[below, mypurple] {\tiny{\text 3}}
	to (-3,3);
	\draw [thick, densely dotted, myblue] (0,2) 
	to (2,2) node[below, mypurple] {\tiny{\text 2}}
	to (3,3);
	\draw [thick, densely dashed, myred] (1,1) to (-1,1) to (-2,2) to (-1,3) to (1,3) to (2,2) to (1,1);
	\node at (0,0) {$\bulletg$};
	\node at (0,2) {$\bulletg$};
	\node at (3,3) {$\bulletg$};
	\node at (-3,3) {$\bulletg$};
	\node at (1,1) {$\bulleto$};
	\node at (-2,2) {$\bulleto$};
	\node at (1,3) {$\bulleto$};
	\node at (2,2) {$\bulletp$};
	\node at (-1,1) {$\bulletp$};
	\node at (-1,3) {$\bulletp$};
\end{tikzpicture}
,\quad\quad
\graphA{3}^{\fu}
=
\begin{tikzpicture}[anchorbase, xscale=.35, yscale=.5]
	\draw [thick, myyellow, directed=.55] (0,0) node[below, mygreen] {\tiny{\phantom{1}}} to (1,1);
	\draw [thick, myyellow, rdirected=.55] (1,1) to (0,2);
	\draw [thick, myyellow, directed=.55] (0,2) to (1,3);
	\draw [thick, myyellow, rdirected=.55] (1,3) to (3,3);
	\draw [thick, myyellow, directed=.55] (0,2) to (-2,2); 
	\draw [thick, myyellow, rdirected=.55] (-2,2) to (-3,3);
	\draw [thick, densely dotted, myblue, rdirected=.55] (0,0) to (-1,1);
	\draw [thick, densely dotted, myblue, directed=.55] (-1,1) to (0,2);
	\draw [thick, densely dotted, myblue, rdirected=.55] (0,2) to (-1,3);
	\draw [thick, densely dotted, myblue, directed=.55] (-1,3) to (-3,3);
	\draw [thick, densely dotted, myblue, rdirected=.55] (0,2) to (2,2);
	\draw [thick, densely dotted, myblue, directed=.55] (2,2) to (3,3);
	\draw [thick, densely dashed, myred, directed=.55] (1,1) to (-1,1);
	\draw [thick, densely dashed, myred, rdirected=.55] (-1,1) to (-2,2);
	\draw [thick, densely dashed, myred, directed=.55] (-2,2) to (-1,3);
	\draw [thick, densely dashed, myred, rdirected=.55] (-1,3) to (1,3);
	\draw [thick, densely dashed, myred, directed=.55] (1,3) to (2,2); 
	\draw [thick, densely dashed, myred, rdirected=.55] (2,2) to (1,1);
	\node at (0,0) {$\bulletg$};
	\node at (0,2) {$\bulletg$};
	\node at (3,3) {$\bulletg$};
	\node at (-3,3) {$\bulletg$};
	\node at (1,1) {$\bulleto$};
	\node at (-2,2) {$\bulleto$};
	\node at (1,3) {$\bulleto$};
	\node at (2,2) {$\bulletp$};
	\node at (-1,1) {$\bulletp$};
	\node at (-1,3) {$\bulletp$};
\end{tikzpicture}
,\quad\quad
\graphA{3}^{\fud}
=
\begin{tikzpicture}[anchorbase, xscale=.35, yscale=.5]
	\draw [thick, myyellow, rdirected=.55] (0,0) node[below, mygreen] {\tiny{\phantom{1}}} to (1,1);
	\draw [thick, myyellow, directed=.55] (1,1) to (0,2);
	\draw [thick, myyellow, rdirected=.55] (0,2) to (1,3);
	\draw [thick, myyellow, directed=.55] (1,3) to (3,3);
	\draw [thick, myyellow, rdirected=.55] (0,2) to (-2,2); 
	\draw [thick, myyellow, directed=.55] (-2,2) to (-3,3);
	\draw [thick, densely dotted, myblue, directed=.55] (0,0) to (-1,1);
	\draw [thick, densely dotted, myblue, rdirected=.55] (-1,1) to (0,2);
	\draw [thick, densely dotted, myblue, directed=.55] (0,2) to (-1,3);
	\draw [thick, densely dotted, myblue, rdirected=.55] (-1,3) to (-3,3);
	\draw [thick, densely dotted, myblue, directed=.55] (0,2) to (2,2);
	\draw [thick, densely dotted, myblue, rdirected=.55] (2,2) to (3,3);
	\draw [thick, densely dashed, myred, rdirected=.55] (1,1) to (-1,1);
	\draw [thick, densely dashed, myred, directed=.55] (-1,1) to (-2,2);
	\draw [thick, densely dashed, myred, rdirected=.55] (-2,2) to (-1,3);
	\draw [thick, densely dashed, myred, directed=.55] (-1,3) to (1,3);
	\draw [thick, densely dashed, myred, rdirected=.55] (1,3) to (2,2); 
	\draw [thick, densely dashed, myred, directed=.55] (2,2) to (1,1);
	\node at (0,0) {$\bulletg$};
	\node at (0,2) {$\bulletg$};
	\node at (3,3) {$\bulletg$};
	\node at (-3,3) {$\bulletg$};
	\node at (1,1) {$\bulleto$};
	\node at (-2,2) {$\bulleto$};
	\node at (1,3) {$\bulleto$};
	\node at (2,2) {$\bulletp$};
	\node at (-1,1) {$\bulletp$};
	\node at (-1,3) {$\bulletp$};
\end{tikzpicture}
\\
\begin{tikzpicture}[baseline=(current bounding box.center),yscale=0.6]
  \matrix (m) [matrix of math nodes, row sep=1.65em, column
  sep=1em, text height=1.5ex, text depth=0.25ex, ampersand replacement=\&,font=\scriptsize] {
A=
\begin{pmatrix}
1 & 1 & 0 & 0\\
0 & 1 & 0 & 1\\
0 & 1 & 1 & 0\\
\end{pmatrix}
\&
B=
\begin{pmatrix}
1 & 1 & 0\\
1 & 0 & 1\\
0 & 1 & 1\\
\end{pmatrix}
\&
C=
\begin{pmatrix}
1 & 0 & 0\\
1 & 1 & 1\\
0 & 1 & 0\\
0 & 0 & 1\\
\end{pmatrix}
\\};
\end{tikzpicture}
\end{gather*}
\vspace*{.1cm}

(The matrices $A,B,C$ are given with respect to the ordering of the vertices as indicated in the unoriented graphs.)  
The vertices of these graphs can be identified with the cut-offs of the positive Weyl chamber 
of $\slt$, cf. \eqref{eq:weight-picture}, where e.g the vertex with label 
$4$ in $\graphA{3}$ corresponds to the $\slt$-weight $(0,3)$. 

Moreover, the spectra of these graphs are:
\begin{gather*}
S_{\graphA{1}^{\fu}}
=\left\{
\text{roots of }
(X-1)
(X^2+X+1)
\right\}
,
\\
S_{\graphA{2}^{\fu}}
=
\left\{
\text{roots of }
(X^2-X-1)
(X^4+X^3+2X^2-X+1)
\right\},
\\
S_{\graphA{3}^{\fu}}=
\left\{
\text{roots of }
X
(X-2)
(X^2 + 2X + 4)
(X^6 - X^3 + 1)
\right\}.
\end{gather*}
The reader should compare these to \fullref{example:plot-zeros}.
\end{example}

Next, recall that an oriented graph $\Gg^{\mathrm{or}}$ is called strongly connected, 
if there is a path from $i$ to $j$ for any $i,j\in\Gg^{\mathrm{or}}$. Further, 
we say that $\Gg^{\mathrm{or}}$ 
is quasi regular if, for all $i,j\in\Gg^{\mathrm{or}}$, 
the number of two-step paths $i\rightarrow\placeholder\leftarrow j$ going first with and then 
against the orientation is the same as the number of two-step paths 
$i\leftarrow\placeholder\rightarrow j$ going first against and then 
with the orientation.

\begin{example}\label{example:weakly-regular}
Recall that an oriented graph is called 
weakly regular if the numbers of incoming and outgoing edges 
agree at each vertex, counting $r$ parallel edges $r^2$ times 
(e.g. a vertex with two incoming parallel edges must have two outgoing parallel edges or four outgoing single edges). By considering $i=j$, we 
see that any quasi regular graph is weakly regular, with the latter being 
a local condition which one easily checks. 
(In particular, each vertex is of even degree.) However, the converse is 
not true as e.g.
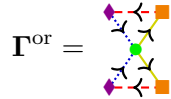
\begin{gather*}
\Gg^{\mathrm{or}}
=
\begin{tikzpicture}[anchorbase, xscale=.35, yscale=.5]
	\draw [thick, myyellow, directed=.55] (0,0) to (1,1);
	\draw [thick, myyellow, directed=.55] (0,0) to (1,-1);
	\draw [thick, densely dotted, myblue, directed=.55] (-1,1) to (0,0);
	\draw [thick, densely dotted, myblue, directed=.55] (-1,-1) to (0,0);
	\draw [thick, densely dashed, myred, directed=.55] (1,1) to (-1,1);
	\draw [thick, densely dashed, myred, directed=.55] (1,-1) to (-1,-1);
	\node at (0,0) {$\bulletg$};
	\node at (1,1) {$\bulleto$};
	\node at (1,-1) {$\bulleto$};
	\node at (-1,1) {$\bulletp$};
	\node at (-1,-1) {$\bulletp$};
\end{tikzpicture}
\end{gather*}
is weakly regular, but not quasi regular.
\end{example}

By convention, we call $\Gg$ as above strongly connected, respectively quasi regular,
if $\oGg^{\fu}$ and $\oGg^{\fud}$ are both strongly connected, respectively quasi regular.

\begin{definition}\label{definition:admissible}
A graph $\Gg$ is called admissible if it admits a tricoloring, 
such that $\Gg$ is strongly connected and quasi regular. 
\end{definition}

\begin{example}\label{example:triangle2}
All of our main examples from \fullref{subsec:gen-D-list} 
are admissible. 
\end{example}

\begin{lemma}\label{lemma:weakly-regular}
The matrices $A,B,C$ in \eqref{eq:ad-matrix}, which are blocks of $A(\Gg)$, satisfy
\begin{gather}\label{eq:main-transposes}
A^{\mathrm{T}}A=
CC^{\mathrm{T}},
\quad\quad
AA^{\mathrm{T}}=
B^{\mathrm{T}}B,
\quad\quad
C^{\mathrm{T}}C=
BB^{\mathrm{T}}
\end{gather}
if and only if $\Gg$ is quasi regular.
\end{lemma}

In particular, $AA^{\mathrm{T}}$, $A^{\mathrm{T}}A$, 
$BB^{\mathrm{T}}$, $B^{\mathrm{T}}B$,
$CC^{\mathrm{T}}$ and $C^{\mathrm{T}}C$ have the same non-zero eigenvalues 
for any quasi regular graph $\Gg$.

\begin{proof}
Assume that $\Gg$ is quasi regular.
Then, in $\Gg^{\fu}$, the entries of $A^{\mathrm{T}}A$ count the number of two-step paths 
$\Gset\rightarrow\Oset\leftarrow\Gset$, while the entries of $CC^{\mathrm{T}}$ 
count the number of two-step paths 
$\Gset\leftarrow\Pset\rightarrow\Gset$. 
A similar statement holds 
for the other colors respectively matrix equations in \eqref{eq:main-transposes}. 
Hence, all equations in \eqref{eq:main-transposes} 
hold if and only if
$\Gg$ is quasi regular.
\end{proof}
 
\subsubsection{Some trihedral \texorpdfstring{$\N_{[\vpar]}$}{Nv}-represenations}\label{subsec:three-colored-graphs-reps}
 
We denote by $\Cv\{\Gset,\Oset,\Pset\}$ 
the free ($\Cv$-)vector space on the vertex set of $\Gg$.

\begin{definition}\label{definition:n-modules}
We define a $\subquo$-representation
\[
\M_{\Gg}\colon\subquo\to\End_{\Cv}(\Cv\{\Gset,\Oset,\Pset\})
\]
by associating the following matrices to the generators $\theta_{\gc}$, $\theta_{\oc}$, $\theta_{\pc}$:
\begin{gather}\label{eq:main-matrices}
\begin{gathered}
\M_{\Gg}(\theta_{\gc})=
\vnumber{2}
{\scriptstyle
\begin{pmatrix}
\vnumber{3}\Idmatrix & A^{\mathrm{T}} & C \\
0 & 0 & 0 \\
0 & 0 & 0
\end{pmatrix}
}
,\quad\quad
\M_{\Gg}(\theta_{\oc})=
\vnumber{2}
{\scriptstyle
\begin{pmatrix}
0 & 0 & 0 \\
A & \vnumber{3}\Idmatrix & B^{\mathrm{T}} \\
0 & 0 & 0
\end{pmatrix}
}
,
\\
\M_{\Gg}(\theta_{\pc})
=
\vnumber{2}
{\scriptstyle
\begin{pmatrix}
0 & 0 & 0 \\
0 & 0 & 0 \\
C^{\mathrm{T}} & B & \vnumber{3}\Idmatrix
\end{pmatrix}
}
.
\end{gathered}
\end{gather}
Here $A,B,C$ are as in 
\eqref{eq:ad-matrix}. 
\end{definition}

Note that we have
\[
\Mt[\Gg]
=
\M_{\Gg}(\theta_{\gc})
+
\M_{\Gg}(\theta_{\oc})
+
\M_{\Gg}(\theta_{\pc})
=
\vnumber{2}\left(
\vnumber{3}\Idmatrix
+
A(\oGg)
\right).
\]

\begin{remark}\label{remark:spectrum-matrices}
The three-dimensional simple $\subquo[e]$-representations $\M_z$ 
in \eqref{eq:three-dims} are similar to the $\M_{\Gg}$ in \eqref{eq:main-matrices}. 
In $\M_{\Gg}$ the complex entry $z$ of $\M_z$ has been replaced by $\N$-matrices $A,B,C$ 
which have these complex numbers as eigenvalues, as we will see in \fullref{corollary:poly-killed} below. 
However, in $\M_{\Gg}$ the matrices $A,B,C$ need not be equal, whereas in $\M_z$ we only have one complex number.
\end{remark}

We always choose $\{\Gset,\Oset,\Pset\}$ as a basis.
Recalling the setup from \fullref{subsec:three-colored-graphs} we get:

\begin{lemma}\label{lemma:n-modules}
$\M_{\Gg}$ is well-defined 
if and only if $\Gg$ is quasi regular.
\end{lemma}

\begin{proof}
By direct computation, one immediately sees that \eqref{eq:first-rel} always holds, 
irrespective of $A,B$ and $C$. 
Furthermore, note that $\M_{\Gg}$ preserves the relations in \eqref{eq:second-rel} if and only if 
the equations in \eqref{eq:main-transposes} hold. The claim then follows from \fullref{lemma:weakly-regular}.
\end{proof}

From now on we assume that $\Gg$ is quasi regular whenever we write $\M_{\Gg}$. 
Proving that these are $\aformvN$-representations 
is hard and follows from categorification. However, if we drop the positivity condition, then the following is clear 
by noting that the scalars $\vnumber{2}^{-k-l}$ appearing in the 
definition of the colored KL elements cancel against the positive powers of $\vnumber{2}$ in \eqref{eq:main-matrices}.

\begin{lemmaqed}\label{lemma:trans-graphs}
$\M_{\Gg}$ is a transitive $\aformv$-representation if and only if 
$\Gg$ is admissible.
\end{lemmaqed}

\begin{example}\label{example:small-tri}
Take $e=2$ and the graph $\graphA{2}$ as
in \fullref{subsec:gen-D-list}. Fix $\gc$ as a starting color. Then the 
six non-trivial, colored KL basis elements of 
$\subquo[2]$ act on $\M_{\graphA{2}}$ via matrices whose 
entries are all in $\aformvN$. 
For $\RKLg{0,0}=\theta_{\gc}$, 
$\RKLg{1,0}=\vnumber{2}^{-1}\theta_{\oc}\theta_{\gc}$
and $\RKLg{0,1}=\vnumber{2}^{-1}\theta_{\pc}\theta_{\gc}$ 
this is immediately clear. For the other basis elements, one can check the claim by calculation. For example, 
$\RKLg{2,0}=\vnumber{2}^{-2}\theta_{\pc}\theta_{\oc}\theta_{\gc}-
\vnumber{2}^{-1}\theta_{\pc}\theta_{\gc}$, since $\pxy{2,0}=\fu^2-\fud$, so  
\[
\M_{\graphA{2}}(\RKLg{2,0})=
\vnumber{2}
{\scriptstyle
\scalebox{0.85}{$
\begin{pmatrix}
0 & 0 & 0 & 0 & 0 & 0 \\
0 & 0 & 0 & 0 & 0 & 0 \\
0 & 0 & 0 & 0 & 0 & 0 \\
0 & 0 & 0 & 0 & 0 & 0 \\
0 & \vnumber{3} & 1 & 1 & 1 & 1 \\
\vnumber{3} & 0 & 1 & 0 & 1 & 0 \\
\end{pmatrix}$}}.
\]
The matrices associated to $\RKLg{1,1}$ and $\RKLg{0,2}$ can be computed similarly. The fact that we get a 
$\aformvN$-representation is non-trivial, because the expressions for the $\RKLg{m,n}$ 
in terms of the $\rklx{k,l}$ have negative coefficients.
\end{example}

The following can be proved as in the dihedral case \cite[Section 5.4]{MT1}.

\begin{lemmaqed}\label{lemma:find-all-transitives-2}
Let $\Gg$ and $\Gg^{\prime}$ be two admissible graphs. Then
$\M_{\Gg}\cong_{+}\M_{\Gg^{\prime}}$ 
if and only if $\Gg$ and $\Gg^{\prime}$ are 
isomorphic as tricolored graphs.
\end{lemmaqed}

\begin{example}\label{example:find-all-transitives-2}
For the graphs from \fullref{subsec:gen-D-list} we get the following.
The graph $\graphA{e}$ allows three 
non-isomorphic tricolorings in case $e\equiv 0\bmod 3$, but only one otherwise.
The graph $\graphD{e}$ can always be tricolored in three non-isomorphic ways, while the 
graph $\graphC{e}$ admits only one tricoloring up to isomorphism. Finally, 
in type $\typeE$ there are always three non-isomorphic tricolorings except for 
the graph $\graphE{5}$ which has only one such tricoloring up to isomorphism.
Thus, \fullref{lemma:find-all-transitives-2} gives us the corresponding $\aformv$-representations 
which are not $\aformvN$-equivalent.
\end{example}

\begin{lemma}\label{lemma:find-all-transitives}
Let $\M$ be a transitive $\aformvN$-representation of 
$\subquo[\infty]$ which satisfies 
\begin{gather*}
\M(\theta_{\tduc})\algstuff{m}
=
a\algstuff{m}
+
\aformvN
\left(\posmodbasis[\M]{-}\{\algstuff{m}\}\right)
\;\Rightarrow\;
a\in\{0,\vfrac{3}\},
\quad\text{for all }\tduc,m,
\\
\text{and}\quad
a=\vfrac{3}\text{ only if }\M(\theta_{\tduc})\algstuff{m}
=
a\algstuff{m}.
\end{gather*}
Then there exists an admissible 
graph $\Gg$ with $\M\cong_{+}\M_{\Gg}$.
\end{lemma}

\begin{proof}
Recall that $\M$ has a fixed basis $\posmodbasis[\M]$ 
on which all elements of the colored KL basis act by 
matrices with entries in $\aformvN$, that $\theta_{\tduc}^2=\vfrac{3}\theta_{\tduc}$
and that the trace of an idempotent 
matrix is equal to its rank (which thus holds 
for $\vfrac{3}^{-1}\M(\theta_{\tduc})$). In particular,
the assumption implies that for each generator 
$\theta_{\tduc}$ there is an ordering of $\posmodbasis[\M]$ such that
\begin{gather}\label{eq:funny-matrix}
\M(\theta_{\tduc})=
\begin{pmatrix}
\vfrac{3}\Idmatrix & D \\
0 & 0
\end{pmatrix}
\end{gather}
for some matrix $D$ with entries in $\aformvN$. The rest of the proof now follows along the lines of 
\cite[Corollary 5.5]{Zi1} or \cite[Section 4.3]{KMMZ}:
\smallskip
\begin{enumerate}[label=$\blacktriangleright$]

\setlength\itemsep{.15cm}

\item First observe that each $\algstuff{m}$ is a $\vfrac{3}$-eigenvector 
of some $\theta_\tduc$, since otherwise 
$\M(\theta_{\gc})+\M(\theta_{\oc})+\M(\theta_{\pc})$ would have a zero row 
by \eqref{eq:funny-matrix}, which contradicts the transitivity.

\item Secondly, $\algstuff{m}$ is not a $\vfrac{3}$-eigenvector for all the 
$\theta_\tduc$. To see this, assume the contrary. Then, by transitivity, 
$\M$ has to be one-dimensional with all $\theta_\tduc$ acting by $\vfrac{3}$. 
However, as in \fullref{lemma:further-restrictions}, this contradicts 
the fact that $\M$ is a $\subquo[e]$-representation.

\item Finally, $\algstuff{m}$ is not a common $\vfrac{3}$-eigenvector of two of the 
$\theta_\tduc$. Assume on the contrary that $\theta_\gc$ and $\theta_\oc$ had such a common 
eigenvector. Then $\M(\theta_{\gc})\algstuff{m}=\M(\theta_{\oc})\algstuff{m}=\vfrac{3}\algstuff{m}$ and $\M(\theta_{\pc})\algstuff{m}=0$.
This contradicts \eqref{eq:second-rel}.
\end{enumerate}
\smallskip
\makeautorefname{lemma}{Lemmas}
Together with \fullref{lemma:weakly-regular}, \ref{lemma:n-modules} and \ref{lemma:trans-graphs}, 
this proves the claim. \makeautorefname{lemma}{Lemma}
\end{proof}

\subsubsection{The classification problem \fullref{problem:classification}}\label{subsec:Z-reps-2}

Back to the polynomials $\pxy{m,n}$ from \fullref{definition:sl3-polys}. 
Observe that quasi regularity implies that
\begin{gather}\label{eq:weakly-regular}
A(\oGg^{\fu})A(\oGg^{\fud})
=
{\scriptstyle
\begin{pmatrix}
CC^{\mathrm{T}} & 0 & 0 \\
0 & AA^{\mathrm{T}} & 0 \\
0 & 0 & BB^{\mathrm{T}}
\end{pmatrix}
}
\stackrel{\eqref{eq:main-transposes}}{=}
{\scriptstyle
\begin{pmatrix}
A^{\mathrm{T}}A & 0 & 0 \\
0 & B^{\mathrm{T}}B & 0 \\
0 & 0 & C^{\mathrm{T}}C
\end{pmatrix}
}
=
A(\oGg^{\fud})A(\oGg^{\fu}).
\end{gather}
Thus, we can formulate the following classification problem. 
\newline

\begin{problem}\label{problem:classification}
Classify all admissible graphs $\Gg$ such that 
\begin{gather*}
\pxy[A(\oGg^{\fu}),A(\oGg^{\fud})]{m,n}=0, \quad\text{for all }m+n=e+1.
\end{gather*}
(In other words, classify all admissible graphs $\Gg$ such 
that $z\in S_{\Gg^{\fu}}$ only if $(z,\overline{z})\in\vanset{e}$.)
\end{problem}

\begin{proposition}\label{proposition:poly-killed}
A graph $\Gg$ is a solution of \fullref{problem:classification} 
if and only if $\M_{\Gg}$ 
descends to a transitive $\aformv$-representation of $\subquo[e]$.
\end{proposition}

\begin{proof}
Recall that admissible graphs are always strongly connected. Thus,
the claim about transitivity is clear and it remains to check the other claims.

To this end, fix $m,n$. 
Observe that $\pxy[A(\oGg^{\fu}),A(\oGg^{\fud})]{m,n}$ 
has at most one non-zero block matrix entry in each of the $\Gset$-, $\Oset$- and $\Pset$-rows 
(as indicated in \eqref{eq:ad-matrix}), since 
$A(\oGg^{\fu})$ and $A(\oGg^{\fud})$ just permute the $\Gset$-, $\Oset$- and $\Pset$-blocks, 
and multiply them by $A,B,C$ or their transpose. Let us denote these 
block matrix entries by $N^{\Gset}_{m,n}$, $N^{\Oset}_{m,n}$ and $N^{\Pset}_{m,n}$. 

Let us fix $\gc$ as a starting color, the other two cases work verbatim.
In this case an easy calculation yields:
\begin{gather*}
\M_{\Gg}(\RKLg{m,n})
=
\begin{cases}
\vnumber{2}
{\scriptstyle
\begin{pmatrix}
N^{\Gset}_{m,n}\cdot\vnumber{3}\Idmatrix
& 
N^{\Gset}_{m,n}\cdot A^{\mathrm{T}}
& 
N^{\Gset}_{m,n}\cdot C
\\
0 & 0 & 0 \\
0 & 0 & 0
\end{pmatrix}
}
,&
\text{if }
m+2n \equiv 0 \bmod 3,
\\
\vnumber{2}
{\scriptstyle
\begin{pmatrix}
0 & 0 & 0 \\
N^{\Oset}_{m,n}\cdot\vnumber{3}\Idmatrix
& 
N^{\Oset}_{m,n}\cdot A^{\mathrm{T}}
& 
N^{\Oset}_{m,n}\cdot C
\\
0 & 0 & 0
\end{pmatrix}
}
,&
\text{if }
m+2n \equiv 1 \bmod 3,
\\
\vnumber{2}
{\scriptstyle
\begin{pmatrix}
0 & 0 & 0 \\
0 & 0 & 0 \\
N^{\Pset}_{m,n}\cdot\vnumber{3}\Idmatrix
& 
N^{\Pset}_{m,n}\cdot A^{\mathrm{T}}
& 
N^{\Pset}_{m,n}\cdot C
\end{pmatrix}
}
,&
\text{if }
m+2n \equiv 2 \bmod 3.
\end{cases}
\end{gather*}
As in the proof of \fullref{lemma:three-dim-rep1}, 
we note that in the calculation of $\M_{\Gg}(\RKLg{m,n})$ 
the positive powers of $\vnumber{2}$, due to \eqref{eq:main-matrices}, 
cancel against the negative powers of $\vnumber{2}$, 
which appear in \eqref{eq:the-expressions}, up to an 
overall factor $\vnumber{2}$. We see that $\M_{\Gg}(\RKLg{m,n})$ 
vanishes if and only if $\pxy[A(\oGg^{\fu}),A(\oGg^{\fud})]{m,n}=0$.
\end{proof}

Our main examples of solutions of \fullref{problem:classification} are the graphs from \fullref{subsec:gen-D-list}. 
Indeed, as can be seen in \fullref{subsec:gen-D-spectra}, their spectra are such that 
\fullref{proposition:poly-killed} applies:

\begin{corollary}\label{corollary:poly-killed}
The generalized $\ADE$ Dynkin diagrams from \fullref{subsec:gen-D-list} 
give transitive $\aformv$-representations $\M_{\Gg}$ 
for the associated level $e$.
\end{corollary}

\makeautorefname{lemmaqed}{Lemmas}

By \fullref{lemma:find-all-transitives-2}, 
\ref{lemma:find-all-transitives} and \fullref{proposition:poly-killed}, classifying all 
$\aformv$-representations of $\subquo[e]$ boils down to 
\fullref{problem:classification}. We have already seen that the generalized 
$\ADE$ Dynkin diagrams give solutions of \fullref{problem:classification}. 
So two questions remain: whether these are all solutions and whether these are 
$\aformvN$-representations (transitivity is clear because 
the graphs are strongly connected).

\makeautorefname{lemmaqed}{Lemma}

We do not have a complete answer to these questions. However, we are able to prove:

\begin{proposition}\label{proposition:typeA-D-decat}
Let $\gc,\oc,\pc$ indicate the starting color. Then
we have (at least) the following transitive $\aformvN$-representations 
of $\subquo[e]$. 
\begin{gather}\label{eq:the-transitives}
\begin{tikzpicture}[baseline=(current bounding box.center),yscale=0.6]
  \matrix (m) [matrix of math nodes, row sep=1em, column
  sep=1em, text height=1.5ex, text depth=0.25ex, ampersand replacement=\&] {
\phantom{a}
\& 
e\equiv 0\bmod 3 
\&  
e\not\equiv 0\bmod 3 
\\
\text{$\aformvN$-reps.}
\&
\begin{gathered}
\M_{\graphA{e}^{\gc}},\,
\M_{\graphA{e}^{\oc}},\,
\M_{\graphA{e}^{\pc}},
\\
\M_{\graphD{e}^{\gc}},\,
\M_{\graphD{e}^{\oc}},\,
\M_{\graphD{e}^{\pc}}
\end{gathered}
\& 
\M_{\graphA{e}^{\gc}}
\\
\text{quantity}
\&
6
\& 
1
\\
\\};
  \draw[densely dashed] ($(m-1-1.south west)+ (-.75,0)$) to ($(m-1-3.south east)+ (.6,0)$);
  \draw[densely dashed] ($(m-1-3.north west) + (-.3,0)$) to ($(m-1-3.north west) + (-.3,-4.2)$);
  \draw[densely dashed] ($(m-1-2.north west) + (-.85,0)$) to ($(m-1-2.north west) + (-.85,-4.2)$);
\end{tikzpicture}
\end{gather}
Moreover, the representations $\M_{\graphA{e}}$ 
are the cell modules of $\subquo[e]$.
\end{proposition}

\makeautorefname{lemmaqed}{Lemmas}

\begin{proof}
Except for the claim about positivity, 
this is clear by \fullref{corollary:poly-killed}, \fullref{lemma:trans-graphs}, 
\ref{lemma:find-all-transitives-2} and the construction. 

For example, if $e\not\equiv 0\bmod 3$, then there is only one type 
$\typeA$ representation up to $\aformvN$-equivalence, since all tricolorings of $\graphA{e}$ give isomorphic 
tricolored graphs. To see this, note that a tricoloring of the lowest 
triangle fixes the tricoloring of the whole graph, 
and that there are six choices. When $e\not\equiv 0\bmod 3$, they all give isomorphic tricolored graphs, 
as can be easily seen. When $e\equiv 0\bmod 3$, we get three different 
isomorphism classes of tricolored graphs, 
which are determined by the color of the corner vertices. Note that there 
is one more vertex with that color, e.g. $\gc$,  than vertices with one 
of the other two colors, e.g. $\oc$ resp. $\pc$. This explains why tricolored graphs whose 
corner vertices have different colors, are non-isomorphic. Finally, for 
any fixed color of the corner vertices, there are 
two tricolored graphs, which are isomorphic by a $\zeetwo$-symmetry in the 
bisector of the angle at that vertex.
 
Positivity 
follows from \fullref{theorem:typeA-D-cat} 
which we proof later on.
\end{proof}

\makeautorefname{lemmaqed}{Lemma}

In contrast to simples, the 
transitive $\aformvN$-representations of $\subquo[e]$ can 
get arbitrarily large as $e$ grows, c.f. \eqref{eq:the-simples} and \eqref{eq:the-transitives}. We only know 
their complete classification for small values of $e$. 

\subsubsection{Classification for small levels}\label{subsec:decat-story-b}

\begin{theorem}\label{theorem:low-level-classification}
Let $e\in\{0,1,2,3\}$.
An admissible graph $\Gg$ provides a solution 
to \fullref{problem:classification} if and only if 
$\Gg$ is a generalized $\ADE$ Dynkin 
diagram of the corresponding level or
\begin{gather}\label{eq:special-solution}
\Gg=
\begin{tikzpicture}[anchorbase, xscale=.35, yscale=.5]
	\draw [thick, myyellow, double] (0,0) to (1,1);
	\draw [thick, densely dotted, myblue, double] (0,0) to (-1,1);
	\draw [thick, densely dashed, myred, double] (1,1) to (-1,1);
	\node at (0,0) {$\bulletg$};
	\node at (1,1) {$\bulleto$};
	\node at (-1,1) {$\bulletp$};
\end{tikzpicture}
,
\quad
\text{for }e=3,
\end{gather}
called special solution. (Note the double edges.)
\end{theorem}

\begin{proof}
We do the hardest case where $e=3$ and omit the others, all of which 
can be proven similarly. In this case, the vanishing ideal $\vanideal{3}$ is 
generated by
\[
\left\{
\pxy{4,0}
,
\pxy{3,1}
,
\pxy{2,2}
,
\pxy{1,3}
,
\pxy{0,4}
\right\}
\subset\Z[\fu,\fud]
\]
with the polynomials as in \fullref{example:sl3-polys}. 
We proceed as follows: consider the polynomials
\begin{gather}\label{eq:new-cp}
\left\{
\fud^{4}\pxy{4,0}
,
\fud^{2}\pxy{3,1}
,
\fud^{3}\pxy{2,2}
,
\fud\pxy{1,3}
,
\fud^{2}\pxy{0,4}
\right\}
\subset\Z[\fu,\fud],
\end{gather} 
which are now polynomials in the variables $\varx=\fu\fud$ and $\vary=\fud^3$. 
Clearly, any solution of \fullref{problem:classification} gives an admissible graph $\Gg$ 
such that $(\varx=A(\Gg^{\fu})A(\Gg^{\fud}),\vary=A(\Gg^{\fud})^3)$ 
is annihilated by the polynomials in \eqref{eq:new-cp}. 
Hence, one can solve \fullref{problem:classification} by first classifying 
all solutions of \eqref{eq:new-cp} in terms of 
$\varx$ and $\vary$ and then checking which ones give solutions of 
\fullref{problem:classification} in terms of $A(\Gg^{\fu})$ and $A(\Gg^{\fud})$. 

To this end, we can use the theory of Gr{\"o}bner bases, for the 
lexicographical ordering on monomials induced by $\varx<\vary$, 
to rewrite \eqref{eq:new-cp}. This shows that $\Gg$ solves        
\fullref{problem:classification} if and only if $\varx$ and $\vary$ satisfy
\begin{gather}\label{eq:groebner-2}
\varx^3-5\varx^2+4\varx=0
\;\;\&\;\;
\varx\vary-\vary-2\varx^2+2\varx=0
\;\;\&\;\;
\vary^2-\vary-5\varx^2+6\varx=0.
\end{gather}
We observe further that the polynomial 
$\varx^3-5\varx^2+4\varx$ evaluated at $A(\Gg^{\fu})A(\Gg^{\fud})$ vanishes 
if and only if it vanishes evaluated at 
the symmetric matrix $A^{\mathrm{T}}A$, cf. \fullref{lemma:weakly-regular}
and \eqref{eq:weakly-regular}. 
Thus, it suffices to solve the first equation in \eqref{eq:groebner-2} 
for $\varx=A^{\mathrm{T}}A$, and then to check 
whether the candidate solutions one obtains satisfy \fullref{problem:classification}.
The upshot is that the first equation in \eqref{eq:groebner-2} is then 
an equation in one symmetric matrix $A^{\mathrm{T}}A$ 
with entries from $\N$.

In order to check which matrices $A^{\mathrm{T}}A$ are annihilated by 
$\varx^3-5\varx^2+4\varx$ we first note that the complex roots of the polynomial 
$\varx^3-5\varx^2+4\varx=\varx(\varx-1)(\varx-4)$
are $0,1,4$, and that $\begin{psmallmatrix}0 & A^{\mathrm{T}}\\A & 0\end{psmallmatrix}$ 
is the adjacency matrix of the connected, bicolored subgraph of $\Gg$ 
obtained by erasing $\Pset$ (and all edges with a vertex in $\Pset$). 
Moreover, the eigenvalues of 
$\begin{psmallmatrix}0 & A^{\mathrm{T}}\\A & 0\end{psmallmatrix}$ 
are the square roots of the eigenvalues of $A^{\mathrm{T}}A$ 
(this linear algebra fact follows from e.g. \cite[Theorem 3]{Si-block-matrices})
and hence, have to be $0,\pm 1,\pm 2$. Therefore, 
$\begin{psmallmatrix}0 & A^{\mathrm{T}}\\A & 0\end{psmallmatrix}$ has 
to be the adjacency matrix of a finite or affine 
type $\ADE$ Dynkin diagram, by the classification in 
\cite{Sm1}, \cite[Section 3.1.1]{BH1}, 
with its Perron--Frobenius eigenvalue being $1$ or $2$, provided it is not zero.  
Furthermore, again by the classification in 
\cite{Sm1}, \cite[Section 3.1.1]{BH1}, the only 
connected graph such that $\begin{psmallmatrix}0 & A^{\mathrm{T}}\\A & 0\end{psmallmatrix}$ 
has Perron--Frobenius eigenvalue $1$ is of 
finite type $\typea{2}$, all those with Perron--Frobenius 
eigenvalue $2$ correspond to affine types. But for finite type $\typea{2}$ we 
get $A=A^{\mathrm{T}}=\begin{psmallmatrix}1\end{psmallmatrix}$ 
which by \eqref{eq:main-transposes} 
and strong connectivity 
implies $B=B^{\mathrm{T}}=C=C^{\mathrm{T}}=\begin{psmallmatrix}1\end{psmallmatrix}$, and thus, 
\eqref{eq:groebner-2} is not satisfied. 
Hence, we only need to consider affine type $\ADE$ Dynkin diagrams 
whose only eigenvalues are $0,\pm 1,-2$ in addition to $2$.

Using the list of eigenvalues from \cite[Section 3.1.1]{BH1}, 
we obtain the following possibilities for the associated 
bicolored graph.
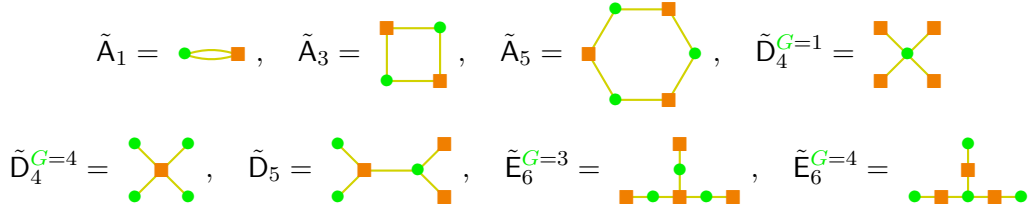
\begin{gather*}
\typeat{1}
=
\begin{tikzpicture}[anchorbase, xscale=.35, yscale=.5]
	\draw [thick, myyellow] (2,0) to [out=192, in=348] (0,0);
	\draw [thick, myyellow] (2,0) to [out=168, in=12] (0,0);
	\node at (0,0) {$\bulletg$};
	\node at (2,0) {$\bulleto$};
\end{tikzpicture}
,
\quad
\typeat{3}
=
\begin{tikzpicture}[anchorbase, xscale=.35, yscale=.5]
	\draw [thick, myyellow] (2,-.7) to (0,-.7) to (0,.7) to (2,.7) to (2,-.7);
	\node at (0,-.7) {$\bulletg$};
	\node at (2,-.7) {$\bulleto$};
	\node at (0,.7) {$\bulleto$};
	\node at (2,.7) {$\bulletg$};
\end{tikzpicture}
,
\quad
\typeat{5}
=
\begin{tikzpicture}[anchorbase, xscale=.35, yscale=.5]
	\draw [thick, myyellow] (3,0) to (2,-1.2) to (0,-1.2) to (-1,0) to (0,1.2) to (2,1.2) to (3,0);
	\node at (0,-1.2) {$\bulletg$};
	\node at (2,-1.2) {$\bulleto$};
	\node at (0,1.2) {$\bulletg$};
	\node at (2,1.2) {$\bulleto$};
	\node at (-1,0) {$\bulleto$};
	\node at (3,0) {$\bulletg$};
\end{tikzpicture}
,
\quad
\typedt{4}^{\Gset=1}
=
\begin{tikzpicture}[anchorbase, xscale=.35, yscale=.5]
	\draw [thick, myyellow] (-1,-.7) to (0,0);
	\draw [thick, myyellow] (1,-.7) to (0,0);
	\draw [thick, myyellow] (-1,.7) to (0,0);
	\draw [thick, myyellow] (1,.7) to (0,0);
	\node at (-1,-.7) {$\bulleto$};
	\node at (-1,.7) {$\bulleto$};
	\node at (0,0) {$\bulletg$};
	\node at (1,-.7) {$\bulleto$};
	\node at (1,.7) {$\bulleto$};
\end{tikzpicture}
\\
\typedt{4}^{\Gset=4}
=
\begin{tikzpicture}[anchorbase, xscale=.35, yscale=.5]
	\draw [thick, myyellow] (-1,-.7) to (0,0);
	\draw [thick, myyellow] (1,-.7) to (0,0);
	\draw [thick, myyellow] (-1,.7) to (0,0);
	\draw [thick, myyellow] (1,.7) to (0,0);
	\node at (-1,-.7) {$\bulletg$};
	\node at (-1,.7) {$\bulletg$};
	\node at (0,0) {$\bulleto$};
	\node at (1,-.7) {$\bulletg$};
	\node at (1,.7) {$\bulletg$};
\end{tikzpicture}
,
\quad
\typedt{5}
=
\begin{tikzpicture}[anchorbase, xscale=.35, yscale=.5]
	\draw [thick, myyellow] (-1,-.7) to (0,0);
	\draw [thick, myyellow] (3,-.7) to (2,0);
	\draw [thick, myyellow] (-1,.7) to (0,0);
	\draw [thick, myyellow] (3,.7) to (2,0);
	\draw [thick, myyellow] (0,0) to (2,0);
	\node at (-1,-.7) {$\bulletg$};
	\node at (-1,.7) {$\bulletg$};
	\node at (0,0) {$\bulleto$};
	\node at (2,0) {$\bulletg$};
	\node at (3,-.7) {$\bulleto$};
	\node at (3,.7) {$\bulleto$};
\end{tikzpicture}
,
\quad
\typeet{6}^{\Gset=3}
=
\begin{tikzpicture}[anchorbase, xscale=.35, yscale=.5]
	\draw [thick, myyellow] (-2,0) to (2,0);
	\draw [thick, myyellow] (0,0) to (0,1.4);
	\node at (-2,0) {$\bulleto$};
	\node at (-1,0) {$\bulletg$};
	\node at (0,0) {$\bulleto$};
	\node at (1,0) {$\bulletg$};
	\node at (2,0) {$\bulleto$};
	\node at (0,.7) {$\bulletg$};
	\node at (0,1.4) {$\bulleto$};
\end{tikzpicture}
,
\quad
\typeet{6}^{\Gset=4}
=
\begin{tikzpicture}[anchorbase, xscale=.35, yscale=.5]
	\draw [thick, myyellow] (-2,0) to (2,0);
	\draw [thick, myyellow] (0,0) to (0,1.4);
	\node at (-2,0) {$\bulletg$};
	\node at (-1,0) {$\bulleto$};
	\node at (0,0) {$\bulletg$};
	\node at (1,0) {$\bulleto$};
	\node at (2,0) {$\bulletg$};
	\node at (0,.7) {$\bulleto$};
	\node at (0,1.4) {$\bulletg$};
\end{tikzpicture}
\end{gather*}
($\typeat{2}$, which has eigenvalues $-1,-1,2$, 
is ruled out since it does not allow a bicoloring.) The same holds for the other colors, of course.

One can now write down all candidates solutions for the bicolored subgraphs:
\begin{center}
\renewcommand{\arraystretch}{1.25}
\begin{tabular}{c||c|c|c|c|c|c|c|c|c|c}
$\Gg_{\gc,\oc}$ & $\typeat{1}$ & $\typeat{1}$ & $\typeat{3}$ & $\typeat{5}$ & $\typeat{5}$ & $\typeat{5}$ & $\typeat{5}$ & $\typeat{5}$ & $\typedt{4}^{\Gset=1}$ & $\typedt{4}^{\Gset=4}$ \\ 
\hline 
$\Gg_{\oc,\pc}$ & $\typeat{1}$ & $\typedt{4}^{\Oset=1}$ & $\typeat{3}$ & $\typeat{5}$ & $\typeat{5}$ & $\typedt{5}$ & $\typedt{5}$ & $\typeet{6}^{\Oset=3}$ & $\typedt{4}^{\Oset=4}$ & $\typeat{1}$ \\ 
\hline 
$\Gg_{\pc,\gc}$ & $\typeat{1}$ & $\typedt{4}^{\Pset=4}$ & $\typeat{3}$ & $\typeat{5}$ & $\typedt{5}$ & $\typeat{5}$ & $\typedt{5}$ & $\typeet{6}^{\Pset=4}$ & $\typeat{1}$ & $\typedt{4}^{\Pset=1}$ \\
\hline\hline
\eqref{eq:groebner-2}? & \eqref{eq:special-solution} & $\graphD{3}^{\pc}$ & $\graphC{3}$ & no & no & no & no & $\graphA{3}^{\pc}$ & $\graphD{3}^{\oc}$ & $\graphD{3}^{\gc}$
\end{tabular} 
\end{center}
\begin{center}
\renewcommand{\arraystretch}{1.25}
\begin{tabular}{c||c|c|c|c|c|c|c|c|c}
$\Gg_{\gc,\oc}$ & $\typedt{5}$ & $\typedt{5}$ & $\typedt{5}$ & $\typedt{5}$ & $\typedt{5}$ & $\typeet{6}^{\Gset=3}$ & $\typeet{6}^{\Gset=3}$ & $\typeet{6}^{\Gset=4}$ & $\typeet{6}^{\Gset=4}$ \\ 
\hline 
$\Gg_{\oc,\pc}$ & $\typeat{5}$ & $\typeat{5}$ & $\typedt{5}$ & $\typedt{5}$ & $\typeet{6}^{\Oset=3}$ & $\typeet{6}^{\Oset=4}$ & $\typeet{6}^{\Oset=4}$ & $\typeat{5}$ & $\typedt{5}$ \\ 
\hline 
$\Gg_{\pc,\gc}$ & $\typeat{5}$ & $\typedt{5}$ & $\typeat{5}$ & $\typedt{5}$ & $\typeet{6}^{\Pset=4}$ & $\typeat{5}$ & $\typedt{5}$ & $\typeet{6}^{\Pset=3}$ & $\typeet{6}^{\Pset=3}$ \\
\hline\hline
\eqref{eq:groebner-2}? & no & no & no & no & no & $\graphA{3}^{\oc}$ & no & $\graphA{3}^{\gc}$ & no
\end{tabular} 
\end{center}
Here we have indicated whether all equations in
\eqref{eq:groebner-2} are satisfied or not.
The remaining possibilities 
give solutions to \fullref{problem:classification}.
\end{proof}

The solution \eqref{eq:special-solution} is not on the list of 
generalized $\ADE$ Dynkin diagrams, and we do not know whether this is an exception 
for $e=3$ or whether there exist more solutions which are not 
generalized $\ADE$ Dynkin diagrams for $e>3$.

\begin{example}\label{example:classification}
For $e=0,1,2,3$ the list of transitive $\aformvN$-representations 
given in \eqref{eq:the-transitives} is almost complete. Adding a 
representation $\M_{\graphC{3}}$ to this list, for any color by 
\fullref{lemma:find-all-transitives-2}, and a representation 
for the special solution \eqref{eq:special-solution} completes the list, where 
one can check by hand that these are $\aformvN$-representations.
\end{example}

\subsection{Categorified story}\label{subsec:cat-story}

Recall that $\GGc{\subcatquo[e]}\cong\subquo[e]$, see \fullref{proposition:cat-the-quoalgebra} 
(excluding $e=0$), and 
assume that we have a transitive $2$-representation $\cM$ of $\subcatquo[e]$. 
Then $\GGc{\cM}$ is a transitive $\aformvN$-representation of $\subquo[e]$. 
So, by the discussion in \fullref{subsec:decat-story}, the 
classification of simple transitive $2$-representations of $\subcatquo[e]$ 
boils down to \fullref{problem:classification} together with the construction of the 
corresponding $2$-representations (i.e. their categorification). We are going to explain this construction 
for types $\typeA$ and $\typeD$.

\subsubsection{Satake and \texorpdfstring{$2$}{2}-representations}\label{subsec:cat-story-0}

Let us first sketch how \fullref{theorem:q-satake} gives rise to a correspondence between 
the simple transitive $2$-representations of $\slqmodgop$ 
and those of $\subcatquo[e]$. We will discuss the details in the sections below.

Recall that there is a bijection 
between the equivalence classes of 
simple transitive $2$-re\-presentations of $\slqmodgop$ 
and the Morita equivalence classes of 
simple algebra $1$-morphisms in $\slqmodgop$, cf. \cite[Theorem 9]{MMMT1}. 

For $\subcatquo[e]$ the 
situation is more complicated, because it is 
additive, but not abelian. However, 
it is still true that, if $\morstuff{A}$ is an indecomposable 
algebra $1$-morphism in $\subcatquo[e]$, 
then $\modcat{\subcatquo[e]}(\morstuff{A})$ 
is a transitive $2$-representation of $\subcatquo[e]$. By 
taking its simple quotient, as 
described in \fullref{remark:graded-finitary-3}, we get a simple transitive 
$2$-representation associated to $\morstuff{A}$.

As we will see, any algebra $1$-morphism in $\morstuff{A}$ in $\slqmod$ gives rise to 
an algebra $1$-morphism $\morstuff{A}^{\tduc}$ in $\slqmodgop[e](\tduc,\tduc)$. (Without loss of generality, 
we will concentrate on the case $\tduc=\gc$ below.)

The Satake $2$-functor from \fullref{lemma:Satake} transports $\morstuff{A}^{\gc}$
to an algebra $1$-morphism in $\Subcatquo[e](\gc,\gc)$ (where we keep the same notation). Biinduction, 
which means gluing white outer regions to the diagrams which 
define the multiplication and unit $2$-morphisms (see also \fullref{example:clasps}), then gives an algebra $1$-morphism 
$\morstuff{B}^{\gc}=\morstuff{B}(\morstuff{A}^{\gc})\in\subcatquo[e]=\Subcatquo[e](\wc,\wc)$. 
As we will show, the algebra $1$-morphism also has to be translated 
(which would correspond to shifting the grading if the algebra $1$-morphism were given by a graded bimodule), 
so that the final degree of the unit and multiplication $2$-morphisms becomes zero. 

The fact that $\morstuff{B}^{\gc}$ is associative 
and unital, follows almost immediately from the associativity and unitality 
of $\morstuff{A}^{\gc}$. (For a detailed proof we refer to \cite[Section 7.3]{MMMT1}.)
By construction, $\morstuff{B}^{\gc}$ is indecomposable if $\morstuff{A}$ is simple. 

Thus, given a simple transitive $2$-representation $\cM$ 
of $\slqmodgop$, let $\morstuff{A}$ 
be the corresponding simple algebra $1$-morphism in 
$\slqmodgop$ and take $\cM_{\gc}$ to be the simple quotient 
(as recalled in \fullref{remark:graded-finitary-3}) of 
$\modcat{\subcatquo[e]}(\morstuff{B}^{\gc})$.

\begin{remark}\label{remark:nothing-new}
Note that all simple algebra $1$-morphisms in 
$\slqmodgop$ arise via coloring from simple algebra 
$1$-morphisms in $\slqmod$. So our first task below is to 
recall the latter, which were already known, see e.g. \cite{Sch}. 
However, we present a self-contained construction in this paper. 
As a service to the reader, we also recall the proof of the 
known classification of their simple module $1$-morphisms 
in $\slqmod$.   
\end{remark}

\subsubsection{Type \texorpdfstring{$\typeA$}{A} \texorpdfstring{$2$}{2}-representations}\label{subsec:cat-story-b}

The object $\obstuff{A}^{\graphA{e}}=\Ll_{0,0}\cong\C$ is clearly 
a simple algebra object in $\slqmod$, because it is the identity 
object. Thus, coloring gives us a simple algebra 
$1$-morphism $\morstuff{A}^{\graphA{e}^{\gc}}$
in $\slqmodgop$. After applying the Satake $2$-functor, 
we get $\elfunctor[e](\morstuff{A}^{\graphA{e}^{\gc}})=\gc$, 
which is the identity $1$-morphism in $\Subcatquo[e](\gc,\gc)$. 

Therefore, we have $\morstuff{B}^{\graphA{e}^{\gc}}=\wc\bc\gc\bc\wc \{-3\}\in\subcatquo[e]$, which is an indecomposable 
$1$-morphism in $\subcatquo[e]$. Recall that $\wc\bc\gc\bc\wc\cong 
\wc\yc\gc\yc\wc$, by \fullref{lemma:clasps-well-defined}. We have translated 
$\wc\bc\gc\bc\wc$ by $-3$, so that the final degree of the unit and multiplication $2$-morphisms 
below becomes zero. 
Note further that $\morstuff{B}^{\graphA{e}}$ is a Frobenius $1$-morphism, and its
(co)unit and (co)multiplication $2$-morphisms (with their respective unshifted degrees) are given by 
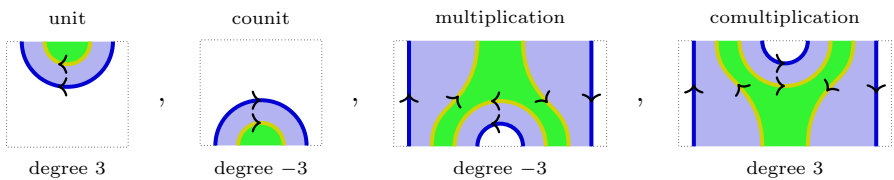
\[
\xy
(0,0)*{
\begin{tikzpicture}[anchorbase, scale=.4, tinynodes]
	\draw[very thin, densely dotted, fill=white] (-.5,0) to [out=270, in=180] (1,-1.5) to [out=0, in=270] (2.5,0) to (3,0) to (3,-3.5) to (-1,-3.5) 
	to (-1,0) to (-.5,0);
	\fill[mygreen, opacity=0.8] (.25,0) to [out=270, in=180] (1,-.75) to [out=0, in=270] (1.75,0) to (.25,0);                   
	\fill[myblue, opacity=0.3] (-.5,0) to [out=270, in=180] (1,-1.5) to [out=0, in=270] (2.5,0) to (1.75,0) to [out=270, in=0] (1,-.75) to [out=180, in=270] (.25,0) to (-.5,0);
	\draw[ystrand, rdirected=.5] (.25,0) to [out=270, in=180] (1,-.75) to [out=0, in=270] (1.75,0);
	\draw[bstrand, rdirected=.5] (-.5,0) to [out=270, in=180] (1,-1.5) to [out=0, in=270] (2.5,0);
\end{tikzpicture}
};
(0,10)*{\text{{\tiny unit}}};
(0,-10)*{\text{{\tiny degree $3$}}};
\endxy
\quad,\quad
\xy
(0,0)*{
\begin{tikzpicture}[anchorbase, scale=.4, tinynodes]
	\draw[very thin, densely dotted, fill=white] (-.5,0) to [out=90, in=180] (1,1.5) to [out=0, in=90] (2.5,0) to (3,0) to (3,3.5) to (-1,3.5) 
	to (-1,0) to (-.5,0);
	\fill[mygreen, opacity=0.8] (.25,0) to [out=90, in=180] (1,.75) to [out=0, in=90] (1.75,0) to (.25,0);                   
	\fill[myblue, opacity=0.3] (-.5,0) to [out=90, in=180] (1,1.5) to [out=0, in=90] (2.5,0) to (1.75,0) to [out=90, in=0] (1,.75) to [out=180, in=90] (.25,0) to (-.5,0);
	\draw[ystrand, directed=.5] (.25,0) to [out=90, in=180] (1,.75) to [out=0, in=90] (1.75,0);
	\draw[bstrand, directed=.5] (-.5,0) to [out=90, in=180] (1,1.5) to [out=0, in=90] (2.5,0);
\end{tikzpicture}
};
(0,10)*{\text{{\tiny counit}}};
(0,-10)*{\text{{\tiny degree $-3$}}};
\endxy
\quad,\quad
\xy
(0,0)*{
\begin{tikzpicture}[anchorbase, scale=.4, tinynodes]
	\draw[very thin, densely dotted, fill=white] (4,0) to (4,3.5) to (4.5,3.5) to (4.5,0) to (4,0);
	\draw[very thin, densely dotted, fill=white] (-2,0) to (-2,3.5) to (-2.5,3.5) to (-2.5,0) to (-2,0);
	\draw[very thin, densely dotted, fill=white] (.25,0) to [out=90, in=180] (1,.75) to [out=0, in=90] (1.75,0) to (.25,0);
	\fill[myblue, opacity=0.3] (-.5,0) to [out=90, in=180] (1,1.5) to [out=0, in=90] (2.5,0) to [out=180, in=0] (1.75,0) to [out=90, in=0] (1,.75) to [out=180, in=90] (.25,0) to [out=180, in=0] (-.5,0);
	\fill[mygreen, opacity=0.8] (2.5,0) to  (3.25,0) to [out=90, in=-45] (2.75,1.25) to [out=135, in=270] (1.75,3.5) 
to (.25,3.5) to [out=270, in=45] (-.75,1.25) to [out=225, in=90] (-1.25,0) to  (-.5,0) to [out=90, in=180] (1,1.5) to [out=0, in=90] (2.5,0);
	\fill[myblue, opacity=0.3]  (3.25,0) to [out=90, in=-45] (2.75,1.25) to [out=135, in=270] (1.75,3.5) to (4,3.5) to (4,0) to (3.25,0);
	\fill[myblue, opacity=0.3]  (-1.25,0) to [out=90, in=225] (-.75,1.25) to [out=45, in=270] (.25,3.5) to (-2,3.5) to (-2,0) to (-1.25,0);
    \draw[bstrand, rdirected=.5] (4,0) to (4,3.5);
    \draw[bstrand, directed=.5] (-2,0) to (-2,3.5);
	\draw[bstrand, rdirected=.5] (.25,0) to [out=90, in=180] (1,.75) to [out=0, in=90] (1.75,0);
	\draw[ystrand, rdirected=.5] (-.5,0) to [out=90, in=180] (1,1.5) to [out=0, in=90] (2.5,0);
    \draw[ystrand, directed=.5] (-1.25,0) to [out=90, in=225] (-.75,1.25) to [out=45, in=270] (.25,3.5);
    \draw[ystrand, rdirected=.5] (3.25,0) to [out=90, in=-45] (2.75,1.25) to [out=135, in=270] (1.75,3.5);
\end{tikzpicture}
};
(0,10)*{\text{{\tiny multiplication}}};
(0,-10)*{\text{{\tiny degree $-3$}}};
\endxy
\quad,\quad
\xy
(0,0)*{
\begin{tikzpicture}[anchorbase, scale=.4, tinynodes]
	\draw[very thin, densely dotted, fill=white] (4,0) to (4,-3.5) to (4.5,-3.5) to (4.5,0) to (4,0);
	\draw[very thin, densely dotted, fill=white] (-2,0) to (-2,-3.5) to (-2.5,-3.5) to (-2.5,0) to (-2,0);
	\draw[very thin, densely dotted, fill=white] (.25,0) to [out=270, in=180] (1,-.75) to [out=0, in=270] (1.75,0) to (.25,0);
	\fill[myblue, opacity=0.3] (-.5,0) to [out=270, in=180] (1,-1.5) to [out=0, in=270] (2.5,0) to [out=180, in=0] (1.75,0) to [out=270, in=0] (1,-.75) to [out=180, in=270] (.25,0) to [out=180, in=0] (-.5,0);
	\fill[mygreen, opacity=0.8] (2.5,0) to (3.25,0) to [out=270, in=45] (2.75,-1.25) to [out=-135, in=90] (1.75,-3.5) 
to (.25,-3.5) to [out=90, in=-45] (-.75,-1.25) to [out=-225, in=270] (-1.25,0) to  (-.5,0) to [out=270, in=180] (1,-1.5) to [out=0, in=270] (2.5,0);
	\fill[myblue, opacity=0.3]  (3.25,0) to [out=270, in=45] (2.75,-1.25) to [out=-135, in=90] (1.75,-3.5) to (4,-3.5) to (4,0) to (3.25,0);
	\fill[myblue, opacity=0.3]  (-1.25,0) to [out=270, in=-225] (-.75,-1.25) to [out=-45, in=90] (.25,-3.5) to (-2,-3.5) to (-2,0) to (-1.25,0);
    \draw[bstrand, directed=.5] (4,0) to (4,-3.5);
    \draw[bstrand, rdirected=.5] (-2,0) to (-2,-3.5);
	\draw[bstrand, directed=.5] (.25,0) to [out=270, in=180] (1,-.75) to [out=0, in=270] (1.75,0);
	\draw[ystrand, directed=.5] (-.5,0) to [out=270, in=180] (1,-1.5) to [out=0, in=270] (2.5,0);
    \draw[ystrand, rdirected=.5] (-1.25,0) to [out=270, in=-225] (-.75,-1.25) to [out=-45, in=90] (.25,-3.5);
    \draw[ystrand, directed=.5] (3.25,0) to [out=270, in=45] (2.75,-1.25) to [out=-135, in=90] (1.75,-3.5);
\end{tikzpicture}
};
(0,10)*{\text{{\tiny comultiplication}}};
(0,-10)*{\text{{\tiny degree $3$}}};
\endxy
\]

By construction,
the corresponding simple transitive 
$2$-representation $\cM_{\graphA{e}^{\gc}}$
of $\subcatquo[e]$ is equivalent to a cell $2$-representation 
and decategorifies to $\M_{\graphA{e}^{\gc}}$ from \fullref{definition:n-modules}.
Similarly for the secondary colors $\oc$ and $\pc$.

\subsubsection{Type \texorpdfstring{$\typeD$}{D} \texorpdfstring{$2$}{2}-representations}\label{subsec:cat-story-c}

Let $e\equiv 0 \bmod 3$. 
In this case, the decomposition of the algebra $1$-morphism into simple $1$-morphisms in $\slqmod$ is given by 
\[
\obstuff{A}^{\graphD{e}}\cong \Ll_{0,0}\oplus\Ll_{e,0}\oplus\Ll_{0,e}.
\]
In order to define the multiplication and unit $2$-morphisms of 
$\obstuff{A}^{\graphD{e}}$ in $\slqmod$,
recall from the representation theory of $\slt$ that
\[
(\Ll_{e,0})^{\ast}
\cong
(\mathrm{Sym}^{e}(\C^{3}))^{\ast}
\cong
\mathrm{Sym}^{e}((\C^{3})^{\ast})
\cong
\Ll_{0,e},
\]
where ${}^{\ast}$ means the dual module and $\mathrm{Sym}^{e}$ the $e^{\mathrm{th}}$ symmetric power. 
Note that, by using e.g. \cite[Proposition 2.11]{BZ1}, we have similar isomorphisms in the quantum 
case as well.
Thus, in order to delineate the monoidal subcategory
generated by the quantum
symmetric powers, we can use the symmetric web categories 
described in \cite{RT1}, \cite{TVW1}, after adding the duals as in \cite{QS1}. 
These symmetric web categories are built from certain labeled,
trivalent graphs, and we need the 
monoidal subcategory of it generated by the objects $e,e^{\ast}$
and the morphisms
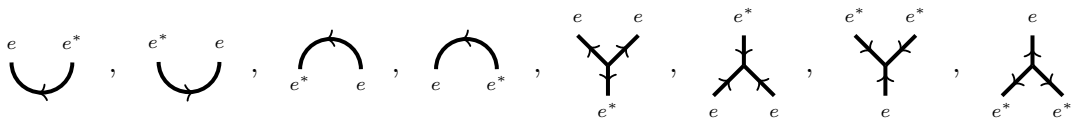
\[
\begin{tikzpicture}[anchorbase, scale=.4, tinynodes]
	\draw[dstrand, Ymarked=.55] (0,0) node [above] {$e$} to [out=270, in=180] (1,-1) to [out=0, in=270] (2,0) node [above] {$e^{\ast}$};
\end{tikzpicture}
\;\;,\;\;
\begin{tikzpicture}[anchorbase, scale=.4, tinynodes]
	\draw[dstrand, Xmarked=.55] (0,0) node [above] {$e^{\ast}$} to [out=270, in=180] (1,-1) to [out=0, in=270] (2,0) node [above] {$e$};
\end{tikzpicture}
\;\;,\;\;
\begin{tikzpicture}[anchorbase, scale=.4, tinynodes]
	\draw[dstrand, Xmarked=.55] (2,0) node [below] {$e$} to [out=90, in=0] (1,1) to [out=180, in=90] (0,0) node [below] {$e^{\ast}$};
\end{tikzpicture}
\;\;,\;\;
\begin{tikzpicture}[anchorbase, scale=.4, tinynodes]
	\draw[dstrand, Xmarked=.55] (2,0) node [below] {$e^{\ast}$} to [out=90, in=0] (1,1) to [out=180, in=90] (0,0) node [below] {$e$};
\end{tikzpicture}
\;\;,\;\;
\begin{tikzpicture}[anchorbase, scale=.4, tinynodes]
	\draw[dstrand, Ymarked=.55] (-1,1) node [above] {$e$} to (0,0);
	\draw[dstrand, Ymarked=.55] (1,1) node [above] {$e$} to (0,0);
	\draw[dstrand, Xmarked=.55] (0,0) to (0,-1) node [below] {$e^{\ast}$};
\end{tikzpicture}
\;\;,\;\;
\begin{tikzpicture}[anchorbase, scale=.4, tinynodes]
	\draw[dstrand, Xmarked=.55] (-1,-1) node [below] {$e$} to (0,0);
	\draw[dstrand, Xmarked=.55] (1,-1) node [below] {$e$} to (0,0);
	\draw[dstrand, Ymarked=.55] (0,0) to (0,1) node [above] {$e^{\ast}$};
\end{tikzpicture}
\;\;,\;\;
\begin{tikzpicture}[anchorbase, scale=.4, tinynodes]
	\draw[dstrand, Xmarked=.55] (-1,1) node [above] {$e^{\ast}$} to (0,0);
	\draw[dstrand, Xmarked=.55] (1,1) node [above] {$e^{\ast}$} to (0,0);
	\draw[dstrand, Ymarked=.55] (0,0) to (0,-1) node [below] {$e$};
\end{tikzpicture}
\;\;,\;\;
\begin{tikzpicture}[anchorbase, scale=.4, tinynodes]
	\draw[dstrand, Ymarked=.55] (-1,-1) node [below] {$e^{\ast}$} to (0,0);
	\draw[dstrand, Ymarked=.55] (1,-1) node [below] {$e^{\ast}$} to (0,0);
	\draw[dstrand, Xmarked=.55] (0,0) to (0,1) node [above] {$e$};
\end{tikzpicture}
\]
(our reading conventions are still from bottom to top), 
subject to some relations which are all stated e.g. in \cite[Section 2.1]{RW1}. 
(We stress that some of the 
cited papers work with $\mathrm{gl}$ instead of $\mathrm{sl}$, and 
we also semisimplify according to the cut-off as in \eqref{eq:weight-picture}. 
In diagrammatic terms this amounts to a slightly different web calculus where 
e.g. an edge of label $2e$ in \cite[Section 2.1]{RW1} is identified 
in our notation with an 
edge of label $e^{\ast}$ with the orientation reversed.)
These are basically thick, but uncolored versions 
of the webs which we met in \fullref{subsec:quotient-category}, and the 
object $e$ corresponds to the $e^{\mathrm{th}}$ quantum symmetric power 
of $\Ll_{1,0}$, which is $\Ll_{e,0}$, and $e^{\ast}$ to its dual, 
which is $\Ll_{0,e}$.

Hence, we can use the diagrammatic calculus of symmetric webs to 
describe the intertwiners in $\sltcat[\qpar]$ that we need. So far, we have 
assumed that $\qpar$ is a generic parameter. By putting it equal to a primitive, complex 
$2(e+3)^{\mathrm{th}}$ root of unity $\qqpar$, we get a projection onto $\sltcat[e]$, and we can use the 
specialized relations of the symmetric web calculus. 

To be absolutely clear, we do not claim that the symmetric web calculi are equivalent 
to the monoidal subcategories in question. All we 
need is that the functor from the web calculus to $\sltcat[e]$ is full, which is true.

We use the following shorthand:
\[
\Ll_{0,0}\leftrightsquigarrow\emptyset,
\quad
\Ll_{e,0}\leftrightsquigarrow e,
\quad
\Ll_{0,e}\leftrightsquigarrow e^{\ast},
\]
and $\Ll_{e,0}\otimes\Ll_{0,e}\leftrightsquigarrow ee^{\ast}$ etc.,
where we omit the $\otimes$ symbol.

\begin{proposition}\label{proposition:typeD-object}
$\obstuff{A}^{\graphD{e}}$ has the structure of a Frobenius object in $\slqmod$ with 
unit $\iota\colon\emptyset\to\obstuff{A}^{\graphD{e}}$, $\iota(1)=1$, 
counit $\epsilon\colon\obstuff{A}^{\graphD{e}}\to\emptyset$, $\epsilon(1)=1$ and multiplication 
$m\colon\obstuff{A}^{\graphD{e}}\otimes\obstuff{A}^{\graphD{e}}\to\obstuff{A}^{\graphD{e}}$ given by
\begin{gather}\label{eq:typeD-mult-table}
\begin{tikzpicture}[baseline=(current bounding box.center),yscale=0.6]
  \matrix (m) [matrix of math nodes, row sep={1.0cm,between origins}, column
  sep={1.0cm,between origins}, text height=1.5ex, text depth=0.25ex, ampersand replacement=\&] {
 \phantom{a} \&
\emptyset\emptyset \&  \emptyset e \& \emptyset e^{\ast} \& 
e\emptyset \& ee \& ee^{\ast} \& 
e^{\ast}\emptyset \& e^{\ast}e \& e^{\ast}e^{\ast}
\\
\emptyset
\&
\emptyset \&  {\color{mygray}0} \& {\color{mygray}0} \& 
{\color{mygray}0} \& {\color{mygray}0} \& \begin{tikzpicture}[anchorbase, scale=.4, tinynodes]
	\draw[dstrand, Ymarked=.55] (1,0) to [out=90, in=0] (.5,.5) to [out=180, in=90] (0,0);
	\node at (0,-.4) {$e$};
	\node at (1.1,-.4) {$e^{\hspace*{-.04cm}\ast}$};
	\node at (0,1) {$\phantom{e}$};
\end{tikzpicture} \& 
{\color{mygray}0} \& \begin{tikzpicture}[anchorbase, scale=.4, tinynodes]
	\draw[dstrand, Ymarked=.55] (0,0) to [out=90, in=180] (.5,.5) to [out=0, in=90] (1,0);
	\node at (.1,-.4) {$e^{\hspace*{-.04cm}\ast}$};
	\node at (1,-.4) {$e$};
	\node at (0,1) {$\phantom{e}$};
\end{tikzpicture} \& {\color{mygray}0}
\\
e
\&
{\color{mygray}0} \& \begin{tikzpicture}[anchorbase, scale=.4, tinynodes]
	\draw[dstrand, Xmarked=.7] (.25,0) to [out=90, in=270] (-.25,1);
	\node at (.25,-.4) {$e$};
	\node at (-.25,1.15) {$e$};
\end{tikzpicture} \& {\color{mygray}0} \& 
\begin{tikzpicture}[anchorbase, scale=.4, tinynodes]
	\draw[dstrand, Xmarked=.7] (-.25,0) to [out=90, in=270] (.25,1);
	\node at (-.25,-.4) {$e$};
	\node at (.25,1.15) {$e$};
\end{tikzpicture} \& {\color{mygray}0} \& {\color{mygray}0} \& 
{\color{mygray}0} \& {\color{mygray}0} \& \begin{tikzpicture}[anchorbase, scale=.4, tinynodes]
	\draw[dstrand, Ymarked=.55] (-.5,-.5) to (0,0);
	\draw[dstrand, Ymarked=.55] (.5,-.5) to (0,0);
	\draw[dstrand, Xmarked=.7] (0,0) to (0,.5);
	\node at (-.35,-.9) {$e^{\hspace*{-.04cm}\ast}$};
	\node at (.6,-.9) {$e^{\hspace*{-.04cm}\ast}$};
	\node at (0,.65) {$e$};
\end{tikzpicture}
\\
e^{\ast}
\&
{\color{mygray}0} \& {\color{mygray}0} \& \begin{tikzpicture}[anchorbase, scale=.4, tinynodes]
	\draw[dstrand, Ymarked=.55] (.25,-1) to [out=90, in=270] (-.25,0);
	\node at (.35,-1.4) {$e^{\hspace*{-.04cm}\ast}$};
	\node at (-.15,.1) {$e^{\hspace*{-.04cm}\ast}$};
\end{tikzpicture} \& 
{\color{mygray}0} \& \begin{tikzpicture}[anchorbase, scale=.4, tinynodes]
	\draw[dstrand, Xmarked=.7] (-.5,-.5) to (0,0);
	\draw[dstrand, Xmarked=.7] (.5,-.5) to (0,0);
	\draw[dstrand, Ymarked=.55] (0,0) to (0,.5);
	\node at (-.5,-.9) {$e$};
	\node at (.5,-.9) {$e$};
	\node at (.1,.6) {$e^{\hspace*{-.04cm}\ast}$};
\end{tikzpicture} \& {\color{mygray}0} \& 
\begin{tikzpicture}[anchorbase, scale=.4, tinynodes]
	\draw[dstrand, Ymarked=.55] (-.25,-1) to [out=90, in=270] (.25,0);
	\node at (-.15,-1.4) {$e^{\hspace*{-.04cm}\ast}$};
	\node at (.35,.1) {$e^{\hspace*{-.04cm}\ast}$};
\end{tikzpicture} \& {\color{mygray}0} \& {\color{mygray}0}
\\
};
\draw[densely dashed] ($(m-1-1.south west)+ (-.2,-.3)$) to ($(m-1-3.south east)+ (7.25,-.3)$);
\draw[densely dashed] ($(m-1-1.south west)+ (-.2,-2.0)$) to ($(m-1-3.south east)+ (7.25,-2.0)$);
\draw[densely dashed] ($(m-1-1.south west)+ (-.2,-3.8)$) to ($(m-1-3.south east)+ (7.25,-3.8)$);
\draw[densely dashed] ($(m-1-1.east)+ (.25,.35)$) to ($(m-3-1.east)+ (.25,-2.4)$);
\draw[densely dashed] ($(m-1-1.east)+ (1.25,.35)$) to ($(m-3-1.east)+ (1.25,-2.4)$);
\draw[densely dashed] ($(m-1-1.east)+ (2.25,.35)$) to ($(m-3-1.east)+ (2.25,-2.4)$);
\draw[densely dashed] ($(m-1-1.east)+ (3.25,.35)$) to ($(m-3-1.east)+ (3.25,-2.4)$);
\draw[densely dashed] ($(m-1-1.east)+ (4.25,.35)$) to ($(m-3-1.east)+ (4.25,-2.4)$);
\draw[densely dashed] ($(m-1-1.east)+ (5.25,.35)$) to ($(m-3-1.east)+ (5.25,-2.4)$);
\draw[densely dashed] ($(m-1-1.east)+ (6.25,.35)$) to ($(m-3-1.east)+ (6.25,-2.4)$);
\draw[densely dashed] ($(m-1-1.east)+ (7.25,.35)$) to ($(m-3-1.east)+ (7.25,-2.4)$);
\draw[densely dashed] ($(m-1-1.east)+ (8.25,.35)$) to ($(m-3-1.east)+ (8.25,-2.4)$);
\end{tikzpicture}
\end{gather}
The comultiplication $\Delta\colon\obstuff{A}^{\graphD{e}}\to\obstuff{A}^{\graphD{e}}\otimes\obstuff{A}^{\graphD{e}}$ 
is given by transposing \eqref{eq:typeD-mult-table} and turning the symmetric webs upside down.
\end{proposition}

We omit the edge labels (which are always $e$ or $e^{\ast}$) from now on, 
and also sometimes silently identify $e\emptyset=e$ etc.

\begin{proof}
First, we observe that the relations in the symmetric web calculus are invariant under 
horizontal and vertical reflections, which 
reduces the number of cases we need to verify. For example, 
checking the unitality of $\obstuff{A}^{\graphD{e}}$ boils down 
to checking the commutativity of
\[
\xymatrix@C+=1.0cm@L+=6pt{
e 
\ar[r]^{
\begin{tikzpicture}[anchorbase, scale=.4, tinynodes]
	\draw[dstrand, Xmarked=.7] (.25,0) to [out=90, in=270] (-.25,1);
\end{tikzpicture}}
\ar[dr]_{
\begin{tikzpicture}[anchorbase, scale=.4, tinynodes]
	\draw[dstrand, Xmarked=.7] (0,0) to (0,1);
\end{tikzpicture}
}
& 
e\emptyset
\ar[d]^{
\begin{tikzpicture}[anchorbase, scale=.4, tinynodes]
	\draw[dstrand, Xmarked=.7] (-.25,0) to [out=90, in=270] (.25,1);
\end{tikzpicture}}
\\
& 
e
}
\]
which follows directly from the symmetric web calculus.

Next, we show that $m$ and $\Delta$ are (co)associative. 
Up to symmetries and trivial compositions, we need 
to check that
\[
\xymatrix@C+=1.0cm@L+=6pt{
eee 
\ar[r]^{
\begin{tikzpicture}[anchorbase, scale=.4, tinynodes]
	\draw[dstrand, Xmarked=.7] (-.5,-.5) to (0,0);
	\draw[dstrand, Xmarked=.7] (.5,-.5) to (0,0);
	\draw[dstrand, Ymarked=.55] (0,0) to (0,.5);
	\draw[dstrand, Xmarked=.7] (1,-.5) to (1,.5);
\end{tikzpicture}}
\ar[d]_{
\begin{tikzpicture}[anchorbase, scale=.4, tinynodes]
	\draw[dstrand, Xmarked=.7] (-.5,-.5) to (0,0);
	\draw[dstrand, Xmarked=.7] (.5,-.5) to (0,0);
	\draw[dstrand, Ymarked=.55] (0,0) to (0,.5);
	\draw[dstrand, Xmarked=.7] (-1,-.5) to (-1,.5);
\end{tikzpicture}}
& 
e^{\ast}e
\ar[d]^{
\begin{tikzpicture}[anchorbase, scale=.4, tinynodes]
	\draw[dstrand, Ymarked=.55] (0,0) to [out=90, in=180] (.5,.5) to [out=0, in=90] (1,0);
\end{tikzpicture}}
\\
ee^{\ast} 
\ar[r]_{
\begin{tikzpicture}[anchorbase, scale=.4, tinynodes]
	\draw[dstrand, Ymarked=.55] (1,0) to [out=90, in=0] (.5,.5) to [out=180, in=90] (0,0);
\end{tikzpicture}}
& 
\emptyset
}
\quad\quad
\xymatrix@C+=1.0cm@L+=6pt{
eee^{\ast} 
\ar[r]^{
\begin{tikzpicture}[anchorbase, scale=.4, tinynodes]
	\draw[dstrand, Xmarked=.7] (-.5,-.5) to (0,0);
	\draw[dstrand, Xmarked=.7] (.5,-.5) to (0,0);
	\draw[dstrand, Ymarked=.55] (0,0) to (0,.5);
	\draw[dstrand, Ymarked=.55] (1,-.5) to (1,.5);
\end{tikzpicture}}
\ar[d]_{
\begin{tikzpicture}[anchorbase, scale=.4, tinynodes]
	\draw[dstrand, Ymarked=.55] (1,0) to [out=90, in=0] (.5,.5) to [out=180, in=90] (0,0);
	\draw[dstrand, Xmarked=.7] (-.5,0) to [out=90, in=270] (-.5,1);
\end{tikzpicture}}
& 
e^{\ast}e^{\ast}
\ar[d]^{
\begin{tikzpicture}[anchorbase, scale=.4, tinynodes]
	\draw[dstrand, Ymarked=.55] (-.5,-.5) to (0,0);
	\draw[dstrand, Ymarked=.55] (.5,-.5) to (0,0);
	\draw[dstrand, Xmarked=.7] (0,0) to (0,.5);
\end{tikzpicture}}
\\
e\emptyset
\ar[r]_{
\begin{tikzpicture}[anchorbase, scale=.4, tinynodes]
	\draw[dstrand, Xmarked=.7] (0,0) to [out=90, in=270] (0,1);
\end{tikzpicture}}
& 
e
}
\quad\quad
\xymatrix@C+=1.0cm@L+=6pt{
ee^{\ast}e 
\ar[r]^{
\begin{tikzpicture}[anchorbase, scale=.4, tinynodes]
	\draw[dstrand, Ymarked=.55] (1,0) to [out=90, in=0] (.5,.5) to [out=180, in=90] (0,0);
	\draw[dstrand, Xmarked=.7] (1.5,0) to [out=90, in=270] (1.5,1);
\end{tikzpicture}}
\ar[d]_{
\begin{tikzpicture}[anchorbase, scale=.4, tinynodes]
	\draw[dstrand, Ymarked=.55] (0,0) to [out=90, in=180] (.5,.5) to [out=0, in=90] (1,0);
	\draw[dstrand, Xmarked=.7] (-.5,0) to [out=90, in=270] (-.5,1);
\end{tikzpicture}}
& 
\emptyset e
\ar[d]^{
\begin{tikzpicture}[anchorbase, scale=.4, tinynodes]
	\draw[dstrand, Xmarked=.7] (0,0) to [out=90, in=270] (0,1);
\end{tikzpicture}}
\\
e\emptyset
\ar[r]_{
\begin{tikzpicture}[anchorbase, scale=.4, tinynodes]
	\draw[dstrand, Xmarked=.7] (0,0) to [out=90, in=270] (0,1);
\end{tikzpicture}}
& 
e
}
\]
commute. The leftmost case is just an isotopy in the symmetric web calculus.
The other two cases follow by observing that we have
\begin{gather}\label{eq:sym-web-rels}
\begin{tikzpicture}[anchorbase, scale=.4, tinynodes]
	\draw[dstrand, Xmarked=.55] (2,0) to [out=90, in=0] (1,1) to [out=180, in=90] (0,0);
	\draw[dstrand, Xmarked=.55] (0,0) to [out=270, in=180] (1,-1) to [out=0, in=270] (2,0);
\end{tikzpicture}
=
\begin{tikzpicture}[anchorbase, scale=.4, tinynodes]
	\draw[dstrand, Ymarked=.55] (2,0) to [out=90, in=0] (1,1) to [out=180, in=90] (0,0);
	\draw[dstrand, Ymarked=.55] (0,0) to [out=270, in=180] (1,-1) to [out=0, in=270] (2,0);
\end{tikzpicture}
=
1
\quad
\text{and}
\quad
\begin{tikzpicture}[anchorbase, scale=.4, tinynodes]
	\draw[dstrand, Xmarked=.7] (0,-3) to (0,-2);
	\draw[dstrand, Ymarked=.55] (0,-2) to [out=180, in=270] (-.75,-1) to [out=90, in=180] (0,0);
	\draw[dstrand, Ymarked=.55] (0,-2) to [out=0, in=270] (.75,-1) to [out=90, in=0] (0,0);
	\draw[dstrand, Xmarked=.7] (0,0) to (0,1);
\end{tikzpicture}
=
\begin{tikzpicture}[anchorbase, scale=.4, tinynodes]
	\draw[dstrand, Xmarked=.55] (0,-3) to (0,1);
\end{tikzpicture}
\quad
\text{and}
\quad
\begin{tikzpicture}[anchorbase, scale=.4, tinynodes]
	\draw[dstrand, Ymarked=.7] (0,-3) to (0,-2);
	\draw[dstrand, Xmarked=.55] (0,-2) to [out=180, in=270] (-.75,-1) to [out=90, in=180] (0,0);
	\draw[dstrand, Xmarked=.55] (0,-2) to [out=0, in=270] (.75,-1) to [out=90, in=0] (0,0);
	\draw[dstrand, Ymarked=.7] (0,0) to (0,1);
\end{tikzpicture}
=
\begin{tikzpicture}[anchorbase, scale=.4, tinynodes]
	\draw[dstrand, Ymarked=.55] (0,-3) to (0,1);
\end{tikzpicture}
\end{gather}
by specializing the relations for symmetric powers
in \cite[Equations (12), (14) and (15)]{RW1}. (As recalled above, 
a label $m+n=2e$ in their picture corresponds to $e^{\ast}$ 
in our notation and all $m+n=2e$ in \cite[Equations (12), (14) and (15)]{RW1}
are then to be replaced by $e$.) Here it is crucial that $\qqpar$ is a 
$2(e+3)^{\mathrm{th}}$ root of unity. For example, $\qqnumber{e{+}1}=\qqnumber{2}$ in this case, 
which implies that $\qbinq{e{+}2}{e}=\qqnumber{e{+}1}\qqnumber{2}^{-1}=1$.

Next, the relations in \eqref{eq:sym-web-rels} give
\begin{gather}\label{eq:sym-rel-1}
\begin{tikzpicture}[anchorbase, scale=.4, tinynodes]
	\draw[dstrand, Ymarked=.55] (2,0) to [out=90, in=0] (1,1) to [out=180, in=90] (0,0);
	\draw[dstrand, Ymarked=.55] (0,3) to [out=270, in=180] (1,2) to [out=0, in=270] (2,3);
\end{tikzpicture}
=
\begin{tikzpicture}[anchorbase, scale=.4, tinynodes]
	\draw[dstrand, Ymarked=.55] (0,3) to (0,0);
	\draw[dstrand, Ymarked=.55] (2,0) to (2,3);
\end{tikzpicture}
\quad
\text{and}
\quad
\begin{tikzpicture}[anchorbase, scale=.4, tinynodes]
	\draw[dstrand, Xmarked=.55] (2,0) to [out=90, in=0] (1,1) to [out=180, in=90] (0,0);
	\draw[dstrand, Xmarked=.55] (0,3) to [out=270, in=180] (1,2) to [out=0, in=270] (2,3);
\end{tikzpicture}
=
\begin{tikzpicture}[anchorbase, scale=.4, tinynodes]
	\draw[dstrand, Xmarked=.55] (0,3) to (0,0);
	\draw[dstrand, Xmarked=.55] (2,0) to (2,3);
\end{tikzpicture}
\quad
\text{and}
\quad
\begin{tikzpicture}[anchorbase, scale=.4, tinynodes]
	\draw[dstrand, Ymarked=.55] (1,1) to (2,0);
	\draw[dstrand, Ymarked=.55] (1,1) to (0,0);
	\draw[dstrand, Ymarked=.55] (0,3) to (1,2);
	\draw[dstrand, Ymarked=.55] (2,3) to (1,2);
	\draw[dstrand, Ymarked=.55] (1,1) to (1,2);
\end{tikzpicture}
=
\begin{tikzpicture}[anchorbase, scale=.4, tinynodes]
	\draw[dstrand, Ymarked=.55] (0,3) to (0,0);
	\draw[dstrand, Ymarked=.55] (2,3) to (2,0);
\end{tikzpicture}
\quad
\text{and}
\quad
\begin{tikzpicture}[anchorbase, scale=.4, tinynodes]
	\draw[dstrand, Xmarked=.55] (1,1) to (2,0);
	\draw[dstrand, Xmarked=.55] (1,1) to (0,0);
	\draw[dstrand, Xmarked=.55] (0,3) to (1,2);
	\draw[dstrand, Xmarked=.55] (2,3) to (1,2);
	\draw[dstrand, Xmarked=.55] (1,1) to (1,2);
\end{tikzpicture}
=
\begin{tikzpicture}[anchorbase, scale=.4, tinynodes]
	\draw[dstrand, Xmarked=.55] (0,3) to (0,0);
	\draw[dstrand, Xmarked=.55] (2,3) to (2,0);
\end{tikzpicture}
\end{gather}
which are needed to show associativity.

For $\obstuff{A}^{\graphD{e}}$ to be Frobenius, we additionally need to check that
\[
\xymatrix@C+=1.0cm@L+=6pt{
ee 
\ar[r]^{
\begin{tikzpicture}[anchorbase, scale=.4, tinynodes]
	\draw[dstrand, Xmarked=.7] (-.5,-.5) to [out=90, in=270] (-1,.5);
	\draw[dstrand, Xmarked=.7] (0,-.5) to (0,.5);
\end{tikzpicture}}
\ar[d]_{
\begin{tikzpicture}[anchorbase, scale=.4, tinynodes]
	\draw[dstrand, Xmarked=.7] (-.5,-.5) to (0,0);
	\draw[dstrand, Xmarked=.7] (.5,-.5) to (0,0);
	\draw[dstrand, Ymarked=.55] (0,0) to (0,.5);
\end{tikzpicture}}
& 
e\emptyset e
\ar[d]^{
\begin{tikzpicture}[anchorbase, scale=.4, tinynodes]
	\draw[dstrand, Xmarked=.7] (1,-.5) to [out=90, in=270] (.5,.5);
	\draw[dstrand, Xmarked=.7] (0,-.5) to (0,.5);
\end{tikzpicture}}
\\
e^{\ast}
\ar[r]_{
\begin{tikzpicture}[anchorbase, scale=.4, tinynodes]
	\draw[dstrand, Ymarked=.55] (-.5,.5) to (0,0);
	\draw[dstrand, Ymarked=.55] (.5,.5) to (0,0);
	\draw[dstrand, Xmarked=.7] (0,0) to (0,-.5);
\end{tikzpicture}}
& 
ee
}
\;\;\;\;
\xymatrix@C+=1.0cm@L+=6pt{
ee^{\ast}
\ar[r]^{
\begin{tikzpicture}[anchorbase, scale=.4, tinynodes]
	\draw[dstrand, Xmarked=.7] (-.5,.5) to (0,0);
	\draw[dstrand, Xmarked=.7] (.5,.5) to (0,0);
	\draw[dstrand, Ymarked=.55] (0,0) to (0,-.5);
	\draw[dstrand, Ymarked=.5] (1,-.5) to (1,.5);
\end{tikzpicture}}
\ar[d]_{
\begin{tikzpicture}[anchorbase, scale=.4, tinynodes]
	\draw[dstrand, Ymarked=.55] (1,0) to [out=90, in=0] (.5,.5) to [out=180, in=90] (0,0);
\end{tikzpicture}}
& 
e^{\ast}e^{\ast}e^{\ast}
\ar[d]^{
\begin{tikzpicture}[anchorbase, scale=.4, tinynodes]
	\draw[dstrand, Ymarked=.55] (-.5,-.5) to (0,0);
	\draw[dstrand, Ymarked=.55] (.5,-.5) to (0,0);
	\draw[dstrand, Xmarked=.7] (0,0) to (0,.5);
	\draw[dstrand, Ymarked=.55] (-1,-.5) to (-1,.5);
\end{tikzpicture}}
\\
\emptyset
\ar[r]_{
\begin{tikzpicture}[anchorbase, scale=.4, tinynodes]
	\draw[dstrand, Ymarked=.55] (1,0) to [out=270, in=0] (.5,-.5) to [out=180, in=270] (0,0);
\end{tikzpicture}}
& 
e^{\ast}e
}
\;\;\;\;
\xymatrix@C+=1.0cm@L+=6pt{
ee^{\ast}
\ar[r]^{
\begin{tikzpicture}[anchorbase, scale=.4, tinynodes]
	\draw[dstrand, Ymarked=.55] (-.5,.5) to (0,0);
	\draw[dstrand, Ymarked=.55] (.5,.5) to (0,0);
	\draw[dstrand, Xmarked=.7] (0,0) to (0,-.5);
	\draw[dstrand, Ymarked=.5] (-1,.5) to (-1,-.5);
\end{tikzpicture}}
\ar[d]_{
\begin{tikzpicture}[anchorbase, scale=.4, tinynodes]
	\draw[dstrand, Ymarked=.55] (1,0) to [out=90, in=0] (.5,.5) to [out=180, in=90] (0,0);
\end{tikzpicture}}
& 
eee
\ar[d]^{
\begin{tikzpicture}[anchorbase, scale=.4, tinynodes]
	\draw[dstrand, Xmarked=.7] (-.5,-.5) to (0,0);
	\draw[dstrand, Xmarked=.7] (.5,-.5) to (0,0);
	\draw[dstrand, Ymarked=.55] (0,0) to (0,.5);
	\draw[dstrand, Ymarked=.55] (1,.5) to (1,-.5);
\end{tikzpicture}}
\\
\emptyset
\ar[r]_{
\begin{tikzpicture}[anchorbase, scale=.4, tinynodes]
	\draw[dstrand, Ymarked=.55] (1,0) to [out=270, in=0] (.5,-.5) to [out=180, in=270] (0,0);
\end{tikzpicture}}
& 
e^{\ast}e
}
\]
commute, which follows from \eqref{eq:sym-rel-1}.
The other diagrammatic equations, which prove the 
compatibility between the multiplication and the comultiplication, are immediate.
\end{proof}

\begin{proposition}\label{proposition:typeD}
The rank of the module category associated to $\obstuff{A}^{\graphD{e}}$ is equal to $\tfrac{t_e-1}{3}+3$, 
the number of nodes of the graph $\graphD{e}$.
\end{proposition}

Note here that $t_e=\tfrac{(e+1)(e+2)}{2}\equiv 1\bmod 3$ since $e\equiv 0\bmod 3$, by assumption.

\begin{proof}
We first recall some 
facts. By \cite[Section 3.3]{BK1} or \cite[Lemma 3.2.1]{Sch}, the so-called twist 
of $\obstuff{A}^{\graphD{e}}$ is the identity morphism. 
(This is false when $e\not\equiv 0 \bmod 3$.) Note also that $\obstuff{A}^{\graphD{e}}$ is
simple as an algebra $1$-morphism since it cannot have a proper, non-zero two-sided ideal, because any ideal containing
$e$ or $e^{\ast}$ has to contain $ee^{\ast}=e^{\ast}e$, which is isomorphic to the unit object 
by e.g. \eqref{eq:sym-web-rels} and \eqref{eq:sym-rel-1}.
Moreover, let $\qdim(\obstuff{O})$ denote the quantum 
dimension of $\obstuff{O}\in\slqmod$, which is defined by specializing the generic quantum dimension 
at the primitive, complex root of unity $\qqpar$, see e.g. \cite[Section 4.7]{EGNO}. 
By Weyl's character formula \cite[Theorem 5.15]{Ja} and its specialization, we have  
\[
\qdim(\Ll_{m,n})=\qqnumber{2}^{-1}\qqnumber{m+1}\qqnumber{n+1}\qqnumber{m+n+2}.
\]
Hence, we get $\qdim(\obstuff{A}^{\graphD{e}})=3\neq 0$.   
By \cite[Lemma 1.20]{K-O}, this 
implies that 
$\obstuff{A}^{\graphD{e}}$ is rigid, 
as defined in \cite[Definition 1.11]{K-O}. 

Therefore, $\modcat{\slqmod}(\obstuff{A}^{\graphD{e}})$ is a semisimple 
category with finitely many isomorphism classes of simples, 
by \cite[Theorem 3.3]{K-O}. Furthermore, any simple module in 
$\modcat{\slqmod}(\obstuff{A}^{\graphD{e}})$ is a direct 
summand of $\morstuff{F}(\obstuff{O})$ for a certain $\obstuff{O}\in\slqmod$, 
by \cite[Lemma 3.4]{K-O}. Here $\morstuff{F}\colon\slqmod\to 
\modcat{\slqmod}(\obstuff{A}^{\graphD{e}})$ is the free functor 
defined by $\morstuff{F}(\obstuff{O})=\obstuff{O}\otimes\obstuff{A}^{\graphD{e}}$. 
By \cite[Lemma 1.16]{K-O}, this functor is biadjoint to the forgetful functor 
$\morstuff{G}\colon\modcat{\slqmod}(\obstuff{A}^{\graphD{e}})\to 
\slqmod$. As a last ingredient recall that
\[
\Ll_{m,n}\otimes\Ll_{e,0}\cong\Ll_{e{-}m{-}n,m},
\quad
\Ll_{m,n}\otimes\Ll_{0,e}\cong\Ll_{n,e{-}m{-}n},
\]
hold in $\slqmod$ by e.g. \cite[Corollary 8]{Saw}.

It is now easy to determine the simples in $\modcat{\slqmod}(\obstuff{A}^{\graphD{e}})$. 
Let us write $e=3r$. 
\smallskip
\begin{enumerate}[label=$\blacktriangleright$]

\setlength\itemsep{.15cm}

\item Assume that $(m,n)\neq(r,r)$. Then
\[
\morstuff{GF}(\Ll_{m,n})\cong\Ll_{m,n}\oplus
\Ll_{e{-}m{-}n,m}\oplus\Ll_{n,e{-}m{-}n}. 
\]
These three summands form a three element orbit of the cut-off of the positive 
Weyl chamber under the rotation by $\tfrac{2\pi}{3}$. 
Therefore, we have 
\[
\dim_{\C}\Hom_{\obstuff{A}^{\graphD{e}}}(\morstuff{F}(\Ll_{m,n}),
\morstuff{F}(\Ll_{m,n}))=\dim_{\C}  
\Hom_{\slqmod}(\Ll_{m,n},\morstuff{GF}(\Ll_{m,n}))=1.
\]
By the categorical version of Schur's lemma, see e.g. 
\cite[Lemma 1.5.2]{EGNO}, $\morstuff{F}(\Ll_{m,n})$ is a 
simple $\obstuff{A}^{\graphD{e}}$-module object.
Note further that 
$\morstuff{GF}(\Ll_{m,n})\cong\morstuff{GF}(\Ll_{e-m-n,m})\cong\morstuff{GF}(\Ll_{n,e-m-n})$. Thus, we have 
\[
\dim_{\C}\Hom_{\obstuff{A}^{\graphD{e}}}(\morstuff{F}(\Ll_{m,n}),\morstuff{F}(\Ll_{e-m-n,m}))=
\dim_{\C}\Hom_{\obstuff{A}^{\graphD{e}}}(\morstuff{F}(\Ll_{m,n}),\morstuff{F}(\Ll_{n,e-m-n}))=1,
\]
by adjointness of 
$\morstuff{F}$ and $\morstuff{G}$. Thus, 
$\morstuff{F}(\Ll_{m,n})
\cong\morstuff{F}(\Ll_{e-m-n,m})
\cong\morstuff{F}(\Ll_{n,e-m-n})$. 

Finally, $\morstuff{F}(\Ll_{m,n})\not\cong\morstuff{F}(\Ll_{m^{\prime},n^{\prime}})$,
if $(m^{\prime},n^{\prime})\not\in\{(m,n),(e-m-n,m),(n,e-m-n)\}$, because 
in that case $\morstuff{GF}(\Ll_{m,n})\not\cong\morstuff{GF}(\Ll_{m^{\prime},n^{\prime}})$.

\item Assume that $(m,n)=(r,r)$. Then 
\[
\morstuff{GF}(\Ll_{r,r})\cong
\Ll_{r,r}\oplus\Ll_{r,r}\oplus\Ll_{r,r}.
\]
Here $\Ll_{r,r}$ is the fixed point in the cut-off of the positive 
Weyl chamber under the rotation by $\tfrac{2\pi}{3}$.
Therefore, we have 
\begin{gather}\label{eq:dim-argument}
\dim_{\C}\Hom_{\obstuff{A}^{\graphD{e}}}(\morstuff{F}(\Ll_{r,r}),\morstuff{F}(\Ll_{r,r}))=\dim_{\C}  
\Hom_{\slqmod}(\Ll_{r,r}, \morstuff{GF}(\Ll_{r,r}))= 3.
\end{gather}
This shows that $\morstuff{F}(\Ll_{r,r})$ decomposes 
into three simples, each of which is mapped to $\Ll_{r,r}$ by 
$\morstuff{G}$, but which are 
pairwise non-isomorphic as $\obstuff{A}^{\graphD{e}}$-module objects. 
(Otherwise, the dimension in \eqref{eq:dim-argument} would be $5$ or $9$.)
\end{enumerate}
\smallskip
The statement now follows by counting.
\end{proof}

\begin{example}\label{example:hom-formula}
In case $e=3$ (the case illustrated on the left in 
\fullref{fig:typeD} below), 
we have ten simple objects in $\slqmod$:
\smallskip
\begin{enumerate}[label=$\blacktriangleright$]

\setlength\itemsep{.15cm}

\item $\Ll_{0,0},\Ll_{3,0},\Ll_{0,3}$, which have quantum dimension $\qqnumber{1}=1$.

\item $\Ll_{1,0},\Ll_{2,1},\Ll_{0,2}$ and $\Ll_{0,1},\Ll_{1,2},\Ll_{2,0}$, which have quantum dimension 
$\qqnumber{3}=2$.

\item $\Ll_{1,1}$, which has quantum dimension $\qqnumber{2}\qqnumber{4}=3$. 
\end{enumerate}
\smallskip

In contrast, the simple $\obstuff{A}^{\graphD{3}}$-module objects are:
\smallskip
\begin{enumerate}[label=$\blacktriangleright$]

\setlength\itemsep{.15cm}

\item $\Ll_{0,0}\oplus\Ll_{3,0}\oplus\Ll_{0,3}$, which have quantum dimension $3\qqnumber{1}=3$.

\item $\Ll_{1,0}\oplus\Ll_{2,1}\oplus\Ll_{0,2}$ and $\Ll_{0,1}\oplus\Ll_{1,2}\oplus\Ll_{2,0}$, 
which have quantum dimension 
$3\qqnumber{3}=6$.

\item Three non-isomorphic 
copies of $\Ll_{1,1}$, which still have quantum dimension $3$.
\end{enumerate}
(The reader should compare this with the $\Z/3\Z$-symmetry in \fullref{fig:typeD} 
and the identification respectively splitting of the vertices illustrated therein.)
\end{example}

\begin{remark}\label{remark:hom-formula}
The above classification is consistent 
with the following analog of \eqref{eq:ineq-semisimple}.
Let 
\[
\qdim(\slqmod)=
{\textstyle\sum_{0\leq m+n\leq e}}\,\qdim(\Ll_{m,n})^2. 
\]  

Since $\obstuff{A}^{\graphD{e}}$ is rigid, we have 
\begin{gather}\label{eq:semisimpleincat}
\qdim(\obstuff{A}^{\graphD{e}})\qdim(\slqmod)=
{\textstyle\sum_{\obstuff{S}}}\; \qdim(\obstuff{S})^2
\end{gather}
where we sum over a complete set of pairwise 
non-isomorphic, simple $\obstuff{A}^{\graphD{e}}$-module objects $\obstuff{S}$. 
The formula in \eqref{eq:semisimpleincat} follows e.g. from \cite[Example 7.16.9(ii)]{EGNO}. 
Note that in this example $\qdim(\obstuff{O})$ equals the Perron--Frobenius dimension of $\obstuff{O}$ 
as used in \cite[Example 7.16.9(ii)]{EGNO}
because $\obstuff{A}^{\graphD{e}}$ is rigid.
\end{remark}

By \fullref{proposition:typeD-object}, we see that $\obstuff{A}^{\graphD{e}}$ 
can be regarded as a Frobenius algebra $1$-morphism in $\slqmodgop$. Hence,
we get a Frobenius algebra $1$-morphism $\algstuff{B}^{\graphD{e}^{\gc}}\{-3\}$ in $\subcatquo[e]$.

\begin{remark}\label{remark:hard-clasps} 
In this case it would be hard to write down explicitly the diagrams which define the structural 
$2$-morphisms of $\algstuff{B}^{\graphD{e}^{\gc}}$, i.e. unit, multiplication, counit and comultiplication. The reason is 
that the symmetric web calculus suppresses the clasps corresponding to 
$\Ll_{e,0}$ and $\Ll_{0,e}$, i.e. the quantum symmetrizers and 
antisymmetrizers on $e$ strands in Kuperberg's web calculus \cite{Kup}, cf. \fullref{example:clasps}. 
\end{remark}    

By \fullref{proposition:typeD}, $\cM_{\graphD{e}^{\gc}}$ does not 
correspond to the cell $2$-representation 
if $e\equiv 0 \bmod 3$, and, by construction, it categorifies $\M_{\graphD{e}^{\gc}}$ 
from \fullref{definition:n-modules}. Similarly for $\oc$ and $\pc$.

\subsubsection{Some simple transitive \texorpdfstring{$2$}{2}-representations}\label{subsec:cat-story-d}

Note that an equivalence of simple transitive $2$-re\-presentations decategorifies 
to a $\aformvN$-equivalence of transitive $\aformvN$-representations. Hence, 
the following is the summary of the above and \fullref{lemma:find-all-transitives-2}:

\begin{theoremqed}\label{theorem:typeA-D-cat}
Let us indicate by 
$\gc,\oc,\pc$ the starting color. Then
we have (at least) the following simple transitive $2$-representations 
of $\subcatquo[e]$.
\begin{gather}\label{eq:the-transitives-2}
\begin{tikzpicture}[baseline=(current bounding box.center),yscale=0.6]
  \matrix (m) [matrix of math nodes, row sep=1em, column
  sep=1em, text height=1.5ex, text depth=0.25ex, ampersand replacement=\&] {
\phantom{a}
\& 
e\equiv 0\bmod 3 
\&  
e\not\equiv 0\bmod 3 
\\
\text{$2$-reps.}
\&
\begin{gathered}
\cM_{\graphA{e}^{\gc}},\,
\cM_{\graphA{e}^{\oc}},\,
\cM_{\graphA{e}^{\pc}},
\\
\cM_{\graphD{e}^{\gc}},\,
\cM_{\graphD{e}^{\oc}},\,
\cM_{\graphD{e}^{\pc}}
\end{gathered}
\& 
\cM_{\graphA{e}^{\gc}} 
\\
\text{quantity}
\&
6
\& 
1
\\};
  \draw[densely dashed] ($(m-1-1.south west)+ (-.75,0)$) to ($(m-1-3.south east)+ (.75,0)$);
  \draw[densely dashed] ($(m-1-2.north west) + (-.9,0)$) to ($(m-1-2.north west) + (-.9,-4.2)$);
  \draw[densely dashed] ($(m-1-2.north west) + (3.2,0)$) to ($(m-1-2.north west) + (3.2,-4.2)$);
\end{tikzpicture}
\end{gather}
Moreover, the $2$-representations $\cM_{\graphA{e}}$ 
are the cell $2$-representations of $\subcatquo[e]$, and 
all of these decategorify 
to the corresponding $\subquo[e]$-representations in \eqref{eq:the-transitives}.
\end{theoremqed}

We note here that the case $e=0$ is not included 
in our discussion above, but \fullref{theorem:typeA-D-cat} holds 
in this case as well (but types $\typeA$ and $\typeD$ coincide), which can 
be checked directly.

\subsection{Trihedral zigzag algebras}\label{subsec:quiver}

We are going to describe a weak categorification of the 
$\aformvN$-representations $\M_{\graphA{\infty}}$ 
and $\M_{\graphA{e}}$
from \fullref{subsec:decat-story} 
using a certain quiver algebra. (By weak categorification 
we mean the same as e.g. in \cite[Definition 1]{KMS1}.)

Below we let $\somevert{i},\somevert{j},\somevert{k}$ 
always be different elements in $\{\somevert{x},\somevert{y},\somevert{z}\}$. 
Moreover, we write $\somevert{i}_{m,n}$ for the 
idempotent at a given vertex labeled by $(m,n)$, and 
a path from $\somevert{i}_{m,n}$ to 
$\somevert{j}_{m^{\prime},n^{\prime}}$ is denoted by
$\pathx{i}{j}=\somevert{j}_{m^{\prime},n^{\prime}}
\circ\pathx{i}{j}\circ\somevert{i}_{m,n}$ 
(abusing notation, 
we omit the idempotents) with compositions $\pathx{j}{k}\circ\pathx{i}{j}=\pathxx{i}{j}{k}$ 
and $\pathx{i}{j}\circ\loopy{k}=\pathx{i}{j}\loopy{k}$ etc.

\subsubsection{The trihedral zigzag algebra of level \texorpdfstring{$\infty$}{infty}}\label{subsec:quiver-algebra}

We work over $\C$ in this section.

\begin{definition}\label{definition:quiver}
Let $\zig[\ast]$ be the path algebra of the following quiver.

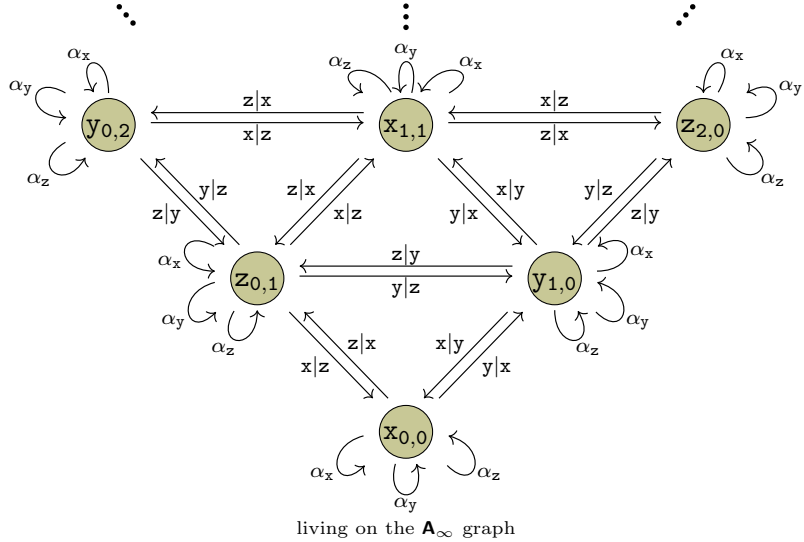
\begin{gather*}
\xy
(0,0)*{
\begin{tikzcd}[ampersand replacement=\&,row sep=large,column sep=scriptsize,arrows={shorten >=-.5ex,shorten <=-.5ex},labels={inner sep=.05ex}]
\dvert{y_{0,2}}
\arrow[out=90,in=120,loop,distance=1.25em,swap]{}{\loopy{x}} 
\arrow[out=150,in=180,loop,distance=1.25em,swap]{}{\loopy{y}}
\arrow[out=210,in=240,loop,distance=1.25em,swap]{}{\loopy{z}}
\arrow[xshift=.6ex,<-]{dr}{\pathy{z}{y}}
\arrow[xshift=-.6ex,->,swap]{dr}{\pathx{y}{z}}
\&
\phantom{.}
\&
\dvert{x_{1,1}}
\arrow[out=30,in=60,loop,distance=1.25em,swap]{}{\loopy{x}} 
\arrow[out=77,in=102,loop,distance=1.25em,swap]{}{\loopy{y}}
\arrow[out=120,in=150,loop,distance=1.25em,swap]{}{\loopy{z}}
\arrow[xshift=-.6ex,->,swap]{dl}{\pathy{x}{z}}
\arrow[xshift=.6ex,<-]{dl}{\pathx{z}{x}}
\arrow[yshift=-.4ex,<-]{ll}{\pathy{z}{x}}
\arrow[yshift=.4ex,->,swap]{ll}{\pathx{x}{z}}
\arrow[xshift=.6ex,<-]{dr}{\pathy{y}{x}}
\arrow[xshift=-.6ex,->,swap]{dr}{\pathx{x}{y}}
\&
\phantom{.}
\&
\dvert{z_{2,0}}
\arrow[out=60,in=90,loop,distance=1.25em,swap]{}{\loopy{x}} 
\arrow[out=0,in=30,loop,distance=1.25em,swap]{}{\loopy{y}}
\arrow[out=300,in=330,loop,distance=1.25em,swap]{}{\loopy{z}}
\arrow[xshift=-.6ex,->,swap]{dl}{\pathy{z}{y}}
\arrow[xshift=.6ex,<-]{dl}{\pathx{y}{z}}
\arrow[yshift=-.4ex,<-]{ll}{\pathy{x}{z}}
\arrow[yshift=.4ex,->,swap]{ll}{\pathx{z}{x}}
\\
\phantom{.}
\&
\dvert{z_{0,1}}
\arrow[out=150,in=180,loop,distance=1.25em,swap,pos=0.575]{}{\loopy{x}} 
\arrow[out=195,in=225,loop,distance=1.25em,swap]{}{\loopy{y}}
\arrow[out=240,in=270,loop,distance=1.25em,swap]{}{\loopy{z}}
\arrow[xshift=.6ex,<-]{dr}{\pathy{x}{z}}
\arrow[xshift=-.6ex,->,swap]{dr}{\pathx{z}{x}}
\&
\phantom{.}
\&
\dvert{y_{1,0}}
\arrow[out=0,in=30,loop,distance=1.25em,swap,pos=0.475]{}{\loopy{x}} 
\arrow[out=315,in=345,loop,distance=1.25em,swap]{}{\loopy{y}}
\arrow[out=270,in=300,loop,distance=1.25em,swap]{}{\loopy{z}}
\arrow[xshift=-.6ex,->,swap]{dl}{\pathy{y}{x}}
\arrow[xshift=.6ex,<-]{dl}{\pathx{x}{y}}
\arrow[yshift=-.4ex,<-]{ll}{\pathy{z}{y}}
\arrow[yshift=.4ex,->,swap]{ll}{\pathx{y}{z}}
\&
\phantom{.}
\\
\phantom{.}
\&
\phantom{.}
\&
\dvert{x_{0,0}}
\arrow[out=195,in=225,loop,distance=1.25em,swap]{}{\loopy{x}} 
\arrow[out=255,in=285,loop,distance=1.25em,swap]{}{\loopy{y}}
\arrow[out=315,in=345,loop,distance=1.25em,swap]{}{\loopy{z}}
\&
\phantom{.}
\&
\phantom{.}
\\
\end{tikzcd}};
(-37.5,38)*{\rotatebox{40}{{\Huge$\vdots$}}};
(0,38)*{\rotatebox{0}{{\Huge$\vdots$}}};
(37.5,38)*{\rotatebox{-40}{{\Huge$\vdots$}}};
(0,-31)*{\text{{\tiny living on the $\graphA{\infty}$ graph}}};
\endxy
\end{gather*}
We view $\zig[\ast]$ as being graded by putting paths $\pathx{i}{j}$ 
and loops $\loopy{i}$ in degree $2$.
\end{definition}

\begin{definition}\label{definition:quiver-infty}
Let $\zig[\infty]$ be the quotient algebra of $\zig[\ast]$ by the following 
relations.
\smallskip
\begin{enumerate}[label=(\alph*)]

\setlength\itemsep{.15cm}

\renewcommand{\theenumi}{(\ref{definition:quiver-infty}.a)}
\renewcommand{\labelenumi}{\theenumi}

\item \label{enum:quiver-1} \textbf{Leaving a triangular face is zero.} Any oriented path of length two 
between non-adjacent vertices is zero. 

\renewcommand{\theenumi}{(\ref{definition:quiver-infty}.b)}
\renewcommand{\labelenumi}{\theenumi}

\item \label{enum:quiver-2} \textbf{The relations of the cohomology ring of 
the variety of full flags in $\C^3$.} $\loopy{i}\loopy{j}=\loopy{j}\loopy{i}$, 
$\loopy{x}+\loopy{y}+\loopy{z}=0$, 
$\loopy{x}\loopy{y}+\loopy{x}\loopy{z}+\loopy{y}\loopy{z}=0$ 
and $\loopy{x}\loopy{y}\loopy{z}=0$.

\renewcommand{\theenumi}{(\ref{definition:quiver-infty}.c)}
\renewcommand{\labelenumi}{\theenumi}

\item \label{enum:quiver-3} \textbf{Sliding loops.} 
$\pathx{i}{j}\loopy{i}=-\loopy{j}\pathx{i}{j}$, 
$\pathx{i}{j}\loopy{j}=-\loopy{i}\pathx{i}{j}$ and
$\pathx{i}{j}\loopy{k}=\loopy{k}\pathx{i}{j}=0$.

\renewcommand{\theenumi}{(\ref{definition:quiver-infty}.d)}
\renewcommand{\labelenumi}{\theenumi}

\item \label{enum:quiver-4} \textbf{Zigzag.} 
$\pathxx{i}{j}{i}=\loopy{i}\loopy{j}$.

\renewcommand{\theenumi}{(\ref{definition:quiver-infty}.e)}
\renewcommand{\labelenumi}{\theenumi}

\item \label{enum:quiver-5} \textbf{Zigzig equals 
zag times loop.} $\pathxx{i}{j}{k}=\pathx{i}{k}\loopy{i}=-\loopy{k}\pathx{i}{k}$.

\end{enumerate}
\smallskip
We call $\zig[\infty]$ the trihedral zigzag algebra of 
level $\infty$.
\end{definition}

The relations \ref{enum:quiver-1} to \ref{enum:quiver-5}
are homogeneous with respect to the degree defined in 
\fullref{definition:quiver}, which endows $\zig[\infty]$ with 
the structure of a graded algebra, and
we write $\qdim[\vpar](\placeholder)\in\N[\vpar]$ for 
the graded dimension, viewing $\vpar$ as a 
variable of degree $1$. Note that $\zig[\infty]$ is zero in all odd degrees, by definition.  

\subsubsection{Some basic properties}\label{subsec:quiver-algebra-props}

\begin{proposition}\label{proposition:quadratic}
$\zig[\infty]$ is quadratic, i.e. it is generated in degree $2$ and the relations 
are generated in degree $4$.
\end{proposition}

\begin{proof}
All relations except $\loopy{x}+\loopy{y}+\loopy{z}=0$ and
$\loopy{x}\loopy{y}\loopy{z}=0$ are of degree $4$. 

The degree two relation shows that our 
presentation is redundant: we could give a presentation with 
fewer generators and no degree two relation. We prefer our 
presentation, which is more symmetric and therefore easier 
to write down. But one could get rid of the degree two relation by using only two degree two loops per vertex. 
 
Thus, up to base change, remains to show that 
$\loopy{x}\loopy{y}\loopy{z}=0$ is a consequence of 
degree $4$ relations, which can be done as follows:
\[
\loopy{x}\loopy{y}\loopy{z}
\stackrel{\ref{enum:quiver-4}}{=}
\pathxx{x}{y}{x}\loopy{z}
\stackrel{\ref{enum:quiver-3}}{=}
0.
\]
This finishes the proof that $\zig[\infty]$ is quadratic.
\end{proof}

\begin{lemma}\label{lemma:quiver-homs}
Let $\obstuff{S}=\{(m{\pm1},n), (m{\pm}1,n{\mp}1), (m,n{\pm}1)\}$.
\begin{gather}\label{eq:basis-quiver}
\begin{aligned}
\Hom_{\C}(\somevert{i}_{m,n},\somevert{i}_{m^{\prime},n^{\prime}})
&=
\begin{cases}
\C
\{\somevert{i}_{m,n},\loopy{i},\loopy{j},\loopy{i}\loopy{j},\loopy{i}\loopy{k},
\loopy{i}^2\loopy{j}\}
, &\text{if }(m,n)=(m^{\prime},n^{\prime}),
\\
0, &\text{else},
\end{cases}
\\
\Hom_{\C}(\somevert{i}_{m,n},\somevert{j}_{m^{\prime},n^{\prime}})
&=
\begin{cases}
\C
\{\pathx{i}{j},\pathx{i}{j}\loopy{i}\}
, &\text{if }(m^{\prime},n^{\prime})\in\obstuff{S},
\\
0, &\text{else}.
\end{cases}
\end{aligned}
\end{gather}
Moreover, the non-trivial graded dimensions are 
$\qdim[\vpar](\Hom_{\C}(\somevert{i}_{m,n},
\somevert{i}_{m,n}))=\vpar^3\vfrac{3}$ and 
$\qdim[\vpar](\Hom_{\C}(\somevert{i}_{m,n},
\somevert{j}_{m^{\prime},n^{\prime}}))=\vpar\vnumber{2}$, when $(m^{\prime},n^{\prime})\in\obstuff{S}$.
\end{lemma}

\begin{proof}
This is clear for the trivial hom-spaces by 
\ref{enum:quiver-1}. So let us focus on the non-trivial ones.
To this end, we first consider  
homogeneous linear combinations of loops of degree $2,4,6$: 
\begin{gather}\label{eq:loopy-rels}
\begin{gathered}
\loopy{x}+\loopy{y}=-\loopy{z},
\\
\loopy{x}^2+\loopy{x}\loopy{y}=-\loopy{x}\loopy{z},
\quad\quad
\loopy{x}\loopy{y}+\loopy{y}^2=-\loopy{y}\loopy{z},
\quad\quad
\loopy{x}\loopy{z}+\loopy{y}\loopy{z}=-\loopy{z}^2,
\\
\loopy{x}^2\loopy{z}=\loopy{x}\loopy{y}^2=\loopy{y}\loopy{z}^2
=
-\loopy{x}^2\loopy{y}=-\loopy{y}^2\loopy{z}=-\loopy{x}\loopy{z}^2.
\end{gathered}
\end{gather}
These relations follow immediately from \ref{enum:quiver-2}, 
and show that the endomorphism space of any vertex is spanned by the ones in \eqref{eq:basis-quiver}.

Next, we consider all homogeneous elements of degree $4$ in $\zig[\ast]$. 
The ones that are composites of two paths leaving a 
triangular face are zero by \ref{enum:quiver-1}, 
and the remaining ones are linear combinations of those 
appearing in \ref{enum:quiver-3}, \ref{enum:quiver-4}, \ref{enum:quiver-5}. 

The homogeneous elements of degree $6$ that are not annihilated by 
\ref{enum:quiver-1} are
\begin{gather*}
\pathxxx{i}{j}{k}{i}
\stackrel{\ref{enum:quiver-5}}{=}
\pathxx{i}{k}{i} \loopy{i}
\stackrel{\ref{enum:quiver-4}}{=}
\loopy{i}^2\loopy{k}
\stackrel{\eqref{eq:loopy-rels}}{=}
-
\loopy{i}^2\loopy{j}
\stackrel{\ref{enum:quiver-4}}{=}
-
\pathxx{i}{j}{i} \loopy{i}
\stackrel{\ref{enum:quiver-5}}{=}
-
\pathxxx{i}{k}{j}{i},
\\
\pathxxx{i}{j}{i}{j}
\stackrel{\ref{enum:quiver-4}}{=}
\pathx{i}{j}
\loopy{i}\loopy{j}
\stackrel{\ref{enum:quiver-2}}{=}
\pathx{i}{j}
(-\loopy{i}\loopy{k}-\loopy{j}\loopy{k})
\stackrel{\ref{enum:quiver-3}}{=}
0,
\\
\pathxxx{i}{j}{i}{k}
\stackrel{\ref{enum:quiver-4}}{=}
\pathx{i}{k}\loopy{i}\loopy{j}
\stackrel{\ref{enum:quiver-3}}{=}
0,
\end{gather*}
including versions obtained by changing sources and targets.

Finally, we claim that
all homogeneous elements of degree $>6$ in $\zig[\ast]$ 
are zero in $\zig[\infty]$. For paths leaving a face or composites of only loops,  
there is nothing to show by \ref{enum:quiver-1} 
and \ref{enum:quiver-2}. For paths of length four around one 
triangular face, we get 
\begin{gather*}
\pathxxxx{i}{j}{k}{i}{j}
\stackrel{\ref{enum:quiver-5}}{=}
\pathxxx{i}{k}{i}{j}\loopy{i}
\stackrel{\ref{enum:quiver-4}}{=}
\pathx{i}{j}\loopy{i}^2\loopy{k}
\stackrel{\ref{enum:quiver-3}}{=}
0,
\\
\pathxxxx{i}{j}{k}{i}{k}
\stackrel{\ref{enum:quiver-5}}{=}
\pathxxx{i}{k}{i}{k}\loopy{i}
\stackrel{\ref{enum:quiver-4}}{=}
\pathx{i}{k}\loopy{i}^2\loopy{k}
\stackrel{\eqref{eq:loopy-rels}}{=}
-\pathx{i}{k}\loopy{i}^2\loopy{j}
\stackrel{\ref{enum:quiver-3}}{=}
0,
\\
\pathxxxx{i}{j}{i}{j}{i}
\stackrel{\ref{enum:quiver-4}}{=}
\pathxx{i}{j}{i}\loopy{i}\loopy{j}
\stackrel{\ref{enum:quiver-4}}{=}
\loopy{i}^2\loopy{j}^2
\stackrel{\eqref{eq:loopy-rels}}{=}
0,
\\
\pathxxxx{i}{j}{i}{k}{i}
\stackrel{\ref{enum:quiver-4}}{=}
\pathxx{i}{k}{i}\loopy{i}\loopy{j}
\stackrel{\ref{enum:quiver-4}}{=}
\loopy{i}^2\loopy{j}\loopy{k}
\stackrel{\ref{enum:quiver-2}}{=}
0.
\end{gather*}
Again, there are analogous relations obtained by changing sources and targets.  
Altogether this shows that the sets in 
\eqref{eq:basis-quiver} span the hom-spaces.

To show linear independence, we consider the following linear map, which on monomials in the path algebra is given by:
\[
\mathrm{tr}\colon
\zig[\infty]\to\C,
\quad
\mathrm{tr}(a)
=
\begin{cases}
1, &\text{if }a\in\{\loopy{x}^2\loopy{z},\loopy{x}\loopy{y}^2,\loopy{y}\loopy{z}^2\},
\\
-1, &\text{if }a\in\{\loopy{x}^2\loopy{y},\loopy{y}^2\loopy{z},\loopy{x}\loopy{z}^2\},
\\
0, &\text{else}.
\end{cases}
\]
Note that $\mathrm{tr}$ is well-defined, which can be 
checked by showing that the relations are preserved. It is also easy to see 
that $\mathrm{tr}$ is a non-degenerate and symmetric trace form, e.g.
\begin{gather*}
\mathrm{tr}(\pathx{x}{y}\cdot\pathxx{y}{z}{x})
=
\mathrm{tr}(\pathxxx{x}{y}{z}{x})
\stackrel{\ref{enum:quiver-5}}{=}
\mathrm{tr}(\pathxx{x}{y}{x}\loopy{x})
\stackrel{\ref{enum:quiver-4}}{=}
\mathrm{tr}(\loopy{x}^2\loopy{y})
\\
\stackrel{\eqref{eq:loopy-rels}}{=}
\mathrm{tr}(-\loopy{x}\loopy{y}^2)
\stackrel{\ref{enum:quiver-4}}{=}
-\mathrm{tr}(\loopy{y}\pathxx{y}{x}{y})
\stackrel{\ref{enum:quiver-5}}{=}
\mathrm{tr}(\pathxxx{y}{z}{x}{y})
=
\mathrm{tr}(\pathxx{y}{z}{x}\cdot\pathx{x}{y})
\end{gather*}
We can now write down sets of morphisms 
which are dual to the ones from \eqref{eq:basis-quiver} with 
respect to $\mathrm{tr}$, e.g.:
\begin{align*}
(
\somevert{i}_{m,n},
\loopy{x},
\loopy{y},
\loopy{x}\loopy{y},
\loopy{x}\loopy{z},
\loopy{x}^2\loopy{z}
)
&
\leftrightsquigarrow
(
\loopy{x}^2\loopy{z},
\loopy{x}\loopy{z},
\loopy{x}\loopy{y},
\loopy{y},
\loopy{x},
\somevert{i}_{m,n}
),
\\
(
\pathx{i}{j}, \pathx{i}{j}\loopy{i}
) 
&
\leftrightsquigarrow
(
\pm \pathx{j}{i}\loopy{j}, \pm  \pathx{j}{i}
).
\end{align*}
Since these sets span the corresponding hom-spaces, we are done.
\end{proof}

The following result follows immediately from the proof of \fullref{lemma:quiver-homs}. 

\begin{corollary}\label{corollary:frob}
$\zig[\infty]$ is a positively graded, symmetric Frobenius algebra.
\end{corollary}

\subsubsection{The quotient of level \texorpdfstring{$e$}{e}}\label{subsec:quotient-algebra-zig}

\begin{definition}\label{definition:quiver-e}
For fixed level $e$, let $\zigideal{e}$ 
be the two-sided ideal in $\zig[\infty]$ generated by
\begin{gather*}
\left\{
\somevert{i}_{m,n} \mid m+n\geq e+1
\right\}.
\end{gather*}
We define the trihedral zigzag algebra of level $e$ as
\[
\zig[e]=\zig[\infty]/\zigideal{e}
\]
and we call $\zigideal{e}$ the vanishing zigzag ideal of level $e$.
\end{definition}

Clearly, $\zig[e]$ has a basis given 
by the elements in \eqref{eq:basis-quiver} for $m+n\leq e$.
Thus, $\zig[e]$ is a finite-dimensional, positively graded algebra, which is a symmetric Frobenius 
algebra by \fullref{corollary:frob}. By the proof of 
\fullref{proposition:quadratic} it is also quadratic, 
as long as $e\neq 0$.

\subsubsection{Weak categorification}\label{subsec:quiver-algebra-weak}

Following ideas from 
\cite{KS1}, \cite{AT1} and
\cite[Section 2]{MT1}, we let 
$\lpro$, respectively $\rpro$, denote the 
left, respectively right, ideals in $\zig[\infty]$ 
generated by $\somevert{i}_{m,n}$. These
are indecomposable, graded projective $\zig[\infty]$-modules, and 
all indecomposable, graded projective left, respectively right, $\zig[\infty]$-modules 
are of this form, up to grading shifts. 

By the above, $\lpro\otimes\rpro$ is a biprojective 
$\zig[\infty]$-bimodule, i.e. it is
projective as a left and as a right $\zig[\infty]$-module. 
Therefore,  
\begin{gather*}
\thetaf{\gc}(\placeholder)
=
{\textstyle\bigoplus_{\somevert{x}_{m,n}}}
(\lpro[\somevert{x}_{m,n}]\otimes \rpro[\somevert{x}_{m,n}]\{-3\})
\otimes_{\zig[\infty]}\placeholder,
\quad\;
\thetaf{\oc}(\placeholder)
=
{\textstyle\bigoplus_{\somevert{y}_{m,n}}}
(\lpro[\somevert{y}_{m,n}]\otimes\rpro[\somevert{y}_{m,n}]\{-3\})
\otimes_{\zig[\infty]}\placeholder,
\\
\thetaf{\pc}(\placeholder)
=
{\textstyle\bigoplus_{\somevert{z}_{m,n}}}
(\lpro[\somevert{z}_{m,n}]\otimes\rpro[\somevert{z}_{m,n}]\{-3\})
\otimes_{\zig[\infty]}\placeholder,
\end{gather*}
define endofunctors on the category $\zigmod$ of 
finite-dimensional, graded projective (left) $\zig$-modules. 
Here $\otimes$ denotes $\otimes_{\C}$. 

To state the weak categorification we denote by 
$\twoEnd(\zigmod)$ the category of endofunctors on $\zigmod$. 
Considering $\subquo$ as a one-object category with the formal 
object $\ast$ and morphisms being 
its elements,  
we get the following.

\begin{lemma}\label{lemma:endo-functors}
The functor $\subquo\to\twoEnd(\zigmod)$ given by $\ast\mapsto\zigmod$ and
\[
\theta_{\gc}
\mapsto
\thetaf{\gc}(\placeholder),
\quad\quad
\theta_{\oc}
\mapsto
\thetaf{\oc}(\placeholder),
\quad\quad
\theta_{\pc}
\mapsto
\thetaf{\pc}(\placeholder)
\]
is well-defined. 

Moreover, decategorification gives the transitive 
$\aformvN$-representation $\M_{\graphA{\infty}}$ of $\subquo$.
\end{lemma}

\begin{proof}
Let us first show that 
$\thetaf{\gc}\thetaf{\gc}\cong \thetaf{\gc}^{\oplus \vfrac{3}}$, with 
the superscript $\oplus \vfrac{3}$ meaning six degree-shifted copies of $\thetaf{\gc}$, with $\vpar$ 
corresponding to a degree-shift by one. Similar arguments show that the same holds for 
$\thetaf{\oc}$ and $\thetaf{\pc}$, of course. Note that $\thetaf{\gc}\thetaf{\gc}$ is given by 
tensoring with the bimodule 
\[
{\textstyle\bigoplus_{\somevert{x}_{m,n}}}\,
\lpro[\somevert{x}_{m,n}]\otimes\rpro[\somevert{x}_{m,n}]\otimes_{\zig[\infty]}
\lpro[\somevert{x}_{m,n}]\otimes\rpro[\somevert{x}_{m,n}]\{-6\},
\]
with all other direct summands being zero. 
By definition, this is isomorphic to 
\begin{gather*}
{\textstyle\bigoplus_{\somevert{x}_{m,n}}}\,
\lpro[\somevert{x}_{m,n}]\otimes \mathrm{End}_{\zig[\infty]}(\somevert{x}_{m,n})\otimes\rpro[\somevert{x}_{m,n}]\{-6\}\cong 
{\textstyle\bigoplus_{\somevert{x}_{m,n}}}\,
\left(\lpro[\somevert{x}_{m,n}]\otimes
\rpro[\somevert{x}_{m,n}]\{-3\}\right)^{\oplus\vnumber{3}!},
\end{gather*}
where the displayed isomorphism follows from \fullref{lemma:quiver-homs}.

Next, we show that 
$\thetaf{\gc}\thetaf{\oc}\thetaf{\gc}\cong\thetaf{\gc}\thetaf{\pc}\thetaf{\gc}$. 
Again, similar arguments show the 
analogous result in the remaining cases. The functor 
$\thetaf{\gc}\thetaf{\oc}\thetaf{\gc}$ 
is given by tensoring with 
\begin{gather}\label{eq:gogbimodule1}
{\textstyle\bigoplus}\,
\lpro[\somevert{x}_{m,n}]\otimes\rpro[\somevert{x}_{m,n}]\otimes_{\zig[\infty]} \lpro[\somevert{y}_{m^{\prime},n^{\prime}}]\otimes\rpro[\somevert{y}_{m^{\prime},n^{\prime}}] \otimes_{\zig[\infty]} \lpro[\somevert{x}_{m^{\prime\prime},n^{\prime\prime}}]\otimes\rpro[\somevert{x}_{m^{\prime\prime},n^{\prime\prime}}]\{-9\},
\end{gather}
where the direct sum is over all neighboring 
pairs $(m,n), (m^{\prime},n^{\prime})$ and $(m^{\prime},n^{\prime}),
(m^{\prime\prime},n^{\prime\prime})$, 
i.e. $(m^{\prime},n^{\prime})\in \{(m\pm 1,n), (m,n\pm 1), (m\pm 1, n\mp 1)\}$ 
and $(m^{\prime\prime},n^{\prime\prime})\in\{(m,n), (m\pm 1,n\pm 1), 
(m\pm 2, n\mp 1), (m\pm 1, n\mp 2)\}$. This is isomorphic to 
\begin{gather}\label{eq:gogbimodule2}
\begin{gathered}
{\textstyle\bigoplus}\,
\lpro[\somevert{x}_{m,n}]\otimes\Hom_{\zig[\infty]}
(\somevert{x}_{m,n}, \somevert{y}_{m^{\prime},n^{\prime}})\otimes
\Hom_{\zig[\infty]}(\somevert{y}_{m^{\prime},n^{\prime}},
\somevert{x}_{m^{\prime\prime},n^{\prime\prime}})\otimes\rpro[\somevert{x}_{m^{\prime\prime},n^{\prime\prime}}]   \{-9\}
\cong
\\
{\textstyle\bigoplus}\,\left(
\lpro[\somevert{x}_{m,n}]\otimes
\rpro[\somevert{x}_{m^{\prime\prime},n^{\prime\prime}}]\{-7\}\right)^{\oplus \vnumber{2}^2},
\end{gathered}
\end{gather}
where the superscript $\oplus \vnumber{2}^2$ should be interpreted as before.
The isomorphism displayed in
\eqref{eq:gogbimodule2} follows from \fullref{lemma:quiver-homs}. 

The functor $\thetaf{\gc}\thetaf{\oc}\thetaf{\gc}$ is 
given by tensoring with the $\zig[\infty]$-bimodule obtained from the 
one in \eqref{eq:gogbimodule1} by replacing 
$\somevert{y}_{m^{\prime},n^{\prime}}$ with $\somevert{z}_{m^{\prime},n^{\prime}}$, which is 
also isomorphic to the $\zig[\infty]$-bimodule in \eqref{eq:gogbimodule2}. 
Although the neighboring pairs change when 
we replace $\somevert{y}_{m^{\prime},n^{\prime}}$ 
with $\somevert{z}_{m^{\prime},n^{\prime}}$, 
their total number is equal by symmetry. 
Thus, the final number of direct summands in \eqref{eq:gogbimodule2} is the 
same in both cases. This finishes the 
proof that $\thetaf{\gc}\thetaf{\oc}\thetaf{\gc}\cong\thetaf{\gc}\thetaf{\pc}\thetaf{\gc}$.
\end{proof}

\begin{proposition}\label{proposition:endo-functors}
The functor from \fullref{lemma:endo-functors} 
descends to a functor $\subquo[e]\to\twoEnd(\zig[e])$. 
Moreover, decategorification gives the transitive 
$\aformvN$-representation $\M_{\graphA{e}}$ of $\subquo[e]$.
\end{proposition}

\begin{proof}
Recall that the
$\aformvN$-representation 
$\M_{\graphA{e}}$ 
of $\subquo[e]$ from \fullref{definition:n-modules} satisfies 
\begin{gather*}
\Mt[\Gg]
=
\M_{\Gg}(\theta_{\gc})
+
\M_{\Gg}(\theta_{\oc})
+
\M_{\Gg}(\theta_{\pc})
=
\vnumber{2}\left(
\vnumber{3}\Idmatrix
+
A(\oGg)
\right).
\end{gather*} 

Up to natural isomorphism, the same holds for the functors 
$\thetaf{\gc},\thetaf{\oc}$ and $\thetaf{\pc}$ when they are 
applied to $\lpro[\somevert{x}_{m,n}], 
\lpro[\somevert{y}_{m,n}]$ and $\lpro[\somevert{z}_{m,n}]$, by 
\fullref{lemma:quiver-homs}. The proof uses exactly the same sort of arguments 
as the proof of \fullref{lemma:endo-functors}. We therefore omit further details. 
The statement then follows from \fullref{corollary:poly-killed}.
\end{proof}

\begin{remark}\label{remark:zig-zag-zog}
The quiver algebra defined in \cite{Gr1} is (in the 
case of $\graphA{e}$ graphs) a subalgebra of $\zig[e]$. 
We do not know any further connection between these two algebras.

In fact, up to certain scalars, the defining relations 
of the quiver algebra in \fullref{definition:quiver-infty} are the ones of the quiver algebra 
underlying the cell $2$-representations 
of $\subcatquo[e]$. Those scalars can be computed 
for small values of $e$, but we have not been able 
to compute them for general $e$, unfortunately. 
However, even without the correct scalars, we 
thought that the trihedral zigzag algebra, 
and its connection to \cite{Gr1}, 
was too nice to exclude it from this paper. 

One also wonders whether any of the constructions 
involving the zigzag algebras in \cite{HK1} have an analogue for the 
trihedral zigzag algebras. 
\end{remark}

\subsection{Generalizing dihedral \texorpdfstring{$2$}{2}-representation theory}\label{subsec:dihedral-group-cat}

\begin{dihedral}\label{remark:dihedral-group4}
Yet another analogy to the dihedral case $\dihquo[e]$ is provided 
by \fullref{problem:classification} and \fullref{proposition:poly-killed}: 
One can define transitive $\aformvN$-representations of $\dihquo$ 
analogously to the $\aformvN$-representations $\M_{\Gg}$, where 
in the dihedral case $\Gg$ is any connected, bicolored graph. These 
descend to the finite-dimensional $\dihquo[e]$ if and only if 
$\pxy[A(\Gg)]{e+1}=0$, where 
$\pxy[\placeholder]{e+1}$ is the Chebyshev 
polynomial as in \fullref{remark:dihedral-group2}. 
(This follows from \cite{KMMZ} and, a bit more directly, from \cite{MT1}.) 
In that case, the analog of \fullref{problem:classification} has a 
well-understood answer, namely $\Gg$ has to be of $\ADE$ 
Dynkin type.
\end{dihedral}

\begin{dihedral}\label{remark-zigzag}
In the dihedral case, the classification of 
simple transitive $2$-representation is an 
$\ADE$-type classification 
(assuming gradeability), cf. \cite{KMMZ} and \cite{MT1}. 
This follows from the classification of graphs recalled in 
\fullref{remark:dihedral-group4} and the associated $2$-representations, which,
in analogy to \fullref{subsec:cat-story}, can be constructed using algebra $1$-morphisms 
in the $\mathfrak{sl}_2$ analog of $\slqmod$ (cf. \cite[Section 7]{MMMT1}).
\end{dihedral}

\begin{dihedral}\label{remark-zigzag-real}
The quiver underlying the cell 
$2$-representa\-tions in the dihedral case is the zigzag algebra from \cite{HK1}, which could be 
presented as in our setup, although this is never done in the literature, 
using two loops $\loopy{x},\loopy{y}$ at each vertex, subject to the relations of the cohomology ring of the variety of full flags in $\C^2$, i.e. 
$\loopy{x}\loopy{y}=\loopy{y}\loopy{x}=0$, $\loopy{x}+\loopy{y}=0$.
(To make the connection with \cite{HK1}, note that this cohomology ring is 
isomorphic to $\C[X]/(X^2)$.) 
The same holds for all other simple transitive $2$-representations (in the dihedral case) 
with the zigzag algebra for the corresponding graph, see \cite{MT1}. Thus, \fullref{subsec:quiver} can 
be seen as the trihedral version of this. We think it would be interesting to work out the trihedral quiver algebras 
for other simple transitive $2$-representations of $\subcatquo[e]$.
\end{dihedral}
\addtocontents{toc}{\protect\setcounter{tocdepth}{1}}
%
\section*{Appendix: Generalized \texorpdfstring{$\ADE$}{ADE} Dynkin diagrams}\label{section:gen-D}

\renewcommand{\thesection}{App}
\renewcommand{\theequation}{\thesection-\arabic{equation}}
\renewcommand{\thefigure}{App-\arabic{figure}}
\setcounter{theoremm}{0}
\setcounter{equation}{0}
\setcounter{figure}{0}
\setcounter{subsection}{0}

In this appendix we have listed certain solutions to \fullref{problem:classification}. Following \cite{Zu}, we call 
these the generalized $\ADE$ Dynkin diagrams. The graphs below depend on the level $e$, which is 
the same as e.g. in \fullref{section:sl3-stuff} and is indicated as a subscript.

\subsection{The list}\label{subsec:gen-D-list}

\makeautorefname{figure}{Figures}

The following are the generalized $\ADE$ Dynkin diagrams.
There are three infinite families, displayed in \fullref{fig:typeA}, \ref{fig:typeD} 
and \ref{fig:typeC}, \makeautorefname{figure}{Figure} and a finite number of exceptions,  
displayed in \fullref{fig:typeE}.

\begin{figure}[ht]
\[
\begin{tikzpicture}[anchorbase, xscale=.35, yscale=.5]
	\draw [thick, myyellow] (0,0) node[below, black] {$\graphA{0}$} to (0,0);
	\node at (0,0) {$\bulletgstart$};
\end{tikzpicture}
\;\;,\;\;
\begin{tikzpicture}[anchorbase, xscale=.35, yscale=.5]
	\draw [thick, myyellow] (0,0) node[below, black] {$\graphA{1}$} to (1,1);
	\draw [thick, densely dotted, myblue] (0,0) to (-1,1);
	\draw [thick, densely dashed, myred] (1,1) to (-1,1);
	\node at (0,0) {$\bulletg$};
	\node at (1,1) {$\bulleto$};
	\node at (-1,1) {$\bulletp$};
\end{tikzpicture}
\;\;,\;\;
\begin{tikzpicture}[anchorbase, xscale=.35, yscale=.5]
	\draw [thick, myyellow] (0,0) node[below, black] {$\graphA{2}$} to (1,1) to (0,2);
	\draw [thick, myyellow] (0,2) to (-2,2);
	\draw [thick, densely dotted, myblue] (0,0) to (-1,1) to (0,2);
	\draw [thick, densely dotted, myblue] (0,2) to (2,2);
	\draw [thick, densely dashed, myred] (2,2) to (1,1) to (-1,1) to (-2,2);
	\node at (0,0) {$\bulletg$};
	\node at (0,2) {$\bulletg$};
	\node at (1,1) {$\bulleto$};
	\node at (-2,2) {$\bulleto$};
	\node at (2,2) {$\bulletp$};
	\node at (-1,1) {$\bulletp$};
\end{tikzpicture}
\;\;,\;\;
\begin{tikzpicture}[anchorbase, xscale=.35, yscale=.5]
	\draw [thick, myyellow] (0,0) node[below, black] {$\graphA{3}$} to (1,1) to (0,2) to (1,3) to (3,3);
	\draw [thick, myyellow] (0,2) to (-2,2) to (-3,3);
	\draw [thick, densely dotted, myblue] (0,0) to (-1,1) to (0,2) to (-1,3) to (-3,3);
	\draw [thick, densely dotted, myblue] (0,2) to (2,2) to (3,3);
	\draw [thick, densely dashed, myred] (1,1) to (-1,1) to (-2,2) to (-1,3) to (1,3) to (2,2) to (1,1);
	\node at (0,0) {$\bulletgstart$};
	\node at (0,2) {$\bulletg$};
	\node at (3,3) {$\bulletg$};
	\node at (-3,3) {$\bulletg$};
	\node at (1,1) {$\bulleto$};
	\node at (-2,2) {$\bulleto$};
	\node at (1,3) {$\bulleto$};
	\node at (2,2) {$\bulletp$};
	\node at (-1,1) {$\bulletp$};
	\node at (-1,3) {$\bulletp$};
\end{tikzpicture}
\;\;,\;\;
\begin{tikzpicture}[anchorbase, xscale=.35, yscale=.5]
	\draw [thick, myyellow] (0,0) node[below, black] {$\graphA{4}$} to (1,1) to (0,2) to (1,3) to (3,3) to (4,4);
	\draw [thick, myyellow] (0,2) to (-2,2) to (-3,3) to (-2,4) to (0,4) to (1,3);
	\draw [thick, densely dotted, myblue] (0,0) to (-1,1) to (0,2) to (-1,3) to (-3,3) to (-4,4);
	\draw [thick, densely dotted, myblue] (0,2) to (2,2) to (3,3) to (2,4) to (0,4) to (-1,3);
	\draw [thick, densely dashed, myred] (1,1) to (-1,1) to (-2,2) to (-1,3) to (1,3) to (2,2) to (1,1);
	\draw [thick, densely dashed, myred] (-4,4) to (-2,4) to (-1,3);
	\draw [thick, densely dashed, myred] (4,4) to (2,4) to (1,3);
	\node at (0,0) {$\bulletg$};
	\node at (0,2) {$\bulletg$};
	\node at (3,3) {$\bulletg$};
	\node at (-3,3) {$\bulletg$};
	\node at (0,4) {$\bulletg$};
	\node at (1,1) {$\bulleto$};
	\node at (-2,2) {$\bulleto$};
	\node at (1,3) {$\bulleto$};
	\node at (-2,4) {$\bulleto$};
	\node at (4,4) {$\bulleto$};
	\node at (2,2) {$\bulletp$};
	\node at (-1,1) {$\bulletp$};
	\node at (-1,3) {$\bulletp$};
	\node at (2,4) {$\bulletp$};
	\node at (-4,4) {$\bulletp$};
\end{tikzpicture}
\dots
\]
\caption{The infinite family of (generalized) type $\typeA$, indexed by $e\in\N$. The graph of type $\graphA{e}$ 
can be obtained by cutting off the $\slt$-weight lattice at level $e+1$, as in \eqref{eq:weight-picture}.}
\label{fig:typeA}
\end{figure}

\begin{figure}[ht]
\[
\begin{tikzpicture}[anchorbase, xscale=.35, yscale=.5]
	\draw [thick, myyellow] (0,0) node[below, black] {$\graphA{3}$} to (1,1) to (0,2) to (1,3) to (3,3);
	\draw [thick, myyellow] (0,2) to (-2,2) to (-3,3);
	\draw [thick, densely dotted, myblue] (0,0) to (-1,1) to (0,2) to (-1,3) to (-3,3);
	\draw [thick, densely dotted, myblue] (0,2) to (2,2) to (3,3);
	\draw [thick, densely dashed, myred] (1,1) to (-1,1) to (-2,2) to (-1,3) to (1,3) to (2,2) to (1,1);
	\node at (0,0) {$\bulletgstart$};
	\node at (0,2) {$\bulletg$};
	\node at (3,3) {$\bulletg$};
	\node at (-3,3) {$\bulletg$};
	\node at (1,1) {$\bulleto$};
	\node at (-2,2) {$\bulleto$};
	\node at (1,3) {$\bulleto$};
	\node at (2,2) {$\bulletp$};
	\node at (-1,1) {$\bulletp$};
	\node at (-1,3) {$\bulletp$};
	\draw [->, mygreen] (3,3.2) to [out=170, in=10] (-3,3.2);
	\draw [->, mygreen] (-3,2.8) to [out=280, in=170] (-.3,0);
	\draw [->, mygreen] (.3,0) to [out=10, in=260] (3,2.8);
	\draw [<-, myorange] (1,2.8) to (1,1.2);
	\draw [<-, myorange] (.8,1.1) to (-1.7,1.9);
	\draw [<-, myorange] (-1.7,2.1) to (.8,2.9);
	\draw [->, mypurple] (-1,2.8) to (-1,1.2);
	\draw [->, mypurple] (-.8,1.1) to (1.7,1.9);
	\draw [->, mypurple] (1.7,2.1) to (-.8,2.9);
\end{tikzpicture}
\rightsquigarrow
\begin{tikzpicture}[anchorbase, xscale=.35, yscale=.5]
	\draw [thick, myyellow] (0,0) node[below, black] {$\graphD{3}$} to (1,1);
	\draw [thick, myyellow] (-2,2) to (1,1) to (0,2);
	\draw [thick, myyellow] (2,2) to (1,1);
	\draw [thick, densely dotted, myblue] (0,0) to (-1,1);
	\draw [thick, densely dotted, myblue] (2,2) to (-1,1) to (0,2);
	\draw [thick, densely dotted, myblue] (-2,2) to (-1,1);
	\draw [thick, densely dashed, myred, double] (1,1) to (-1,1);
	\node at (0,0) {$\bulletgstart$};
	\node at (-2,2) {$\bulletg$};
	\node at (0,2) {$\bulletg$};
	\node at (2,2) {$\bulletg$};
	\node at (1,1) {$\bulleto$};
	\node at (-1,1) {$\bulletp$};
\end{tikzpicture}
,
\begin{tikzpicture}[anchorbase, xscale=.35, yscale=.5]
	\draw [thick, myyellow] (0,0) node[below, black] {$\graphA{6}$} to (1,1) to (0,2) to (1,3) to (3,3) to (4,4);
	\draw [thick, myyellow] (0,2) to (-2,2) to (-3,3) to (-2,4) to (0,4) to (1,3);
	\draw [thick, myyellow] (4,4) to (3,5);
	\draw [thick, myyellow] (-2,4) to (-3,5);
	\draw [thick, myyellow] (0,4) to (1,5);
	\draw [thick, myyellow] (6,6) to (4,6) to (3,5) to (1,5) to (0,6) to (-2,6) to (-3,5) to (-5,5) to (-6,6);
	\draw [thick, densely dotted, myblue] (0,0) to (-1,1) to (0,2) to (-1,3) to (-3,3) to (-4,4);
	\draw [thick, densely dotted, myblue] (0,2) to (2,2) to (3,3) to (2,4) to (0,4) to (-1,3);
	\draw [thick, densely dotted, myblue] (-4,4) to (-3,5);
	\draw [thick, densely dotted, myblue] (2,4) to (3,5);
	\draw [thick, densely dotted, myblue] (0,4) to (-1,5);
	\draw [thick, densely dotted, myblue] (-6,6) to (-4,6) to (-3,5) to (-1,5) to (0,6) to (2,6) to (3,5) to (5,5) to (6,6);
	\draw [thick, densely dashed, myred] (1,1) to (-1,1) to (-2,2) to (-1,3) to (1,3) to (2,2) to (1,1);
	\draw [thick, densely dashed, myred] (-4,4) to (-2,4) to (-1,3);
	\draw [thick, densely dashed, myred] (4,4) to (2,4) to (1,3);
	\draw [thick, densely dashed, myred] (-4,4) to (-5,5) to (-4,6) to (-2,6) to (-1,5) to (-2,4);
	\draw [thick, densely dashed, myred] (4,4) to (5,5) to (4,6) to (2,6) to (1,5) to (2,4);
	\draw [thick, densely dashed, myred] (-1,5) to (1,5);
	\node at (0,0) {$\bulletgstart$};
	\node at (-3,3) {$\bulletg$};
	\node at (-6,6) {$\bulletg$};
	\node at (0,2) {$\bulletg$};
	\node at (-3,5) {$\bulletg$};
	\node at (0,4) {$\bulletg$};
	\node at (3,3) {$\bulletg$};
	\node at (0,6) {$\bulletg$};
	\node at (3,5) {$\bulletg$};
	\node at (6,6) {$\bulletg$};
	\node at (-2,4) {$\bulleto$};
	\node at (1,1) {$\bulleto$};
	\node at (-2,2) {$\bulleto$};
	\node at (-5,5) {$\bulleto$};
	\node at (-2,6) {$\bulleto$};
	\node at (1,3) {$\bulleto$};
	\node at (1,5) {$\bulleto$};
	\node at (4,4) {$\bulleto$};
	\node at (4,6) {$\bulleto$};
	\node at (-4,4) {$\bulletp$};
	\node at (-1,1) {$\bulletp$};
	\node at (-1,3) {$\bulletp$};
	\node at (-4,6) {$\bulletp$};
	\node at (2,2) {$\bulletp$};
	\node at (-1,5) {$\bulletp$};
	\node at (2,4) {$\bulletp$};
	\node at (2,6) {$\bulletp$};
	\node at (5,5) {$\bulletp$};
	\draw [thick, mypurple, ->] (1.8,6) to [out=180, in=30] (-3.8,4.2);
	\draw [thick, mypurple, ->] (-3.8,3.8) to [out=330, in=180] (1.8,2);
	\draw [thick, mypurple, ->] (2,2.2) to [out=60, in=300] (2,5.8);
	\draw [thick, myorange, ->] (1,3.2) to (1,4.8);
	\draw [thick, myorange, ->] (.8,4.925) to (-1.8,4.075);
	\draw [thick, myorange, ->] (-1.8,3.925) to (.8,3.075);
	\draw [thick, mygreen, ->] (-6,5.8) to [out=270, in=180] (-.2,0);
	\draw [thick, mygreen, ->] (.2,0) to [out=0, in=270] (6,5.8);
	\draw [thick, mygreen, ->] (5.9,6.1) to [out=165, in=15] (-5.9,6.1);
\end{tikzpicture}
\rightsquigarrow
\begin{tikzpicture}[anchorbase, xscale=.35, yscale=.5]
	\draw [thick, myyellow] (0,0) node[below, black] {$\graphD{6}$} to (1,1) to (0,2) to (1,3);
	\draw [thick, myyellow] (0,2) to (-2,2);
	\draw [thick, myyellow] (-2,4) to (1,3) to (0,4);
	\draw [thick, myyellow] (2,4) to (1,3);
	\draw [thick, densely dotted, myblue] (0,0) to (-1,1) to (0,2) to (-1,3);
	\draw [thick, densely dotted, myblue] (0,2) to (2,2);
	\draw [thick, densely dotted, myblue] (-2,4) to (-1,3) to (0,4);
	\draw [thick, densely dotted, myblue] (2,4) to (-1,3);
	\draw [thick, densely dashed, myred] (1,1) to (-1,1) to (-2,2) to (-1,3);
	\draw [thick, densely dashed, myred] (1,3) to (2,2) to (1,1);
	\draw [thick, densely dashed, myred, double] (-1,3) to (1,3);
	\draw [thick, myyellow] (0,6) to [out=0, in=90] (2.5,4) to [out=270, in=0] (1,3);
	\draw [thick, densely dotted, myblue] (0,6) to [out=180, in=90] (-2.5,4) to [out=270, in=180] (-1,3);
	\draw [thick, densely dotted, myblue] (0,6) to [out=0, in=90] (3,4) to [out=270, in=0] (2,2);
	\draw [thick, myyellow] (0,6) to [out=180, in=90] (-3,4) to [out=270, in=180] (-2,2);
	\node at (0,0) {$\bulletgstart$};
	\node at (-2,4) {$\bulletg$};
	\node at (0,2) {$\bulletg$};
	\node at (0,4) {$\bulletg$};
	\node at (2,4) {$\bulletg$};
	\node at (0,6) {$\bulletg$};
	\node at (1,1) {$\bulleto$};
	\node at (-2,2) {$\bulleto$};
	\node at (1,3) {$\bulleto$};
	\node at (-1,1) {$\bulletp$};
	\node at (-1,3) {$\bulletp$};
	\node at (2,2) {$\bulletp$};
\end{tikzpicture}
\dots
\]
\caption{The infinite family of (generalized) type $\typeD$, indexed by $e\equiv 0\bmod 3$ and $e\neq 0$. 
The graph of type $\graphD{e}$ comes from the $\zeethree$-symmetry of the graph of type $\graphA{e}$ with 
the fixed points splitting into three copies, cf. \fullref{example:hom-formula}. 
(Note the double edges.) By convention, $\graphA{0}=\graphD{0}$.
}\label{fig:typeD}
\end{figure}

\begin{figure}[ht]
\[
\begin{tikzpicture}[anchorbase, xscale=.35, yscale=.5]
	\draw [thick, myyellow] (0,0) node[below, black] {$\graphC{1}\cong\graphA{1}$} to (1,1);
	\draw [thick, densely dotted, myblue] (0,0) to (-1,1);
	\draw [thick, densely dashed, myred] (1,1) to (-1,1);
	\node at (0,0) {$\bulletg$};
	\node at (1,1) {$\bulleto$};
	\node at (-1,1) {$\bulletp$};
\end{tikzpicture}
,
\begin{tikzpicture}[anchorbase, xscale=.35, yscale=.5]
	\draw [thick, myyellow] (0,0) node[below, black] {$\graphC{2}\cong\graphA{2}$} to (1,1) to (0,2) to (-2,2);
	\draw [thick, densely dotted, myblue] (0,0) to (-1,1) to (0,2) to (2,2);
	\draw [thick, densely dashed, myred] (2,2) to (1,1) to (-1,1) to (-2,2);
	\node at (0,0) {$\bulletg$};
	\node at (0,2) {$\bulletg$};
	\node at (1,1) {$\bulleto$};
	\node at (-2,2) {$\bulleto$};
	\node at (-1,1) {$\bulletp$};
	\node at (2,2) {$\bulletp$};
	\node at (0,1.35) {\text{\tiny$\graphC{1}$}};
	\node at (0,2.35) {\text{\tiny$\phantom{3}$}};
\end{tikzpicture}
,
\begin{tikzpicture}[anchorbase, xscale=.35, yscale=.5]
	\draw [thick, myyellow] (0,0) node[below, black] {$\graphC{3}$} to (1,1) to (0,2) to (-2,2) to [out=270, in=180] (0,0);
	\draw [thick, densely dotted, myblue] (0,0) to (-1,1) to (0,2) to (2,2) to [out=270, in=0] (0,0);
	\draw [thick, densely dashed, myred] (2,2) to (1,1) to (-1,1) to (-2,2) to [out=20, in=160] (2,2);
	\node at (0,0) {$\bulletg$};
	\node at (0,2) {$\bulletg$};
	\node at (1,1) {$\bulleto$};
	\node at (-2,2) {$\bulleto$};
	\node at (-1,1) {$\bulletp$};
	\node at (2,2) {$\bulletp$};
	\node at (0,2.35) {\text{\tiny$\phantom{3}$}};
\end{tikzpicture}
,
\begin{tikzpicture}[anchorbase, xscale=.35, yscale=.5]
	\draw [thick, myyellow] (0,0) node[below, black] {$\graphC{4}$} to (2,2) to [out=90, in=0] (0,4) to (-4,4);
	\draw [thick, myyellow] (0,4) to (-1,3) to (0,2) to (2,2);
	\draw [thick, densely dotted, myblue] (0,0) to (-2,2) to [out=90, in=180] (0,4) to (4,4);
	\draw [thick, densely dotted, myblue] (0,4) to (1,3) to (0,2) to (-2,2);
	\draw [thick, densely dashed, myred] (4,4) to (2,2) to [out=200, in=340] (-2,2) to (-4,4);
	\draw [thick, densely dashed, myred] (-2,2) to (-1,3) to (1,3) to (2,2);
	\node at (0,0) {$\bulletg$};
	\node at (0,4) {$\bulletg$};
	\node at (0,2) {$\bulletg$};
	\node at (2,2) {$\bulleto$};
	\node at (-4,4) {$\bulleto$};
	\node at (-1,3) {$\bulleto$};
	\node at (-2,2) {$\bulletp$};
	\node at (4,4) {$\bulletp$};
	\node at (1,3) {$\bulletp$};
	\node at (-.05,2.7) {\text{\tiny$\graphC{3}$}};
	\node at (0,4.75) {\text{\tiny$\phantom{3}$}};
\end{tikzpicture}
,
\begin{tikzpicture}[anchorbase, xscale=.35, yscale=.5]
	\draw [thick, myyellow] (0,0) node[below, black] {$\graphC{5}$} to (2,2) to [out=90, in=0] (0,4) to (-4,4) to [out=270, in=180] (0,0);
	\draw [thick, myyellow] (0,4) to (-1,3) to (0,2) to (2,2);
	\draw [thick, densely dotted, myblue] (0,0) to (-2,2) to [out=90, in=180] (0,4) to (4,4) to [out=270, in=0] (0,0);
	\draw [thick, densely dotted, myblue] (0,4) to (1,3) to (0,2) to (-2,2);
	\draw [thick, densely dashed, myred] (4,4) to (2,2) to [out=200, in=340] (-2,2) to (-4,4) to [out=20, in=160] (4,4);
	\draw [thick, densely dashed, myred] (-2,2) to (-1,3) to (1,3) to (2,2);
	\node at (0,0) {$\bulletg$};
	\node at (0,4) {$\bulletg$};
	\node at (0,2) {$\bulletg$};
	\node at (2,2) {$\bulleto$};
	\node at (-4,4) {$\bulleto$};
	\node at (-1,3) {$\bulleto$};
	\node at (-2,2) {$\bulletp$};
	\node at (4,4) {$\bulletp$};
	\node at (1,3) {$\bulletp$};
	\node at (0,4.75) {\text{\tiny$\phantom{3}$}};
\end{tikzpicture}
\cdots
\]
\caption{The infinite family of conjugate type $\typeA$, indexed by $e\in\N$. 
The graph of type $\graphC{e}$ comes from an iterative procedure on the graph of type $\graphA{e}$. 
By convention, $\graphA{0}=\graphC{0}$.
}\label{fig:typeC}
\end{figure}

\begin{figure}[ht]
\[
\begin{tikzpicture}[anchorbase, xscale=.35, yscale=.5]
	\draw [thick, myyellow] (-1,3) to (2,2) to (1,1);
	\draw [thick, myyellow] (-1,3) to (-2,2) to (1,1);
	\draw [thick, myyellow] (0,4) to (-1,3);
	\draw [thick, myyellow] (0,0) node[below, black] {$\graphE{5}$} to (1,1);
	\draw [thick, myyellow] (2,2) to (3,3);
	\draw [thick, myyellow] (-2,2) to (-3,1);
	\draw [thick, densely dotted, myblue] (1,3) to (2,2) to (-1,1);
	\draw [thick, densely dotted, myblue] (1,3) to (-2,2) to (-1,1);
	\draw [thick, densely dotted, myblue] (0,4) to (1,3);
	\draw [thick, densely dotted, myblue] (0,0) to (-1,1);
	\draw [thick, densely dotted, myblue] (-2,2) to (-3,3);
	\draw [thick, densely dotted, myblue] (2,2) to (3,1);
	\draw [thick, densely dashed, myred] (1,3) to (-1,3) to (-1,1);
	\draw [thick, densely dashed, myred] (1,3) to (1,1) to (-1,1);
	\draw [thick, densely dashed, myred] (1,1) to (3,1);
	\draw [thick, densely dashed, myred] (-1,3) to (-3,3);
	\draw [thick, densely dashed, myred] (-3,1) to (-1,1);
	\draw [thick, densely dashed, myred] (3,3) to (1,3);
	\node at (0,0) {$\bulletg$};
	\node at (-2,2) {$\bulletg$};
	\node at (2,2) {$\bulletg$};
	\node at (0,4) {$\bulletg$};
	\node at (1,1) {$\bulleto$};
	\node at (-1,3) {$\bulleto$};
	\node at (-3,1) {$\bulleto$};
	\node at (3,3) {$\bulleto$};
	\node at (-1,1) {$\bulletp$};
	\node at (1,3) {$\bulletp$};
	\node at (3,1) {$\bulletp$};
	\node at (-3,3) {$\bulletp$};
\end{tikzpicture}
\;\;,\;\;
\begin{tikzpicture}[anchorbase, xscale=.35, yscale=.5]
	\draw [thick, myyellow] (0,0) node[below, black] {$\graphE{9_1}$} to (1,1) to (0,3);
	\draw [thick, myyellow] (-4,0) to (-3,1) to (0,3);
	\draw [thick, myyellow] (4,0) to (5,1) to (0,3);
	\draw [thick, myyellow, double] (0,3) to (4,3);
	\draw [thick, densely dotted, myblue] (0,0) to (-1,1) to (0,3);
	\draw [thick, densely dotted, myblue] (4,0) to (3,1) to (0,3);
	\draw [thick, densely dotted, myblue] (-4,0) to (-5,1) to (0,3);
	\draw [thick, densely dotted, myblue, double] (0,3) to (-4,3);
	\draw [thick, densely dashed, myred] (-3,1) to (-5,1);
	\draw [thick, densely dashed, myred] (1,1) to (-1,1);
	\draw [thick, densely dashed, myred] (5,1) to (3,1);
	\draw [thick, densely dashed, myred] (-5,1) to (4,3);
	\draw [thick, densely dashed, myred] (-1,1) to (4,3);
	\draw [thick, densely dashed, myred] (3,1) to (4,3);
	\draw [thick, densely dashed, myred] (5,1) to (-4,3);
	\draw [thick, densely dashed, myred] (1,1) to (-4,3);
	\draw [thick, densely dashed, myred] (-3,1) to (-4,3);
	\draw [thick, densely dashed, myred] (-4,3) to [out=10, in=180] (0,3.5) to [out=0, in=170] (4,3);
	\node at (0,0) {$\bulletg$};
	\node at (-4,0) {$\bulletg$};
	\node at (4,0) {$\bulletg$};
	\node at (0,3) {$\bulletg$};
	\node at (-3,1) {$\bulleto$};
	\node at (1,1) {$\bulleto$};
	\node at (5,1) {$\bulleto$};
	\node at (4,3) {$\bulleto$};
	\node at (-5,1) {$\bulletp$};
	\node at (-1,1) {$\bulletp$};
	\node at (3,1) {$\bulletp$};
	\node at (-4,3) {$\bulletp$};
\end{tikzpicture}
\;\;,\;\;
\begin{tikzpicture}[anchorbase, xscale=.35, yscale=.5]
	\draw [thick, myyellow] (0,-1) node[below, black] {$\graphE{9_2}$} to (1,0) to (0,2);
	\draw [thick, myyellow] (1,0) to (0,2);
	\draw [thick, myyellow] (1,0) to (0,4);
	\draw [thick, myyellow] (1,0) to (0,6);
	\draw [thick, myyellow] (0,2) to (-3,1) to (0,6);
	\draw [thick, myyellow] (0,2) to (-3,3) to (0,4);
	\draw [thick, myyellow] (0,4) to (-3,5) to (0,6);
	\draw [thick, densely dotted, myblue] (0,-1) to (-1,0) to (0,2);
	\draw [thick, densely dotted, myblue] (-1,0) to (0,2);
	\draw [thick, densely dotted, myblue] (-1,0) to (0,4);
	\draw [thick, densely dotted, myblue] (-1,0) to (0,6);
	\draw [thick, densely dotted, myblue] (0,2) to (3,1) to (0,6);
	\draw [thick, densely dotted, myblue] (0,2) to (3,3) to (0,4);
	\draw [thick, densely dotted, myblue] (0,4) to (3,5) to (0,6);
	\draw [thick, densely dashed, myred] (-3,1) to (-1,0) to (1,0) to (3,1);
	\draw [thick, densely dashed, myred] (-3,3) to (-1,0);
	\draw [thick, densely dashed, myred] (1,0) to (3,3);
	\draw [thick, densely dashed, myred] (-3,5) to (-1,0);
	\draw [thick, densely dashed, myred] (1,0) to (3,5);
	\draw [thick, densely dashed, myred] (-3,1) to (3,1);
	\draw [thick, densely dashed, myred] (-3,3) to (3,3);
	\draw [thick, densely dashed, myred] (-3,5) to (3,5);
	\node at (0,-1) {$\bulletg$};
	\node at (0,2) {$\bulletg$};
	\node at (0,4) {$\bulletg$};
	\node at (0,6) {$\bulletg$};
	\node at (1,0) {$\bulleto$};
	\node at (-3,1) {$\bulleto$};
	\node at (-3,3) {$\bulleto$};
	\node at (-3,5) {$\bulleto$};
	\node at (-1,0) {$\bulletp$};
	\node at (3,1) {$\bulletp$};
	\node at (3,3) {$\bulletp$};
	\node at (3,5) {$\bulletp$};
\end{tikzpicture}
\]
\[
\begin{tikzpicture}[anchorbase, xscale=.7, yscale=1]
	\draw [thick, myyellow] (1,1) to (0,2) to (1,3) to (3,3) to (4,4);
	\draw [thick, myyellow] (0,2) to (-2,2) to (-3,3) to (-2,4) to (0,4) to (1,3);
	\draw [thick, myyellow] (1,3) to (0,2.5) to [out=180, in=20] (-2,2);
	\draw [thick, myyellow] (0,2.5) to (-.35,3.25) to (0,4);
	\draw [thick, densely dotted, myblue] (-1,1) to (0,2) to (-1,3) to (-3,3) to (-4,4);
	\draw [thick, densely dotted, myblue] (0,2) to (2,2) to (3,3) to (2,4) to (0,4) to (-1,3);
	\draw [thick, densely dotted, myblue] (-1,3) to (0,2.5) to [out=0, in=160] (2,2);
	\draw [thick, densely dotted, myblue] (0,2.5) to (.35,3.25) to (0,4);
	\draw [thick, densely dashed, myred] (-1,1) to (-2,2) to (-1,3) to (1,3) to (2,2) to (1,1);
	\draw [thick, densely dashed, myred] (-4,4) to (-2,4) to (-1,3);
	\draw [thick, densely dashed, myred] (4,4) to (2,4) to (1,3);
	\draw [thick, densely dashed, myred] (-1,3) to (-.35,3.25) to (.35,3.25) to (1,3);
	\draw [thick, densely dashed, myred] (2,2) to [out=200, in=0] (0,1.5) to [out=180, in=340] (-2,2);
	\node at (0,.5) {$\graphE{9_3}$};
	\node at (0,2) {$\bulletgstart$};
	\node at (3,3) {$\bulletg$};
	\node at (-3,3) {$\bulletg$};
	\node at (0,4) {$\bulletg$};
	\node at (0,2.5) {$\bulletg$};
	\node at (1,1) {$\bulleto$};
	\node at (-2,2) {$\bulleto$};
	\node at (1,3) {$\bulleto$};
	\node at (-2,4) {$\bulleto$};
	\node at (4,4) {$\bulleto$};
	\node at (-.35,3.25) {$\bulleto$};
	\node at (2,2) {$\bulletp$};
	\node at (-1,1) {$\bulletp$};
	\node at (-1,3) {$\bulletp$};
	\node at (2,4) {$\bulletp$};
	\node at (-4,4) {$\bulletp$};
	\node at (.35,3.25) {$\bulletp$};
\end{tikzpicture}
\;\;,\;\;
\begin{tikzpicture}[anchorbase, xscale=.35, yscale=.5]
	\draw [thick, myyellow] (0,0) node[below, black] {$\graphE{9_4}$} to (2,2);
	\draw [thick, myyellow] (0,1) to (2,2) to (0,3) to (-3,5) to (0,4) to (2,2) to (0,6) to (-3,5);
	\draw [thick, myyellow] (0,3) to (2,7) to (0,4) to (2,7) to (0,6);
	\draw [thick, densely dotted, myblue] (0,0) to (-2,2);
	\draw [thick, densely dotted, myblue] (0,1) to (-2,2) to (0,3) to (3,5) to (0,4) to (-2,2) to (0,6) to (3,5);
	\draw [thick, densely dotted, myblue] (0,3) to (-2,7);
	\draw [thick, densely dotted, myblue] (0,4) to (-2,7);
	\draw [thick, densely dotted, myblue] (-2,7) to (0,6);
	\draw [thick, densely dashed, myred, double] (2,2) to (-2,2);
	\draw [thick, densely dashed, myred] (2,2) to (3,5) to (2,7) to (-2,7) to (-3,5) to (-2,2);
	\draw [thick, densely dashed, myred] (3,5) to (-3,5);
	\node at (0,0) {$\bulletgstart$};
	\node at (0,1) {$\bulletg$};
	\node at (0,3) {$\bulletg$};	
	\node at (0,4) {$\bulletg$};
	\node at (0,6) {$\bulletg$};	
	\node at (2,2) {$\bulleto$};
	\node at (-3,5) {$\bulleto$};
	\node at (2,7) {$\bulleto$};
	\node at (-2,2) {$\bulletp$};
	\node at (3,5) {$\bulletp$};
	\node at (-2,7) {$\bulletp$};
\end{tikzpicture}
\;\;,\;\;
\begin{tikzpicture}[anchorbase, xscale=.35, yscale=.5]
	\draw [thick, myyellow] (0,0) node[below, black] {$\graphE{21}$} to (1,1) to (0,2);
	\draw [thick, myyellow] (1,3) to (0,2) to (-2,2);
	\draw [thick, myyellow] (0,8) to (-1,7) to (0,6);
	\draw [thick, myyellow] (-1,5) to (0,6) to (2,6);
	\draw [thick, myyellow] (-1,5) to (2,4) to (1,3) to (-2,4) to (-1,5);
	\draw [thick, myyellow] (-5,3) to (-4,4) to (-1,5);
	\draw [thick, myyellow] (-2,2) to (-4,4);
	\draw [thick, myyellow] (5,5) to (4,4) to (1,3);
	\draw [thick, myyellow] (2,6) to (4,4);
	\draw [thick, myyellow] (-5,3) to (-2,4);
	\draw [thick, myyellow] (5,5) to (2,4);
	\draw [thick, densely dotted, myblue] (0,0) to (-1,1) to (0,2);
	\draw [thick, densely dotted, myblue] (-1,3) to (0,2) to (2,2);
	\draw [thick, densely dotted, myblue] (0,8) to (1,7) to (0,6);
	\draw [thick, densely dotted, myblue] (1,5) to (0,6) to (-2,6);
	\draw [thick, densely dotted, myblue] (1,5) to (-2,4) to (-1,3) to (2,4) to (1,5);
	\draw [thick, densely dotted, myblue] (5,3) to (4,4) to (1,5);
	\draw [thick, densely dotted, myblue] (2,2) to (4,4);
	\draw [thick, densely dotted, myblue] (-5,5) to (-4,4) to (-1,3);
	\draw [thick, densely dotted, myblue] (-2,6) to (-4,4);
	\draw [thick, densely dotted, myblue] (5,3) to (2,4);
	\draw [thick, densely dotted, myblue] (-5,5) to (-2,4);
	\draw [thick, densely dashed, myred] (-1,1) to (-2,2) to (-1,3) to (1,3) to (2,2) to (1,1) to (-1,1);
	\draw [thick, densely dashed, myred] (-1,7) to (-2,6) to (-1,5) to (1,5) to (2,6) to (1,7) to (-1,7);
	\draw [thick, densely dashed, myred] (-5,3) to (-5,5);
	\draw [thick, densely dashed, myred] (-1,3) to (-1,5);
	\draw [thick, densely dashed, myred] (1,3) to (1,5);
	\draw [thick, densely dashed, myred] (5,3) to (5,5);
	\draw [thick, densely dashed, myred] (-5,3) to (-1,3);
	\draw [thick, densely dashed, myred] (-5,5) to (-1,5);
	\draw [thick, densely dashed, myred] (5,3) to (1,3);
	\draw [thick, densely dashed, myred] (5,5) to (1,5);
	\node at (0,0) {$\bulletgstart$};
	\node at (0,2) {$\bulletg$};
	\node at (-2,4) {$\bulletg$};
	\node at (2,4) {$\bulletg$};
	\node at (-4,4) {$\bulletg$};
	\node at (4,4) {$\bulletg$};
	\node at (0,6) {$\bulletg$};
	\node at (0,8) {$\bulletg$};
	\node at (1,1) {$\bulleto$};
	\node at (-2,2) {$\bulleto$};
	\node at (1,3) {$\bulleto$};
	\node at (-1,5) {$\bulleto$};
	\node at (2,6) {$\bulleto$};
	\node at (-1,7) {$\bulleto$};
	\node at (5,5) {$\bulleto$};
	\node at (-5,3) {$\bulleto$};
	\node at (-1,1) {$\bulletp$};
	\node at (2,2) {$\bulletp$};
	\node at (-1,3) {$\bulletp$};
	\node at (1,5) {$\bulletp$};
	\node at (-2,6) {$\bulletp$};
	\node at (1,7) {$\bulletp$};
	\node at (-5,5) {$\bulletp$};
	\node at (5,3) {$\bulletp$};
\end{tikzpicture}
\]
\caption{The finite exceptional family of (generalized) type $\typeE$, indexed as indicated and 
denoted by $\graphE{e}$. (Note that there are four for $e=9$.)
}\label{fig:typeE}
\end{figure}
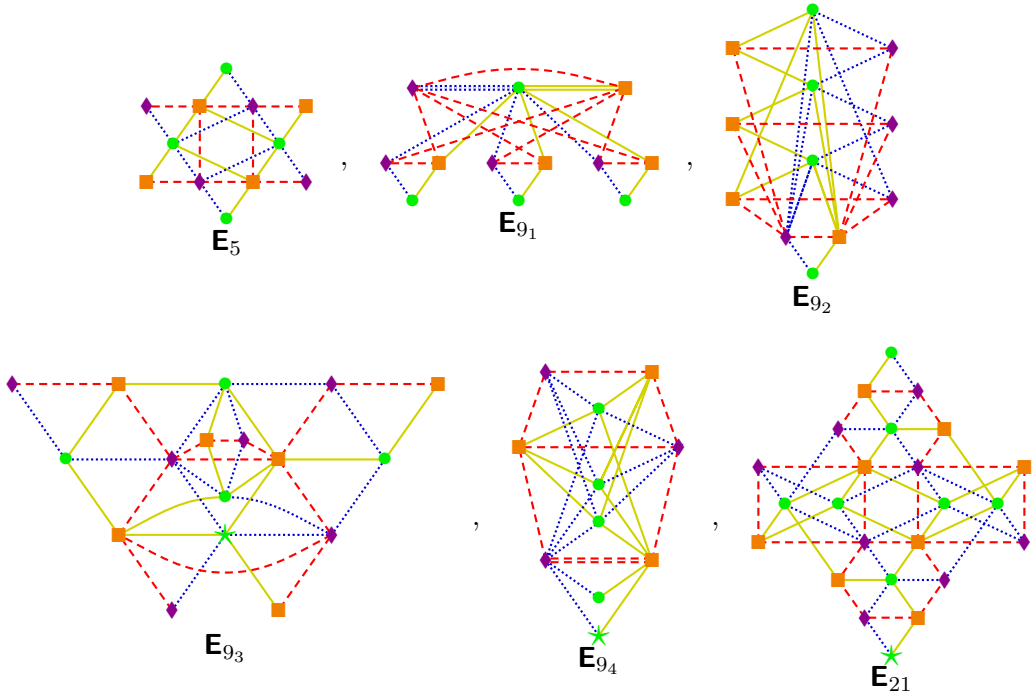

All the above graphs exist for color variations as well. Note further that
we have also indicated a starting vertex $\star$ in case such a choice is essential, 
i.e. in case different tricolorings give non-isomorphic tricolored graphs.

We point out that the above list was obtained from \cite[Section 4]{Oc} by excluding the ones 
that are not tricolorable.

\subsection{The spectra}\label{subsec:gen-D-spectra}

The spectra of the graphs from \fullref{subsec:gen-D-list} are 
known, cf. \cite[Section 2]{EP}. Let us sketch 
how they look like. To this end, recall vanishing set $\vanset{e}$ of level $e$ 
and the discoid $\discoid$ from \fullref{subsec:roots-poly}.
\newline

\noindent\textit{\setword{`\fullref{subsec:gen-D-spectra}.Claim$\typeA$'}{A-spectra}.}
$z\in S_{\graphA{e}^{\fu}}$ if and only if $(z,\overline{z})\in\vanset{e}$.
\newline

\noindent\textit{Proof(Sketch) of \ref{A-spectra}.} 
Observe that $\graphA{e}^{\fu}$ and $\graphA{e}^{\fud}$
are the graphs encoding the action of $[\fu\otimes\underline{\phantom{a}}\,]$, 
respectively of $[\fud\otimes\underline{\phantom{a}}\,]$, on $\GGc{\slqmod}$, 
and the claim follows.
\newline

\noindent\textit{\setword{`\fullref{subsec:gen-D-spectra}.Claim$\typeD\typeE$'}{ADE-spectra}.}
We have, without counting multiplicities of the zero eigenvalue, $S_{\Gg^{\fu}}\subset S_{\graphA{e}^{\fu}}$ 
for any $\Gg$ as in \fullref{subsec:gen-D-list}.
\newline

\noindent\textit{Proof(Sketch) of \ref{ADE-spectra}.} 
For graphs of type $\graphD{e}$ this holds by virtue of their construction, using the $\zeethree$-symmetry of the $\graphA{e}$ graphs. 
In fact, one can get the eigenvalues of $A(\graphD{e})$ from the ones of $A(\graphA{e})$ 
by deleting two out of every three eigenvalues and adding 
two additional eigenvalues $0$, e.g. for $e=3$:
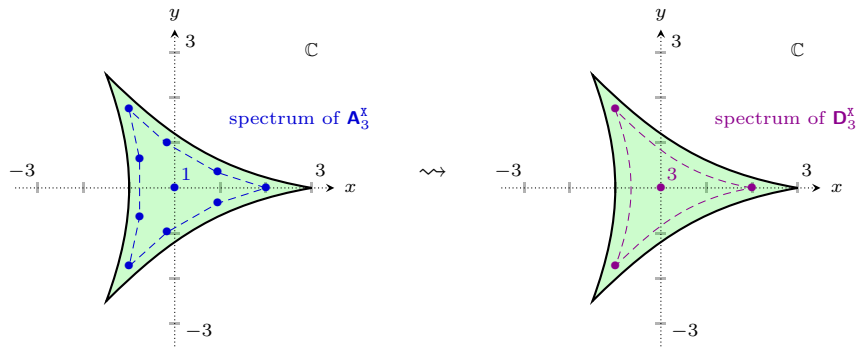
\[
\begin{tikzpicture}[anchorbase, scale=.6, tinynodes]
\draw[thin, marked=.0, marked=.166, marked=.333, marked=.666, marked=.833, marked=1.0, white] (0,-3) to (0,3);
\draw[thin, marked=.0, marked=.166, marked=.333, marked=.666, marked=.833, marked=1.0, white] (-3,0) to (3,0);
\draw[thick, white, fill=mygreen, opacity=.2] (3,0) to [out=170, in=315] (-1.5,2.5) to [out=290, in=70] (-1.5,-2.5) to [out=45, in=190] (3,0);
\draw[thin, densely dotted, ->, >=stealth] (-3.5,0) 
to (-3.35,0) node [above] {$-3$}
to (3.2,0) node [above] {$3$}
to (3.5,0) node[right] {$x$};
\draw[thin, densely dotted, ->, >=stealth] (0,-3.5) 
to (0,-3.2) node [right] {$-3$}
to (0,3.2) node [right] {$3$}
to (0,3.5) node[above] {$y$};
\draw[thick] (3,0) to [out=170, in=315] (-1.5,2.5) to [out=290, in=70] (-1.5,-2.5) to [out=45, in=190] (3,0);
\node at (3,3) {$\C$};
\node[myblue] at (0,0) {$\bullet$};
\node[myblue] at (.25,.25) {$1$};
\node[myblue] at (2,0) {$\bullet$};
\node[myblue] at (-1,1.73) {$\bullet$};
\node[myblue] at (-1,-1.73) {$\bullet$};
\node[myblue] at (-.77,.64) {$\bullet$};
\node[myblue] at (-.77,-.64) {$\bullet$};
\node[myblue] at (-.17,.98) {$\bullet$};
\node[myblue] at (-.17,-.98) {$\bullet$};
\node[myblue] at (.94,.34) {$\bullet$};
\node[myblue] at (.94,-.34) {$\bullet$};
\draw[very thin, densely dashed, myblue] (2,0) to (.94,.34) to (-.17,.98) to (-1,1.73) to (-.77,.64) to (-.77,-.64) to (-1,-1.73) to (-.17,-.98) to (.94,-.34) to (2,0);
\node[myblue] at (2.75,1.5) {spectrum of $\graphA{3}^{\fu}$};
\end{tikzpicture}
\quad
\rightsquigarrow
\quad
\begin{tikzpicture}[anchorbase, scale=.6, tinynodes]
\draw[thin, marked=.0, marked=.166, marked=.333, marked=.666, marked=.833, marked=1.0, white] (0,-3) to (0,3);
\draw[thin, marked=.0, marked=.166, marked=.333, marked=.666, marked=.833, marked=1.0, white] (-3,0) to (3,0);
\draw[thick, white, fill=mygreen, opacity=.2] (3,0) to [out=170, in=315] (-1.5,2.5) to [out=290, in=70] (-1.5,-2.5) to [out=45, in=190] (3,0);
\draw[thin, densely dotted, ->, >=stealth] (-3.5,0) 
to (-3.35,0) node [above] {$-3$}
to (3.2,0) node [above] {$3$}
to (3.5,0) node[right] {$x$};
\draw[thin, densely dotted, ->, >=stealth] (0,-3.5) 
to (0,-3.2) node [right] {$-3$}
to (0,3.2) node [right] {$3$}
to (0,3.5) node[above] {$y$};
\draw[thick] (3,0) to [out=170, in=315] (-1.5,2.5) to [out=290, in=70] (-1.5,-2.5) to [out=45, in=190] (3,0);
\node at (3,3) {$\C$};
\node[mypurple] at (0,0) {$\bullet$};
\node[mypurple] at (.25,.25) {$3$};
\node[mypurple] at (2,0) {$\bullet$};
\node[mypurple] at (-1,1.73) {$\bullet$};
\node[mypurple] at (-1,-1.73) {$\bullet$};
\draw[very thin, densely dashed, mypurple] (2,0) to [out=170, in=315] (-1,1.73) to [out=290, in=70] (-1,-1.73) to [out=45, in=190] (2,0);
\node[mypurple] at (2.75,1.5) {spectrum of $\graphD{3}^{\fu}$};
\end{tikzpicture}
\]
The case of $\graphC{e}$ can be shown similarly (precisely which eigenvalues of $S_{\graphA{e}}$ also 
belong to $S_{\graphC{e}}$ depends on $e\bmod 3$), with a prototypical 
example given by:
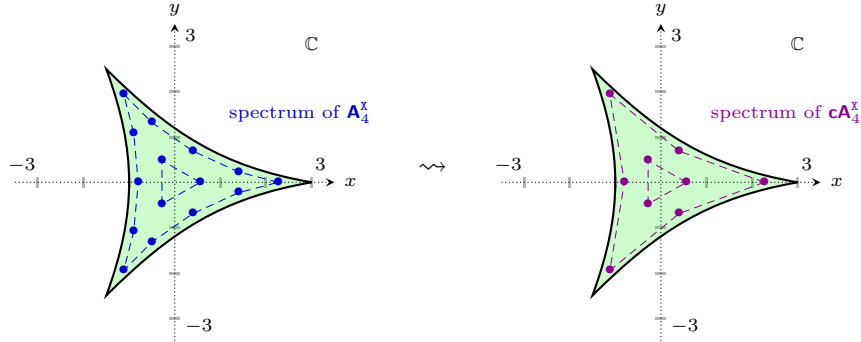
\[
\begin{tikzpicture}[anchorbase, scale=.6, tinynodes]
\draw[thin, marked=.0, marked=.166, marked=.333, marked=.666, marked=.833, marked=1.0, white] (0,-3) to (0,3);
\draw[thin, marked=.0, marked=.166, marked=.333, marked=.666, marked=.833, marked=1.0, white] (-3,0) to (3,0);
\draw[thick, white, fill=mygreen, opacity=.2] (3,0) to [out=170, in=315] (-1.5,2.5) to [out=290, in=70] (-1.5,-2.5) to [out=45, in=190] (3,0);
\draw[thin, densely dotted, ->, >=stealth] (-3.5,0) 
to (-3.35,0) node [above] {$-3$}
to (3.2,0) node [above] {$3$}
to (3.5,0) node[right] {$x$};
\draw[thin, densely dotted, ->, >=stealth] (0,-3.5) 
to (0,-3.2) node [right] {$-3$}
to (0,3.2) node [right] {$3$}
to (0,3.5) node[above] {$y$};
\draw[thick] (3,0) to [out=170, in=315] (-1.5,2.5) to [out=290, in=70] (-1.5,-2.5) to [out=45, in=190] (3,0);
\node at (3,3) {$\C$};
\node[myblue] at (-.80,0) {$\bullet$};
\node[myblue] at (.56,0) {$\bullet$};
\node[myblue] at (2.27,0) {$\bullet$};
\node[myblue] at (-1.12,1.95) {$\bullet$};
\node[myblue] at (-1.12,-1.95) {$\bullet$};
\node[myblue] at (-.28,.48) {$\bullet$};
\node[myblue] at (-.28,-.48) {$\bullet$};
\node[myblue] at (.40,.69) {$\bullet$};
\node[myblue] at (.40,-.69) {$\bullet$};
\node[myblue] at (-.50,1.32) {$\bullet$};
\node[myblue] at (-.50,-1.32) {$\bullet$};
\node[myblue] at (-.90,1.09) {$\bullet$};
\node[myblue] at (-.90,-1.09) {$\bullet$};
\node[myblue] at (1.40,.23) {$\bullet$};
\node[myblue] at (1.40,-.23) {$\bullet$};
\draw[very thin, densely dashed, myblue] (.56,0) to (-.28,.48) to (-.28,-.48) to (.56,0);
\draw[very thin, densely dashed, myblue] (2.27,0) to (1.40,.23) to (.40,.69) to (-.50,1.32) to (-1.12,1.95) to (-.90,1.09) to (-.80,0) to (-.90,-1.09) to (-1.12,-1.95) to (-.50,-1.32) to (.40,-.69) to (1.40,-.23) to (2.27,0);
\node[myblue] at (2.75,1.5) {spectrum of $\graphA{4}^{\fu}$};
\end{tikzpicture}
\quad
\rightsquigarrow
\quad
\begin{tikzpicture}[anchorbase, scale=.6, tinynodes]
\draw[thin, marked=.0, marked=.166, marked=.333, marked=.666, marked=.833, marked=1.0, white] (0,-3) to (0,3);
\draw[thin, marked=.0, marked=.166, marked=.333, marked=.666, marked=.833, marked=1.0, white] (-3,0) to (3,0);
\draw[thick, white, fill=mygreen, opacity=.2] (3,0) to [out=170, in=315] (-1.5,2.5) to [out=290, in=70] (-1.5,-2.5) to [out=45, in=190] (3,0);
\draw[thin, densely dotted, ->, >=stealth] (-3.5,0) 
to (-3.35,0) node [above] {$-3$}
to (3.2,0) node [above] {$3$}
to (3.5,0) node[right] {$x$};
\draw[thin, densely dotted, ->, >=stealth] (0,-3.5) 
to (0,-3.2) node [right] {$-3$}
to (0,3.2) node [right] {$3$}
to (0,3.5) node[above] {$y$};
\draw[thick] (3,0) to [out=170, in=315] (-1.5,2.5) to [out=290, in=70] (-1.5,-2.5) to [out=45, in=190] (3,0);
\node at (3,3) {$\C$};
\node[mypurple] at (-.80,0) {$\bullet$};
\node[mypurple] at (.56,0) {$\bullet$};
\node[mypurple] at (2.27,0) {$\bullet$};
\node[mypurple] at (-1.12,1.95) {$\bullet$};
\node[mypurple] at (-1.12,-1.95) {$\bullet$};
\node[mypurple] at (-.28,.48) {$\bullet$};
\node[mypurple] at (-.28,-.48) {$\bullet$};
\node[mypurple] at (.40,.69) {$\bullet$};
\node[mypurple] at (.40,-.69) {$\bullet$};
\draw[very thin, densely dashed, mypurple] (.56,0) to (-.28,.48) to (-.28,-.48) to (.56,0);
\draw[very thin, densely dashed, mypurple] (2.27,0) to (.40,.69) to (-1.12,1.95) to (-.80,0) to (-1.12,-1.95) to (.40,-.69) to (2.27,0);
\node[mypurple] at (2.75,1.5) {spectrum of $\graphC{4}^{\fu}$};
\end{tikzpicture}
\]
For the exceptional type $\typeE$ graphs the claim can be checked case-by-case.
\newline

In particular, the spectra of the generalized $\ADE$ Dynkin diagrams 
are all inside $\discoid$.

\begin{example}\label{example:triangle3}
The spectra of $\graphA{1}^{\fu},\graphA{2}^{\fu}$ and $\graphA{3}^{\fu}$ 
are given in \fullref{example:triangle1}. (Again, these should be 
compared to \fullref{example:plot-zeros}.)
Additionally we have
\begin{gather*}
S_{\graphD{3}^{\fu}}
=S_{\graphC{3}^{\fu}}
=
\left\{
\text{roots of }
X^3
(X-2)
(X^2 + 2X + 4)
\right\}
\end{gather*}
Forgetting the multiplicity of zero, we get the inclusion of the corresponding spectra.
\end{example}

\begin{example}\label{example:sce}
The graphs $\graphD{3}^{\fu}$ and $\graphC{3}^{\fu}$ have the same 
spectrum, cf. \fullref{example:triangle3}. 
However, they are not isomorphic as graphs, e.g. $\graphD{3}$ has a double-edge and $\graphC{3}^{\fu}$ does not. 
Both observations are true in general for $\graphD{e}$ and $\graphC{e}$.
\end{example}

\subsection{Zuber's classification problem and \texorpdfstring{\fullref{problem:classification}}{CP}}\label{subsec:gen-D-zuber}

Zuber (based on joint work with Di Francesco \cite{DFZ} and 
Petkova \cite{PZ}) introduced the notion of a generalized $\ADE$
Dynkin diagram.
These graphs appear in various 
disguises in the literature, e.g. 
in conformal field
theories, integrable lattice models, topological field theories 
for $3$-manifold invariants and subfactor theory.

Zuber wrote down a list of six axioms 
which these graphs should satisfy, see \cite[Section 1.2]{Zu}, 
and asked for the classification of such graphs. 
In \cite{Oc}, Ocneanu argued that Zuber's classification problem is related to the classification problem 
of the so-called quantum subgroups of $\mathrm{SU}(N)$. He also proposed a list of graphs which should solve 
Zuber's classification problem. The ones that are tricolorable are the graphs 
that we reproduced in \fullref{subsec:gen-D-list}.
However, we already saw in \fullref{theorem:low-level-classification} 
that we get solutions which are not on Ocneanu's list, 
so we do not know whether \fullref{problem:classification} 
and Zuber's classification problem are the same or not, or even how they are related.
\addtocontents{toc}{\protect\setcounter{tocdepth}{2}}
%
\bibliographystyle{alphaurl}
\bibliography{trihedral-alg}
\end{document}